\documentclass[11pt,b5paper,notitlepage]{article}
\usepackage[b5paper, margin={0.5in,0.75in}]{geometry}
\usepackage{amsmath,amscd,amssymb,amsthm,mathrsfs,amsfonts,layout,indentfirst,graphicx,caption,mathabx, stmaryrd,appendix,calc,imakeidx,upgreek} 
\usepackage[dvipsnames]{xcolor}
\usepackage{palatino}  
\usepackage{slashed} 
\usepackage{mathrsfs} 
\usepackage{extarrows} 
\usepackage{enumitem} 

\usepackage{fancyhdr} 

\usepackage{pdfrender}

\newcommand*{\hollowcolon}{%
	\textpdfrender{
		TextRenderingMode=Stroke,
		LineWidth=.1bp,
	}{:}%
}

\newcommand{\hcolondel}[1]{%
	\mathopen{\hollowcolon}#1\mathclose{\hollowcolon}%
}

\usepackage{tikz-cd}
\usepackage[nottoc]{tocbibind}   

\makeindex

\usepackage{lipsum}
\let\OLDthebibliography\thebibliography
\renewcommand\thebibliography[1]{
	\OLDthebibliography{#1}
	\setlength{\parskip}{0pt}
	\setlength{\itemsep}{2pt} 
}

\allowdisplaybreaks  
\usepackage{latexsym}
\usepackage{chngcntr}
\usepackage[colorlinks,linkcolor=blue,anchorcolor=blue, linktocpage,
]{hyperref}
\hypersetup{ urlcolor=cyan,
	citecolor=[rgb]{0,0.5,0}}

\counterwithin{figure}{subsection}

\theoremstyle{definition}
\newtheorem{df}{Definition}[section]
\newtheorem{eg}[df]{Example}
\newtheorem{exe}[df]{Exercise}
\newtheorem{rem}[df]{Remark}
\newtheorem{ass}[df]{Assumption}
\newtheorem{cv}[df]{Convention}
\newtheorem{prin}[df]{Principle}

\theoremstyle{plain}
\newtheorem{thm}[df]{Theorem}
\newtheorem{ccl}[df]{Conclusion}

\newtheorem{pp}[df]{Proposition}
\newtheorem{co}[df]{Corollary}
\newtheorem{lm}[df]{Lemma}

\usepackage{calligra}
\DeclareMathOperator{\shom}{\mathscr{H}\text{\kern -3pt {\calligra\large om}}\,}

\newcommand{\fk}{\mathfrak}
\newcommand{\mc}{\mathcal}
\newcommand{\wtd}{\widetilde}
\newcommand{\wht}{\widehat}
\newcommand{\wch}{\widecheck}
\newcommand{\ovl}{\overline}

\newcommand{\tr}{\mathrm{t}} 

\newcommand{\End}{\mathrm{End}} 
\newcommand{\id}{\mathbf{1}}
\newcommand{\Hom}{\mathrm{Hom}}
\newcommand{\Conf}{\mathrm{Conf}}
\newcommand{\Res}{\mathrm{Res}}

\newcommand{\Rep}{\mathrm{Rep}}

\newcommand{\Diffp}{\mathrm{Diff}^+}

\newcommand{\Vir}{\mathrm{Vir}}

\newcommand{\Span}{\mathrm{Span}}

\newcommand{\bk}[1]{\langle {#1}\rangle}
\newcommand{\prth}[1]{( {#1})}
\newcommand{\bigbk}[1]{\big\langle {#1}\big\rangle}
\newcommand{\Bigbk}[1]{\Big\langle {#1}\Big\rangle}

\newcommand{\Vect}{\mathrm{Vec}}

\newcommand{\scr}{\mathscr}

\newcommand{\gk}{\mathfrak g}
\newcommand{\hk}{\mathfrak h}
\newcommand{\xk}{\mathfrak x}
\newcommand{\yk}{\mathfrak y}
\newcommand{\zk}{\mathfrak z}

\newcommand{\im}{\mathbf{i}}
\newcommand{\Co}{\complement}

\newcommand{\sgm}{\varsigma}
\newcommand{\SX}{{S_{\fk X}}}

\newcommand{\mbb}{\mathbb}
\newcommand{\mbf}{\mathbf}

\newcommand{\blt}{\bullet}
\newcommand{\coker}{\mathrm{coker}}
\newcommand{\Vbb}{\mathbb V}
\newcommand{\Ubb}{\mathbb U}
\newcommand{\Xbb}{\mathbb X}

\newcommand{\Abb}{\mathbb A}
\newcommand{\Wbb}{\mathbb W}
\newcommand{\Mbb}{\mathbb M}
\newcommand{\Gbb}{\mathbb G}
\newcommand{\Cbb}{\mathbb C}
\newcommand{\Nbb}{\mathbb N}
\newcommand{\Zbb}{\mathbb Z}
\newcommand{\Pbb}{\mathbb P}
\newcommand{\Rbb}{\mathbb R}

\newcommand{\Dbb}{\mathbb D}
\newcommand{\Hbb}{\mathbb H}
\newcommand{\cbf}{\mathbf c}
\newcommand{\Rbf}{\mathbf R}
\newcommand{\wt}{\mathrm{wt}}

\newcommand{\Ker}{\mathrm{Ker}}
\newcommand{\Coker}{\mathrm{Coker}}

\newcommand{\Rng}{\mathrm{Rng}}

\newcommand{\SXb}{{S_{\fk X_b}}}

\newcommand{\vbf}{\mathbf v}
\newcommand{\ubf}{\mathbf u}
\newcommand{\wbf}{\mathbf w}

\newcommand{\Sbb}{{\mathbb S}}

\newcommand{\fin}{\mathrm{fin}}
\newcommand{\Ann}{\mathbf{Ann}}
\newcommand{\Real}{\mathrm{Re}}
\newcommand{\Imag}{\mathrm{Im}}
\newcommand{\cl}{\mathrm{cl}}
\newcommand{\Ind}{\mathrm{Ind}}

\usepackage{tipa} 

\usepackage{tipx}

\numberwithin{equation}{section}

\title{Lectures on Vertex Operator Algebras and Conformal Blocks}
\author{{\sc Bin Gui}
}
\date{Last update: February 2023}
\begin{document}\sloppy 
	\pagenumbering{arabic}
	\setcounter{page}{0}
	\setcounter{section}{-1}


	\maketitle
\thispagestyle{empty}	
\begin{gather*}
{~}\\[3ex]
	\vcenter{\hbox{{
				\includegraphics[height=3.3cm]{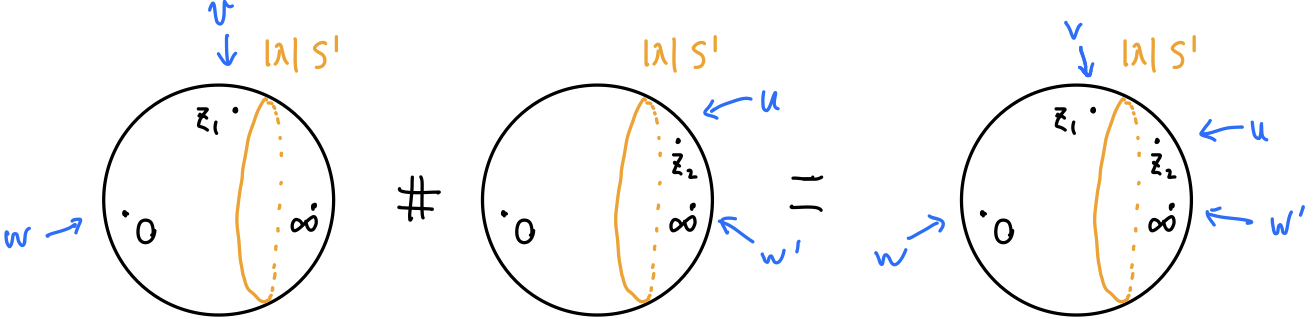}}}}\\[2ex]
	\vcenter{\hbox{{
				\includegraphics[height=3.9cm]{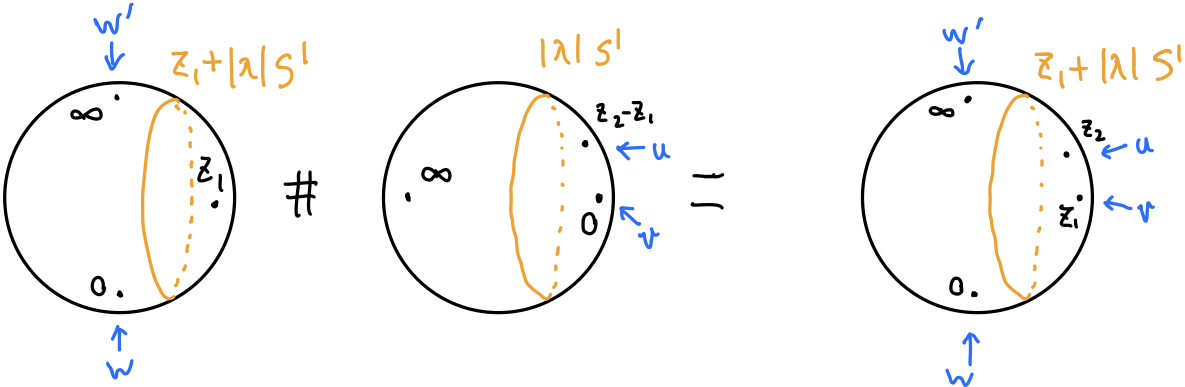}}}}\\[2ex]
	\vcenter{\hbox{{
				\includegraphics[height=3cm]{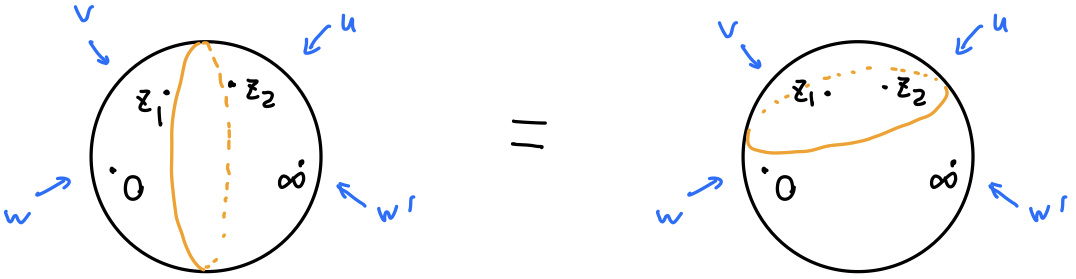}}}}
\end{gather*}

	\newpage
\tableofcontents
	


\newpage

\section*{Preface}

\subsection*{1}

This monograph is the lecture notes of a course I gave in 2022 spring at Tsinghua University, Yau Mathematical Sciences Center. It is an introduction to the basic theory of vertex operator algebras (VOAs) and conformal blocks. The audience of that course is assumed to be familiar with complex analysis, differential manifolds, and the relationship between (the representations of) Lie groups and Lie algebras. 

A key feature of this monograph is the emphasis on the \emph{complex analytic} aspects of VOAs and conformal blocks: We prove many well-known results by first proving the convergence (more precisely: absolute and locally uniform (a.l.u.) convergence) of  correlation functions which are defined a priori as formal power/Laurent series of some formal variables. These results include Dong's Lemma and Goddard uniqueness (and hence the reconstruction theorem), construction of contragredient modules, local freeness of sheaves of conformal blocks for $C_2$-cofinite VOAs (and families of compact Riemann surfaces). 

The algebraic construction of these correlation functions as formal series corresponds geometrically to deformations of pointed compact Riemann surfaces or pointed nodal curves. The usual algebraic approaches to VOAs (e.g. \cite{Kac,LL}) avoid showing the convergence of such series. In geometry, this means not considering analytic sewing, but only formal and infinitesimal sewing. As a compensation, formal calculus and delta functions are heavily used in these approaches. An advantage of our complex analytic approach is that by showing that these formal series are convergent, one can view the correlation functions as genuine functions but not just formal series, so one can understand the nature of VOAs and conformal blocks in a similar way as physicists do.

\subsection*{2}

Another feature of this monograph is that we give motivations for many definitions and results from the perspective of Segal's picture of conformal field theory. These include: the definitions of VOA and in particular Jacobi identity; the definitions of conformal blocks and sheaves of VOAs; translation covariance, scale covariance, and more generally Huang's change of coordinate theorem for vertex operators; the formula for the vertex operators of contragredient modules; the definition of connections for sheaves of conformal blocks. These motivational explanations are known to experts, but are not easily accessible in existing textbooks and articles. We hope that by incorporating these motivations into a monograph, we can make it easier for beginners to get started on these topics.

\subsection*{3}

The theory of conformal blocks is not only very beautiful, but also crucial to a geometric understanding of the representation theory of VOAs and conformal field theory. In recent years, there has been an increasing interplay between different approaches to conformal field theory. These approaches include VOAs, conformal nets and operator algebras, tensor categories, low-dimensional topology, probability, etc. We believe that conformal blocks are a key to understanding the relationships between these approaches. Unfortunately, most of the literature on conformal blocks is written in the language of algebraic geometry, which makes it more difficult for people from different fields to understand this subject. In our monograph, many central ideas about conformal blocks are explained in the language of (complex) differential manifolds and basic sheaf theory, so that they can (hopefully) be understood by a much wider audience. Of course, such elementary language is not sufficient for proving profound results. So we leave the technical proofs to my monograph \cite{Gui}. Indeed, the present monograph can also be viewed as an introduction to \cite{Gui}.

\subsection*{4}

There is some confusion in the proof of local freeness of sheaves of coinvariants and conformal blocks. This result has two versions: the algebraic one for algebraic families of smooth curves, and the analytic one for complex analytic families of compact Riemann surfaces. However, we believe that the following point is not sufficiently recognized in the literature: the proof for the algebraic version is not directly applicable to the analytic version.

The algebraic local freeness is proved in the following steps: 1. Prove that the sheaf of coinvariants over a base scheme $\mc B$ is a coherent $\scr O_{\mc B}$-module. 2. Prove that the sheaf of coinvariants admits a connection. 3. By a standard result, coherence and connections imply local freeness. When adapting this proof to the analytic setting, the difficulties arise in step 1: one can only show that the sheaf of coinvariants for a complex analytic family (over base manifold $\mc B$) is a finite-type $\scr O_{\mc B}$-module, but not that it is coherent (i.e., that moreover the sheaves of relations are of finite-type). The proof of algebraic coherence relies essentially on the \emph{Noetherian property} of the structure sheaf $\scr O_{\mc B}$ (as well as the quasi-coherence of the sheaves of coinvariants), which does not hold in the complex-analytic setting. 

In \cite{Gui}, we have given an analytic proof, first proving a finiteness result slightly stronger than the finite-type condition, and then using some sheaf-theoretic arguments. In this monograph, a different proof was given (cf. Sec. \ref{lb153}). Though this proof  relies on the same finiteness result, the subsequent argument is more analytic and has a clearer physical meaning: the crucial step is to prove the convergence of a formal series $\upphi_\tau$ of $\tau$ using differential equations, where $\upphi_\tau$ is constructed from a given conformal block $\upphi_0$ of a fixed fiber $\fk X_0$ of the analytic family $\fk X$; $\upphi_\tau$ is expected to be a conformal block for the nearby fiber $\fk X_\tau$ if the convergence were proved. Therefore, this proof is very similar to that of convergence of sewing conformal blocks in \cite{Gui} or \cite{Gui20}. Note that $\fk X_\tau$ is the deformation of a compact Riemann surface $\fk X_0$ (with marked points), while sewing is the deformation of a nodal curve. Thus, by presenting such a proof, we want to convey the idea that both types of deformations can be treated in the same way in the (complex analytic) theory of conformal blocks.

\subsection*{Acknowledgment}

The title of this monograph is a little boring. I have considered using a more interesting one such as \emph{Vertex Operator Algebras and Conformal Blocks for analysts/by an analyst}. I didn't do so because I don't want the monograph to look like it is only for analysts.

I have to confess that my preference for the analytic methods is due first and foremost to my own taste and the fact that I am (or I consider myself) a mathematical analyst. But it is also due in large part to the influence of my postdoctoral supervisor, Yi-Zhi Huang. I have benefited greatly from reading many of his works and from many discussions with him. I would like to express my deep gratitude to him. I also want to thank Hao Zhang for reading through this monograph and pointing out many typos and errors.\\[2ex]

\hfill May 2022

\newpage
	
Subsections marked with $\star$ can be skipped on first reading.

\section{Notations}

\begin{itemize}
\item $\Nbb=\{0,1,2,\dots\}$, $\Zbb_+=\{1,2,\dots\}$. 
\item $\im=\sqrt{-1}$, $\Sbb^1=$unit circle, $\Cbb^\times=\Cbb\setminus\{0\}$.
\item $\mbb D_r=\{z\in\Cbb:|z|<r\}$, $\mbb D_r^\times=\{z\in\Cbb:0<|z|<r\}$, $\Dbb_r^\cl=\{z\in\Cbb:|z|\leq r\}$
\item $\scr O(X)$ (resp. $\scr O_X$) is the space (resp. sheaf) of holomorphic functions on a complex manifold $X$. $\scr O_{X,x}$ is the stalk of $\scr O_X$ at $x$.
\item Configuration space $\Conf^n(X)=\{(x_1,\dots,x_n)\in X^n:x_i\neq x_j\text{ if }i\neq j\}$.
\item $z$ and $\zeta$ could mean either points, or the standard coordinate of $\Cbb$, or formal variables. We will give their meanings when the context is unclear. 
\item All vector spaces are over $\Cbb$, unless otherwise stated. If $W$ is a vector space equipped with a Hermitian form $\bk{\cdot|\cdot}$, we let $|\cdot\rangle$ be the linear variable and $\langle\cdot|$ be the antilinear (i.e. conjugate linear) one, following physicists' convention. 
\item If $W,W'$ are vector spaces, then $\Hom(W,W')$ denote the space of linear operators from $W$ to $W'$. We let $\End(W)=\Hom(W,W)$. 
\item We use symbols $\bk{\cdot,\cdot}$ or $(\cdot,\cdot)$ to denote bilinear forms (i.e., linear on both variables).
\item Given a vector space $W$ and a formal variable $z$,
\begin{subequations}
\begin{gather*}
W[z]=\{\text{polynomials of $z$ whose coefficients are elements of $W$}\}\\
W[[z]]=\Big\{\sum_{n\in\Nbb}w_nz^n:w_n\in W\Big\}\\
W((z))=\Big\{\sum_{n\in\Zbb}w_nz^n:w_n\in W,\text{ and }w_n=0\text{ when $n$ is sufficiently negative}\Big\}\\
W[[z^{\pm1}]]=\Big\{\sum_{n\in\Zbb}w_nz^n:w_n\in W\Big\}.
\end{gather*}
\end{subequations}
Each line is a subspace of the subsequent line. In case there are several formal variables, the spaces are defined in a similar way, expect $W((\cdots))$. For instance,
\begin{gather*}
W[[z,\zeta^{\pm1}]]:=W[[z]][[\zeta^{\pm1}]]=W[[\zeta^{\pm1}]][[z]]
\end{gather*}
consists of $\sum_{m\in\Nbb,n\in\Zbb}w_{m,n}z^m\zeta^n$ where each $w_{m,n}\in W$. However, note that $W((z))((\zeta))$ and $W((\zeta))((z))$ are not equal. (For instance, $\sum_{m\geq -n}\sum_{n\geq -1}z^m\zeta^n$ belongs to $\Cbb((z))((\zeta))$ but not $\Cbb((\zeta))((z))$.)

Elements in $W[[z^{\pm1}]]$ are called \textbf{formal Laurent series} of $z$.

\item We let
\begin{align*}
W((z_1,\dots,z_N))=\Big\{\sum_{n_1,\dots,n_N\geq L}w_{n_1,\dots,n_N}z_1^{n_1}\cdots z_N^{n_N}\text{ for some }L\in\Zbb\Big\}.
\end{align*}
Then $W((z_1,z_2))$ is a proper subspace of both $W((z_1))((z_2))$ and $W((z_2))((z_1))$.
\item We set
\begin{align}\label{eq18}
	\Res_{z=0}~\sum_{n\in\Zbb}w_nz^ndz=w_{-1}.
\end{align}
This is in line with the complex analytic residue.

\item A vector of $W_1\otimes\cdots\otimes W_N$ writen as $w_\blt$ means that it is of the form $w_1\otimes\cdots\otimes w_N$ where each $w_i\in W_i$. Depending on the context, $w_\blt$ will also mean a tuple $(w_1,\dots,w_N)$. Similarly, $W_\blt$ may mean $W_1\otimes\cdots\otimes W_N$ or $(W_1,\dots,W_N)$ depending on the context.

\item Unless otherwise stated, by a manifold, we mean one \emph{without} boundaries. Also, "with boundaries" means "possibly with boundaries".
\end{itemize}

\newpage

\section{Segal's picture of 2d CFT; motivations of VOAs and conformal blocks}

\subsection{}

Vertex operator algebras (VOAs) are mathematical objects defined to understand and construct 2-dimensional conformal field theory (CFT for short). A CFT describes propagations and interactions of strings. The are two types of strings: closed strings $\simeq\Sbb^1$ and open strings $\simeq[0,1]$. In this course, we will mainly focus on closed strings.

Let me explain how mathematicians understand CFT. Just like any quantum field theory (QFT), in CFT we must have a Hilbert space $\mc H$. The vectors in $\mc H$ are called ``states", but unlike ordinary QFT, a vector $\xi\in\mc H$ is not a state of a particle, but a state of a closed string $\Sbb^1$.

The most important and non-trivial part in CFT is to define/understand string interactions. According to Segal's picture \cite{Seg88}, an interaction is uniquely determined by a compact Riemann surface $\Sigma$ with boundaries $\partial\Sigma$, where $\partial\Sigma$ is a disjoint union of some circles (strings). Each string is called either an incoming string or an outcomming one. Suppose $\partial\Sigma$ has $N$ incoming strings and $M$ outgoing ones, then this picture describes an interaction where $N$ strings are going inside, and $M$ strings are going outside. 

Moreover, the boundary $\partial\Sigma$ must be \textbf{parametrized}. This means that to each connected component $\partial\Sigma_i$ a diffeomorphism $\eta_i:\partial\Sigma_i\xrightarrow{\simeq}\Sbb^1$ is associated. The orientation on $\partial\Sigma_i$ defined by pulling back the one of $\Sbb^1$ along $\eta_i$ is assumed to be the opposite of the one defined in Stokes' theorem, shown as follows
\begin{align}
	\vcenter{\hbox{{
				\includegraphics[width=0.35\linewidth]{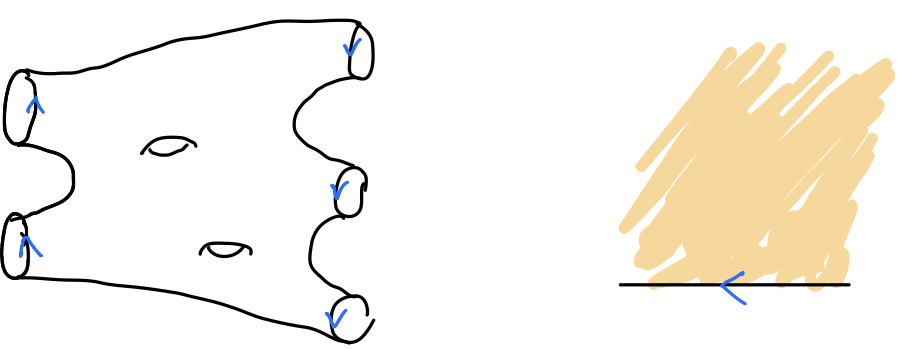}}}}\label{fig1}
\end{align}

\subsection{}\label{lb17}

Unless otherwise stated, we assume that the boundary parametrization is also \textbf{analytic}. Roughly speaking, this means that $\Sigma$ can be obtained by removing some open  disks from a compact Riemann surface $C$ (without boundary) such that the parametrizations of $\partial\Sigma$ are given by local holomorphic functions of $C$. 

Here is a more rigorous explanation. By a \textbf{local coordinate} $\eta$ of $C$ at $x\in C$, we mean $\eta$ is a holomorphic injective function  on a neighborhood $U$ of $x$ such that $\eta(x)=0$. So $\eta$ is a biholomorphism between $U$ and a neighborhood $\eta(U)$ of $0$. Now, suppose we have local coordinates $\eta_1,\dots,\eta_N$ at distinct points $x_1,\dots,x_N\in C$. The data
\begin{align}
\fk X:=(C;x_\blt;\eta_\blt)=(C;x_1,\dots,x_N;\eta_1,\dots,\eta_N)	
\end{align}
is called an \textbf{$N$-pointed compact Riemann surface with local coordinates}. 

Let each $\eta_i$ be defined on a neighborhood $U_i\ni x_i$. We assume moreover the following
\begin{ass}\label{lb11}
$U_i\cap U_j=\emptyset$ if $i\neq j$ (indeed, $\eta_i^{-1}(\Dbb_1^\cl)\cap \eta_j^{-1}(\Dbb_1^\cl)=\emptyset$ is sufficient), and $\eta_i(U_i)\supset\Dbb_1^\cl$ for each $i$. Here $\Dbb_1^\cl$ is the closed unit  disk.
\end{ass} 
By removing all $\eta_i^{-1}(\Dbb_1)$, we get $\Sigma$ with boundary strings $\eta_i^{-1}(\partial\Dbb_1^\cl)=\eta_i^{-1}(\Sbb^1)$ whose parametrizations are $\eta_i$.
\begin{align*}
	\vcenter{\hbox{{
				\includegraphics[height=2.3cm]{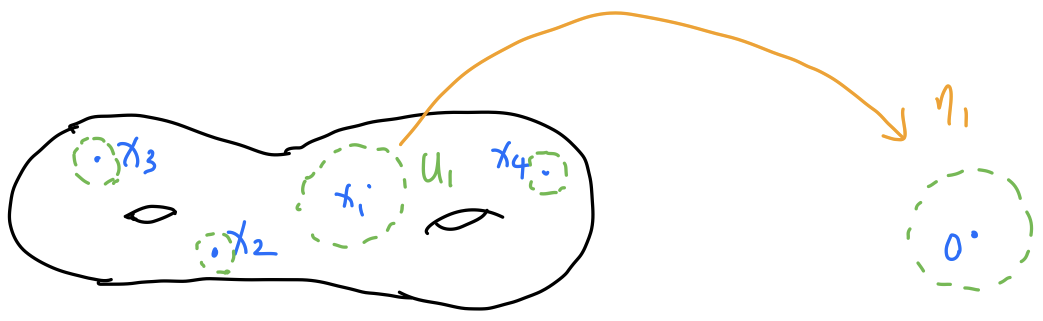}}}}
\end{align*}

\subsection{}

Any $\Sigma$ as above determines uniquely an interaction of strings. Suppose it has $N$ incoming strings and $M$ outgoing ones. Then mathematically, such an interaction is described by a bounded linear map $T=T_\Sigma:\mc H^{\otimes N}\rightarrow\mc H^{\otimes M}$.\index{T@$T_{\Sigma},T_\fk X$: The interaction map/correlation function} (The boundedness is automatic thanks to the uniform boundedness principle. But this is not an important point in this course.) Given $\xi_\blt=\xi_1\otimes\cdots\otimes \xi_N\in\mc H^{\otimes N}$ and $\eta_\blt=\eta_1\otimes\cdots\otimes\eta_M\in\mc H^{\otimes M}$, the value
\begin{align}
\bk{\eta_\blt|T\xi_\blt}	
\end{align}
describes the probability amplitude that the $N$ incoming closed strings with states $\xi_1,\dots,\xi_N$ become $\eta_1,\dots,\eta_M$ after interaction.

The word ``conformal" in conformal field theory reflects the fact that $T$ depends only on the complex structure of $\Sigma$ and its parametrization, but not on the metric for instance. Thus, a CFT is more rigid than a topological quantum field theory (TQFT): in the latter case, $T$ depends only on the topological structures of the manifolds.

\subsection{}\label{lb4}
Suppose we have another interaction $S:\mc H^{\otimes M}\rightarrow\mc H^{\otimes L}$ corresponding to the parametrized surface $\Sigma'$, then the composition of them $S\circ T:\mc H^{\otimes N}\rightarrow\mc H^{\otimes L}$ corresponds to the \textbf{sewing} $\Sigma\#\Sigma'$ of $\Sigma$ and $\Sigma'$, where the $j$-th outgoing string $\partial_+\Sigma_j$ of $\Sigma$ is sewn with the $j$-th incoming one $\partial_-\Sigma'_j$ of $\Sigma'$.
\begin{align*}
	\vcenter{\hbox{{
				\includegraphics[width=0.4\linewidth]{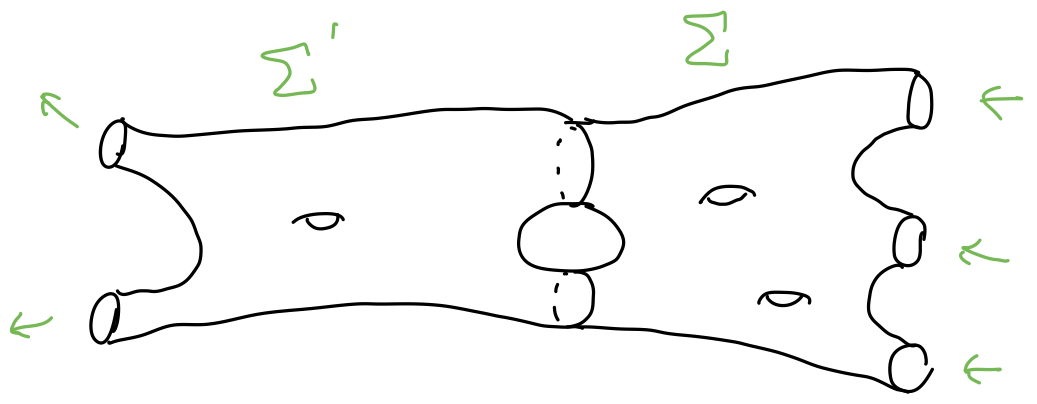}}}}
\end{align*}

It is important to specify how $\partial_+\Sigma_j$ (with parametrization $\eta_j$) is identified with $\partial_-\Sigma_j'$ (with parametrization $\eta_j'$). Pick $x\in\partial_+\Sigma_j$ and $y\in\partial_-\Sigma_j'$. Then
\begin{align}
x=y\quad\Longleftrightarrow \quad\eta_j(x)\eta_j'(y)=1.\label{eq3}
\end{align}

It is clear from the picture that the orientations of $\partial_+\Sigma_j$ and $\partial_-\Sigma_j$ are opposite to each other. This is related to the fact that our rule for sewing is $\eta_j(x)=1/\eta_j'(y)$ but not (say) $\eta_j(x)=\eta_j'(y)$.

Recall we assume that the parametrizations are analytic. We leave it to the readers to check that the sewing of $\Sigma$ and $\Sigma'$, a priori only a topological surface, has a natural complex analytic structure.

\subsection{}\label{lb5}
Suppose $T_1:\mc H^{\otimes N_1}\rightarrow\mc H^{\otimes M_1}$ corresponds to $\Sigma_1$ and $T_2:\mc H^{\otimes N_2}\rightarrow\mc H^{\otimes M_2}$ to $\Sigma_2$, then $T_1\otimes T_2:\mc H^{\otimes (N_1+N_2)}\rightarrow\mc H^{\otimes(M_1+M_2)}$ corresponds to the disjoint union $\Sigma_1\sqcup\Sigma_2$.
\begin{align*}
	\vcenter{\hbox{{
				\includegraphics[width=0.25\linewidth]{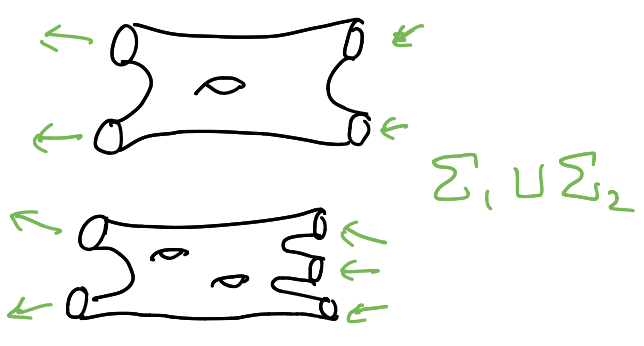}}}}
\end{align*}

\subsection{}\label{lb3}

Consider an annulus $A_{r,R}=\{z\in\Cbb:r<|z|<R\}$ obtained by removing two open  disks from the compact Riemann sphere $\Pbb^1$ via the local coordinate $\eta_1(z)=z/r$ at $0$ and $\eta_2(z)=R/z$ at $\infty$. We call such $A_{r,R}$ (with the given boundary parametrization) a \textbf{standard annlus}. Let $r\nearrow 1,R\searrow 1$. The limit of this annulus  is a ``degenerate" Riemann surface with 1 incoming boundary circle and 1 outing one. Both circles are $\Sbb^1$. The incoming one has parametrization $z\mapsto z$ and the outgoing one $z\mapsto z^{-1}$. We call this annulus the \textbf{standard thin annulus} and denote it by  $A_{1,1}$\index{A11@$A_{r,R}$ and the standard thin annulus $A_{1,1}$}. \emph{The map $T:\mc H\rightarrow\mc H$ associated to  $A_{1,1}$ is the identity map.} This reflects the fact that sewing any $\Sigma$ with a disjoint union of $A_{1,1}$ gives $\Sigma$.
\begin{align*}
	\vcenter{\hbox{{
				\includegraphics[height=1.5cm]{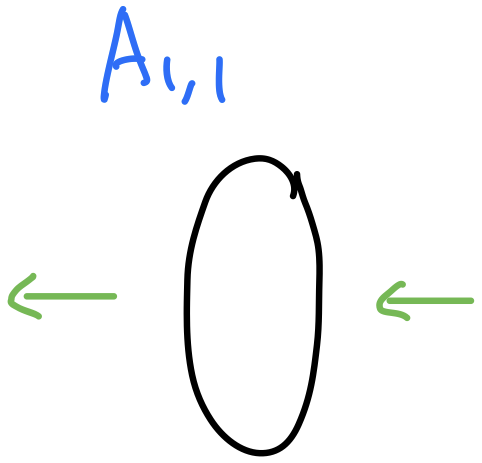}}}}
\end{align*}

\subsection{}\label{lb1}
We give a fancy way to summarize what we have so far: Let $\scr C$ be the monoidal category of compact $1$-dimensional smooth manifolds such that a morphism from an object $S_1$ to another $S_2$ is a compact Riemann surface with incomming parametrized boundary $\simeq S_1$ and outgoing one $\simeq S_2$, that the identity morphism for a union of $N$ circles  is a disjoint union of $N$ pieces of $A_{1,1}$, that the unit object is the empty set, and that the tensor product of objects and morphisms are respectively the disjoint unions of strings and Riemann surfaces. Then a CFT is a monoidal functor from $\scr C$ to the monoidal category of Hilbert spaces. So, roughly speaking, a CFT is a representation of $\scr C$.

Since we choose Hilbert spaces as our underlying spaces, we should expect that the representation of $\scr C$ is unitary. Technically, the functor mentioned above should be a $*$-functor: this means that for each morphism $\Sigma$ from $N$ strings to $M$ strings, we should define its adjoint morphism $\Sigma^*$ from $M$ strings to $N$ ones whose corresponding map is the adjoint $T^*:\mc H^{\otimes M}\rightarrow\mc H^{\otimes N}$ of $T$. $\Sigma^*$ is defined simply to be the \textbf{complex conjugate} $\ovl\Sigma$ \index{CSigma@$\ovl C,\ovl\Sigma$, the complex conjugate of $C$ and $\Sigma$} of $\Sigma$: 

\begin{df}\label{lb7}
$\ovl\Sigma$ consists of points $\ovl x$ where $x\in \Sigma$; the local holomorphic functions on $\ovl\Sigma$ are $\eta^*$ where $\eta$ is a locally defined holomorphic function on $\Sigma$ and \index{fzz@$f^*,\eta^*,\dots$ and $\ovl f,\ovl \eta,\dots$, where $f^*(\ovl x)=\ovl{f(x)}$ and $\ovl f(x)=\ovl{f(x)}$.}
\begin{align}
\eta^*(\ovl x)=\ovl{\eta(x)}	
\end{align}
whenever $\eta$ is defined on $x\in \Sigma$; similarly, boundary parametrizations are given by $\eta_j^*$. Note that if $\Sigma$ is obtained by removing open  disks from an $N$ pointed $\fk X=(C;x_\blt;\eta_\blt)$, then $\ovl\Sigma$ is obtained by removing  disks from
\begin{align}
\ovl{\fk X}:=(\ovl C;\ovl x_1,\dots,\ovl x_N;\eta_1^*,\dots,\eta_N^*)	
\end{align}
\end{df}

$\eta^*$ should not be confused with $\ovl \eta$ defined on $\Sigma$ by
\begin{align*}
\ovl \eta(x)=\ovl{\eta(x)}.
\end{align*}

In the present context, we should assume that an incoming (resp. outgoing) string of $\Sigma$ becomes an outgoing (resp. incoming) one of $\ovl\Sigma$ via the conjugate map $\Co:x\in \Sigma\mapsto\ovl x\in\ovl\Sigma$. In the future, we will often consider all strings as incoming ones if necessary (cf. \ref{lb2}). In that case, we shall also assume all the boundary strings of $\ovl\Sigma$ as incoming.

We should point out that although unitarity is a very important condition, there are important non-unitary CFTs, for instance, the logarithmic CFTs. (In such cases, $\mc H$ is a vector space without inner products.) Also, many VOA results and techniques do not rely on the unitarity. Nevertheless, assuming unitarity will often reasonably simply discussions or give motivations.

\begin{eg}
Let $\fk X=(\Pbb^1;0;\lambda \zeta)$ where $\zeta$ is the standard coordinate of $\Cbb$ and $\lambda\in\Cbb^\times$. We can identify the conjugate of $\Pbb^1$ with $\Pbb^1$ by letting $x\in\Pbb^1\mapsto \ovl x$ be the standard conjugate of $\Cbb$: $z\mapsto \ovl z$. Then $(\lambda\zeta)^*(\ovl z)=\ovl{\lambda\zeta(z)}=\ovl\lambda\cdot\ovl z=\ovl\lambda\zeta(\ovl z)$. So the conjugate of $\fk X$ is isomorphic to $\ovl{\fk X}=(\Pbb^1;0;\ovl\lambda\zeta)$.	
\end{eg}

\subsection{}\label{lb23}

An interaction process could have no incoming or outgoing strings. \emph{The Hilbert space for the empty string $\emptyset$ is $\Cbb$}. The most elementary and important example with no incoming boundary is the closed unit  disk $\Dbb_1^\cl$ with 1 outgoing boundary parametrized by $z\mapsto z^{-1}$. The corresponding map $\Cbb\rightarrow \mc H$ can be identified with its value at $1$. This element in $\mc H$ is denoted by $\id$ \index{1@$\id$, the vacuum vector} and called the \textbf{vacuum vector}.
\begin{align}
	\vcenter{\hbox{{
				\includegraphics[height=1cm]{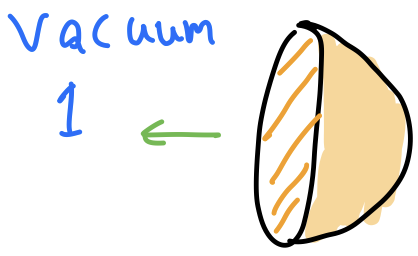}}}}\label{eq27}
\end{align}
Assume as before that out theory is unitary. Then conjugate of the above disk is the same disk and boundary parametrization, but the original outgoing string is now the incoming one. The corresponding map $\mc H\rightarrow \Cbb$ is, according to \ref{lb1}, the linear functional $\bk{\id|\cdot}$.

\subsection{}\label{lb2}

In general, one may wonder what the interaction $T:\mc H^{\otimes N}\rightarrow \Cbb$ means physically for a surface $\Sigma$ with $N$ incoming strings but no outgoing ones.  Choose $0<M<N$, and make $M$ of the $N$ strings of $\partial\Sigma$ be outgoing strings. Then the corresponding interaction is a map $\wtd T:\mc H^{\otimes (N-M)}\rightarrow \mc H^{\otimes M}$. In unitary CFT, $T$ can be related to $\wtd T$ by a anti-unitary (i.e. conjugate-unitary) map $\Theta$ \index{zz@$\Theta$, the CPT operator} on $\mc H$, called the \textbf{CPT operator}, such that for $\xi_1,\dots,\xi_N\in\mc H$ (where the last $M$ vectors are associated to the outgoing strings), we have
\begin{align}
T(\xi_1\otimes\cdots\otimes\xi_N)=\bk{\Theta \xi_{N-M+1}\otimes\cdots\otimes \Theta \xi_N|\wtd T(\xi_1\otimes\cdots\otimes\xi_{N-M})},	\label{eq1}
\end{align}
interpreted pictorially as
\begin{align*}
	\vcenter{\hbox{{
				\includegraphics[height=2cm]{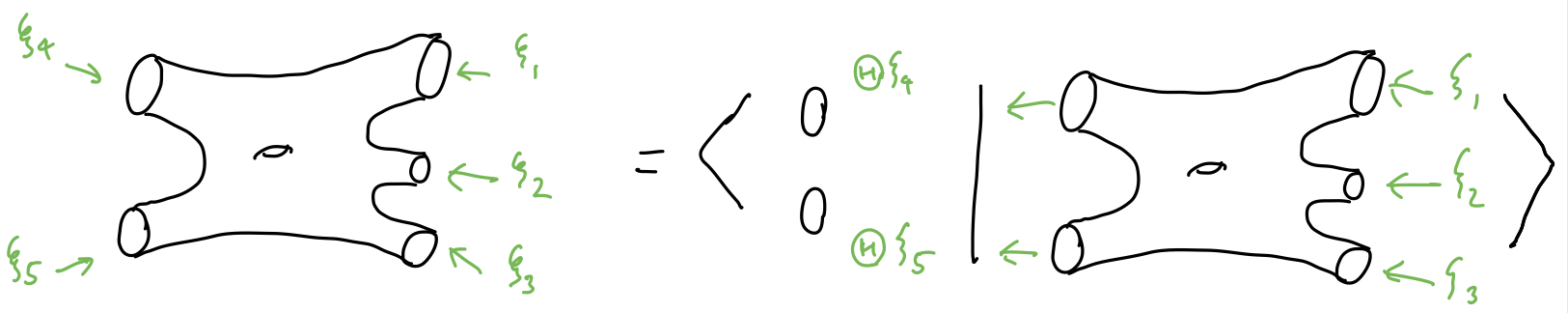}}}}
\end{align*}
The operator $\Theta$ is an involution, i.e., $\Theta^2=\id_{\mc H}$.

Such a linear functional $T$ corresponding to an interaction with no outgoing strings is called a \textbf{correlation function} (or an \textbf{$N$-point function}). These functions are the central objects in CFT (and indeed, in any quantum field theory). Relation \eqref{eq1} teaches us that: (1) correlation functions can be interpreted as probability amplitudes in string interactions with the help of $\Theta$, and (2) to study arbitrary interactions, it suffices to study those with no outgoing strings.

Let me close this subsection by mentioning an important fact: suppose the complex structure of $\Sigma$ and the (assumed analtytic) boundary parametrizations are parametrized holomorphically by some complex variables $\tau_\blt=(\tau_1,\dots,\tau_k)$, then the value of $T(\xi_\blt)$ is now a \emph{real analytic function} of $\tau_\blt$, i.e., it is locally a power series of $\tau_1,\dots,\tau_k$ and their conjugates. Actually, the word ``function" in ``correlation function" means a function of $\tau_\blt$, but not of $\xi_\blt$.

\subsection{}

You must be curious what CPT means. Indeed, $\Theta$ is responsible for the simultaneous symmetry of charge conjugation (C), parity transformation (P), and time reversal (T). P+T together means  an \emph{anti-biholomorphism} $\Sigma\rightarrow\Sigma'$. Now we have arrived at a point that we missed previously: since anti-holomorphic maps are also conformal maps, should we expect that the interaction maps (or the correlation functions) for   anti-biholomorphic surfaces are equal? The answer is no. (Namely, P+T are not preserved.) Indeed, if we let $\Sigma$ have $N$ incomes and no outcomes, let $\ovl\Sigma$ be its complex conjugate (cf. \ref{lb1}) but still with $N$ incomes, and let $T_\Sigma,T_{\ovl\Sigma}$ be the correlation functions associated to them. Then from \ref{lb1} and relation \eqref{eq1}, we have 
\begin{align}
T_\Sigma(\xi_1\otimes\cdots\otimes\xi_N)=\ovl{T_{\ovl\Sigma}(\Theta\xi_1\otimes\cdots\otimes\Theta\xi_N)}.\label{eq30}	
\end{align}

\begin{proof}
By the description in Subsec. \ref{lb1}, the interaction map $\wtd T_{\ovl\Sigma}$ associated $\ovl\Sigma$ with no input and $N$ outputs is $T_\Sigma^*:\Cbb\rightarrow\mc H^{\otimes N}$, the adjoint of $T_\Sigma$. By  $\Theta^2=1$, we have
\begin{align*}
&T_\Sigma(\xi_1\otimes\cdots\xi_N)=\bk{1|T_\Sigma(\xi_1\otimes\cdots\otimes\xi_N)}=\bk{T_\Sigma^*1|\xi_1\otimes\cdots\otimes\xi_N}\\
=&\ovl{\bk{\xi_1\otimes\cdots\otimes \xi_N|\wtd T_{\ovl\Sigma}1}}\xlongequal{\eqref{eq1}}\ovl{T_{\ovl\Sigma}(\Theta\xi_1\otimes\cdots\otimes\Theta\xi_N)}.	
\end{align*}
Note that mathematically, the point of formula \eqref{eq30} is to translate (using \eqref{eq1}) the relation $\wtd T_{\ovl\Sigma}=T_\Sigma^*$ (regarding all the strings of $\ovl\Sigma$ as outgoing) to the case that all the strings of $\ovl\Sigma$ are incoming. 
\end{proof}
Formula \eqref{eq30} explains CPT symmetry: the symmetries of charge (taking complex conjugate of the values of correlation functions) and parity+time (the conjugate biholomophism $\Co:\Sigma\rightarrow\ovl\Sigma$) are preserved, and the operator realizing this simultaneous symmetry is $\Theta$.

Note that mathematically, charge conjugation $C$  is related to taking complex conjugate of numbers (but not of $\Sigma$). Physically, it means making a string  into its ``antistring", or (in general QFT) making a particle (e.g. an election with negative charge) to its anti-particle (e.g. an antielectron with positive charge).

\subsection{}

The CFT we have described so far is actually very special: it has no conformal anomaly. There are indeed no nontrivial CFTs which are both unitary and without anomaly. In this course, we will be mainly interested in CFTs with {conformal anomaly}. Technically, the conformal anomaly is determined by a complex number $c$ (positive for unitary CFT), called \textbf{central charge}. To describe such CFT, we modify the previous descriptions as follows: The map (or the correlation function) $T_\Sigma$ for $\Sigma$ is only up to a positive scalar multiplication depending on $\Sigma$. $T_{\Sigma_1}\circ T_{\Sigma_2}=\lambda T_{\Sigma_1\#\Sigma_2}$ where $\lambda>0$. (The constants are not necessarily positive in non-unitary CFT.) If $\Sigma$ is parametrized holomorphically by some complex variables $\tau_\blt$, then by shrinking the domain of $\tau_\blt$, we can choose $T_\Sigma$ depending real analytically on $\tau_\blt$. 

There are many important cases where a real analytic (or even a holomorphic) $T_\Sigma$ can be chosen globally for $\tau_\blt$. This will be studied later in details. 

Unless otherwise stated, a CFT always means one with (possible) conformal anomaly. Using the fancy language of \ref{lb1}, one can say that a unitary CFT is a \emph{projective} monoidal $*$-functor from the category $\scr C$ in \ref{lb1} to the category of Hilbert spaces. Namely, it is a projective unitary representation of $\scr C$.

\subsection{}\label{lb35}

To study the representations of a topological group $G$, one must first understand very well the topological and the algebraic structures of $G$. Similarly, the study of CFTs relies heavily on the geometric and analytic structures of compact Riemann surfaces. However, from what we have discussed, there is a huge obstacle for studying CFTs: the correlation functions are real analytic, but not complex analytic (i.e. holomorphic) functions of the parameters $\tau_\blt$. Thus, in order to study CFTs using the powerful tools of complex analysis (residue theorem, for instance), we make the following Ansatz: A correlation function $T$ is a sum : $T_\Sigma=\sum_j \Phi^j_{\Sigma}\cdot \Psi^j_{\ovl\Sigma}$, where each $\Phi^j$ and $\Psi^j$ relies holomorphically on $\Sigma$ and $\ovl\Sigma$ respectively (so $\Psi^j_{\ovl\Sigma}$ relies anti-holomorphically on $\Sigma$).

This Ansatz is very vague. Let me explain it in more details. Consider the annulus $A_{r,R}$ with boundary parametrization as in \ref{lb3}. We move the inside circle to another one centered at $z$ (where $z\in A_{r,R}$ is reasonably small), still with radius $r$. The new eccentric annulus $A_{z,r,R}$ has larger outgoing string parametrized by $R/\zeta$ and the smaller incoming one parametrized by $(\zeta-z)/r$, where $\zeta$ is the standard coordinate of $\Pbb^1$. Namely, it is determined by the data
\begin{align}
(\Pbb^1;z,\infty;(\zeta-z)/r,R/\zeta).	\label{eq224}
\end{align}
Let $T_z:\mc H\rightarrow\mc H$ be the corresponding map. As we have said, for general vectors $\xi,\eta\in\mc H$, the expression $\bk{\eta|T_z\xi}=\bk{\Theta\eta,T_z\xi}$ can be chosen to be real analytic with respect to $z$. We now let
\begin{align}\label{eq2}
	\begin{aligned}
\Vbb=\{&\xi\in\mc H:\text{For all $r,R$, the map $T$ can be chosen such that} \\
	&\text{$z\mapsto\bk{\nu|T_z\xi}$ is holomorphic for all $\nu\in\mc H$, and}\\
	&\text{$\xi$ has ``finite energy"}\}		
	\end{aligned}	
\end{align}
``Finite energy" is a minor condition to be explained later. (See \ref{lb10}.)

We can sew $A_{z,r,R}$ with any $\Sigma$, and the motion of the smaller string inside the annulus becomes, after sewing, the motion of a boundary string of $\Sigma$: 
\begin{align}
	\vcenter{\hbox{{
				\includegraphics[height=2.3cm]{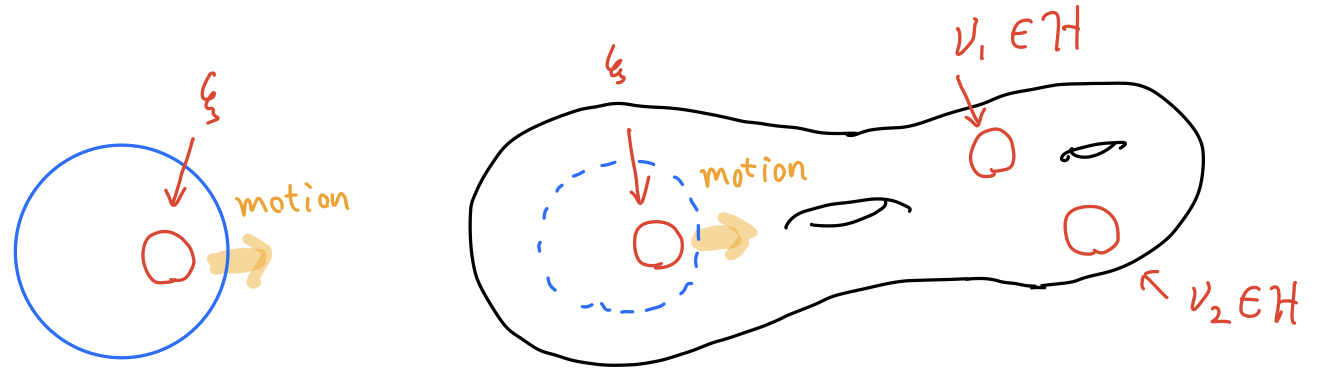}}}}\label{eq176}
\end{align}
Therefore, if a vector $\xi\in \Vbb$ is assigned to an incoming string of $\Sigma$ with (analytic) boundary parametrization $\eta_i$, then, when translating this parametrized string with respect to $\eta_i$, the correlation function $T_\Sigma(\xi\otimes\cdots)$ should be holomorphic with respect to the motion, whatever states we assign to the other strings.  We can therefore study $\Vbb$ with the help of complex analysis. $\Vbb$ is called a \textbf{vertex operator algebra} (VOA).

We have only described $\Vbb$ as a vector space. But in which sense is $\Vbb$ an algebra? An obvious candidate is as follows: consider $\Pbb^1$ with three marked points $0,z,\infty$ and usual coordinates, e.g. $\eta_0=\zeta/r_1,\eta_z=(\zeta-z)/r_2,\eta_\infty=R/\zeta$ at $0,z,\infty$ where $r_1,r_2>0$ are small and $R>0$ is large, and $\zeta$ is again the standard coordinate of $\Cbb$. We assume the strings around $0$ and $z$ are ingoing and that around $\infty$ outgoing. If we assign $\xi_1,\xi_2\in\Vbb$ to the incoming strings, then the outcome  can be viewed as a product of $\xi_1$ and $\xi_2$.
\begin{align*}
	\vcenter{\hbox{{
				\includegraphics[height=2.3cm]{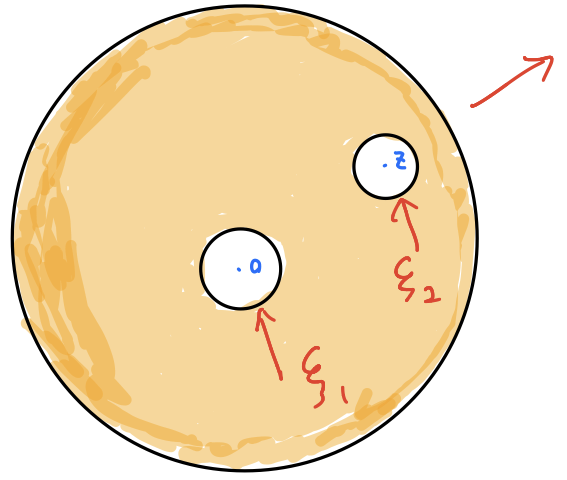}}}}
\end{align*} 
Although this product does not have finite energy, it does satisfy the statement before the last line in \eqref{eq2}. Thus, this product is almost a vector in $\Vbb$. By modifying this product suitably, we can ensure that the products of vectors in $\Vbb$ are always in $\Vbb$. Details will be give in later sections.

Similarly to \eqref{eq2}, we define $\wht\Vbb\subset \mc H$ to be the set of finite energy vectors $\xi$ such that $\bk{\nu|T_z\xi}$ is anti-holomorphic over $z$. The vacuum vector $\id$ belongs to $\Vbb\cap\wht\Vbb$: The result of gluing the unit  disk into the inside of $A_{z,r,R}$ is just the  disk with radius $R$ and parametrization $R/\zeta$, which is independent of $z$. So $T_z\id$ and hence $\bk{\nu|T_z\id}$ are constant over $z$, and hence both holomorphic and anti-holomorphic over $z$.

\subsection{}\label{lb6}

Now we can give a more detailed presentation of our Ansatz. We let $\mc H^\fin$ be the (indeed dense) subspace of vectors in $\mc H$ with ``finite energy", which is acted on by $\Vbb\otimes\wht\Vbb$. Ansatz:
\begin{enumerate}
\item $\mc H^\fin$ as a $\Vbb\otimes\wht\Vbb$-module has decomposition
\begin{align}
\mc H^\fin=\bigoplus_{i\in\fk I} \Wbb_i\otimes\wht\Wbb_i\qquad\supset \Vbb\otimes\wht\Vbb	
\end{align}
where each $\Wbb_i,\wht\Wbb_i$ are respectively irreducible $\Vbb$-modules and $\wht\Vbb$-modules. $\Vbb$ and $\wht \Vbb$ are (according to their definition cf. \eqref{eq2}) subspaces of $\mc H^\fin$ by identifying them with $\Vbb\otimes\id$ and $\id\otimes\wht\Vbb$ respectively. The vacuum vector $\id$ of $\mc H$ is identified with $\id\otimes\id$ (which belongs to $\Vbb\otimes\wht\Vbb$).
\item For some $\Sigma$ without outgoing boundaries, let $T_\Sigma:\mc H^{\otimes N}\rightarrow\Cbb$ be the corresponding map. Then, corresponding to the above direct sum decomposition, we have
\begin{gather}
T_\Sigma\Big|_{(\mc H^\fin)^{\otimes N}}=\sum_{i_1,\dots,i_N\in\fk I}\Phi_{\Sigma,i_\blt}\otimes\Psi_{\ovl\Sigma,i_\blt}	\label{eq174}
\end{gather}
where
\begin{gather*}
\Phi_{\Sigma,i_\blt}:\Wbb_{i_1}\otimes\cdots\otimes\Wbb_{i_N}\rightarrow\Cbb,\\
\Psi_{\ovl\Sigma,i_\blt}:\wht\Wbb_{i_1}\otimes\cdots\otimes\wht\Wbb_{i_N}\rightarrow\Cbb\end{gather*}
are linear. Moreover, when the complex structure and boundary parametrization  are parametrized analytically by complex variables $\tau_\blt$, then locally (with respect to the domain of $\tau_\blt$), $T_\Sigma,\Psi_{\Sigma,i_\blt},\Psi_{\ovl\Sigma,i_\blt}$ can be chosen such that $\Psi_{\Sigma,i_\blt}$ is holomorphic over $\tau_\blt$ (for all input vectors), and $\Psi_{\ovl\Sigma,i_\blt}$ holomorphic over $\ovl\tau_\blt$. $\Phi_{\Sigma,i_\blt}$ and $\Psi_{\ovl\Sigma,i_\blt}$ are called \textbf{conformal blocks} associated to $\Sigma$ (resp. $\ovl\Sigma$) and $\Vbb$ (resp. $\wht\Vbb$).
\end{enumerate}

In part one, $\bigoplus$ could be finite (our main focus in this course), infinite but discrete, or continuous. 

The second part can be summarized by saying that the CFT is separated into the  \textbf{chiral halves} (those $\Phi$ or $\Wbb_i$) and the \textbf{anti-chiral halves} (those $\Psi$ or $\wht\Wbb_i$). Here, ``chiral"=``holomorphic".

When physicists say a CFT is \textbf{rational}, they usually mean that the above direct sum is finite, and each $\Wbb_{i_k},\wht\Wbb_{i_k}$ are semi-simple (hence, by further decomposition, can be irreducible). So far, the mathematical theory of conformal blocks is complete almost only for rational CFTs. These will be the main examples of this course. For non-rational logarithmic CFTs, even the above Ansatz needs to be modified. (So far, it is not even clear how to do it.)

Physicists more or less consider the above description as the definition of conformal blocks. We mathematicians should do the opposite:  define conformal blocks in a different way, and use them to \emph{construct} CFTs following the above Ansatz.

\subsection{}\label{lb27}
You may notice that to make this Ansatz compatible with \ref{lb4} and \ref{lb5}, it is necessarily to assume that
\begin{enumerate}
	\item The tensor product of conformal blocks $\Phi_{\Sigma_1},\Phi_{\Sigma_2}$ associated to $\Sigma_1,\Sigma_2$ respectively should be a conformal block associated to $\Sigma_1\sqcup\Sigma_2$.
	\item The composition of $\Phi_{\Sigma_1},\Phi_{\Sigma_2}$ (or more precisely, their contractions) should be conformal blocks associated to the sewings of $\Sigma_1$ and $\Sigma_2$, where the pair of $\Vbb$-modules to be contracted must be dual to each other.
\begin{gather*}
	\vcenter{\hbox{{
				\includegraphics[height=1.5cm]{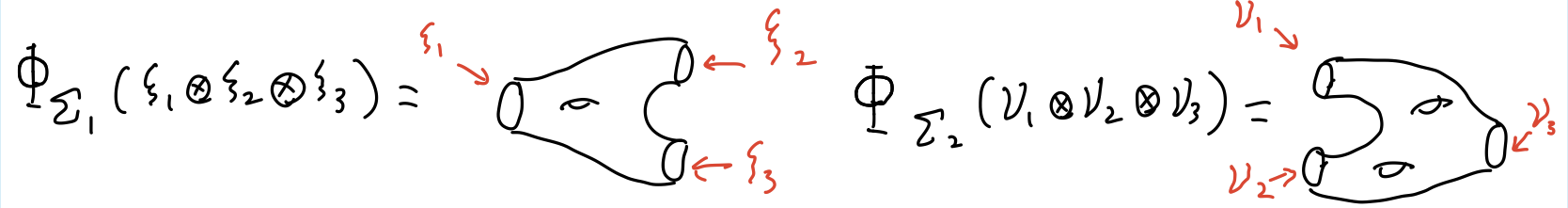}}}}\\
\vcenter{\hbox{{
				\includegraphics[height=1.5cm]{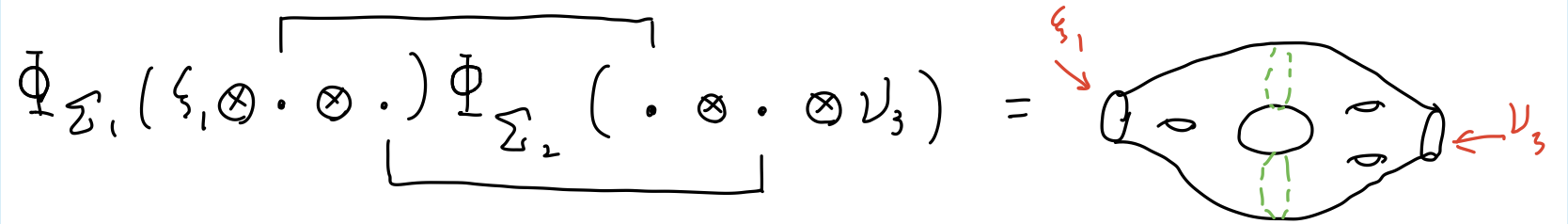}}}}
\end{gather*} 
\end{enumerate}

(A side note on linear algebra: If $V^\vee$ is the dual space (or a suitable dense subspace of the dual space) of a vetor space $V$, we choose a basis $\{v_\alpha\}_{\alpha\in\fk A}$ labeled by elements of $\fk A$, and choose a dual basis $\{v^\vee_\alpha\}_{\alpha\in\fk A}$ of $V^\vee$ (i.e. the one determined by $\bk{v_\alpha,v_\beta}=\delta_{\alpha,\beta}$), then taking  contraction means substituting $\sum_{\alpha\in\fk A}v_\alpha\otimes v^\vee_\alpha$ inside the linear functional on a tensor product of vector spaces such that $V,V^\vee$ are tensor components.)

After we define conformal blocks rigorously, we will see that the first point is obvious, while the second one is a non-trivial theorem.

We briefly explain the meaning of ``dual", and why the dual modules appear in $\mc H$. For instance, in the above picture, the unitary $\Vbb$-module containing $\xi_2$ is dual to the one containing $\eta_1$. As vector spaces, they are ``graded" dual spaces of each other.  (It is a dense subspace of the full dual space, the subspace of ``finite energy" linear functionals. We will talk about this in future sections.) In unitary CFTs, all $\Vbb$ and $\wht\Vbb$ modules are unitary, and  $\Theta(\Wbb_i\otimes\wht\Wbb_i)$ is equivalent to $\Wbb_i'\otimes\wht\Wbb_i'$ where $\Wbb_i'$
is a  $\Vbb$-module dual to $\Wbb_i$, and $\wht\Wbb_i'$ a $\wht\Vbb$-module dual to $\wht\Wbb_i$. The formal name for dual module is \textbf{contragredient module}, to be defined rigorously in Sec. \ref{lb192}.

\subsection{}\label{lb26}

Let us describe the equivalence $\Theta(\Wbb_i\otimes\wht\Wbb_i)\simeq\Wbb_i'\otimes\wht\Wbb_i'$ in more details.

For each $w_i\otimes\wht w_i\in\Wbb_i\otimes\wht\Wbb_i$, the vector $\Theta(w_i\otimes\wht w_i)$ is regarded as a linear functional on $\Wbb_i\otimes\wht\Wbb_i$ in the following way. Let the (clearly symmetric) bilinear form $\bk{\cdot,\cdot}:\mc H^{\otimes 2}\rightarrow\Cbb$ be the correlation function $T_{A_{1,1}}$ for the standard thin annulus $A_{1,1}$ (with two inputs and no outputs).\index{zz@$\bk{\cdot,\cdot}:\mc H^{\otimes 2}\rightarrow\Cbb$, the correlation function $T_{A_{1,1}}$ for $A_{1,1}$} Note that by \eqref{eq1}, for each $\xi,\nu\in\mc H$, we have
\begin{align}
	\bk{\Theta\xi,\nu}=\bk{\xi|\nu}.	\label{eq29}
\end{align} 
Then $\Theta(w_i\otimes\wht w_i)$ is equivalent to the linear functional
\begin{align}
\bk{\Theta(w_i\otimes\wht w_i),\cdot}=\bk{w_i\otimes\wht w_i|\cdot}	
\end{align}
restricted onto $\Wbb_i\otimes\wht\Wbb_i$.

A conformal block with $M+N$ inputs $\Phi_\Sigma:\Wbb_{i_1}\otimes\cdots\otimes\Wbb_{i_N}\otimes\Wbb_{j_1}\otimes\cdots\otimes\Wbb_{j_N}\rightarrow\Cbb$ can be regarded as one with $N$ inputs and $M$ outputs $\Phi_\Sigma:\Wbb_{j_1}\otimes\cdots\otimes\Wbb_{j_N}\rightarrow\mc H_{i_1}'\otimes\cdots\otimes\mc H_{i_M}'$ where $\mc H_{i_k}'$ is the Hilbert space completion of $\Wbb_{i_k}'$ and $\Wbb_{i_k}'$ is the contragredient $\Vbb$-module of $\Wbb_{i_k}$. Using \eqref{eq29}, it is not hard to show that taking compositions of conformal blocks with outputs is equivalent to taking contractions for conformal blocks without outputs.

\section{Virasoro relations; change of boundary  parametrizations; strings vs. punctures}\label{lb41}

\subsection{}\label{lb9}
The goal of this section is to understand conformal blocks associated to $2$-pointed Riemann spheres, equivalently, genus-$0$ surfaces with two boundary strings. We simply call them \textbf{annuli}, although their complex structures and boundary parametrizations are not necessarily the standard ones as in \ref{lb3}.

Let us first consider some degenerate examples whose boundary parametrizations are not necessarily analyic. Let $\Diffp(\Sbb^1)$ \index{Diff@$\Diffp(\Sbb^1)$} be the topological group of orientation preserving diffeomorphisms of $\Sbb^1$. For each $g\in\Diffp(\Sbb^1)$, we let $A_{1,1}^g$ be the thin annulus whose incoming and outgoing strings are both $\Sbb^1$ with parametrizations
\begin{gather*}
\text{Incoming}: z\mapsto z,\qquad \text{Outgoing}: z\mapsto 1/g(z).	
\end{gather*}

\begin{lm}
If $h\in\Diffp(\Sbb^1)$, then $A_{1,1}^{gh}$ is obtained by gluing the incoming circle of  $A_{1,1}^g$ with the outgoing one of $A_{1,1}^h$.
\end{lm}

\begin{proof}
By \eqref{eq3}, a point $z\in A_{1,1}^h$ is glued with $\zeta\in A_{1,1}^g$ iff $\zeta\cdot 1/h(z)=1$, i.e., $\zeta=h(z)$. Now, a point $z$ of $A_{1,1}^h$ becomes the point $h(z)$ of $A_{1,1}^g$ after gluing, which is sent by the outgoing parametrization of $A_{1,1}^g$ to $1/g(h(z))$.
\end{proof}

This proof is not rigorous since we are considering degenerate annuli. A rigorous one would be approximating $A_{1,1}^g$ and $A_{1,1}^h$ by genuine annuli, identifying the sewn annuli, and then taking the limit. This proof is not easy, unless when $g$ and $h$ are real-analytic (e.g., rotations). Nevertheless, we only need this lemma to motivate our following discussions.

\subsection{}
Thus, we may consider $\Diffp(\Sbb^1)$ as the group of thin annuli whose product is the sewing. The merit of this viewpoint is that it convinces us to \emph{consider the semi-group $\Ann$ of  annuli as the complexification of $\Diffp(\Sbb^1)$}. The multiplication $A_1A_2$ of $A_1,A_2\in\Ann$ is the sewing of $A_1,A_2$ defined by gluing the inside of $A_1$ with the outside of $A_2$ using their parametrizations.

As an example, consider $\Pbb^1$ with marked points $0,\infty$ and local coordinates $\eta_0(z)=z,\eta_\infty(z)=e^{-\im \tau}/z$, which gives a thin annulus corresponding to the rotation $z\mapsto e^{\im\tau} z$ when $\tau$ is real. Now consider $\tau$ as a complex variable $\tau=s+\im t$. Then the outgoing circle is the one with radius $e^t$. This gives a genuine annulus whenever $t>0$.

The Ansatz in \ref{lb6} should be expanded to include the following point: for each annulus $A\in \Ann$, the comformal block decomposition of the interaction $T_A:\mc H\rightarrow\mc H$ (with one income and one outcome) with respect to $\mc H^\fin=\bigoplus_i\Wbb_i\otimes\wht\Wbb_i$ is of the form
\begin{align}
T_A=\sum_i 	\pi_i(A)\otimes\wht\pi_i(\ovl A)
\end{align}
where $\pi_i(A)$ is a bounded linear operator on the Hilbert space completion $\mc H_i$ of $\Wbb_i$, and $\wht\pi_i(\ovl A)$ is one on the completion $\wht {\mc H}_i$ of $\wht \Wbb_i$. ($\ovl A$ is the complex conjugate of $A$; see Def. \ref{lb7}. We assume the conjugate of the incomming string of $\ovl A$ is the incoming of $A$, and similarly for the outcoming strings.) The choice of $\pi_i(A)$ and $\wht\pi_i(\ovl A)$ are unique up to scalar multiplications, and if $A$ vary holomorphically over some complex variable $\tau_\blt$, then locally $\pi_i(A)$ and $\wht\pi_i(\ovl A)$ can be chosen to vary holomorphically with respect to $\tau_\blt$ and $\ovl\tau_\blt$ respectively. Finally, if $A_1,A_2\in\Ann$, then $\pi_i(A_1 A_2)$ equals $\pi_i(A_1)\pi_i(A_2)$ up to scalar multiplication, and a similar thing can be said about $\wht\pi_i$.

Namely, each $\pi_i$ is a projective representation of $\Ann$ on $\mc H_i$, and so is $\wht\pi_i$ on $\wht{\mc H}_i$. They should be the analytic extensions of projective unitary representations of $\Diffp(\Sbb^1)$. 

We emphasize that $\pi_i(A)$ and $\wht\pi_i(\ovl A)$ are conformal blocks associated to $A$ and $\ovl A$ respectively.  Roughly speaking, $\pi_i$ describes the conformal symmetries of chiral halves and $\wht\pi_i$ the anti-chiral halves. $A$ and $\ovl A$ have to act jointly on the full space $\mc H$.

\subsection{}

Thus, the study of CFT interactions for annuli reduces to that of the projective representations of $\Ann$. Our goal is to describe such representations in terms of Lie algebras. 

Let $\Vect(\Sbb^1)$ be the Lie algebra of smooth real vector fields of $\Sbb^1$, whose elements are of the form $f\partial_\theta$ where $\partial_\theta$ is the pushforward of the standard unit vector of the real line under the map $\theta\mapsto e^{\im\theta}$, and $f\in C^\infty(\Sbb^1,\Rbb)$. The action of  $f\partial_\theta$ on $h\in C^\infty(\Sbb^1,\Rbb)$ is the negative of the usual one, $-f(e^{\im\theta})\cdot \frac\partial{\partial\theta}h(e^{\im\theta})$. This is because the action of $g\in\Diffp(\Sbb^1)$ on $h$ should be $h\circ g^{-1}$ in order to respect the order of group multiplication. Therefore, the Lie bracket in $\Vect(\Sbb^1)$ is the negative of the usual one:
\begin{align}
[f_1\partial_\theta,f_2\partial_\theta]_{\Vect(\Sbb^1)}=(-f_1\partial_\theta f_2+f_2\partial_\theta f_1)\partial_\theta.	
\end{align}

\subsection{}\label{lb8}

A projective unitary representation $\pi$ of $\Vect(\Sbb^1)$ and the corresponding one $\pi$ of $\Diffp(\Sbb^1)$ (if exists) are related as follows. (Here unitary means that for each vector field $f\partial_\theta$, we have $\pi(f\partial_\theta)^\dagger=-\pi(f\partial_\theta)$, where $\dagger$ is the adjoint, or ``formal adjoint" when the underlying inner product space is not Cauchy-complete.) 

Let $t\in(-\epsilon,\epsilon)\mapsto g_t\in\Diffp(\Sbb^1)$ be a smooth family of diffeomorphisms satisfying $g_0=\id$. Then up to addition by a number of $\im\Rbb$, 
\begin{align}
\frac d{dt}\pi(g_t)\Big|_{t=0}=\pi(\partial_t g_0)	\label{eq5}
\end{align}
where $\partial_t g_0\in\Vect(\Sbb^1)$, the derivative of $g$ at $t_0$, is the vector field determined by
\begin{align}
(\partial_t g_0)(h)=\frac d{dt}(h\circ g_t)\Big|_{t=0}\label{eq4}
\end{align}
for all smooth function $h$ on $\Sbb^1$.

Let now $t\in\Rbb\mapsto\exp(tf\partial_\theta)\in\Diffp(\Sbb^1)$ be the flow generated by $f\partial_\theta\in\Vect(\Sbb^1)$. So its derivative at $t=0$ is $f\partial_\theta$, and $\exp((t_1+t_2)f\partial_\theta)=\exp(t_1f\partial_\theta)\circ\exp(t_2 f\partial_\theta)$. Then \eqref{eq4} implies that up to $\Sbb^1$-multiplication,
\begin{align}
\pi(\exp(tf\partial_\theta))=e^{t\pi(f\partial_\theta)},\label{eq6}
\end{align}
since the derivative of $\pi(\exp(tf\partial_\theta))e^{-t\pi(f\partial_\theta)}$ is $\pi(\exp(tf\partial_\theta))(\pi(f\partial_\theta)-\pi(f\partial_\theta))e^{-t\pi(f\partial_\theta)}=0$.

\subsection{}

The Witt algebra $\Span_\Cbb=\{l_n:n\in\Zbb\}$ is a complex dense Lie subalgebra of the complexification $\Vect(\Sbb^1)\otimes_\Rbb\Cbb$. Here,
\begin{align}
l_n=z^{n+1}\partial_z=-\im e^{\im n\theta}\partial_\theta	
\end{align}
where $z=e^{\im\theta}$ and $\partial_z=\frac 1{\im e^{\im\theta}}\partial_\theta$. (We use the chain rule to ``define" $\partial_z$.) One checks
\begin{align}
[l_m,l_n]=(m-n)l_{m+n}
\end{align}
where the bracket is the negative of the usual one for vector fields.

Let us assume for simplicity that the CFT is unitary. In the decomposition $\mc H^\fin=\bigoplus_i \Wbb_i\otimes\wht\Wbb_i$, each $\Wbb_i$ is a projective unitary representation $\pi_i$ of $\{l_n\}$ , and similarly $\wht\Wbb_i$ is one $\wht\pi_i$ of $\{l_n\}$.  We know that the choice of $\pi_i(l_n)$ is unique up to $\im\Rbb$-scalar addition. Here is a well-known fact about projective representations of Witt algebra (cf. for instance \cite[Sec. IV.1]{Was10}): one can make a particular choice of $\pi_i(l_n)$ (for each $n$), denoted by $L_n$, such that the \textbf{Virasoro relation}
\begin{gather}
		[L_m,L_n]=(m-n)L_{m+n}+\frac c{12}(m+1)m(m-1)\delta_{m,-n}\label{eq16}
	\end{gather} 
holds and $c\in\Cbb$ is called the \textbf{central charge}. In the case that $\pi_i$ is projectively unitary, $L_n$ can be chosen such that $L_n^\dagger=L_{-n}$ also holds.

We have abused the notation by writing the actions of $l_n$ on all $\Vbb$-modules $\Wbb_i$ (as chiral halves of the CFT) as $L_n$. We are justified to do so because, as we will see later, the actions of $l_n$ come from those of $\Vbb$. Technically: Virasoro algebra is inside the VOA. So the action of $\{l_n\}$ on $\Wbb_i$ is the restriction of that of $\Vbb$. In particular, all chiral halves $\Wbb_i$ share the same central charge $c$.

Similarly, we write the actions of $l_n$ on all $\wht\Wbb_i$ as $\ovl L_n$. (The bar over $L_n$ reflects the fact that  $\ovl L_n$ describes the conformal symmetries of the anti-chiral halves of the CFT. $\ovl L_n$ is not related with $L_n$ by the CPT operator $\Theta$.) The central charge $\wht c$ for $\{\ovl L_n\}$ is independent of $\wht\Wbb_i$ and in general could be different from the one $c$ of $\{L_n\}$, although in most important cases they are equal. (E.g., when the CFT contains both closed and open strings.)

\subsection{}\label{lb12}

We shall generalize \eqref{eq6} to complex vector fields. First of all, we consider an element
\begin{align*}
f(z)\partial_z=\sum_{n\in\Zbb}a_nz^{n+1}\partial_z	
\end{align*}
where the sum could be infinite. We treat $f(z)=\sum_n a_nz^{n+1}$ as a Laurent series. Let us now assume that $f(z)$ \emph{is a holomorphic function on a neighborhood $U\subset\Cbb$ of $\Sbb^1$.} 

$f\partial_z$ is a complex holomorphic vector field of $U$, which (after shrinking  $U$) gives a \textbf{holomorphic flow} $\tau\in \Delta\mapsto \exp(\tau f\partial_z)\in\scr O(U)$ where $\Delta\subset\Cbb$ is a neighborhood of $0$. (Recall from the notation section that $\scr O(U)$ is the space of holomorphic functions on $U$.) This means:
\begin{enumerate}[label=(\arabic*)]
\item $(\tau,z)\in\Delta\times U\mapsto \exp(\tau f\partial_z)(z)$ is holomorphic whose restriction to each slice $\tau\times U$ is injective (and hence, a biholomorphism onto its image).
\item $\exp(0 f\partial_z)(z)=z$.
\item $\exp((\tau_1+\tau_2)f\partial_z)=\exp(\tau_1 f\partial_z)\circ\exp (\tau_2 f\partial_z)$ on an open subset of $U$ containing $\Sbb^1$.
\item For any holomorphic function $h$ defined on an open set inside $U$,
\begin{align}
f\partial_z h=\frac \partial{\partial\tau} h\circ \exp(\tau f\partial_z)\Big|_{\tau=0}.\label{eq7}
\end{align}
(Compare \eqref{eq4}.) This condition is equivalent to
\begin{align}
\frac \partial{\partial\tau}\exp(\tau f\partial_z)\Big|_{\tau=0}=f.\label{eq8}	
\end{align}
\end{enumerate}
(To see the equivalence, set $h(z)=z$ for one direction, and use chain rule for the other one.)
\begin{align*}
\vcenter{\hbox{{
			\includegraphics[height=2.5cm]{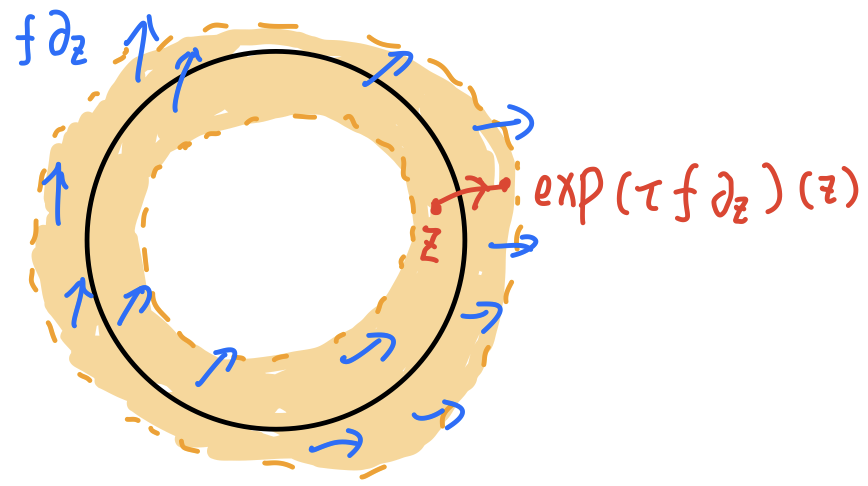}}}}	
\end{align*}


\begin{rem}
A caveat: The notations $f\partial_z$ and $\exp(\tau f\partial_z)$ are not compatible with those in the real case. Indeed, if we assume that $\tau$ only takes real values $\tau=t$, then by taking the real and the imaginary parts of \eqref{eq8}, we see that $\sigma_t$ is a real flow on the real surfaces $U$ generated by the real vector field $\Real f\cdot\partial_x+\Imag f\cdot \partial_y$. Writing $\partial_x=\partial_z+\partial_{\ovl z},\partial_y=\im(\partial_z-\partial_{\ovl z})$, we see that this vector field $f\partial_z$ should more precisely be written as $f\partial_z+\ovl f\partial_{\ovl z}$ where $\ovl f(x)=\ovl {f(x)}$.

This point  is also justified by the fact that if $k$ is antiholomorphic, then
\begin{align}
\ovl f\partial_{\ovl z}k=\frac\partial{\partial_{\ovl \tau}} k\circ \exp (\tau  f\partial_z)\Big|_{\tau=0}.	
\end{align}
(Proof: take $k=\ovl h$  in \eqref{eq8}.) Thus, a more precise notation for $\exp(\tau f\partial_z)$ should be $\exp(\tau f\partial_z+\ovl\tau\ovl f\partial_{\ovl z})$. But we prefer to suppress the term $\ovl\tau\ovl f\partial_{\ovl z}$ to keep the notations shorter.
\end{rem}

\subsection{}

One way to find the expression of $\sigma_\tau=\exp(\tau f\partial_z)$ is to solve the holomorphic nonlinear differential equation with initial condition:
\begin{gather}
\begin{gathered}\label{eq10}
\frac \partial{\partial\tau}\sigma_\tau(z)=f(\sigma_\tau(z)),\\
\sigma_0(z)=z.
\end{gathered}
\end{gather} 
This is due to \eqref{eq8} and $\sigma_{\tau_1+\tau_2}=\sigma_{\tau_1}\circ\sigma_{\tau_2}$. (Indeed, the existence of holomorphic flows is due to that of the solutions of such equations.)

Alternatively, one may calculate the flow by brutal force using the formula
\begin{align}\label{eq14}
\begin{aligned}
\exp(f\partial_z)(z)=&\sum_{k\in\Nbb}\frac 1{k!}(f(z)\partial_z)^kz\\
=&\sum_{k\in\Nbb}\frac 1{k!}\underbrace{f(z)\partial_z\Big(f(z)\partial_z\big(\cdots f(z)\partial_z}_{k\text{ times}} z)\cdots \big)\Big).
\end{aligned}
\end{align}
(One may treat this formula as a formal sum if one worries about the convergence issue.) To see why this formula is valid,  check that such defined $\exp(\tau f\partial_z)(z)=:\sigma_\tau(z)$ satisfies that $\sigma_{\tau_1+\tau_2}=\sigma_{\tau_1}\circ\sigma_{\tau_2}$, that $\partial_\tau\sigma_\tau|_{\tau=0}=f$, and that $\sigma_0(z)=z$. This is easy.

\subsection{}\label{lb10}

\begin{eg}\label{lb14}
$\sigma_\tau(z)=e^{\tau}z$ is the holomorphic flow generated by the vector field $l_0=z\partial_z$ since $\frac\partial{\partial\tau}e^\tau z|_{\tau=0}=z$. Namely,
\begin{align*}
\exp(\tau z\partial_z)(z)=e^\tau z.	
\end{align*}
\end{eg}

Set $\lambda=e^\tau$. In view of the  $A_{1,1}^g$ in \ref{lb9}, we consider the 2-pointed sphere $\fk X=(\Pbb^1;0,\infty;\zeta,\lambda^{-1}\zeta^{-1})$ where $\zeta:z\mapsto z$ is the standard coordinate of $\Cbb$. Then, when $|\lambda|\leq 1$, $\fk X$ defines an annulus $A$, either genuine or thin, whose incoming circle has radius $1$ and outcoming $1/|\lambda|$. Thus, the conformal block $\pi_i(A)$ associated to this annulus, which is a linear operator on the Hilbert space completion $\mc H_i$, should be $e^{\tau L_0}=\lambda^{L_0}$ (by replacing $z\partial_z$ with $L_0$). 

It is easy to check that $\ovl A$ is isomorphic to the annulus defined by $(\Pbb^1;0,\infty;\zeta,\ovl{\lambda^{-1}}\zeta^{-1})$. So the corresponding conformal block should be $\wht\pi_i(\ovl A)=\ovl\lambda^{\ovl L_0}$. Therefore, the interaction map $T_A:\mc H\rightarrow\mc H$ is determined by
\begin{align}
	T_A\big|_{\mc H_i\otimes\wht{\mc H}_i}=\lambda^{L_0}\otimes\ovl\lambda^{\ovl L_0}.	\label{eq9}
\end{align}

In a unitary CFT, $L_0$ and $\ovl L_0$ (or more precisely, their closures) are self-adjoint operators so that $\lambda^{L_0}$ and $\ovl\lambda^{\ovl L_0}$ can be defined and are unitary when $|\lambda|=1$. Moreover, in a unitary CFT:
\begin{ass}[Positive energy]\label{lb18}
The spectra of $L_0$ and $\ovl L_0$ are both positive (i.e. $\geq 0$). In these notes, we are mainly interested in the case that the spectra are discrete. We identify $L_0$ with $L_0\otimes\id$ and $\ovl L_0$ with $\id\otimes\ovl L_0$ so that $L_0,\ovl L_0$ are commuting diagonalizable operators on $\mc H^\fin$ with $\geq0$ eigenvalues.
\end{ass}


Now we can explain what we meant by finite energy: A vector $\xi$ of $\mc H$ has \textbf{finite energy} if $\xi$ is a finite sum of eigenvectors of both $L_0$ and $\ovl L_0$. (In general, a vector of $\mc H$ is an $l^2$-convergent sum, either finite or infinite, of eigenvectors.)

\subsection{}

\begin{eg}\label{lb20}
Let $n\neq 0$. To understand the geometric meanings of $e^{\tau L_{-n}}$ and $e^{\ovl\tau\ovl L_{-n}}$, we find the expression of $\sigma_\tau=\exp(\tau z^{-n+1}\partial_z)$ by solving the differential equation $\partial_\tau\sigma_\tau=(\sigma_\tau)^{-n+1}$ with initial condition $\sigma_0(z)=z$ (cf. \eqref{eq10}). The solution is
\begin{align}
\exp (\tau z^{-n+1}\partial_z)(z)=(z^n+n\tau)^{\frac 1n}.\label{eq17}
\end{align}
\hfill $\Box$
\end{eg}

Unfortunately, this flow does not give us any annulus in the usual sense. Take $n=1$ for instance. Then the flow is just the translation by $\tau$. However, the circle after a small translation will intersect the original one. So there is no annulus whose outgoing circle is the translation of the incoming one. In fact, in most cases,   $\exp(f\partial_z)$ is not the action of an annulus. We have to pursue another way of understanding this operator.

\subsection{}\label{lb13}

There are two ways to look at a group action $G\curvearrowright X$: (1) The action of $g\in G$ on $X$ is a transformation. So $gx\neq x$ in general. (2) $gx$ and $x$ are different expressions (under different coordinates) of the same element. The rule for change of coordinate is given by the action of $G$. We shall take the second viewpoint.

Let $\fk X=(C;x_1,\dots,x_N;\eta_1,\dots,\eta_N)$ be an $N$-pointed compact Riemann surface with local coordinates satisfying Assumption \ref{lb11}. Assume the setting of \ref{lb12}. Write $\sigma_\tau=\exp(\tau f\partial_z)$ and $f(z)=\sum_{n\in\Zbb}a_nz^{n+1}$ be defined on $U\supset\Sbb^1$. Let $\tau\in\Delta$ be close to $0$.

\begin{rem}
In case you want to know the precise meaning of ``close": for the local coordinate $\eta_i$ we are to discuss in the following, we choose $\epsilon>0$ such that  $\sigma_\tau(U\cap\Rng(\eta_i))$ contains $\Sbb^1$ for all  $\tau\in\Dbb_\epsilon$, where the open set $\Rng(\eta_i)$ is the range of $\eta_i$.
\end{rem}

\begin{prin}[Change of boundary parametrizations]
Suppose that the local coordinate $\eta_i$ at $x_i$ is changed to the boundary parametrization $\sigma_\tau\circ\eta_i$ and the boundary string $\eta_i^{-1}\circ(\Sbb^1)$ is gradually changed (with respect to the change of  $\tau$) to $\eta_i^{-1}\big(\sigma_\tau^{-1}(\Sbb^1)\big)$. Then, in the expressions of conformal blocks and correlation functions (without outputs), each $w_i\in\Wbb_i$ is replaced by $e^{\tau\sum_n a_nL_n}w_i$, and each $\wht w_i\in\wht\Wbb_i$ by $e^{\ovl\tau\sum \ovl{a_n}\ovl L_n}\wht w_i$.
\end{prin}

To be more precise, let $T_\Sigma:\mc H^{\otimes N}\rightarrow\Cbb$ be the correlation function where $\Sigma$ is obtained from $\fk X$. Assume $i=1$ for simplicity. Changing the local coordinate $\eta_1$ to $\sigma_\tau\circ\eta_1$ gives a new surface with parametrized boundary $\Sigma'$. Then up to scalar multiplication, $T_{\Sigma'}$ and $T_\Sigma$ are related by
\begin{align}\label{eq12}
T_\Sigma(\xi_1\otimes\xi_2\otimes\cdots\otimes\xi_N)=T_{\Sigma'}\Big(\big(e^{\tau\sum_n a_nL_n}\otimes e^{\ovl\tau\sum_n\ovl{a_n}\ovl L_n}\big)\xi_1\otimes\xi_2\otimes\cdots\otimes\xi_N \Big)	
\end{align}
for all $\xi_1,\dots,\xi_N$. Similarly, if $\Phi_\Sigma:\Wbb_{i_1}\otimes\cdots\otimes\Wbb_{i_N}\rightarrow\Cbb$ is a conformal block for $\Sigma$, then $\Phi_{\Sigma'}$ defined by
\begin{align}\label{eq11}
\Phi_{\Sigma}(w_1\otimes w_2\otimes\cdots\otimes w_N)=\Phi_{\Sigma'}\big(e^{\tau\sum_n a_nL_n} w_1\otimes w_2\otimes\cdots\otimes w_N\big)	
\end{align}
is one for $\Sigma'$.

\subsection{}\label{lb15}

The geometric intuition in the above subsection is the following: $\xi_1$ in the $\eta_1$-parametrization is the same (up to scalar multiplication) vector as $(e^{\tau\sum_n a_nL_n}\otimes e^{\ovl \tau\sum_n \ovl {a_n}\cdot\ovl L_n})\xi_1$ in the $\sigma_\tau\circ\eta_1$-parametrization.  We call this same ``abstract" vector $\wtd\xi_1$, which is unique up to scalar multiplication. We write $\xi_1=(\mc U(\eta_1)\otimes\mc U(\eta_1^*))\wtd \xi_1$, understanding $\mc U(\eta_1)\otimes\mc U(\eta_1^*)$ as the map sending an abstract vector to its concrete expression under the boundary parametrization $\eta_1$. Namely, $\mc U(\eta_1)\otimes\mc U(\eta_1^*)$ is a vector bundle trivialization.  The transition function from the $\eta_1$-parametrization to the $\sigma_\tau\circ\eta_1$-parametrization is
\begin{align}
\big(\mc U(\sigma_\tau\circ\eta_1)\otimes\mc U((\sigma_\tau\circ\eta_1)^*)\big)\big(\mc U(\eta_1)\otimes\mc U(\eta_1^*)\big)^{-1}=e^{\tau\sum_n a_nL_n}\otimes e^{\ovl \tau\sum_n \ovl {a_n}\cdot\ovl L_n}.	
\end{align}
We have a parametrization independent $T$ (more precisely, independent of a small change of parametrizations)  whose expressions under the concrete boundary parametrizations are (up to scalar multiplications)
\begin{align*}
&T(\wtd\xi_1\otimes\cdots)=T_\Sigma\Big(\big(\mc U(\eta_1)\otimes\mc U(\eta_1^*)\big)^{-1}\wtd\xi_1\otimes\cdots\Big)\\
=&T_{\Sigma'}\Big(\big(\mc U(\sigma_\tau\circ\eta_1)\otimes\mc U((\sigma_\tau\circ\eta_1)^*)\big)^{-1}\wtd\xi_1\otimes\cdots\Big).
\end{align*}

\subsection{}

Let us do an example to see how the change of parametrization formula works. 

\begin{eg}
Let $\fk X=(\Pbb^1;1/3,\infty;2(\zeta-1/3),\zeta^{-1})$ where $\zeta:z\mapsto z$ is the standard coordinate of $\Cbb$. We choose $1/3$ to be the input point, and $\infty$ the outgoing one. The associated boundary parametrized surface $\Sigma$ is an annulus whose incoming circle $\{z:|2(z-1/3)|=1\}$ has center $1/3$ and radius $1/2$, and whose outgoing circle is $\Sbb^1$. Let us find an expression for $T_\Sigma:\mc H\rightarrow\mc H$.

We know that the map for the standard thin annulus $A_{1,1}$ is $T_{A_{1,1}}=\id_{\mc H}$. Let $\fk X_1=(\Pbb^1;0,\infty;2\zeta,\zeta^{-1})$, which gives an annlus $\Sigma_1$ with incoming string $\frac 12\Sbb^1$ and outgoing one $\Sbb^1$. $A_{1,1}$ is changed to $\Sigma_1$ by changing the incoming boundary parametrization $\zeta$ to $2\zeta$. By Ex. \ref{lb14}, $2\zeta=\exp(\log 2\cdot z\partial_z)$. So, as $e^{\log 2L_0}=2^{L_0}$ and similarly $e^{\log 2 \ovl L_0}=2^{\ovl L_0}$, by \eqref{eq12}, $T_{\Sigma_1}$ could be $(1/2)^{L_0}\otimes (1/2)^{\ovl L_0}$.

$\Sigma_1$ is changed to $\Sigma$ by adding $2\zeta$ by $-2/3$. According to Ex. \ref{lb20},  $\exp(-2/3\partial_z)(z)=z-2/3$. Therefore, up to a scalar multiplication, $T_{\Sigma_1}(\xi)=T_\Sigma((e^{-\frac 23 L_{-1}}\otimes e^{-\frac 23 \ovl L_{-1}})\xi)$. Thus, the answer is
\begin{align*}
T_\Sigma=\big((1/2)^{L_0}\otimes (1/2)^{\ovl L_0}\big)\cdot \big((e^{\frac 23 L_{-1}}\otimes e^{\frac 23 \ovl L_{-1}})\big)=\big((1/2)^{L_0}e^{\frac 23L_{-1}}\big)\otimes 	\big((1/2)^{\ovl L_0}e^{\frac 23\ovl L_{-1}}\big).
\end{align*}
$(1/2)^{L_0}e^{\frac 23L_{-1}}$ is a conformal block for $\Sigma$. \hfill\qedsymbol
\end{eg}

\subsection{}\label{lb28}

What is the change of parametrization formula for $T_\Sigma$ (and hence $\Phi_\Sigma$) when some output strings are involved? Recall from Subsec. \ref{lb26} that the correlation function $T_{A_{1,1}}:\mc H^{\otimes 2}\rightarrow\Cbb$ is a symmetric bilinear form $\bk{\xi,\nu}=\bk{\nu,\xi}=\bk{\Theta\nu|\xi}$. With respect to this form, we actually have
\begin{align}
(L_n\otimes\id)^\tr=L_{-n}\otimes\id,\qquad (\id\otimes\ovl L_n)^\tr=\id\otimes \ovl L_{-n}.	\label{eq15}
\end{align}
More precisely, for each $\xi,\nu\in\mc H^\fin$, we have
\begin{align*}
\bk{(L_n\otimes\id)\xi,\nu}=\bk{\xi,(L_{-n}\otimes\id)\nu}	
\end{align*} 
and a similar relation for $\ovl L_n$. Rewrite the above relation in terms of $\bk{\cdot|\cdot}$, we have $\bk{\Theta(L_n\otimes\id)\xi|\nu}=\bk{\Theta\xi|(L_{-n}\otimes\id)\nu}	$, and noticing the unitarity property $L_n^\dagger=L_{-n}$, we get
\begin{align}
\Theta(L_n\otimes\id)=(L_n\otimes\id)\Theta,\qquad \Theta(\id\otimes\ovl L_n)=(\id\otimes\ovl L_n)\Theta.\label{eq33}
\end{align}
These relations truly hold, not just up to scalar addition or multiply.

From this, we see that for the maps $T_\Sigma,T_{\Sigma'}:\mc H^{\otimes(N-1)}\rightarrow\mc H$ with $N-1$ inputs and $1$ output,
\begin{align}
T_\Sigma=\Big(e^{\tau\sum_n a_nL_{-n}}\otimes e^{\ovl \tau\sum_n \ovl{a_n}\ovl L_{-n}}\Big)\circ T_{\Sigma'}.	\label{eq13}
\end{align}
You can easily generalize this formula to the case of more than one outputs.
\begin{proof}
Let $\xi_\blt\in\mc H^{\otimes(N-1)}$ and $\nu\in\mc H$. By \eqref{eq1}, the correlation function (with $N$-inputs and no outputs) for $\Sigma$ and $\Sigma'$ are $\bk{\Theta\cdot|T_\Sigma\cdot}$ and $\bk{\Theta\cdot|T_{\Sigma'}\cdot}$ respectively. So by  \eqref{eq12}, 
\begin{align*}
&\bk{\Theta\nu|T_\Sigma(\xi_\blt)}=\bk{\Theta (e^{\tau\sum_n a_nL_n}\otimes e^{\ovl \tau\sum_n \ovl{a_n}\ovl L_n})\nu|T_{\Sigma'}(\xi_\blt)}\\
\xlongequal{\eqref{eq33}}&	\bk{(e^{\ovl \tau\sum_n \ovl{a_n}L_n}\otimes e^{\tau\sum_n a_n\ovl L_n})\Theta \nu|T_{\Sigma'}(\xi_\blt)}\\
\xlongequal{\text{unitarity}}&\bk{\Theta \nu|(e^{\tau\sum_n a_nL_{-n}}\otimes e^{\ovl \tau\sum_n \ovl{a_n}\ovl L_{-n}})T_{\Sigma'}(\xi_\blt)}.
\end{align*}
\end{proof}

\begin{exe}
Show that the formula \eqref{eq9} in Example \ref{lb14} follows from \eqref{eq13}.
\end{exe}

\subsection{}

In case you want to know why $(L_{-n}\otimes\id)=(L_n\otimes\id)^\tr$, we give a geometric explanation below, in which we pretend to ignore the issue of the uniqueness up to scalar additions/multiplications.

\begin{proof}
Let $\fk X=(\Pbb^1;0,\infty;z,z^{-1})$ where $z$ is the standard coordinate of $\Cbb$, which gives the standard thin annulus $A_{1,1}$. Assume the two strings are incoming.  We know the correlation function is $\bk{\xi,\nu}$, where we assume $\xi$ is associated to the string around $0$ and $\nu$ the one around $\infty$.

Change the local coordinate $z$ at $0$ to $\sigma_\tau$, and keep the other data of $\fk X$.  This changes $A_{1,1}$ to a new weird annulus $A$. By \eqref{eq12}, the correlation function for $A$ is
\begin{align*}
T_A(\xi\otimes\nu)=\bk{(e^{-\tau\sum_n a_nL_n}\otimes e^{-\ovl\tau\sum_n \ovl{a_n}\cdot \ovl L_n }) \xi,\nu}.	
\end{align*}
Note that if we set $\zeta=\sigma_\tau(z)$, then $z^{-1}=1/\sigma_\tau^{-1}(\zeta)$, which equals $1/\sigma_{-\tau}(\zeta)$ by the definition of flows. Namely, $A$ is equivalent to the weird annulus whose incoming boundary parametrization is $z$ and outcoming $1/\sigma_{-\tau}(z)$. To compute the correlation function for this choice of boundary parametrization, we note that the original $1/z$ at $\infty$ is changed to $1/\sigma_{-\tau}(z)$. Therefore, if we let $\gamma_\tau(z)=1/\sigma_{-\tau}(1/z)$ which is a holomorphic flow generated by some $\sum_n b_n z^{n+1}$, then the expression for $T_A$ is
\begin{align*}
T_A(\xi\otimes\nu)=\bk{\xi,(e^{-\tau\sum_n b_nL_n}\otimes e^{-\ovl\tau\sum_n \ovl{b_n}\cdot\ovl L_n })\nu}.	
\end{align*}
For the two expressions of $T_A$, we take the holomorphic derivative $-\partial_\tau$ at $\tau=0$ to get
\begin{align*}
\sum a_n\bk{(L_n\otimes\id)\xi,\nu}=\sum b_n\bk{\xi,(L_n\otimes\id)\nu}.	
\end{align*}
To finish the proof, it suffices to prove $b_n=a_{-n}$.

Recall $\sum a_nz^{n+1}=\partial_\tau\sigma_\tau|_{\tau=0}$.  Similarly,  $\sum b_nz^{n+1}=\partial_\tau\gamma_\tau|_{\tau=0}$, which is
\begin{align*}
&\partial_\tau(1/\sigma_{-\tau}(1/z))\big|_{\tau=0}=-\frac 1{\sigma_0(1/z)^2}\cdot \partial_\tau(\sigma_{-\tau}(1/z))\big|_{\tau=0}	\\
=&z^2\cdot \sum a_n(1/z)^{n+1}=\sum a_nz^{-n+1}=\sum a_{-n}z^{n+1}.
\end{align*}
\end{proof}

\subsection{}\label{lb16}

As an easy application of our change of parametrization formula, we are able to describe the map $T_A:\mc H\rightarrow\mc H$ for an analytic annulus $A\in\Ann$ obtained from $(\Pbb^1;0,\infty;\eta_0,\eta_\infty)$ where $\eta_0$ and $\eta_\infty$ are local coordinates at $0,\infty$ respectively. Set $\varpi=1/z$. One can write
\begin{align*}
\eta_0(z)=\exp\Big(\sum_{n\in\Nbb} a_nz^{n+1}\partial_z\Big)(z),\qquad \eta_\infty(\varpi)=\exp\Big(\sum_{n\in\Nbb} b_n \varpi^{n+1}\partial_\varpi\Big)(\varpi),	
\end{align*}
where the coefficients $a_n,b_n$ can be determined using \eqref{eq14}. (We will say more about determining the coefficients in the future.) When $A$ is the standard thin annulus (i.e., when $\eta_0:z\mapsto z,\eta_\infty:z\mapsto z^{-1}$), we know $T_A=\id$. Thus, in general, by \eqref{eq12} and \eqref{eq15}, the map $T_A$ is (up to scalar multiplications)
\begin{align*}
T_A=\Big(e^{\sum_{n\in\Nbb}-b_nL_{-n}}\otimes e^{\sum_{n\in\Nbb}-\ovl{b_n}\cdot \ovl L_{-n}}\Big) \cdot \Big(e^{\sum_{n\in\Nbb} -a_nL_n}\otimes e^{\sum_{n\in\Nbb} -\ovl{a_n}\cdot \ovl L_n}\Big).	
\end{align*}

The reason that only $n\in\Nbb$ are involved is because $\eta_0$ and $\eta_\infty$ can be defined near $0$ and send $0$ to $0$. Indeed, for $f(z)=\sum_{n\in\Zbb} a_n z^{n+1}$, assume that $\exp(\tau f\partial_z)(z)$ is defined near $0$ and sends $0$ to $0$ for all small $\tau$. Then its derivative over $\tau$ at $z=0$, which is $f(\exp(\tau f\partial_z)(0))=f(0)$ by \eqref{eq9}, should also be $0$. So $f$ must be of the form $\sum_{n\geq 0}a_nz^{n+1}$.

\subsection{}

We call those in \ref{lb13} and \ref{lb15} \textbf{change of (boundary) parametrizations} in general, and those in \ref{lb16} \textbf{change of (local) coordinates}. The former contains the latter. 

When changing the boundary parametrizations, the standard coordinate $z$ could be changed to $\sigma_\tau$ not necessarily  defined at $0$, or more generally, a local coordinate (say) $\eta_1$ of an $N$-pointed $\fk X=(C;x_\blt;\eta_\blt)$ is changed to $\sigma_\tau\circ\eta_1$. This changes the boundary-parametrized Riemann surface $\Sigma$ to $\Sigma'$. Note that this process does not violate our definition of \emph{analytic} boundary parametrizations in \ref{lb17}: The new surface $\Sigma'$ is obtained from a new $N$-pointed one $\fk X'=(C';x_\blt;\sigma_\tau\circ\eta_1,\eta_1,\dots,\eta_N)$ where $C'$ is a new compact Riemann surface, which is defined by gluing $\Sigma$ with $N$ pieces of unit  disks $\Dbb_1$ using the maps $\sigma_\tau\circ\eta_1,\eta_2,\dots,\eta_N$. (If you use the maps $\eta_1,\dots,\eta_N$ instead, you simply get $C$.) Thus, \emph{for the change of boundary parametrizations in general, the underlying compact Riemann surfaces $C$ could be changed}. More details will be given in Examples \ref{lb109} and \ref{lb111}.

By change of coordinates, we mean $\fk X$ is changed to $\fk X'=(C;x_\blt;\eta_\blt')$ with the same underlying compact Riemann surface $C$ and the same marked points $x_\blt$ as the original ones but different local coordiates at these marked points. As mentioned in \ref{lb16}, in this process, only $L_0,L_1,L_2,\dots$ (and also $\ovl L_0,\ovl L_1,\ovl L_2,\dots$) are involved, while in the change of boundary parametrizations, all $L_n$ are involved.

In the previous discussions, almost all formulas hold only up to scalar multiplications or additions. However, when only $L_{-1}, L_0, L_1, L_2,\dots$ are involved, the interaction maps $T_\Sigma$ can indeed be chosen such that all the formulas truly hold, not just up to scalar multiplications or additions. This is because the conformal anomaly is due to the central term $c\cdot (m^3-m)\delta_{m,-n}/12$ in the Virasoro relation \eqref{eq16}, which vanishes when $m,n\geq -1$. Note that $L_{-1}$ is responsible for translation. Thus:

\begin{prin}\label{lb19}
$T_\Sigma$ can be chosen to have no ambiguity when changing the local coordinates, or when translating a marked point $x_i$ with respect to its local coordinate $\eta_i$. 

To be more precise: We fix a compact Riemann surface $C$. Then for each choice of $N$ marked points $x_\blt$ and local coordinates $\eta_\blt$, we can choose the correlation function $T_{\fk X}:\mc H^{\otimes N}\rightarrow\Cbb$ \index{T@$T_{\Sigma},T_\fk X$: The interaction map/correlation function} associated to the boundary parametrized surface associated to $\fk X=(C;x_\blt;\eta_\blt)$ such that
\begin{itemize}
\item For another choice of $N$-pointed $\fk X'=(C;x_\blt;\eta_\blt')$ with the same marked points and different local coordinates $\eta_\blt'$, $T_{\fk X}$ and $T_{\fk X'}$ are related by \eqref{eq12}.
\item If $\fk X'=(C;x_1',x_2,\dots,x_N;\eta_1',\eta_2,\dots,\eta_N)$ where $\eta_1'=\eta_1-\eta_1(x_1')$, and if $x_1'$ is inside an open  disk $U_1$ centered at $x_1$ on which $\eta_1$ is holomorphically defined (more precisely, this means $\eta_1(U_1)$ is an open  disk centered at $\eta_1(x_1)=0$), then $T_{\fk X}$ and $T_{\fk X'}$ are related by \eqref{eq12}, namely, (noticing \eqref{eq17} for $n=1$)
\begin{align}
T_{\fk X}(\xi_1\otimes\cdots\otimes \xi_N)=T_{\fk X'}\Big(\big(e^{-\eta_1(x_1')L_{-1}}\otimes e^{-\ovl{\eta_1(x_1')}\cdot \ovl L_{-1}}\big)\xi_1\otimes\xi_2\otimes\cdots\otimes\xi_N \Big).	
\end{align}
\end{itemize}
A similar principle holds when $T_{\fk X}$ has output strings.  \hfill$\qedsymbol$
\end{prin}

Recall the geometric picture described in \ref{lb15}. We see that when changing local coordinates, everything in \ref{lb15} truly holds, not just up to scalar multiplications. In particular, the abstract vector $\wtd\xi_1$ is uniquely determined when only the change of local coordinates are allowed.

\subsection{}\label{lb25}

\begin{ass}
We drop Assumption \ref{lb11} for the incoming strings when we associate only finite energy vectors (i.e., vectors of $\Wbb_i\otimes\wht\Wbb_i$, $\Vbb\otimes\wht\Vbb$, etc.) to the incoming strings. Instead, we only assume that the (distinct) incoming points are outside the outgoing strings.
\end{ass}

In this course, we will be mainly interested in finite energy vectors. Therefore, we do not assume that that each $\eta_i(U_i)$ contains $\Dbb^\cl_1$, or that $U_i$ and $U_j$ are disjoint for different $i$ and $j$. In the latter case, the two boundary strings $\eta_i^{-1}(\Sbb^1)$ and $\eta_j^{-1}(\Sbb^1)$ possibly overlap. What does this picture actually mean?

Note that multiplying $\eta_i$ by $\lambda\eta_i$ amounts to shrinking the size of the string $\eta_i^{-1}(\Sbb^1)$ by $|\lambda|$ and then rotating the string. If $\lambda>0$ then there is only shrinking but not rotating. Thus, \emph{for an local coordinated $N$-pointed $\fk X=(C;x_\blt;\eta_\blt)$, we can find  $\lambda_1,\dots,\lambda_N\in\Cbb^\times$ with large enough absolute values such that the new data $\fk X'=(C;x_\blt;\lambda_1\eta_1,\dots,\lambda_N\eta_N)$ satisfies Assumption \ref{lb11}.} Then for finite energy vectors $\xi_1,\dots,\xi_N\in\mc H^\fin=\bigoplus_i \Wbb_i\otimes\wht\Wbb_i$, $T_{\fk X}(\xi_1\otimes\cdots\otimes\xi_N)$ is understood as
\begin{align}
T_{\fk X}(\xi_1\otimes\cdots\otimes\xi_N):=T_{\fk X'}\Big(\big(\lambda_1^{L_0}\otimes\ovl{\lambda_1}^{\ovl L_0}\big)\xi_1\otimes\cdots\otimes \big(\lambda_N^{L_0}\otimes\ovl{\lambda_N}^{\ovl L_0}\big)\xi_N \Big).	
\end{align}
This definition is independent of the choice of sufficiently large $\lambda_1,\dots,\lambda_N$. And each $\lambda_j^{L_0}\otimes\ovl{\lambda_j}^{\ovl L_0}$ acts diagonally on $\mc H^\fin$ since $L_0\otimes\ovl L_0$ does. (Recall Assumption \ref{lb18}.)
\begin{align*}
	\vcenter{\hbox{{
			\includegraphics[height=2.3cm]{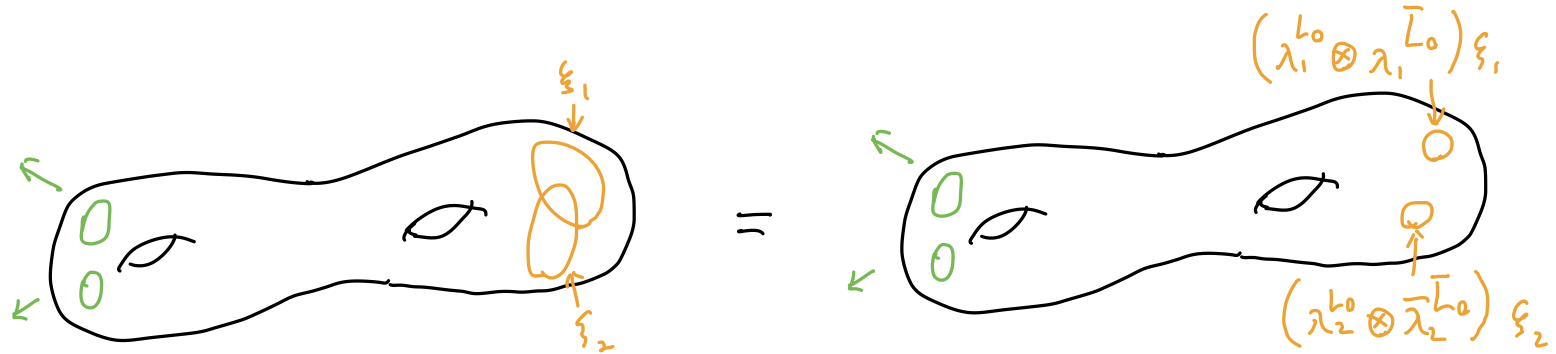}}}}	
\end{align*}

In the spirit of the previous subsection, you should view the finite energy vectors $\xi_j$ and $\big(\lambda_j^{L_0}\otimes\ovl{\lambda_j}^{\ovl L_0}\big)\xi_j$ not as different vectors, but as two coordinate representations of the same vector $\wtd\xi_j$. When $|\lambda_j|$ becomes infinitely large, the string for $\xi_j$ shrinks to an infinitesimal one around $x_j$, i.e., it shrinks to $x_j$ as a \textbf{puncture}. It is very useful to view the abstract finite energy vector $\wtd\xi_j$ not associated to any particular string, but associated to that puncture $x_j$. Thus, the marked points $x_\blt$ of $\fk X$ are also called punctures.

\begin{rem}
A side note: When we do local coordinate changes, finite energy vectors are changed to finite energy ones. 
\end{rem}
Therefore, in the above discussion, we don't have to stick to change of coordinates of the form $\eta_j\mapsto \lambda_j\eta_j$: any local coordinate change is valid. We will prove the above claim in later sections.

\subsection{}

Let us choose $\Wbb_i\otimes\wht\Wbb_i$ inside $\mc H^\fin$. According to Assumption \ref{lb18}, the eigenvalues of the diagonalizable operators $L_0$ (on $\Wbb_i$) and $\ovl L_0$ (on $\wht\Wbb_i$) are $\geq0$. Now choose eigenvectors $w\in\Wbb_i$ and $\wht w\in\wht\Wbb_i$ with $L_0w=\Delta w, \ovl L_0\wht w=\wht\Delta\wht w$ where $\Delta,\wht\Delta\geq0$.

Here is an important point about the two eigenvalues. They are not necessarily integers, which means that $\lambda^{L_0} w$ and $\ovl\lambda^{\ovl L_0}\wht w$ might be \emph{multivalued with respect to $\lambda$}, i.e., they may also depend on the choice of argument $\arg\lambda$. However, according to the No-Ambiguity Principle \ref{lb19}, the expression
\begin{align*}
	\big(\lambda^{L_0}\otimes\lambda^{\ovl L_0}\big)(w\otimes\wht w)=\lambda^\Delta\ovl\lambda^{\wht\Delta}\cdot w\otimes\wht w
\end{align*}
must be single-valued with respect to $\lambda$, namely, it does not rely on the choice of $\arg\lambda$. As $\lambda=|\lambda|e^{\im\arg\lambda}$ and hence $\lambda^\Delta\ovl\lambda^{\wht\Delta}=|\lambda|^{\Delta+\wht\Delta}e^{\im(\Delta-\wht\Delta)\arg\lambda}$, we conclude that
\begin{align}
\Delta-\wht\Delta\in\Zbb.	\label{eq28}
\end{align}
This gives a constraint on the possible $\Vbb\otimes\wht\Vbb$-submodules of $\mc H^\fin$.

That $\lambda^{L_0}w$ could be multivalued is a crucial property in CFT, and it is not related to conformal anomaly. Indeed, it is related to the non-uniqueness of decomposing $T_\Sigma$ into conformal blocks. Thus, \emph{the No-Ambiguity Principle \ref{lb19} does not hold for conformal blocks}.

\section{Definition of VOAs, I}\label{lb42}

\subsection{}

We first give the rigorous definition of vertex operators algebras and a slightly weaker version, graded vertex algebras. Then we explain the meanings of the axioms.

\begin{df}\label{lb24}
A \textbf{graded vertex algebra} is a (complex) vector space $\Vbb$ together with a diagonalizable operator $L_0$ acting on $\Vbb$ whose eigenvalues are inside $\Nbb$. We write the $L_0$-grading of $\Vbb$ as $\Vbb=\bigoplus_{n\in\Nbb}\Vbb(n)$. (Note: Starting from Sec. \ref{lb155}, we will assume that all $\Vbb(n)$ are finite-dimensional.) Any eigenvector $v$ of $L_0$ (including $0$) is called ($L_0$)-\textbf{homogeneous}, and if $v\in \Vbb(n)$ (i.e. $L_0v=nv$), we write $\wt v=n$ \index{vw@$\wt v,\wtd\wt w$} and call $\wt v$ the \textbf{weight} of $v$. Moreover, we have a linear map
\begin{gather}
\begin{gathered}
\Vbb\rightarrow\big(\End(\Vbb)\big)[[z^{\pm1}]]\\
u\mapsto Y(u,z)\equiv\sum_{n\in\Zbb}Y(u)_nz^{-n-1}
\end{gathered}	
\end{gather}
where each $Y(u)_n\in\End(\Vbb)$ is called a \textbf{(Fourier) mode}.\index{Y@$Y(u)_n$} Here, $z$ is treated as a formal variable. Thus $Y(u,z)v\in\Vbb[[z^{\pm1}]]$ for each $v\in\Vbb$. The reason for associating $z^{-n-1}$ to $Y(u)_n$ is because we could have (recalling \eqref{eq18})
\begin{align}
\Res_{z=0}~Y(u,z)z^ndz=Y(u)_n.	
\end{align}
$Y(u,z)$ is called a \textbf{vertex operator}. 

Moreover, the following axioms are satisfied:
\begin{itemize}
\item There is a distinguished vector $\id\in\Vbb(0)$ called \textbf{vacuum vector} such that
\begin{align*}
	Y(\id,z)=\id_\Vbb.
\end{align*}
Namely $Y(\id)_{-1}=\id_\Vbb$ and $Y(\id)_n=0$ if $n\neq -1$.
\item \textbf{Creation property}: For each $v\in\Vbb$, $Y(v,z)\id=v+\blt z+\blt z^2+\cdots$ where each $\blt$ is in $\Vbb$. Namely,
\begin{align}
Y(v)_{-1}\id =v,	
\end{align}
and $Y(v)_{n}\id=0$ for all $n>-1$. This property is abbreviated to
\begin{align*}
\lim_{z\rightarrow0}Y(v,z)\id=v.	
\end{align*}
\item \textbf{Grading property}: For each $v\in\Vbb$,
\begin{align}
[L_0,Y(v,z)]=Y(L_0v,z)+z\frac d{dz}Y(v,z).	\label{eq19}
\end{align}
\item \textbf{Translation property}: There is a distinguished linear operator $L_{-1}$ on $\Vbb$  such that 
\begin{align}
L_{-1}\id=0,	
\end{align}
and that for each $v\in\Vbb$,
\begin{align}
[L_{-1},Y(v,z)]=\frac d{dz}Y(v,z).
\end{align}
\item \textbf{Jacobi identity}: This is the most crucial yet complicated axiom. We postpone its definition to the next section. (See Def. \ref{lb36}.)
\end{itemize}

We say that $\Vbb$ is a \textbf{vertex operator algebra} (VOA) if $L_0,L_{-1}$ can be extended to a sequence of linear operators $(L_n)_{n\in\Zbb}$ on $\Vbb$ satisfying the Virasoro relation \eqref{eq16} for some central charge $c\in\Cbb$, and if there is a distinguished vector $\cbf\in\Vbb$, called the \textbf{conformal vector}, such that
\begin{align}
Y(\cbf)_n=L_{n-1},\label{eq22}
\end{align}
or equivalently,
\begin{align}
Y(\cbf,z)=\sum_{n\in\Zbb} L_nz^{-n-2}.	
\end{align}
\hfill\qedsymbol
\end{df}

You may wonder why the right hand side of \eqref{eq22} is not $L_n$ or $L_{n-a}$ for some constant $a\neq 1$. Indeed, if it were not $L_{n-1}$, then the Virasoro relation would be not compatible with the Jacobi identity. We will explain this in more details after defining the Jacobi identity. (See Exercise \ref{lb37}.)

We warn the readers that our definitions of graded vertex algebras and VOAs are slightly stronger than the usual ones in the VOA literature, which do not require  $L_0$ to have non-negative eigenvalues. This positivity condition $L_0\geq 0$ is very mild and satisfied by most interesting examples including all unitary ones. Since assuming this condition will simply proofs, we keep it in our definition.

Also, in most interesting cases, each $\Vbb(n)$ is finite-dimensional. We do not include this in our definition of VOA here, but we will assume this fact from Sec. \ref{lb155}.

Most VOA textbooks and articles use either $\omega$ or $\nu$ to denote the conformal vector $\cbf$. In our notes, $\omega$ and $\nu$ are reserved for other meanings and hence do not denote conformal vectors in order to avoid conflicts of notations. 

The reason why we should assume that $\sum L_nz^{-n-2}$ can be written as $Y(\cbf,z)$ for some $\cbf\in\Vbb$ will not be explained in this section. We will explain it in Subsec. \ref{lb43}.

There is a notion of \textbf{unitary VOA} which we do not define in this course (although our motivations are mainly from unitary CFTs). We refer the readers to \cite{CKLW18,DL14} for details.

\subsection{}

Before we give the motivations for these axioms, let us first derive some useful facts.

Expand the series \eqref{eq19} and take the coefficients before each $z^{-n-1}$. This gives us  the following equivalent form of grading property:
\begin{align}
[L_0,Y(v)_n]=Y(L_0v)_n-(n+1)Y(v)_n.	\label{eq20}
\end{align}
To be more concrete, assuming that $v$ is homogeneous, then
\begin{align}
[L_0,Y(v)_n]=(\wt v-n-1)Y(v)_n.	\label{eq21}
\end{align}
Namely: \emph{$Y(v)_n$ raises the weights by $\wt v-n-1$}. It is useful to keep in mind that in the VOA theory, $Y(v)_n$ raises weights when $n$ is sufficiently negative, and lowers weights when $n$ is sufficiently positive. As a related fact, as
\begin{align}\label{eq57}
[L_0,L_n]=-nL_n	
\end{align}
by the Virasoro relation \eqref{eq16}, $L_{-n}$ raises (resp. $L_n$ lowers) the weights by $n$. 

\begin{rem}\label{lb80}
As an application of \eqref{eq57}, we compute $L_n\cbf$ when $n\geq0$. Since
\begin{align}
	\cbf=Y(\cbf)_{-1}\id=L_{-2}\id,	
\end{align}
and since $L_{-2}$ raises the weights by $2$, we see that
\begin{align}
	L_0\cbf=2\cbf.
\end{align}	
By $[L_1,L_{-2}]=3L_{-1}$, $[L_2,L_{-2}]=4L_0+\frac 12 c$, and that $L_n\id=0$ whenever $n>0$ (since its weight is $<0$), we have
\begin{align}
L_1\cbf=0,\qquad L_2\cbf=\frac c2\id.	
\end{align}
\end{rem}

\subsection{}

By \eqref{eq21}, for each $u,v\in\Vbb$, we know that $Y(u)_nv$ vanishes when $n$ is sufficiently large. Equivalently, we have
\begin{align}\label{eq46}
Y(u,z)v\in\Cbb((z)).	
\end{align}
This important fact is called the \textbf{lower truncation property}. It allows us to use meromorphic functions to study VOAs.

In the definition of graded vertex algebras, if  the grading property is replaced by the lower truncation property, and if in particular the diagonalizable $L_0$ is not introduced, then $\Vbb$ is called a \textbf{vertex algebra}. We will not address this most general notion in our notes.

\subsection{}

We let \index{VW@$\Vbb',\Wbb'$, the graded dual spaces}
\begin{align*}
	\Vbb'=\bigoplus_{n\in\Nbb} \Vbb(n)^*
\end{align*}
where $\Vbb(n)^*$ is the dual space of $\Vbb$. $\Vbb'$ is called the \textbf{graded dual space} of $\Vbb$. We let $L_0$ act on $\Vbb'$ such that $L_0v'=nv'$ whenever $v'\in\Vbb(n)^*$. Then $L_0^\tr=L_0$. As before, a \textbf{homogeneous} vector of $\Vbb'$ is either $0$ or an eigenvector of $L_0$. From our definition, it is clear that the evaluation between $\Vbb'(m)=\Vbb(m)^*$ and $\Vbb(n)$ vanishes if $m\neq n$.

\begin{pp}\label{lb29}
For each $u,v\in\Vbb,v'\in\Vbb'$, $\bk{v',Y(u,z)v}:=\sum_{n\in\Zbb}\bk{v',Y(u)_nv}z^{-n-1}$ is a \textbf{Laurent polynomial} of $z$, i.e.,
\begin{align*}
\bk{v',Y(u,z)v}\in\Cbb[z^{\pm1}].	
\end{align*}
\end{pp}

Thus, when evaluating between \textbf{finite energy vectors} (i.e., vectors of $\Vbb$ and $\Vbb'$), $Y(u,z)$ is not only a formal series, but a meromorphic function of $\Pbb^1$ with poles at $0,\infty$.

\begin{proof}
We must show that $\sum_{n\in\Zbb}\bk{v',Y(u)_nv}z^{-n-1}$ is a finite sum. By linearity, it suffices to assume that $u,v,v'$ are homogeneous. Then $Y(u)_nv$ is homogeneous with weight $\wt u+\wt v-n-1$. So $\bk{v',Y(u)_nv}$ is non-zero only if $\wt v'=\wt u+\wt v-n-1$. Thus
\begin{align*}
\bk{v',Y(u,z)v}=\bk{v',Y(u)_{\wt u+\wt v-\wt v'-1}\cdot v}\cdot z^{\wt v'-\wt u-\wt v}.	
\end{align*}
\end{proof}

\begin{rem}\label{lb30}
The formula $\lim_{z\rightarrow 0}Y(u,z)\id$ can now be understood in an analytic sense: By the creation property, for each $v'\in\Vbb$, $\bk{v',Y(u,z)\id}$ is a polynomial of $z$ since it has no negative powers of $z$. So 
\begin{align}
	\lim_{z\rightarrow 0}\bk{v',Y(u,z)\id}=\bk{v',u}	\label{eq38}
\end{align}
where the left hand side is the limit of a polynomial function.
\end{rem}

\subsection{}
The grading  and the translation properties were presented in the ``derivative form". We shall present them in the integral form. To prepare for this task, we introduce \index{VW@$\Vbb^\cl,\Wbb^\cl$, the algebraic completions}
\begin{align}
	\Vbb^\cl:=\prod_{n\in\Nbb}\Vbb(n)=\big\{(v_0,v_1,v_2,\dots):v_n\in\Vbb(n)\big\},	
\end{align}
called the \textbf{algebraic completion} of $\Vbb$. $\Vbb^\cl$ is a naturally a subspace of the dual space $(\Vbb')^*$ of $\Vbb'$. (Indeed, we are mostly interested in the case that each $\Vbb(n)$ is finite dimensional. In such case, one checks easily that $\Vbb^\cl=(\Vbb')^*$.) We let \index{Pn@$P_n$}
\begin{align}
P_n:\Vbb^\cl\rightarrow \Vbb(n),\qquad (v_0,v_1,v_2,\dots)\mapsto v_n\label{eq89}	
\end{align}
be the canonical projection onto the $n$-th component. Then for each $z\in\Cbb^\times=\Cbb\setminus\{0\}$, we have
\begin{align*}
	Y(u,z)v\in\Vbb^\cl
\end{align*} 
whose projection onto $\Vbb(\wt u+\wt v-n-1)$ is $Y(u)_nv\cdot z^{-n-1}$. 

Note that $L_0$ and $\lambda^{L_0}$ act on $\Vbb^\cl$ in an  obvious way:
\begin{gather*}
L_0(v_n)_{n\in\Nbb}=(nv_n)_{n\in\Nbb},\qquad \lambda^{L_0}	(v_n)_{n\in\Nbb}=(\lambda^n v_n)_{n\in\Nbb}.
\end{gather*}

\subsection{}

\begin{pp}[\textbf{Scale covariance}]\label{lb63}
For each $\lambda\in\Cbb^\times$, we have
\begin{align}
\lambda^{L_0}Y(u,z)\lambda^{-L_0}v=Y(\lambda^{L_0}u,\lambda z)v\label{eq23}
\end{align}
on the level of $\Vbb^\cl$. We drop the symbol $v$ and simply write the above relation as
\begin{align*}
\lambda^{L_0}Y(u,z)\lambda^{-L_0}=Y(\lambda^{L_0}u,\lambda z).
\end{align*}
\end{pp}

The method in the following proof will appear repeatedly in our notes. 

\begin{proof}
Recall $L_0^\tr=L_0$. Fix $z\in\Cbb^\times$. We prove that for each homogeneous $u,v,v'$,
\begin{align}
\bk{\lambda^{L_0}v',Y(u,z)\lambda^{-L_0}v}	=\bk{v',Y(\lambda^{L_0}u,\lambda z)v}.
\end{align}	
The left hand side $f$ is a scalar times $\lambda^{\wt v'-\wt v}$, and the right hand side $g$ is a Laurent polynomial of $\lambda$. So both are holomorphic functions on $\Cbb^\times$. Clearly these two expressions are equal when $\lambda=1$. Let us prove that they are equal for all $\lambda\neq 0$ by showing that they satisfy the same differential equation.

From the form of $f$, it is clear that $\partial_\lambda f(\lambda)=(\wt v'-\wt v)\lambda^{-1}f(\lambda)$. To compute $\partial_\lambda g$, we first  compute an easier derivative $\partial_\lambda\bk {v',Y(u,\lambda z)v}$. By the chain rule, we have
\begin{align*}
	\frac\partial{\partial\lambda} \bk{v',Y(u,\lambda z)v}=	z\frac d{d\zeta}\bk{v',Y(u,\zeta)v}\Big|_{\zeta=\lambda z},
\end{align*}
which, due to the grading property, equals 
\begin{align*}
&\lambda^{-1}\Bigbk{v',\big([L_0,Y(u,\lambda z)]-Y(L_0u,\lambda z)\big)v}\\
=&(\wt v'-\wt v-\wt u)\lambda^{-1}\bigbk{v',Y(u,\lambda z)v}.
\end{align*}
So
\begin{align*}
\partial_\lambda g(\lambda)=\partial_\lambda \bigbk{v',Y(\lambda^{L_0}u,\lambda z)v}=	\partial_\lambda \big(\lambda^{\wt u}\bigbk{v',Y(u,\lambda z)v}\big)=(\wt v'-\wt v)\lambda^{-1}g(\lambda).
\end{align*}
\end{proof}

Informally, the integral form \eqref{eq23} (i.e., the scale covariance) also implies the derivative form \eqref{eq20} by taking partial derivative over $\lambda$. Thus, on a non-rigorous level, these two forms are equivalent. But the integral form has a clearer geometric meaning, which we shall give later.

In the above proof, we have done our first serious VOA calculation. You should be so familiar with these computations that you can ``immediately see" the equivalence of the two forms.

The integral form of $[L_{-1},Y(u,z)]=\partial_z Y(u,z)$ is
\begin{align*}
e^{\tau L_{-1}}Y(u,z)e^{-\tau L_{-1}}=Y(u,z+\tau),
\end{align*}
called the \textbf{translation covariance}. You may give an informal proof yourself by checking that both sides satisfy the same ``linear differential equation". A rigorous treatment is more difficult than the scale covariance. So we leave it to the end of this section.

\subsection{}

We now explain the motivations behind the definition of VOAs. Namely, we shall explain how the axioms are natural assumptions from the point of view of the previous two sections. The following explanations are heuristic and non-rigorous.

Recall the non-rigorous ``definition" of $\Vbb$ in \eqref{eq2}. We know that $\Vbb$ and $\wht\Vbb$ are subspaces of $\mc H^\fin$, and the decomposition of $\mc H^\fin$ into $\Vbb\otimes\wht\Vbb$-submodules contains a piece $\Vbb\otimes\wht\Vbb$, which furthermore contains $\Vbb\simeq\Vbb\otimes\id$ and $\wht\Vbb\simeq\id\otimes\wht\Vbb$. The vacuum vector is $\id\simeq\id\otimes\id$.

We have said in Subsection \ref{lb23} that the standard unit closed  disk $\Dbb_1^\cl$ with no input and whose boundary $\Sbb^1$ is parametrized by $z\mapsto z^{-1}$ produces from nothing the vacuum vector $\id\otimes\id$. Namely, the vacuum vector comes from the data $(\Pbb^1;\infty;\zeta^{-1})$ where $\zeta$ is the standard coordinate. This data is equivalent to $(\Pbb^1;\infty;\lambda^{-1}\zeta^{-1})$ (where $\lambda\in\Cbb^\times$) via the biholomorphism $z\in\Pbb^1\mapsto \lambda z\in\Pbb^1$. By the change of local coordinate formula (Principle \ref{lb19}), the later geometric data produces uniquely the vector $(\lambda^{L_0}\otimes {\ovl\lambda}^{\ovl L_0})\id$, which is equal to $\id$ by the equivalence of the two geometric data. Apply $\partial_{\lambda}$ and $\partial_{\ovl\lambda}$ to $(\lambda^{L_0}\otimes {\ovl\lambda}^{\ovl L_0})\id=\id$, we see that $L_0\id=\ovl L_0\id=0$. This explain $\id\in\Vbb(0)$ in Def. \ref{lb24}.
\begin{align*}
	\vcenter{\hbox{{
				\includegraphics[height=1.5cm]{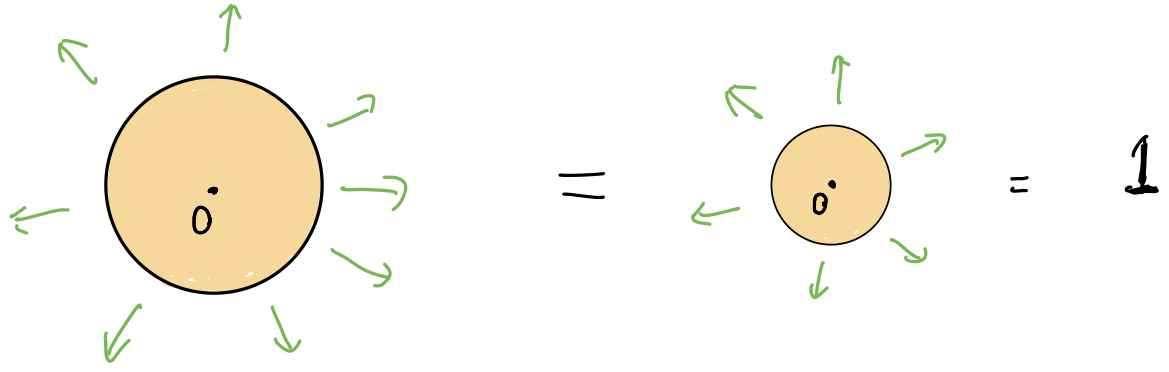}}}}	
\end{align*}
Consequently, by \eqref{eq28}, the eigenvalues of $L_0$ are integers, and hence $\geq0$ integers by the positive energy Assumption \ref{lb18}. This explains $\mathrm{Spec}(L_0)\subset\Nbb$.

Similarly, the standard  disk $\Dbb_1^\cl$ is equivalent to its translation by some $\tau\in\Cbb$. So we must have $(e^{\tau L_{-1}}\otimes e^{\ovl \tau\ovl L_{-1}})\id=\id$ and hence, similarly, $L_{-1}\id=\ovl L_{-1}\id=0$. This explains part of the translation property.
\begin{align*}
	\vcenter{\hbox{{
				\includegraphics[height=1.6cm]{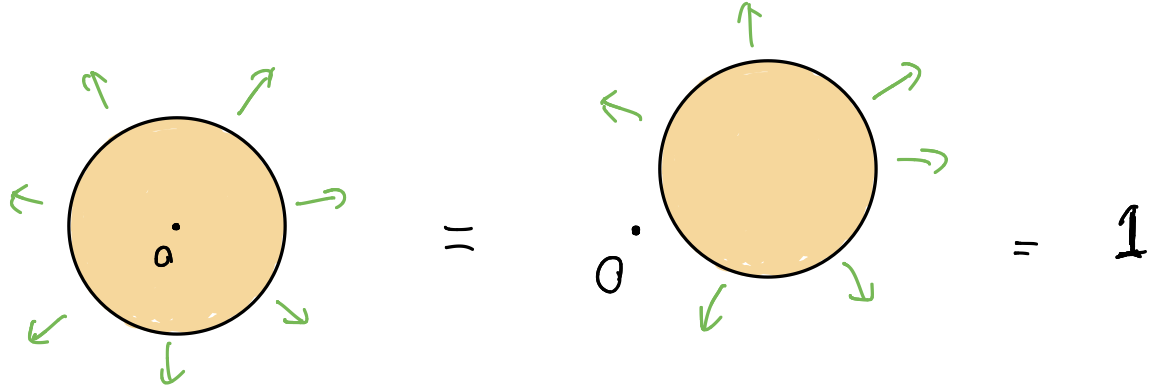}}}}	
\end{align*}

\subsection{}

Recall
\begin{align}
[L_0,L_n]=-nL_n,\qquad [\ovl L_0,\ovl L_n]=-n\ovl L_n.	
\end{align}
As the $L_0$ and $\ovl{L_0}$ spectral are $\geq 0$, and since $\id$ is a zero eigenvectors of them, we must have
\begin{align}
L_n\id=\ovl L_n\id=0\qquad(n\geq -1).\label{eq31}
\end{align}

From \eqref{eq31}, we see that for each $v\in\Vbb$, if the change of boundary parametrization does not involve $L_{-2},L_{-3},\dots$ and $\ovl L_{-2},\ovl L_{-3},\dots$, then all $\ovl L_n$ can be ignored:
\begin{align}
\big(e^{\sum_{n\geq-1}a_n L_n}\otimes  e^{\sum_{n\geq-1}\ovl{a_n} \ovl L_n}\big)v=	e^{\sum_{n\geq-1}a_n L_n}v.\label{eq32}
\end{align}
To see this, identify $v$ with $v\otimes\id\in\Vbb\otimes\wht\Vbb\subset\mc H$ and note that $\id$ is fixed by $e^{\sum_{n\geq-1}\ovl{a_n} \ovl L_n}$.

Thus, we conclude: \emph{The translation of the change of local coordinates formula for vectors of $\Vbb$ does not involve $\ovl L_n$.} In particular, note that the right hand side of \eqref{eq32} is almost a vector of $\Vbb$. It is a genuine vector of $\Vbb$ when it has finite energy. Thus, \emph{the change of local coordinates and the translation almost preserve $\Vbb$}. Indeed, the change of local coordinates  truly preserve $\Vbb$, as we will see in later sections.

A general change of \emph{boundary parametrization} does not necessarily preserve $\Vbb$ in any weak sense.

\subsection{}

Let us describe the meaning of $Y(u,z)v$. For each $z\in\Cbb^\times$, we define a local-coordinated $3$-pointed sphere
\begin{align}
\fk P_z=\{\Pbb^1;0,z,\infty;\zeta,\zeta-z,\zeta^{-1}\}	\label{eq34}
\end{align}
where $\zeta$ is the standard coordinate of $\Cbb$. 

Let us regard $0,z$ as incoming punctures and $\infty$ outgoing. Roughly speaking, $Y(u,z)v$ is just $T_{\fk P_z}(v\otimes u)$ where $v$ is associate to $0$ and $u$ to $z$, understood in a suitable way by change of coordinates. Assume first of all that $0<|z|<1$. After scaling $\zeta$ and $\zeta-z$ to $\lambda_1\zeta,\lambda_2(\zeta-z)$ and hence shrinking the two incoming strings, Assumption \ref{lb11} is satisfied. Let the new $N$-pointed sphere by denoted by $\fk P_z^{\lambda_1,\lambda_2}$. Note that $v$ in the $\zeta$ coordinate becomes $(\lambda_1^{L_0}\otimes\ovl{\lambda_1}^{\ovl L_0})v=\lambda_1^{L_0}v$ in the $\lambda_1\zeta$ coordinate. Similarly, $u$ becomes $\lambda_2^{L_0}u$ in the new coordinate. Then $Y(u,z)v$ is (physically) defined as $T_{\fk P_z^{\lambda_1,\lambda_2}}(\lambda_1^{L_0} v\otimes \lambda_2^{L_0}u)$. 

As in Subsec. \ref{lb25}, we can use the \emph{puncture picture} to view $u$ and $v$ as the states associated to the punctures $0,z$ with respect to the local coordinates $\zeta,\zeta-z$. Or moreover, formulated in a coordinate independent way as in Subsec. \ref{lb15}, we associate the abstract vector $\mc U(\zeta)^{-1}v$ (the one whose explicit expression under the coordinate $\zeta$ is $v$) to the puncture $0$ and $\mc U(\zeta-z)^{-1}v$ to $z$. Then:
\begin{align}
	\vcenter{\hbox{{
				\includegraphics[height=3.5cm]{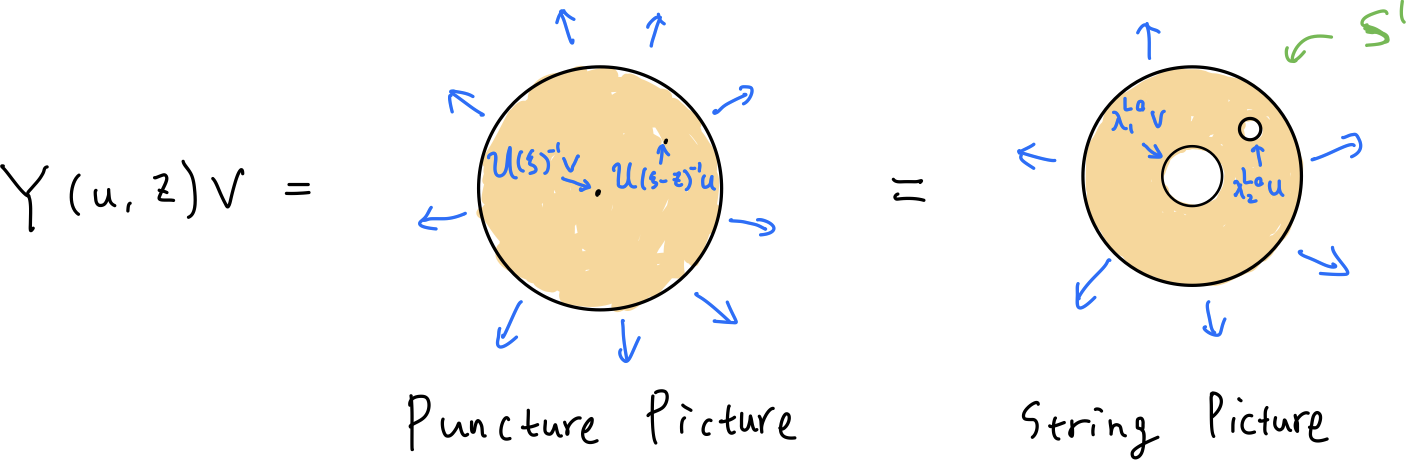}}}}	\label{eq35}
\end{align}
According to the notation in Subsec. \ref{lb15}, the abstract vectors should be written as $\big(\mc U(\zeta)\otimes \mc U(\zeta^*)\big)^{-1}v$ and $\big(\mc U(\zeta-z)\otimes \mc U((\zeta-z)^*)\big)^{-1}u$. Here we suppress the second tensor component because, by \eqref{eq32}, the change of local coordinates for vectors of $\Vbb$ does not involve $\ovl L_n$.

\subsection{}

In the string picture of \eqref{eq35}, setting $u$ to be $\id$ means filling the hole around $z$ using the solid disc. The result we get is an annulus $\mc A_{\lambda_1,1}$ with inside parametrization $\lambda_1\zeta$ and outside one $\zeta^{-1}$.\footnote{We have previously defined an annulus $A_{r,R}$ with incoming string $|z|=r$ and outgoing $|z|=R$. According to that definition, $\mc A_{r^{-1},1}=A_{r,1}$.} According to the change of coordinate formula, the interaction map $\mc H\rightarrow\mc H$ for this annulus satisfies $T_{\mc A_{\lambda_1,1}}(\lambda_1^{L_0}v)=T_{A_{1,1}}v=v$. This explains $Y(\id,z)v=v$.

If we set $v=\id$ instead, then we fill the hold around $0$ with the solid disc. The result we get is an eccentric annulus $\mc A_{z,\lambda_2,1}$ with inside boundary parametrization $\lambda_2(\zeta-z)$ and outside one $\zeta^{-1}$. Let $T_{\mc A_{z,\lambda_2,1}}:\mc H\rightarrow\mc H$ be the interaction map. Then, by \eqref{eq35}, $Y(u,z)\id=T_{\mc A_{z,\lambda_2,1}}(\lambda_2^{L_0}u)$. Let $z\rightarrow 0$. Then $\mc A_{z,\lambda_2,1}$ converges to $\mc A_{0,\lambda_2,1}$, which is just the concentric annulus $\mc A_{\lambda_2,1}$. We have $T_{\mc A_{\lambda_2,1}}(\lambda_2^{L_0}u)=u$. This explains $\lim_{z\rightarrow 0} Y(u,z)\id=u$.

\subsection{}

For a general $z\in\Cbb^\times$, in the string picture, we must also shrink the outgoing string in order to get a true surface $\Sigma$. We thus choose $\lambda\in\Cbb^\times$ with $|\lambda|>1$. Let
\begin{align*}
\fk P_z^{\lambda_1,\lambda_2,\lambda}=\{\Pbb^1;0,z,\infty;\lambda_1\zeta,\lambda_2(\zeta-z);\lambda\zeta^{-1}\}.	
\end{align*}
Then $Y(u,z)v$ is physically ``defined" to be
\begin{align}
Y(u,z)v=\lambda^{L_0}T_{\fk P_z^{\lambda_1,\lambda_2,\lambda}}(\lambda_1^{L_0}v\otimes\lambda_2^{L_0}u).	
\end{align}
In the puncture picture, it is
\begin{align*}
	\vcenter{\hbox{{
				\includegraphics[height=3.5cm]{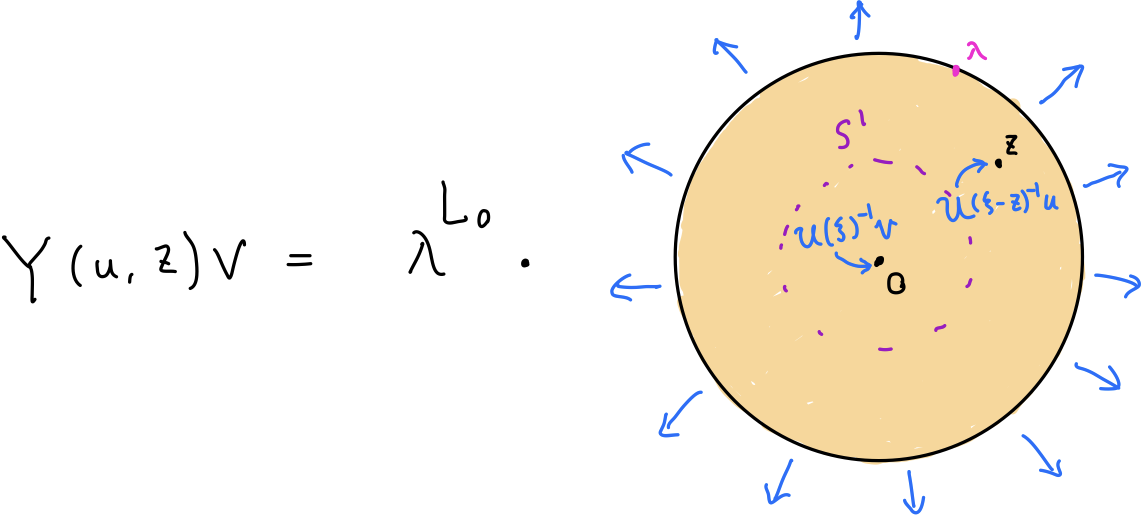}}}}	
\end{align*}
The meaning of $Y(u)_n$ is clear:
\begin{align*}
\bk{v',Y(u)_n v}=\oint_0 \bk{v',Y(u,z)v} z^n\cdot\frac{dz}{2\im\pi}=\Res_{z=0}~	\bk{v',Y(u,z)v} z^ndz.
\end{align*}
where the subscript under $\oint$ means that the integral is over any loop around $0$.

\subsection{}

If we prefer not to scale $\zeta^{-1}$, we can make the output point $\infty$ input. To do this, note that from Subsec. \ref{lb27} and \ref{lb26}, we know that each $\Theta(\Wbb_i\otimes\wht\Wbb_i)$ is equivalent to $\Wbb_i'\otimes\wht\Wbb_i'$, the space of finite energy dual vectors on $\Wbb_i\otimes\wht\Wbb_i$. In the case of $\Vbb$, we get an equivalence
\begin{align*}
\Theta\Vbb\xrightarrow{\simeq}\Vbb',\qquad \Theta w\mapsto\bk{\Theta w,\cdot}\big|_\Vbb=\bk{w|\cdot}\big|_\Vbb	
\end{align*}
where $\bk{\cdot,\cdot}$ is the correlation function associated to $A_{1,1}$. (From $\bk{w|\cdot}\big|_\Vbb$ you can see why this linear map is an isomorphism. Here, you may assume each $\Vbb(n)$ is finite-dimensional, or even pretend that $\Vbb$ is finite dimensional.) Then in the puncture picture, the vertex operator and the correlation function of $\fk P_z$ (restricted to a linear functional on $\Vbb\otimes\Vbb\otimes\Theta\Vbb\simeq \Vbb\otimes\Vbb\otimes\Vbb'$) are related by
\begin{align*}
\bk{\Theta w,Y(u,z)v}=\bk{w|Y(u,z)v}\xlongequal{\eqref{eq1}}		\vcenter{\hbox{{
			\includegraphics[height=2cm]{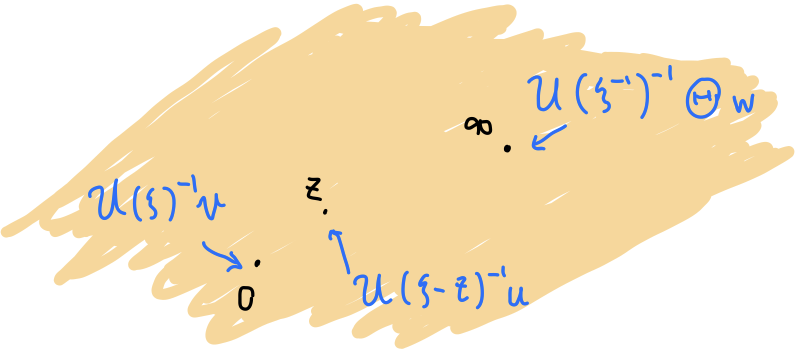}}}}	
\end{align*}
for all $u,v,w\in\Vbb$ and hence $\Theta w\in\Theta\Vbb\simeq\Vbb'$.

\subsection{}

We actually have
\begin{align}
	\Theta\Vbb=\Vbb
\end{align}
and similarly $\Theta\wht\Vbb=\wht\Vbb$. An explanation is as follows:

\begin{proof}
First of all, $\Theta$ maps finite energy vectors to finite energy ones since $\Theta$ commutes with the energy operators $L_0\otimes\id,\id\otimes\ovl L_0$. (See Subsec. \ref{lb28}.) By the physical definition of $\Vbb$ in \eqref{eq2}, for each $u\in\Vbb$, the correlation function $T_z$ associated to $\fk P_{z,r,R}=(\Pbb^1;z,\infty;(\zeta-z)/r,R/\zeta)$ varies holomorphically if $u$ is associated to the puncture $z$. Namely, $T_z(u\otimes\nu)$ is holomorphic for all $\nu\in\mc H$. It is easy to see that the conjugate of $\fk P_{z,r,R}$ is equivalent (via the standard conjugation of the complex plane) to $\fk P_{\ovl z,r,R}=(\Pbb^1;\ovl z,\infty;(\zeta-\ovl z)/r,R/\zeta)$, whose correlation function is $T_{\ovl z}$. Thus, by \eqref{eq30} 
\begin{align*}
	T_z(\Theta u\otimes \nu)=\ovl{T_{\ovl z}(u\otimes\Theta\nu)},	
\end{align*}
which is also holomorphic over $z$. This proves $\Theta u\in\Vbb$ if $u\in\Vbb$.
\end{proof}

Consequently, $\Vbb=\Theta\Vbb\simeq\Vbb'$. The equivalence is given by
\begin{align}
	\Vbb\xrightarrow{\simeq}\Vbb',\qquad u\mapsto \bk{u,\cdot}	
\end{align}
Due to this equivalence, we call the VOA $\Vbb$ to be \textbf{self-dual}.

So, in all unitary CFTs (and indeed, also in many non-unitary CFTs), the VOAs are self-dual. We remark that there is a mathematically rigorous definition of self-dualness, which plays an important role in the tensor categories of $\Vbb$-modules. However, the definition of a general VOA does not require self-dualness, because many properties can be derived without assuming self-dualness.

\subsection{}\label{lb31}

Let $\zeta$ be the standard coordinate of $\Cbb$ as usual. For each $\lambda\neq0$, we have an equivalence
\begin{align}
(\Pbb^1;0,z,\infty;\zeta,\zeta-z,\zeta^{-1})\simeq (\Pbb^1;0,\lambda z,\infty;\lambda^{-1}\zeta,\lambda^{-1}\zeta-z,\lambda\zeta^{-1})\label{eq36}
\end{align}
realized by the biholomorphism $\gamma\mapsto \lambda\gamma$ of $\Pbb^1$. (You should check that the pullback of the local coordinates on the right hand side equal those on the left.) The correlation function for the left hand side, evaluating on $v\otimes u\otimes w\in\Vbb^{\otimes 3}$, is $\bk{w,Y(u,z)v}$. The right hand side of \eqref{eq36} is obtained by scaling the local coordinates of $(\Pbb^1;0,\lambda z,\infty;\zeta,\zeta-\lambda z,\zeta^{-1})$ (whose correlation function on $\Vbb^{\otimes 3}$ takes the form $\bk{w,Y(u,\lambda z)v}$) by $\lambda^{-1},\lambda^{-1},\lambda$ respectively. By the change of coordinate formula, the correlation function for the right hand side of \eqref{eq36}, denoted temporarily by $\omega$, must satisfy
\begin{align*}
\bk{w,Y(u,\lambda z)v}=\omega(\lambda^{-L_0}v\otimes \lambda^{-L_0}u\otimes\lambda^{L_0}w),	
\end{align*}
namely, $\omega$ should be $\bk{\lambda^{-L_0}w,Y(\lambda^{L_0}u,\lambda z)\lambda^{L_0}v}$. This last equation must equal  $\bk{w,Y(u,z)v}$ due to the equivalence \eqref{eq36}. This explains the scale covariance.

\subsection{}\label{lb40}

Similarly, for each $\tau\in\Cbb$, consider the equivalence
\begin{align}
	(\Pbb^1;0,z,\infty;\zeta,\zeta-z,\zeta^{-1})\simeq \big(\Pbb^1;\tau,z+\tau,\infty;\zeta-\tau,\zeta-z-\tau,\frac 1{\zeta-\tau}\big)\label{eq37}
\end{align}
induced by the biholomorphism $\gamma\mapsto\gamma+\tau$ of $\Pbb^1$. The right hand side is a change of parametrization from $(\Pbb^1;0,z+\tau,\infty;\zeta,\zeta-z-\tau,\zeta^{-1})$ (whose correlation function is $\bk{w,Y(u,z+\tau)v}$), where $\zeta$ is changed to $\zeta-\tau$ (which is a translation), and $\eta:=\zeta^{-1}$ is changed to $1/(\eta^{-1}-\tau)$. The translation corresponds to $e^{-\tau L_{-1}}$. The second change of coordinate is $\exp(\tau z^{2}\partial_z)$ due to Ex. \ref{lb20}, which gives $e^{\tau L_1}$.

Let $\omega$ now be the correlation function (restricted to $\Vbb^{\otimes 3}$) of the right hand side. Then we have
\begin{align*}
\bk{w,Y(u,z+\tau)v}=\omega(e^{-\tau L_{-1}}v\otimes u\otimes e^{\tau L_1}w).
\end{align*} 
So $\omega$ is $\bk{e^{-\tau L_1}w,Y(u,z+\tau)e^{\tau L_1}v}=\bk{w,e^{-\tau L_{-1}}Y(u,z+\tau)e^{\tau L_1}v}$, which must equal $\bk{w,Y(u,z)v}$ due to the equivalence \eqref{eq37}. This explains the translation covariance.

\begin{exe}
Find a geometric explanation of $Y(u,z+\tau)=Y(e^{\tau L_{-1}}u,z)$.
\end{exe}

There is a another shorter geometric explanation of translation covariance: $e^{\tau L_{-1}}Y(u,z)v$ amounts to moving the outgoing large string in the string picture in \eqref{eq35} by $-\tau$. This is the same as fixing the outgoing string and translating the two incoming strings by $\tau$. Translating the one around $0$ changes $v$ to $e^{\tau L_{-1}}v$, and translating the one around $z$ just changes $z$ to $z+\tau$.

This second explanation is however less rigorous than the first one. But the first one is not  rigorous anyway. So why should we care about the issue of rigor here? Well, our first geometric explanation for translation covariance, as well as the one in Subsec. \ref{lb31} for rotation covariance, is much more rigorous in the sense that you can easily get the correct formulas using this method. You may try and give a short explanation for rotation covariance using our second method. Then you will realize that it is not easy to get the correct formula since the change of local coordinates is not so easy to visualize.

\subsection{}

Now we return to rigorous mathematics. We are going to prove translation covariance rigorously. For that purpose, we need to generalize the differential equation method in the proof of scale covariance to the following vector-valued form:

\begin{lm}\label{lb21}
Let $W$ be a (non-necessarily finite dimensional) vector space, and $f\in W[[z]]$. Suppose that $\frac d{dz}f(z)=Af(z)$ for some $A\in\End(W)$. Suppose also that $f(0)=0$, namely, the constant term in the power series $f(z)$ is $0$. Then $f=0$.	
\end{lm}

\begin{proof}
Write $f(z)=\sum_{n\in\Nbb}f_nz^n$ where each $f_n\in W$. The assumptions say that $f_0=0$ and
\begin{align*}
\sum_{n\in\Nbb} nf_nz^{n-1}=\sum_{n\in\Nbb} Af_nz^n.	
\end{align*}
So $nf_n=Af_{n-1}$ where $n>0$. This proves that all $f_n$ are $0$.
\end{proof}

\subsection{}

We have said that the integral form of $[L_{-1},Y(u,z)]=\partial_z Y(u,z)$ is
\begin{align}
\bigbk{v',e^{\tau L_{-1}}Y(u,z)e^{-\tau L_{-1}}v}=\bigbk{v',Y(u,z+\tau)v}.\label{eq24}	
\end{align}
This relation is more difficult to address than the scale covariance since both sides actually involve infinite sums of powers of $\tau$. Our goal is to understand: on which domain does this relation hold? Certainly we need $\tau\neq -z$. But this condition is far from enough.

Let us first understand the two sides as infinite series of $\tau$ and $z$. Assume without loss of generality that $u,v,v'$ are homogeneous. The right hand side is of the form $a(z+\tau)^m$ for some $a\in\Cbb,m\in\Zbb$. Certainly this expression makes sense as a rational function, but we shall first regard it as a formal series of $\tau,z$ by expanding $(z+\tau)^m$ on the domain $|\tau|<|z|$, namely $(z+\tau)^m=\sum_{k\in\Nbb}{m\choose k} z^{m-k}\tau^k$. Thus, the right hand side of \eqref{eq24}, as an element of $\Cbb[z^{\pm 1}][[\tau]]$, is understood as
\begin{align*}
\bigbk{v',Y(u,z+\tau)v}=\sum_{n\in\Zbb}\sum_{k\in\Nbb}{-n-1\choose k}	\bigbk{v',Y(u)_nv}\cdot z^{-n-1-k}\tau^k.
\end{align*}
Here, the sum over $n\in\Zbb$ is finite, and when the vectors are homogeneous, there is only one possibly non-zero summand.

But why do we expand $(z+\tau)^m$ on $|\tau|<|z|$? Why not $|z|<|\tau|$? Well, this will give us $\sum_{k\in\Nbb}{m\choose k}z^k\tau^{m-k}$ which contains negative powers of $\tau$. But the left hand side of \eqref{eq24} actually has only non-negative powers of $\tau$.

So let us turn to the left hand side of \eqref{eq24}. It would be easier to first understand why
\begin{align}
\bigbk{v',e^{\lambda L_{-1}}Y(u,z)e^{-\mu L_{-1}}v}	\label{eq25}
\end{align}
is an element of $\Cbb[z^{\pm1}][[\lambda,\mu]]$. We first want to move $e^{\lambda L_{-1}}$ to the left hand side of the bracket. In general, if $L_n$ is defined on $\Vbb$, we define $L_{-n}$ on $\Vbb'$ to be the transpose of $L_n$: $L_{-n}=L_n^\tr$, or more precisely,
\begin{align}
\bk{L_{-n}v',v}:=\bk{v',L_nv}.	
\end{align}
In case you doubt why this transpose exists, we can write the definition even more precisely: Assume $v'\in\Vbb'(m)$. Then $L_{-n}v'$ is a linear functional on $\Vbb(m+n)$ (so $L_{-n}$ raises the weights by $n$) whose value at any $v\in\Vbb(m+n)$ is $\bk{v',L_nv}$. (Recall that $L_n$ lowers the weights by $n$ so $L_nv\in\Vbb(m)$.) And $L_{-n}v'$ vanishes on $\Vbb(a)$ if $a\neq m+n$.

Now,  \eqref{eq25} equals
\begin{align}
f(z,\lambda,\mu):=\bigbk{e^{\lambda L_1}v',Y(u,z)e^{-\mu L_{-1}}v}=\sum_{n,l\in\Nbb}\frac{\lambda^n(-\mu)^l}{n!l!}\bigbk{L_1^nv',Y(u,z)L_{-1}^lv}.	
\end{align}
This is in $\Cbb[z^{\pm1}][[\lambda,\mu]]$. Indeed, it is in $\Cbb[z^{\pm1}][[\mu]][\lambda]$ since $L_1^nv'$ lowers the weight by $n$, and hence vanishes when $n>\wt v'$. But we will not need this fact here.

Now, the left hand side of \eqref{eq24} can be understood as $f(z,\tau,\tau)$, noting the following fact:
\begin{lm}
Let $W$ be a vector space. If $\varphi(z_1,\dots,z_N)\in W[[z_1,\dots,z_N]]$, then $\varphi(z,\dots,z)$ naturally makes sense as an element of $W[[z]]$.
\end{lm}
\begin{proof}
Write $\varphi(z_\blt)=\sum a_{n_1,\dots,n_N}z_1^{n_1}\cdots z_N^{n_N}$. Then
\begin{align*}
\varphi(z,\dots,z)=\sum_{n\in\Nbb} \sum_{n_1+\cdots+n_N=n}a_{n_1,\dots,n_N} z^n
\end{align*}
where the inside sum is clearly finite.
\end{proof}

\subsection{}

\begin{pp}[\textbf{Translation covariance}]\label{lb22}
For each $u,v\in\Vbb,v'\in\Vbb'$, the following equation holds on the level of $\Cbb[z^{\pm1}][[\tau]]$:
\begin{align}
\bigbk{v',e^{\tau L_{-1}}Y(u,z)e^{-\tau L_{-1}}v}=\bigbk{v',Y(u,z+\tau)v}.\label{eq26}	
\end{align}
Here, the right hand side, which is a priori a Laurent polynomial of $z+\tau$, is expanded as if $|\tau|<|z|$.
\end{pp}

\begin{proof}
Let $f_z(\tau)$ and $g_z(\tau)$ be the left and the right hand sides of \eqref{eq26}, considered as formal power series of $\tau$ whose coefficients are elements of $\Cbb[z^{\pm 1}]$. Then clearly $f_z(0)=g_z(0)$ as polynomials of $z^{\pm1}$. So, it suffices to prove that $f_z$ and $g_z$ satisfy the same linear differential equation. The left hand side is $f_z(\tau,\tau)$ where
\begin{align*}
f_z(\lambda,\mu)=\bigbk{e^{\lambda L_1}v',Y(u,z)e^{-\mu L_{-1}}v}\qquad\in\Cbb[z^{\pm1}][[\lambda,\mu]].
\end{align*}
As a general result about multivariable formal power series, we have chain rule
\begin{align*}
\partial_\tau f_z(\tau,\tau)=(\partial_\lambda+\partial_\mu)f_z(\lambda,\mu)\big|_{\lambda=\mu=\tau}.	
\end{align*}
(It is reasonable to believe that this is true. But you can also give a rigorous proof by expanding the two series  and check that their coefficients agree!) So, as
\begin{gather*}
\partial_\lambda f_z(\lambda,\mu)=\bigbk{e^{\lambda L_1}L_1v',Y(u,z)e^{-\mu L_{-1}}v},\\
\partial_\mu f_z(\lambda,\mu)=-\bigbk{e^{\lambda L_1}v',Y(u,z)e^{-\mu L_{-1}}L_{-1}v},	
\end{gather*}
we have
\begin{align*}
\partial_\tau f_z(\tau)=\bigbk{e^{\tau L_1}L_1v',Y(u,z)e^{-\tau L_{-1}}v}-\bigbk{e^{\tau L_1}v',Y(u,z)e^{-\tau L_{-1}}L_{-1}v}.
\end{align*}

This expression is not a differential equation of the $\Cbb[z^{\pm1}]$-coefficients power series $f_z$. But we can make it an ODE by fixing $u$, varying $v,v'$, and view $f_z$ as a $\mc V:=\Hom(\Vbb\otimes\Vbb',\Cbb[z^{\pm1}])$-valued power series of $\tau$. Then $\partial_\tau f_z=Af_z$ where $A\in\End\mc V$ is defined by sending each $\Phi:\Vbb\otimes\Vbb'\rightarrow\Cbb[z^{\pm1}]$ to
\begin{align*}
A\Phi=\Phi\circ (\id\otimes L_1-L_{-1}\otimes\id).
\end{align*}

Now, we compute (noting that the following sum is finite for each fixed $u,v$)
\begin{align*}
&\partial_\tau g_z(\tau)=\partial_\tau \bigbk{v',Y(u,z+\tau)v}=\partial_\tau\Big(\sum_n a_n(z+\tau)^n\Big)\\
=& 	\sum_n na_n(z+\tau)^{n-1}=\partial_\zeta\Big(\sum_n a_n\zeta^n\Big)\Big|_{\zeta=z+\tau}=\partial_\zeta \bigbk{v',Y(u,\zeta)v}\big|_{\zeta=z+\tau}.
\end{align*}
By the translation property, the above equals
\begin{align*}
\partial_\tau g_z(\tau)=	\bigbk{v',[L_{-1},Y(u,\zeta)]v}\big|_{\zeta=z+\tau},
\end{align*}
which also equals $Ag_z(\tau)$ if we now vary $v,v'$ and regard $g$ as $\mc V$-valued. Therefore, $f_z(\tau)=g_z(\tau)$ due to Lemma \ref{lb21}.
\end{proof}

\subsection{}

Let us consider a useful  variant of Prop. \ref{lb22}. Notice that  \eqref{eq26} holds if $v'$ is replaced by $L_1^n$ and also both sides are multiplied by $\tau^n$. Thus, \eqref{eq26} holds on the level of $\Cbb[z^{\pm1}][[\tau]]$ if $v'$ is replaced by $e^{-\tau L_1}v'$. Namely:
\begin{align}
	\bigbk{v',Y(u,z)e^{-\tau L_{-1}}v}=\bigbk{e^{-\tau L_1}v',Y(u,z+\tau)v}.\label{eq40}
\end{align}

\begin{rem}
The left hand sides of \eqref{eq26} and \eqref{eq40} converges absolutely when $|\tau|<|z|$ since the right hand side does. These right hand sides are  linear combinations of $(z+\tau)^m$ for some $m\in\Zbb$, whose expansion $\sum_{j,k\in\Zbb}a_{j,k}z^j\tau^k:=\sum_{n\in\Nbb}{m\choose n}z^{m-n}\tau^n$ clearly satisfies
\begin{align}
\sup_{(z,\tau)\in K}	\sum_{j,k\in\Zbb}|a_{j,k}z^j\tau^k|<+\infty
\end{align}
on every compact subset $K$ of $\{(z,\tau):|\tau|<|z|\}$. Thus, the same convergence property holds for the left hand sides of \eqref{eq26} and \eqref{eq40}. We call this property the \textbf{absolute and locally uniform convergence}, which will be the focus of our study in this course.
\end{rem}

Thus, we have actually proved our first convergence result in this course. The method used here is standard in the VOA theory: we show that a formal power series converges by identifying it with the power series expansion of a holomorphic function, which can be achieved with the help of linear differential equations.

\subsection{}

Let us choose $v=\id$ in the formula \eqref{eq26}. Then, as $L_{-1}\id=0$, we obtain
\begin{align}
	\bigbk{v',e^{\tau L_{-1}}Y(u,z)\id}=\bigbk{v',Y(u,z+\tau)\id}\label{eq39}
\end{align}
on the level of $\Cbb[z,\tau]$, since, by Rem. \ref{lb30},  the right hand side is a polynomial of $z+\tau$.  As $z\rightarrow 0$, the left hand side converges to $\bigbk{e^{\tau L_1}v',u}=\bigbk{v',e^{\tau L_{-1}}u}$ by \eqref{eq38}. So we conclude:

\begin{co}\label{lb78}
For each $u\in\Vbb,v\in\Vbb'$, the equation
\begin{align*}
\bigbk{v',e^{\tau L_{-1}}u}=\bigbk{v',Y(u,\tau)\id}	
\end{align*}
holds as polynomials of $\tau$. Equivalently, the equation
\begin{align*}
e^{\tau L_{-1}}u=Y(u,\tau)\id	
\end{align*}
holds on the level of $\Vbb[[\tau]]$, which is equivalent to that for each $n\in\Nbb$,
\begin{align}
Y(u)_{-n-1}\id=\frac 1{n!}L_{-1}^n u.\label{eq52}	
\end{align}
\end{co}

We leave it to the reader to find a geometric explanation of  $e^{\tau L_{-1}}u=Y(u,\tau)\id$.

\section{Definition of VOAs, II: Jacobi Identity}

\subsection{}

\begin{prin}
When gluing Riemann spheres to get new spheres, the formula $T_{\Sigma_1}\circ T_{\Sigma_2}=T_{\Sigma_1\#\Sigma_2}$ truely holds if the local coordinates at the points for sewing are M\"obius transformations, i.e. of the form $z\mapsto \frac {az+b}{cz+d}$ where $ad-bc\neq0$.
\end{prin}

A rough reason for this No-Ambiguity Principle is that only $L_0,L_{\pm 1}$ are involved in the change of coordinate formulas between M\"obius transformations, and the Lie bracket relations between them do not involve the central charge.

\subsection{}\label{lb156}

We shall give motivations for the Jacobi identity.

We first remark on the sewing of compact Riemann surfaces in Subsec. \ref{lb4}. Suppose we have data $\fk X=(C;x_\blt;\eta_\blt)$ and $\fk X'=(C';y_\blt;\eta_\blt')$ and we sew them along $x_1$ and $x_1'$. For simplicity, we set $\xi=\eta_1,\varpi=\eta_1'$. From \eqref{eq3}, we know that the gluing law is that any $x\in\xi^{-1}(\Sbb^1)$ (recall that $\xi^{-1}(\Sbb^1)$ is a boundary string of the corresponding surface $\Sigma$ for $\fk X$) and any $y\in\varpi^{-1}(\Sbb^1)$ are identified following the rule
\begin{align}
x=y\qquad \Longleftrightarrow\qquad\xi(x)\varpi(y)=1.\label{eq41}
\end{align}

This definition of gluing is topological, but not complex analytic. Analytically, we are actually gluing a neighborhood of $\xi^{-1}(\Sbb^1)$ and one of $\varpi^{-1}(\Sbb^1)$ using the rule \eqref{eq41} for all $x$ in the first neighborhood and $y$ in the second one. It is clear that a (locally defined) function on the first neighborhood is holomorphic if and only if it is so on the second one. This defines the complex analytic structure on $C\#C'$.

\begin{rem}\label{lb33}
Let us be more precise on the shape of the neighborhoods. Let $\xi$ and $\varpi$ be defined (and injective) on $U,U'$ respectively. Choose $r>1,\rho>1$ such that $\xi(U)\supset \Dbb_r,\varpi(U')\supset\Dbb_\rho$. Then the following neighborhoods of $\xi^{-1}(\Sbb^1)$ and $\varpi^{-1}(\Sbb^1)$ are glued via the relation \eqref{eq41}:
\begin{gather}\label{eq154}
	\begin{gathered}
\xi^{-1}(A_{\rho^{-1},r})=\{x\in U: \rho^{-1}<|\xi(x)|<r\}\\
\Big\updownarrow\text{identified via } \eqref{eq41}\\	
\varpi^{-1}(A_{r^{-1},\rho})=\{y\in U':r^{-1}<|\varpi(y)|<\rho\}
	\end{gathered}	
\end{gather}
The parts $\{x\in U:|\xi(x)|\leq \rho^{-1}\}$ and $\{y\in U':|\varpi(y)|\leq r^{-1}\}$ are discarded when gluing.
\begin{align}
	\vcenter{\hbox{{
				\includegraphics[height=2cm]{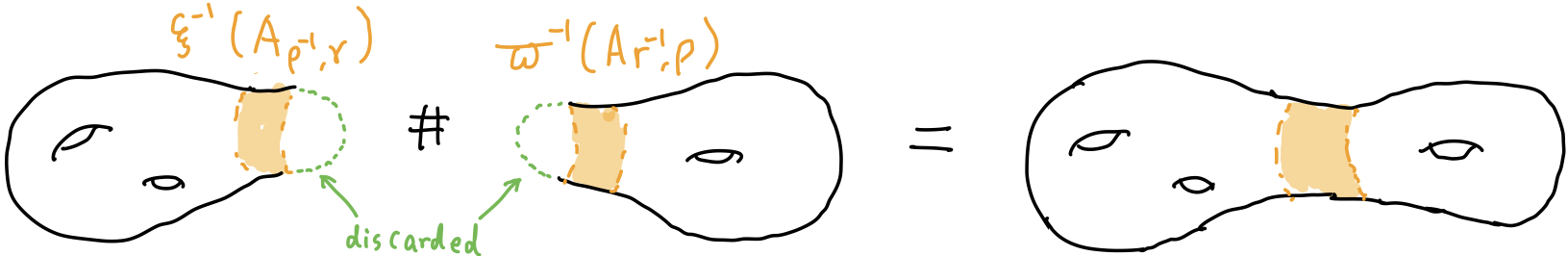}}}}  \label{eq43}
\end{align}
\end{rem}

\subsection{}

As pointed out before, when we associate finite energy vectors to the incoming  strings/points, we may scale their local coordiates. However, for the local coordinates at the output points and the points to be sewn, an arbitrary scaling is not allowed. We thus assume that Assumption \ref{lb11} holds after scaling (by some $\lambda$ with arbitrarily large $|\lambda|$) the local coordinates at the incoming points. This amounts to the following

\begin{ass}\label{lb32}
If $x_i$ is either an outgoing point or a point to be sewn with another point, then the local coordinate $\eta_i$ at $x_i$ defined on a neighborhood $U_i\ni x_i$ satisfies that $\eta_i(U_i)\supset\Dbb_1^\cl$, that $\eta_i^{-1}(\Dbb_1^\cl)\cap \eta_j^{-1}(\Dbb_1^\cl)=\emptyset$ if $x_j$ is either outgoing or a point to be sewn, and that $x_j\in\eta_i^{-1}(\Dbb_1^\cl)$ if $x_j$ is incoming and not to be sewn.
\end{ass}


\begin{rem}\label{lb34}
There is indeed one way we can slightly loosen the above assumption. Using the notation of \eqref{eq41}. Then we may assume that Assumption \ref{lb32} after scaling $\xi$ by some $\lambda\in\Cbb^\times$ and $\varpi$ by $\lambda^{-1}$. Then the rule for gluing \eqref{eq41} is not changed. On the side of interaction maps $T_{\Sigma}$, the change $\xi\rightsquigarrow \lambda \xi$ adds a factor $\lambda^{-L_0}\otimes(\ovl\lambda)^{-\ovl L_0}$ to one tensor component in $T_\Sigma$, and  $\xi\rightsquigarrow \lambda^{-1} \xi$ adds a factor $\lambda^{L_0}\otimes\ovl\lambda^{\ovl L_0}$. These two are canceled after taking contraction or composition.
\end{rem}

\subsection{}

We want to understand the product $\bk{w',Y(u,z_2)Y(v,z_1)w}$. Let $\zeta$ be the standard coordinate of $\Cbb$. By the sewing property in Segal's picture, this expression should correspond to the sewing of
\begin{gather*}
\fk P_{z_1}=(\Pbb^1_1;0,z_1,\infty;\zeta,\zeta-z_1,\zeta^{-1}),\qquad \fk P_{z_2}=(\Pbb^1_2;0,z_2,\infty;\zeta,\zeta-z_2,\zeta^{-1})	
\end{gather*}
along the points $\infty$ of $\fk P_{z_1}$ and $0$ of $\fk P_{z_2}$. (Here, both $\Pbb^1_1$ and $\Pbb^1_2$ are $\Pbb^1$. We assume the two $\infty$ are outgoing before sewing.) Assumption \ref{lb32} is satisfied when $0<|z_1|<1<|z_2|<+\infty$ if we consider all the points not for sewing as incoming. The sewing rule is that $\gamma_1\in\Pbb^1_1,0<|\gamma_1^{-1}|<+\infty$ is identified with $\gamma_2\in\Pbb^1_2,0<|\gamma_2|<+\infty$  if and only if  $\gamma_1^{-1}\cdot \gamma_2=1$, namely $\gamma_1=\gamma_2$. (Here, we set $r=\rho=\infty$ in order to apply Rem. \ref{lb33}. The discarded points are the $\infty$ of $\Pbb^1_1$ and the $0$ of $\Pbb^1_2$.) Thus, the sewing is just placing the first sphere onto the second one. 
\begin{align*}
\vcenter{\hbox{{
\includegraphics[height=2.2cm]{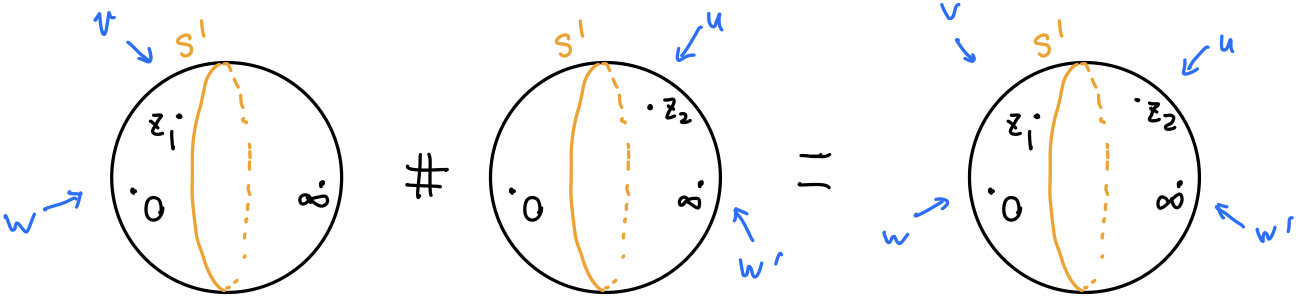}}}}
\end{align*}
The result of sewing is
\begin{gather}
\fk P_{z_1,z_2}=(\Pbb^1;0,z_1,z_2,\infty,\zeta-z_1,\zeta-z_2,\zeta^{-1})\label{eq61}	
\end{gather}
Assuming all the points of $\fk P_{z_1,z_2}$ as incoming,  for each $u,v,w,w'\in\Vbb$,
\begin{align}
T_{\fk P_{z_1,z_2}}(w,v,u,w')=\bk{w',Y(u,z_2)Y(v,z_1)w}\qquad(\text{if }0<|z_1|<|z_2|<+\infty).\label{eq42}	
\end{align}
The reason why the conditions $|z_1|<1$ and $1<|z_2|$ can be dropped is explained below.

\subsection{}\label{lb185}

We explain why \eqref{eq42} holds provided  $0<|z_1|<|z_2|<+\infty$.

Pick $\lambda\in\Cbb$ such that $|z_1|<|\lambda|<|z_2|$. Following the guide of Rem. \ref{lb34}, we replace the local coordinate $\zeta^{-1}$ of $\fk P_{z_1}$ by $\lambda\zeta^{-1}$ and the one $\zeta$ of $\fk P_{z_2}$ by $\zeta/\lambda$. Then Assumption \ref{lb32} is again satisfied. In particular, the outgoing string of $\Pbb^1_1$ around $\infty$ and the incoming one of $\Pbb^1_2$ around $0$ are both $|\lambda|\Sbb^1$.

The interaction map $T_{\fk P_{z_1}}:\mc H^{\otimes 2}\rightarrow\mc H$ acting on $w\otimes v$ is $\lambda^{-L_0}Y(v,z_1)w$. $T_{\fk P_{z_2}}$ sends $u\otimes \underline{~~~}\in\Vbb\otimes\Vbb$ to $Y(u,z_2)\lambda^{L_0}\underline{~~~}$. The composition of these two expressions, evaluated with $w'\in\Vbb$, is again the right hand side of \eqref{eq42}. And the result of sewing is again $\fk P_{z_1,z_2}$. So \eqref{eq42} holds in general.
\begin{align*}
	\vcenter{\hbox{{
				\includegraphics[height=2.2cm]{fig21.png}}}}
\end{align*}

\subsection{}

According to the physical definition of $\Vbb$ in Subsec. \ref{lb35} as well as the No-Ambiguity Principle \ref{lb19}, we know that when the vectors of $\Vbb$ are inserted, the correlation functions change holomorphically with respect to the translation of the marked points and their local coordinates. Thus $T_{\fk P_{z_1,z_2}}(w,v,u,v')$ is a holomorphic function on $\Conf^2(\Cbb^\times)=\{(z_1,z_2)\in\Cbb^\times:z_1\neq z_2\}$. Since, similar to \eqref{eq42}, we also have
\begin{align}
	T_{\fk P_{z_1,z_2}}(w,v,u,w')=\bk{w',Y(v,z_1)Y(u,z_2)w}\qquad(\text{if }0<|z_2|<|z_1|<+\infty),\label{eq44}	
\end{align}
we conclude that $\bk{w',Y(u,z_2)Y(v,z_1)w}$ defined on $0<|z_1|<|z_2|$ and $\bk{w',Y(v,z_1)Y(u,z_2)w}$ defined on $0<|z_2|<|z_1|$ can be continued to the same holomorphic function on $\Conf^2(\Cbb^\times)$. That this fact is true for all $w,w'\in\Vbb$ (or more generally, all $w\in\Vbb,w'\in\Vbb'$ if $\Vbb\simeq\Vbb'$ is not assumed) is simply written as
\begin{align}
Y(u,z_2)Y(v,z_1)\sim Y(v,z_1)Y(u,z_2).	
\end{align}
This property is called \textbf{commutativity}.

\subsection{}\label{lb186}
We now consider the sewing of
\begin{gather*}
	\fk P_{z_1}=(\Pbb^1_1;0,z_1,\infty;\zeta,\zeta-z_1,\zeta^{-1}),\qquad \fk P_{z_2-z_1}=(\Pbb^1_{21};0,z_2-z_1,\infty;\zeta,\zeta-z_2+z_1,\zeta^{-1})	
\end{gather*}
(where $\Pbb^1_{21}=\Pbb^1$) along the points $z_1\in\Pbb^1_1$ and $\infty\in\Pbb^1_{21}$. We assume $0<|z_2-z_1|<|z_1|<+\infty$. Choose $\lambda\in\Cbb$ satisfying $|z_2-z_1|<|\lambda|<|z_1|$.  Replace the local coordinate $\zeta-z_1$ of $\fk P_{z_1}$ by $\lambda^{-1}(\zeta-z_1)$ and the one $\zeta^{-1}$ of $\fk P_{z_2-z_1}$ by $\lambda \zeta^{-1}$. Then Assumption \ref{lb32} is satisfied. The rule for sewing is identifying $\gamma_1\in \Pbb^1_1,0<|\lambda^{-1}(\gamma_1-z_1)|<+\infty$ with $\gamma_{21}\in\Pbb^1_{21},0<|\lambda\gamma_{21}^{-1}|<+\infty$ if and only if $(\gamma_1-z_1)=\gamma_{21}$. Thus, gluing $\fk P_{z_1-z_1}$ to $\fk P_{z_1}$ amounts to translating $\fk P_{z_2-z_1}$ to $\fk P_{z_2}$. After sewing, the points $0$ and $z_2-z_1$ of $\fk P_{z_2-z_1}$ become $z_1$ and $z_2$. (The points $z_1$ of $\fk P_{z_1}$ and $\infty$ of $\fk P_{z_2-z_1}$ are discarded.) 
\begin{align*}
	\vcenter{\hbox{{
				\includegraphics[height=2.6cm]{fig22.png}}}}
\end{align*}
This sewing picture gives
\begin{align}
T_{\fk P_{z_1,z_2}}(w,v,u,w')=\bk{w',Y(Y(u,z_2-z_1)v,z_1)w}\quad(\text{if }0<|z_2-z_1|<|z_1|<+\infty).\label{eq45}	
\end{align}

We therefore have the \textbf{associativity} property
\begin{gather}\label{eq60}
\begin{gathered}
\bk{w',Y(u,z_2)Y(v,z_1)w}=\bk{w',Y(Y(u,z_2-z_1)v,z_1)w}\\
\text{if }0<|z_2-z_1|<|z_1|<|z_2|.
\end{gathered}	
\end{gather}
Geometrically, it means the equivalence of sewing spheres in the following way:
\begin{align*}
	\vcenter{\hbox{{
				\includegraphics[height=2cm]{fig23.png}}}}
\end{align*}

\subsection{}\label{lb38}

The fact that for all $u,v,w\in\Vbb,w'\in\Vbb'$, \eqref{eq42}, \eqref{eq44}, and \eqref{eq45} can be defined as holomorphic functions of $z_1,z_2$ on the given domain (the precise meaning will be given later), and that these three expressions can be extended to the same holomorphic function (namely $T_{\fk P_{z_1,z_2}}(w\otimes v\otimes u\otimes w')$) on $\Conf^2(\Cbb^\times)$ is called the \textbf{Jacobi identity} in the complex analytic form. (See Def. \ref{lb194} for the precise statement.) Roughly speaking,
\begin{align}
\text{Jacobi identity $=$ Commutativity $+$ Associativity}.\label{eq73}
\end{align}
For the moment, we will derive an algebraic version, and use it as the formal definition of Jacobi identity in Def. \ref{lb24}. 

Write $f(z_1,z_2)=T_{\fk P_{z_1,z_2}}(w\otimes v\otimes u\otimes w')$.  Fix $z_1\in \Cbb^\times$, and consider $f$ as a holomorphic function of $z_2$ on $\Cbb^\times\setminus\{z_1\}$. (Moreover, from \eqref{eq42}, \eqref{eq44}, \eqref{eq45}, and by the lower truncation property \eqref{eq46}, it is easy to see that $f$ has finite poles at $z_1=0,z_2,\infty$. So $f$ is a meromorphic function.) By the residue theorem, for each meromorphic  $1$-form $\mu$ on $\Pbb^1$ with possible poles only at $0,z_1,\infty$, we must have $(\Res_{z_2=0}+\Res_{z_2=z_1}+\Res_{z_2=\infty})f\mu=0$. It is easy to see that such $\mu$ are linear combinations of those of the form $z_2^m(z_2-z_1)^ndz_2$.

Equivalently, choose $C_+$ to be a circle around $0$ whose radius is  $>|z_1|$, $C_-$ is one around $0$ whose radius is $<|z_1|$, and $C_0$  a small circle around $z_1$ between $C_+$ and $C_-$. 
\begin{align*}
	\vcenter{\hbox{{
				\includegraphics[height=1.8cm]{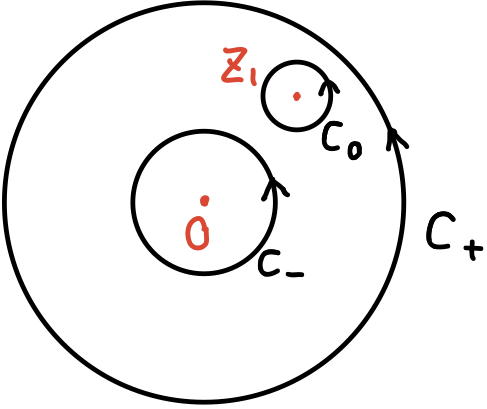}}}}
\end{align*}
Let $f_+,f_-,f_0$ be respectively the right hand sides of \eqref{eq42}, \eqref{eq44}, \eqref{eq45}. Then, when $z_2$ is on $C_+,C_-,C_0$ respectively, $f$ equals $f_+,f_-,f_0$. Then the fact that $f_+,f_-,f_0$ defined on their domains extend to the same meromophic function on $\Pbb^1$ with poles $0,z_1,\infty$ implies for any $m,n\in\Zbb$ and $\mu=z_2^m(z_2-z_1)^ndz_2$ that
\begin{align}
	\oint_{C_+}\frac{f_+\mu}{2\im\pi} -\oint_{C_-}\frac{f_-\mu}{2\im\pi}=\oint_{C_0} \frac{f_0\mu}{2\im\pi}.\label{eq48}
\end{align}
Indeed, the latter one also implies the previous one. This is guaranteed by the so called \emph{strong residue theorem}, which will be discussed in Subsec. \ref{lb193}. The strong residue theorem will imply that the analytic form and the algebraic form of Jacobi identity are equivalent.

Recall the general formula $\oint_C Y(u,z)z^k\frac {dz}{2\im\pi}=Y(u)_k$ if $C$ is a circle around the origin. When $z_2\in C_+$, $\mu$ has absolutely convergent expansion $\mu=\sum_{l\in\Nbb}{n\choose l}(-z_1)^lz_2^{m+n-l}dz_2$. So
\begin{align*}
&\oint_{C_+}\frac{f_+\mu}{2\im\pi}=\sum_{l\in\Nbb}	\oint_{C_+}{n\choose l}(-z_1)^lz_2^{m+n-l}\bk{w',Y(u,z_2)Y(v,z_1)w}\frac{dz_2}{2\im\pi}\\
=&\sum_{l\in\Nbb} {n\choose l}(-z_1)^l\bk{w',Y(u)_{m+n-l}Y(v,z_1)w}=:a(z_1)
\end{align*}
When $z_2\in C_-$, 
\begin{align*}
	&\oint_{C_-}\frac{f_-\mu}{2\im\pi}=\sum_{l\in\Nbb}	\oint_{C_+}{n\choose l}(-z_1)^{n-l}z_2^{m+l}\bk{w',Y(v,z_1)Y(u,z_2)w}\frac{dz_2}{2\im\pi}\\
	=&\sum_{l\in\Nbb}{n\choose l}(-z_1)^{n-l}\bk{w',Y(v,z_1)Y(u)_{m+l}w}=:b(z_1)
\end{align*}
When $z_2\in C_0$, since $0<|z_2-z_1|<|z_1|$, we have the absolutely convergent expansion $\mu=(z_1+(z_2-z_1))^m(z_2-z_1)^ndz_2=\sum_{l\in\Nbb}{m\choose l}z_1^{m-l}(z_2-z_1)^{n+l}dz_2$. So
\begin{align*}
&\oint_{C_0}\frac{f_0\mu}{2\im\pi}=\sum_{l\in\Nbb}\oint_{C_0}{m\choose l}z_1^{m-l}(z_2-z_1)^{n+l}\bk{w',Y(Y(u,z_2-z_1)v,z_1)w}\frac{dz_2}{2\im\pi}\\
=&\sum_{l\in\Nbb}{m\choose l}z_1^{m-l}\bk{w',Y(Y(u)_{n+l}v,z_1)w}:=c(z_1)
\end{align*}

Now we have $c(z_1)=a(z_1)-b(z_1)$. We vary $z_1$. For each $k\in\Zbb$, multiply both sides by $z_1^k\frac{dz_1}{2\im\pi}$ and apply the residue at $z_1=0$. We then get (by suppressing $w'$ and $w$)

\begin{df}[\textbf{Jacobi identity} (algebraic version)]\label{lb36} \index{00@Jacobi identity, algebraic version}
For each $u,v,w\in\Vbb$, and each $m,n,k\in\Zbb$, we have
\begin{align}\label{eq47}
	\begin{aligned}
&\sum_{l\in\Nbb}{m\choose l}Y\big(Y(u)_{n+l}v\big)_{m+k-l}\\
=&\sum_{l\in\Nbb}(-1)^l{n\choose l}Y(u)_{m+n-l}Y(v)_{k+l}-\sum_{l\in\Nbb}(-1)^{n+l}{n\choose l}Y(v)_{n+k-l} Y(u)_{m+l}	.
	\end{aligned}
\end{align}
This completes Definition \ref{lb24}.
\end{df}
In the above three terms, when acting on every $w\in\Vbb$, each sum over $l\in\Nbb$ is finite thanks to the lower truncation property.

\section{Consequences of Jacobi identity; reconstruction theorem}

\subsection{}
The algebraic form of Jacobi identity is very complicated. Very few people can write down exactly the right formula without checking the references or reproving this formula using the long argument in Subsec. \ref{lb38}. But we shall try our best to explain how to use this formula and what this formula implies.

First of all, if \eqref{eq47} holds whenever $m=0$ or $n=0$, then it holds in general. We will not give a rigorous proof for this statement. But, since \eqref{eq47} is derived from \eqref{eq48} for all $\mu=z_2^m(z_2-z_1)^ndz_2$, the readers can be convinced of this statement by the following elementary fact:

\begin{exe}
Show that $z_2^m(z_2-z_1)^n$ is a $\Cbb[z_1^{\pm1}]$-linear combination of $z_2^k$ and $(z_2-z_1)^l$ where $k,l\in\Zbb$ and $l<0$. (Hint: Assume without loss of generality that $m,n<0$. Prove the statement by induction on $|m|$ and $|n|$.)
\end{exe}

Thus, we may understand \eqref{eq47} by restricting to the special cases $m=0,n<0$ or $n=0$. 
\subsection{}

We now return to rigorous mathematics. Consider  the case that $n=0$, i.e., $\mu=z_2^mdz_2$. Then \eqref{eq47} reads
\begin{align}
\big[Y(u)_m,Y(v)_k\big]=\sum_{l\in\Nbb}{m\choose l}Y\big(Y(u)_lv\big)_{m+k-l}.	\label{eq49}
\end{align}
This is a Lie bracket relation. Interestingly, this general formula does not come from Lie groups, but from the residue theorem. However, in many concrete examples, such Lie bracket relations do have Lie-theoretic origins. 

Let me take this chance to say a few words about the similarity and the difference between the VOA theory and the Lie theory. In the VOA theory, the residue theorem is the standard way of passing from the complex analytic world to the algebraic world. The opposite direction is through the strong residue theorem. This is strikingly different from the Lie theory, in which one passes from the differential geometric formulation (i.e. Lie groups) to the algebraic one (i.e. Lie algebras) by taking derivatives, and vice versa by taking exponentiation/integral. Thus, although Lie brackets do appear in VOAs, it is not always fruitful to think of VOAs as generalizations of Lie algebras. These two mathematical objects have very different geometric intuitions. Also, if we view VOAs in the complex analytic way, then by \eqref{eq73}, VOAs are more like commutative algebras. Thus, VOAs can be viewed as a quantum version of both the Lie algebras and the commutative algebras.

\subsection{}

Take $u$ to be the conformal vector $\cbf$ in \eqref{eq49} and recall that $Y(\cbf)_{m+1}=L_m$. We obtain
\begin{align}
&[L_m,Y(v)_k]=\sum_{l\in\Nbb}{m+1\choose l}Y(L_{l-1}v)_{m+k+1-l}\nonumber\\	=&Y(L_{-1}v)_{m+k+1}+\sum_{l\in\Nbb}{m+1\choose l+1}Y(L_l v)_{m+k-l}.\label{eq54}
\end{align}
Multiply $z^{-k-1}$ to both sides and take the sum over all $k\in\Zbb$, we obtain
\begin{align}
[L_m,Y(v,z)]=z^{m+1}Y(L_{-1}v,z)+\sum_{l\in\Nbb}{m+1\choose l+1}z^{m-l}Y(L_l v,z)\label{eq68}
\end{align}
either on the level of $\End(\Vbb)[[z^{\pm1}]]$, or as Laurent polynomials of $z$ when evaluating between any $w\in\Vbb$ and $w'\in\Vbb'$. Then the cases $m=-1$ and $m=0$ imply
\begin{subequations}\label{eq50}
\begin{gather}
[L_{-1},Y(v,z)]=Y(L_{-1}v,z)\\
[L_0,Y(v,z)]=zY(L_{-1}v,z)+Y(L_0v,z).		
\end{gather}
\end{subequations}
Note that these two equations follow solely from the Jacobi identity. By the translation property, we have
\begin{align}
Y(L_{-1}v,z)=\frac d{dz}Y(v,z).	\label{eq51}
\end{align}
Equivalently, by applying $\Res_{z=0}(\cdot)z^ndz$, we get a crucial relation
\begin{align}
Y(L_{-1}v)_n=-nY(v)_{n-1}.\label{eq53}	
\end{align}
(The quickest way to get the formula on the right hand side is integration by parts.)

\begin{exe}
Show that \eqref{eq52} follows from \eqref{eq53} and the creation property.
\end{exe}

\begin{exe}
Assume that $\Vbb$ satisfies the lower truncation property \eqref{eq46} and all the axioms of VOAs in Def. \ref{lb24} except the grading and the translation property. Use \eqref{eq50} to prove that the following conditions are equivalent.
\begin{enumerate}
	\item The grading property.
	\item $Y(L_{-1}v,z)=\partial_zY(v,z)$ for all $v\in\Vbb$.
	\item The translation property.
	\item The translation property without assuming $L_{-1}\id=0$.
\end{enumerate} 
Thus, we may use the lower truncation property and any of these four conditions to replace the grading and the translation properties in the definition of VOAs.
\end{exe}

\begin{exe}\label{lb37}
In \eqref{eq54}, set $v=\cbf$, and show that this formula is compatible with the Virasoro relation. 
\end{exe}

\subsection{}\label{lb43}

We see that \eqref{eq68} for $m=0,-1$ (together with \eqref{eq51}) means the grading and the translation properties, which integrate to the rotation and the translation covariance. For general $m$, \eqref{eq68} also has a geometric explanation. To simplify discussions, we give such an explanation by assuming that $v$ is primary.

\begin{df}
A vector $v\in\Vbb$ is called a \textbf{primay vector} if it is homogeneous and $L_nv=0$ for all $n>0$.
\end{df}

Some important VOAs (affine VOAs for instance) are generated by primary vectors. And many important formulas in CFT were first proved by physics who assumed that their theories are generated by primary vectors in the following sense:
\begin{df}\label{lb163} \index{00@Generatig subsets of the VOA $\Vbb$}
	We say that a VOA $\Vbb$ is \textbf{generated} by a subset $E\subset\Vbb$ if $\Vbb$ is spanned by vectors of the form $Y(v_1)_{n_1}\cdots Y(v_k)_{n_k}\id$ where $k\in\Nbb$, $n_1,\dots,n_k\in\Zbb$, and $v_1,\dots,v_k\in E$.
\end{df}
Indeed, formula \eqref{eq68} for any primary vector $v$ is one such example, which (combined with \eqref{eq51}) reads
\begin{align}
	[L_m,Y(v,z)]=z^{m+1}\partial_zY(v,z)+(m+1)\wt v\cdot z^mY(v,z).\label{eq69}
\end{align}
This is called by physicists (or more precisely, is equivalent to what physicists call) the \textbf{conformal Ward identity}. 

Choose a holomorphic vector field $f(z)\partial_z=\sum_{n\in\Zbb} a_n z^{n+1}\partial_z$ on a neighborhood of $\Sbb^1$. Let $\sigma_\tau=\exp(\tau f\partial_z)$ be the holomorphic flow. Then \eqref{eq69} (with $L_m,z^m$ replaced by $\sum_m a_mL_m,\sum_m a_mz_m$) integrates to
\begin{align}
e^{\tau\sum_{n\in\Zbb}a_nL_n}Y(v,z)e^{-\tau\sum_{n\in\Zbb}a_nL_n}=\big(\partial_z\sigma_\tau(z)\big)^{\wt v}Y\big(v,\sigma_\tau(z) \big),\label{eq70}
\end{align}
called \textbf{conformal covariance}.  For now, we do not treat this formula in a rigorous way. But the readers can convince themselves by checking that both sides satisfy the same linear differental equation over $\tau$. 

The right hand side of \eqref{eq70} looks familiar to us. Set $\tau=1$, $\sigma=\sigma_1$, and $\Delta=\wt v$. Then formula \eqref{eq70} resembles the change of variable formula $\big(\partial(\varphi\circ\sigma)\big)^\Delta=\big(\partial\varphi\circ\sigma\big)^\Delta\cdot (\partial_z\sigma)^\Delta$ for a function $\varphi=\varphi(z)$ and $\partial$ is the standard holomorphic derivative. Indeed, the primary field $Y(v,z)$ can be viewed as the quantization of $(\partial\varphi)^\Delta$, or more generally, of $\partial\varphi_1\cdots\partial\varphi_\Delta$. It is also interesting to write \eqref{eq70} in the form
\begin{align}
	e^{\sum a_nL_n}\big(Y(v,z)dz^\Delta \big)e^{-\sum a_nL_n}=Y(v,\sigma)d\sigma^\Delta.
\end{align}

Conformal covariance \eqref{eq70} can be interpreted in a similar geometric way as we did for rotation and translation covariance in Subsec. \ref{lb31} and \ref{lb40}. (We will give this explanation in the future assuming $f=\sum_{n\geq 0}a_nz^{n+1}\partial_z$.) So, from the CFT point of view, this formula follows naturally from our change of parametrization formula in Sec. \ref{lb41} and the physical definition of the vertex operator $Y(v,z)$ in Sec. \ref{lb42} (if we ignore the issue of uniqueness up to scalar multiplications). In particular, the geometric intuition we are using for formula \eqref{eq69} is Lie theoretic, because the relationship between Virasoro algebras and change of parametrization formula is the one between the representations of Lie algebras and Lie groups. But we have also derived \eqref{eq69} from the Jacobi identity, whose geometric intuition relies on the residue theorem. How should we view this coincidence of the two geometric pictures?

My answer is that we should regard the Lie theoretic explanation as the fundamental one for conformal covariance/Ward identity. In fact, to use the Jacobi identity to   obtain \eqref{eq69}, we have assumed that $\sum L_nz^{-n-2}$ is the vertex operator of a vector of $\Vbb$, namely the conformal vector $\cbf$. But the reason that this assumption should be included in the definition of VOA was not explained in Sec. \ref{lb42}. Here we give a short explanation: we will see later (cf. the reconstruction Thm. \ref{lb39} and Rem. \ref{lb44}) that if the Fourier modes $A_m\in\End(\Vbb)$ of a  field $A(z)$ satisfy the correct Jacobi identity (such as \eqref{eq49} or \eqref{eq69}) with the modes $Y(v)_k$ for $v$ inside a generating subset $E\subset\Vbb$, then $A(z)$ must be $Y(u,z)$ for some $u\in \Vbb$. Thus, (in my opinion) the better point of view is that we use the conformal Ward identity (whose geometric intuition relies on the change of parametrization formula and the physical meaning of $Y(u,z)$) and the Jacobi identity to explain the fact that $\sum L_nz^{-n-2}$ is represented by a vector $\cbf$ in $\Vbb$, but not that we explain the Ward identity using the VOA Jacobi identity.

\subsection{}\label{lb56}
We say that $\Vbb$ is of \textbf{CFT-type} if $\dim\Vbb(n)<+\infty$ for each $n$, and $\Vbb(0)=\Cbb\id$. The CFT-type condition is a very natural and mild one satisfied by all the examples in our notes. It says that the only quantum states with zero energy are the vacuum.

In this subsection, we assume $\Vbb$ is CFT-type, and study \eqref{eq49} for vectors in $\Vbb(1)$. For each $u\in\Vbb(1)$, we write $Y(u)_m$ as $u_m$ for short. By \eqref{eq21}, $u_l$ lowers the weights by $l$. Then \eqref{eq49} says $[u_m,v_n]=(u_0v)_{m+n}+m(u_1v)_{m+n-1}$, where $u_lv$ vanishes when $l>1$ since its weight is $1-l$. Since $u_1v\in\Vbb(0)\in\Cbb$, we may write
\begin{align}
u_1v=\prth{u,v}\id	\label{eq58}
\end{align}
where $\prth{\cdot,\cdot}$ is a bilinear form on $\Vbb(1)$. Thus $(u_1v)_{m+n-1}=\prth{u,v}\delta_{m,-n}$ since $Y(\id,z)=\id$. Set
\begin{align}
[u,v]:=u_0v.	\label{eq55}
\end{align}
Then
\begin{align}
[u_m,v_n]=[u,v]_{m+n}+m\prth{u,v}\delta_{m,-n}.	\label{eq56}
\end{align}

\begin{pp}
$[\cdot,\cdot]$ defines a Lie algebra structure on $\Vbb(1)$, and $\prth{\cdot,\cdot}$ is an invariant symmetric bilinear form, namely, $\prth{u,v}=\prth{v,u}$ and $\prth{[w,u],v}=-\prth{u,[w,v]}$.
\end{pp}

\begin{proof}
$w\in\Vbb(1)\mapsto w_{-1}$ is injective since $w_{-1}\id=w$  by the creation property. By \eqref{eq56}, $[u,v]_{-1}=[u_0,v_{-1}]=-[v_{-1},u_0]=-[v,u]_{-1}$. This proves $[u,v]=-[v,u]$. By calculating $[u_1,v_{-1}]$ and $[v_{-1},u_1]$ using \eqref{eq56}, we obtain $\prth{u,v}=\prth{v,u}$. \eqref{eq56} implies 
\begin{align*}
[w_k,[u_m,v_n]]=[w,[u,v]]_{k+m+n}+k\prth{w,[u,v]}\delta_{k+m+n,0}.	
\end{align*}
Apply the Jacobi identity for the Lie bracket of linear operators, we obtain the Jacobi identity for $[\cdot,\cdot]$ on $\Vbb(1)$ if we set $k=-1,m=n=0$, and we obtain the invariance of $\prth{\cdot,\cdot}$ if we set $k=0,m=1,n=-1$.
\end{proof}

The vector space $\Span_\Cbb\{v_n,\id_\Vbb:n\in\Zbb\}$ is a Lie algebra whose bracket is the standard one for  linear operators. Since it satisfies \eqref{eq56}, we call it an \textbf{affine Lie algebra} associated to the finite-dimensional complex Lie algebra $\Vbb(1)$.  When $\Vbb$ is generated by $\Vbb(1)$, we say $\Vbb$ is an \textbf{affine VOA}.

We are mostly interested in the case that $\prth{\cdot,\cdot}$ is non-degenerate. This is always true when the CFT (or the VOA) is unitary, since $\prth{\cdot,\cdot}$ is indeed the negative of the correlation function $\bk{\cdot,\cdot}=\bk{\Theta\cdot|\cdot}$ of $A_{1,1}$ restricted to $\Vbb^{\otimes 2}$. Moreover, a unitary affine VOA $\Vbb$ is indeed uniquely determined by its Lie subalgebra $\Vbb(1)$, where $\Vbb(1)$ is a direct sum of an abelian Lie algebra and a semisimple one. (We refer the readers to \cite[Sec. 1 and 2]{Gui19} for a detailed account of the relationship between unitary VOAs and their ``unitary" Lie subalgebras $\Vbb(1)$.) Affine Lie algebras and affine VOAs in the strict sense are those such that $\Vbb(1)$ are simple Lie algebras. If on the other hand $\Vbb(1)$ is abelian, then $\Vbb$ is called a \textbf{free boson VOA} or a \textbf{Heisenberg VOA}.

If $\Vbb$ is generated by $\cbf$, we call $\Vbb$ a \textbf{Virasoro VOA}.

\subsection{}


We now turn to the case $m=0,n<0$ in the VOA Jacobi identity \eqref{eq47}. First consider $n=-1$. Then \eqref{eq47} reads
\begin{align}
Y\big(Y(u)_{-1}v\big)_k=\sum_{l\in\Nbb}Y(u)_{-1-l}Y(v)_{k+l}+\sum_{l\in\Nbb}Y(v)_{k-1-l}Y(u)_l.\label{eq59}
\end{align}
This formula can be written in a compact way. For a general series $f(z)=\sum_{l\in\Zbb} a_lz^{-l-1}\in W[[z^{\pm1}]]$ where $W$ is a vector space, we let
\begin{align}
f(z)_+=\sum_{l\in\Nbb}a_lz^{-1-l},\qquad f(z)_-=\sum_{l\in\Nbb}a_{-l-1}z^l	
\end{align}
(so we have $f(z)=f(z)_++f(z)_-$). Define the \textbf{normal-ordered product}
\begin{align}
\hcolondel {Y(u,z)Y(v,z)}=Y(u,z)_-Y(v,z)+Y(v,z)Y(u,z)_+	
\end{align}
which is non-commutative in general. Then \eqref{eq59} can be abbreviated to
\begin{align}
Y\big(Y(u)_{-1}v,z\big)=\hcolondel {Y(u,z)Y(v,z)}\label{eq88}
\end{align}

By \eqref{eq53} we have
\begin{align}
	Y(u)_{-j-1}=\frac 1{j!}Y(L_{-1}^ju)_{-1}
\end{align}
when $j\geq 0$. Combine this with $Y(L_{-1}^ju,z)=\partial_z^j Y(u,z)$, we obtain
\begin{align}
Y\big(Y(u)_{-j-1}v,z\big)=\frac{1}{j!}\hcolondel{\big(\partial_z^jY(u,z)\big) Y(v,z)}	\label{eq65}
\end{align}
where the normal-ordered product is defined in a similar way using the positive and the negative parts of $\partial_z^jY(u,z)$.  We leave it to the readers to check that this formula agrees with the Jacobi identity \eqref{eq47} when $m=0,n<0$.

Thus, once we know how $Y(u,z)$ looks like for all $u$ in a small generating subset $E$ of $\Vbb$, we can write down the formula of $Y(w,z)$ for any $w\in\Vbb$ using the formula
\begin{align}
Y\big(Y(u_1)_{-j_1-1}\cdots Y(u_k)_{-j_k-1}v,z\big)=\frac 1{j_1!\cdots j_k!}\hcolondel{\partial_z^{j_1} Y(u_1,z)\cdots \partial_z^{j_k}Y(u_k,z)\cdot Y(v,z)}	\label{eq67}
\end{align}
where the normal-ordered product for several operators is defined inductively by
\begin{align}
\hcolondel{A_1A_2\cdots A_n}=\hcolondel{A_1 (\hcolondel{A_2\cdots A_n})}	
\end{align}

\subsection{}\label{lb188}
One can also write down the explicit formula of $Y(Y(u)_nv,z)$ for $n\geq 0$ using \eqref{eq47} where $m=0,n\geq 0$. But as I said, \eqref{eq47} is determined by the special cases $m=0,n<0$ and $n=0$. So we  hope that $Y(Y(u)_nv,z),n\geq 0$ can be calculated using \eqref{eq49}. This is true.

Write \eqref{eq49} in the equivalent form
\begin{align}
\big[Y(u)_m,Y(v,z)\big]=\sum_{l\in\Nbb}{m\choose l}z^{m-l}Y\big(Y(u)_lv,z\big).\label{eq63}
\end{align}
Thus, for $m\geq0$, $Y(Y(u)_mv,z)$ can be computed inductively by
\begin{gather}\label{eq64}
\begin{gathered}
Y\big(Y(u)_0v,z\big)=\big[Y(u)_0,Y(v,z)\big]\\
Y\big(Y(u)_mv,z\big)=\big[Y(u)_m,Y(v,z)\big]-\sum_{l=0}^{m-1}{m\choose l}z^{m-l}Y\big(Y(u)_lv,z\big).
\end{gathered}	
\end{gather}

We now see the close relation between the Lie brackets of vertex operators and the data $Y\big(Y(u)_mv,z\big),m\geq0$. The latter plays a very different role from $Y\big(Y(u)_mv,z\big),m<0$. To understand this relation better, we write the associativity relation \eqref{eq60} as
\begin{align}
Y(u,z_2)Y(v,z_1)=\sum_{m\in\Zbb}(z_2-z_1)^{-m-1}Y(Y(u)_mv,z_1)\label{eq62}	
\end{align}
when $0<|z_2-z_1|<|z_1|$. Here, we understand $Y(u,z_2)Y(v,z_1)$ as $Y(v,z_1)Y(u,z_2)$ when $0<|z_1|<|z_2|$ or more generally, as a linear functional on $\Vbb^{\otimes 2}$ sending $w\otimes w'$ to $T_{\fk P_{z_1,z_2}}(w\otimes v\otimes u\otimes w')$ (the correlation function associated to \eqref{eq61}) for all $(z_1,z_2)\in\Conf^2(\Cbb^\times)$. Then the part $m\geq 0$ in \eqref{eq62} accounts for the poles of $T_{\fk P_{z_1,z_2}}(w\otimes v\otimes u\otimes w')$ at $z_2=z_1$.

The summand in \eqref{eq62} vanishes for sufficiently positive $m$. In physics, a series expansion of the form
\begin{align*}
	A(z_2)B(z_1)=\sum_{m\geq -N}(z_2-z_1)^mC^m(z_1)
\end{align*}
is called the  \textbf{operator product expansion (OPE)} of the fields $A(z_2),B(z_1)$. Thus, in the VOA context, \emph{OPEs are just the associativity property \eqref{eq60}}. OPE is useful to physicists because it allows them to reduce the calculation of $4$-point correlations functions to that of $3$-point ones, or in general, $N$-point to $(N-1)$-point.

We split the right hand side of \eqref{eq62} into two parts: $m\geq 0$, which is called the \textbf{regular terms} since it has no poles at $z_2=z_1$, and $m<0$ called the \textbf{singular terms}. Thus
\begin{align*}
Y(u,z_2)Y(v,z_1)=\frac{Y(Y(u)_{N-1}v,z_1)}{(z_2-z_1)^N}+\cdots+	\frac{Y(Y(u)_0v,z_1)}{(z_2-z_1)}+\text{regular terms},
\end{align*}
or, written in physics language,
\begin{align}
	Y(u,z_2)Y(v,z_1)\sim\frac{Y(Y(u)_{N-1}v,z_1)}{(z_2-z_1)^N}+\cdots+	\frac{Y(Y(u)_0v,z_1)}{(z_2-z_1)}.
\end{align}

Thus, to summarize, \emph{\eqref{eq49} establishes a close relationship between the Lie brackets of vertex operators,  the finite poles of the correlation function $T_{\fk P_{z_1,z_2}}$ at $z_1=z_2$, and the finitely may singular terms in the OPE of vertex operators}. As a special case, from \eqref{eq63} and \eqref{eq64} one sees that two vertex operators $Y(u,z_2),Y(v,z_1)$ commute (namely, their Fourier modes $Y(u)_m,Y(v)_k$ commute) iff there are no singular terms in the OPE of $Y(u,z_2)Y(v,z_1)$, iff $T_{\fk P_{z_1,z_2}}(\cdot\otimes v\otimes u\otimes\cdot)$ is holomorphic on a neighborhood of $z_2=z_1$.

\subsection{}\label{lb57}

In the previous subsection, we derived the relationship from the definition of VOAs (in particular, from the VOA Jacobi identity). So one may ask this natural question: does this relationship rely on the full Jacobi identity? For instance, does it rely on \eqref{eq65}?

The answer is no. In a very vague sense, any of the following three implies the others without assuming the full Jacobi identity.
\begin{enumerate}
\item Suitable Lie bracket relations hold for a pair of field operators $A(z_2),B(z_1)$.
\item The finite poles of (the analytic continuation of) $\bk{w',A(z_2)B(z_1)w}$ at $z_2=z_1$.
\item The finitely many singular terms in the OPE of $A(z_2)B(z_1)$ and, in particular, \emph{the existence of such OPE}.
\end{enumerate}
Clearly, the third one a priori implies the second one, since the second does not assume the existence of OPE. Thus, as we have said that OPEs are roughly the same as associativity, we see that the associativity (and indeed, the full Jacobi identity) can be derived from the first or the second statement above. This is called the \textbf{reconstruction theorem} because it allows us to build examples of VOAs by checking only a small part of the Jacobi identity, namely the Lie bracket relations. This theorem is the most important one for constructing examples of VOAs.

A rigorous and detailed discussion of the equivalence of the above three statements will be given in Sec. \ref{lb83}. The first and the second statements correspond to three seeming different but indeed equivalent definitions of the \textbf{locality} of $A(z_2),B(z_1)$. (There are two ways to describe the second one, a formal variable way and a complex analytic way.) Here, we first state the rigorous definition of the first one.

\subsection{}

We let $\Vbb=\bigoplus_{n\in\Nbb}\Vbb(n)$ be an $\Nbb$-graded vector space, graded by a diagonalizable operator $L_0$. We do not assume that $\Vbb$ and $L_0$ are  from any graded vertex algebra. 
\begin{df}\label{lb61}
An \textbf{($L_0$-)homogeneous field (operator)} on $\Vbb$ is an element
\begin{align*}
A(z)=\sum_{n\in\Zbb} A_nz^{-n-1}\in\End(\Vbb)[[z^{\pm1}]]	
\end{align*}
(where each $A_n$ is in $\End(\Vbb)$) satisfying
\begin{align}
[L_0,A(z)]=\Delta_A\cdot A(z)+z\partial_z A(z)	
\end{align}
or equivalently,
\begin{align}
[L_0,A_n]=(\Delta_A-n-1)A_n\qquad(\forall n\in\Zbb)	.
\end{align}
$\Delta_A$ is called the \textbf{weight} of $A(z)$.
\end{df}

Clearly, a homogeneous field $A(z)$ satisfies the \textbf{lower truncation property} $A(z)w\in\Cbb((z))$ (for all $w\in\Vbb$).

\begin{df}[\textbf{Local fields} (Lie algebraic version)]\label{lb59}
Given homogeneous fields $A(z)$ and $B(z)$, we say $A(z)$ is  \textbf{local} to $B(z)$ if there exist $C^j(z)=\sum_{n\in\Zbb}C^j_nz^{-n-1}\in\End(\Vbb)[[z^{\pm1}]]$ (where $j=0,1,\dots,N-1$ for some $N\in\Nbb$) satisfying
\begin{align}
[A_m,B_k]=\sum_{l=0}^{N-1}{m\choose l}C^l_{m+k-l}	\label{eq66}
\end{align}
for all $m,k\in\Zbb$.  We consider the right hand side of \eqref{eq66} as $0$ if $N=0$.
\end{df}

\begin{rem}\label{lb60}
$A(z)$ is local to $B(z)$ if and only if there exist $D^0(z),\dots,D^{N-1}(z)\in\End(\Vbb)[[z^{\pm1}]]$ satisfying for all $m,k\in\Zbb$ that
\begin{align}
[A_m,B_k]=\sum_{l=0}^{N-1}m^lD^l_{m+k}.	\label{eq76}
\end{align}
This is because $\wtd C^l_j:=C^l_{j-l}$ and $D^l_j$ are related by   $\wtd C^l_j+\sum_{p=l+1}^{N-1}a_{p,l}\cdot\wtd C^{p}_j=D^l_j$ where each $a_{p,l}\in\Rbb$ is determined by ${m\choose p}=m^p+\sum_{l=1}^{p-1}a_{p,l}\cdot m^{l}$.
\end{rem}

\begin{exe}
Use \eqref{eq76} to show that if $A(z)$ is local to $B(z)$ then $B(z)$ is local to $A(z)$.
\end{exe}

\subsection{}

Roughly speaking, reconstruction theorem says that if we have a small set $\mc E$ of operators $A(z)\in\End(\Vbb)$ that generates $\Vbb$ and satisfies all the axioms in the definition of graded vertex algebras/VOAs, except that the Jacobi identity is replaced by the weaker condition that the operators in $\mc E$ are mutually local, then the Jacobi identity is automatically satisfies, and hence $\Vbb$ is a graded vertex algebra/VOA. This theorem will be proved in Sec. \ref{lb84}.

\begin{thm}[\textbf{Reconstruction theorem}]\label{lb39}
Let $\mc E$ be a set of $L_0$-homogeneous fields on $\Vbb$.  Assume that the following conditions are satisfied. Then $\Vbb$ has a unique graded vertex algebra structure such that each $A(z)\in\mc E$ is a vertex operator (namely, is of the form $Y(u,z)$ for some $u\in\Vbb$), and that the vacuum vector $\id$ and the operator $L_{-1}$  are those described in the following.
\begin{itemize}
\item Creation property: There is a distinguished vector $\id\in\Vbb(0)$ such that $A(z)\id$ has no negative powers of $z$ for all $A(z)\in\mc E$.
\item Translation property: There is a distinguished $L_{-1}\in\End(\Vbb)$ such that $L_{-1}\id=0$, and that for each $A(z)\in\mc E$ we have $[L_{-1},A(z)]=\partial_zA(z)$.
\item Generating property: Vectors of the form $A^1_{n_1}\cdots A^k_{n_k}\id$ (where $k\in\Nbb$, $A^1(z),\dots,A^k(z)\in\mc E$, and $n_1,\dots,n_k\in\Zbb$) span $\Vbb$. \index{00@Generating sets of homogeneous fields}
\item Locality: Any two fields of $\mc E$ are local.
\end{itemize}
Moreover, if $L_0,L_{-1}$ can be extended to a sequence of operators $(L_n)_{n\in\Zbb}$ on $\Vbb$ such that $\sum_{n\in\Zbb}L_nz^{-n-2}$ belongs to $\mc E$, and that the Virasoro relation \eqref{eq16} is satisfied for some $c\in\Cbb$, then $\Vbb$ is a VOA whose conformal vector $\cbf$ satisfies $Y(\cbf,z)=\sum_{n\in\Zbb}L_nz^{-n-2}$.
\end{thm}
Note that the uniqueness of the graded vertex algebra/VOA structure follows directly from \eqref{eq67}. The non-trivial part of this theorem is of course the existence of such structure.

\begin{rem}\label{lb44}
The end of the reconstruction Thm. \ref{lb39} means that in order to show that a graded vertex algebra $\Vbb$ is a VOA, it suffices to show that $L_0,L_{-1}$ can be extended to $(L_n)_{n\in\Zbb}$ satisfying the Virasoro relation, that $T(z)=\sum L_nz^{-n-2}$ satisfies the creation property (namely, $L_n\id=0$ for all $n\geq-1$), and that $T(z)$ is local with any field in $\mc E$ (by showing for instance the conformal Ward identity $[L_m,A(z)]=z^{m+1}\partial_zA(z)+\Delta_A\cdot z^mA(z)$ for all $A(z)\in\mc E$ if one expects that all $A(z)$ are ``primary"). The translation property is automatically satisfied due to the Virasoro relation $[L_{-1},L_n]=-(n+1)L_{n-1}$.
\end{rem}

\section{Constructing examples of VOAs}

\subsection{}
In the previous section, we have mentioned some important examples of VOAs:  affine VOAs and Virasoro VOAs. But we didn't explain why they exist. This is the task of this section. The standard references for this section  are \cite[Chapter 6]{LL} and \cite{Was10} (with emphasis on the unitarity aspect). 

The style of this section is different from the previous ones: it has a strong flavor of Lie theory. The methods in this section will not be used in the future (except when we discuss examples of VOA modules). So the readers can safely skip this section if they do not want to bother with the existence issue. (But they should at least read Subsec. \ref{lb58} on tensor product VOAs.)

Our first class of examples are Virasoro VOAs, namely, those generated by the conformal vector $\cbf$. To begin with, the \textbf{Virasoro algebra} is a Lie algebra $\Vir=\Span_\Cbb\{L_n,K:n\in\Zbb\}$ satisfying the bracket relation
\begin{subequations}
\begin{gather*}
		[L_m,L_n]=(m-n)L_{m+n}+\frac K{12}(m+1)m(m-1)\delta_{m,-n},\\
		[K,L_n]=0.	
\end{gather*}
\end{subequations}

We know that any VOA must satisfy $L_n\id=0$ for all $n\geq-1$. Motivated by this fact, we have:
\begin{pp}\label{lb45}
Let $\Vbb$ be a representation of $\Vir$ such that $L_0$ is diagonalizable and has $\Nbb$-spectrum. Assume that $\Vbb$ has a distinguished vector $\id$ killed by $L_n$ for all $n\geq -1$,  that vectors of the form $L_{n_1}\cdots L_{n_k}\id$ (where $k\in\Nbb,n_1,\dots,n_k\in\Zbb$) span $\Vbb$, and that $K$ acts as a constant $c\in\Cbb$. Then $\Vbb$ has a unique natural structure of a Virasoro VOA. Its central charge is $c$.
\end{pp} 

\begin{proof}
This follows immediately from the reconstruction Thm. \ref{lb39}. Note that by \eqref{eq76}, $\sum_{n\in\Zbb}L_nz^{-n-2}$ is local to itself due to the Virasoro relation.
\end{proof}

\subsection{}

Thus, it remains to construct $\Vir$-modules satisfying the conditions in Prop. \ref{lb45}. Let us first find a ``largest" such module. We expect that this module should have basis $L_{-n_1}\cdots L_{-n_k}\id$ where $n_1\geq\cdots\geq n_k\geq 2$,  because:
\begin{exe}
Let $\Vbb$ be as in Prop. \ref{lb45}. Prove by induction on $k$ that $L_{n_1}\cdots L_{n_k}\id$ (for any $n_1,\dots,n_k$) can be written as a linear combination of $L_{-m_1}\cdots L_{-m_l}\id$ where $l\in\Nbb,m_1,\dots,m_l\geq2$. (Hint: if $n_j\leq -2$, move $L_{n_j}$ to the rightmost by using the Virasoro relation.)
\end{exe}

Now let us construct this largest module $V_\Vir(c,0)$ for each $c\in\Cbb$. Its basis consists of $(-n_1,\dots,-n_k)$ where $k\in\Nbb$ and $n_1\geq \cdots\geq n_k\geq2$. The one with $k=0$ is denoted by $\id$. If $n\geq n_1$, we simply define the action of $L_{-n}$ on each $(-n_1,\dots,-n_k)$ to be $(-n,-n_1,\dots,-n_k)$. But we also want to define the action of $L_n$ on $(-n_1,\dots,-n_k)=L_{-n_1}\cdots L_{-n_k}\id$ for all $n\in\Zbb$. In practice, we can write down the formula explicitly using the Virasoro relation. For instance: $L_0L_{-n_1}\cdots L_{-n_k}\id=(n_1+\cdots+n_k)L_{-n_1}\cdots L_{-n_k}\id$, and
\begin{align}
&L_3L_{-4}L_{-3}\id=[L_3,L_{-4}]L_{-3}\id+L_{-4}[L_3,L_{-3}]\id\nonumber\\
=&7L_{-1}L_{-3}\id+6L_{-4}L_0\id+2cL_{-4}\id=(14+2c)L_{-4}\id.\label{eq72}
\end{align}
There is a natural question about this approach: how do we verify that such defined action of $\Vir$ on $V_\Vir(c,0)$ preserves the Lie bracket relations of $\Vir$?

\subsection{}

The standard way to deal with is issue is to use the \textbf{Poincar\'e–Birkhoff–Witt (PBW)} theorem, which says the following: Let $\fk g$ be a Lie algebra (over any field). Let $U(\fk g)$ be its universal enveloping algebra, i.e., the largest unital associative algebra containing and generated by the vector space $\fk g$ such that $xy-yx=[x,y]$ for all $x,y\in\fk g$. If $E$ is a basis of $U(\gk)$ with a total order $\leq$, then vectors of the form
\begin{align}
x_1x_2\cdots x_k\qquad (k\in\Nbb,x_1\geq x_2\geq\cdots\geq x_k\in E)	\label{eq71}
\end{align}
(when $k=0$, we understand this expression as $1$) form a basis of $U(\fk g)$.

The remarkable point about the PBW theorem is that if we define a vector space $V$ to have a basis of vectors as in \eqref{eq71}, and if we define the action of $x\in\fk g$ using the Lie bracket relations of $\fk g$ (similar to the argument in \eqref{eq72}), then this gives a well defined action of $\fk g$ on $V$ preserving the bracket relations of $\fk g$, i.e., this gives a well defined representation of $\fk g$.

To apply the PBW theorem to our construction of VOAs, we need the following result:
\begin{exe}\label{lb46}
Suppose $\gk=\gk_1\oplus\gk_2$ where $\gk_1,\gk_2$ are Lie subalgebras of $\gk$. Use the PBW theorem to show that there is an isomorphism of vector spaces $U(\gk_1)\otimes U(\gk_2)\rightarrow U(\gk)$ sending each $x_1\cdots x_k\otimes y_1\cdots y_l$ to $x_1\cdots x_ky_1\cdots y_k$ where $x_\blt\in\gk_1,y_\blt\in\gk_2$. 
\end{exe}
The proof is an easy application of the PBW theorem, which we leave to the readers.

\subsection{}
Consider the following Lie subalgebras of $\Vir$:
\begin{align*}
V_-=\Span\{L_n:n\leq -2\},\qquad V_+=\Span\{K,L_n;n\geq -1\}.	
\end{align*}
$\Cbb_c=\Cbb$ is a representation of $V_+$ if we let $L_n$ act as $0$ and $K$ as $c$. So $\Cbb_c$ is also a $U(V_+)$-module. Now $U(\Vir)$ is clearly a right $U(V_+)$-module. So 
\begin{align*}
\Ind_{U(V_+)}^{U(V)}\Cbb_c:=	 U(\Vir)\otimes_{U(V_+)}\Cbb_c
\end{align*}
is a (left) $U(\Vir)$-module, called the \textbf{induced representation} of $\Cbb_c$. This is a $\Vir$-module, and by Exercise \ref{lb46}, this vector space is isomorphic to $U(V_-)\otimes_\Cbb U(V_+)\otimes_{(U(V_+))}\Cbb_c\simeq U(V_-)$, which by the PBW theorem has a basis of vectors the form $L_{-n_1}\cdots L_{-n_k}\id$ where $\id$ is the unit $1$ and $n_1\geq\dots\geq n_k\geq 2$. So we can view $V_\Vir(c,0)$ as $\Ind_{U(V_+)}^{U(V)}\Cbb_c$. In particular, this proves that $V_\Vir(c,0)$ carries a ntural structure of representation of $\Vir$. Hence, by Prop. \ref{lb45}, $V_\Vir(c,0)$ is a Virasoro VOA with central charge $c$.

\begin{exe}
Find an explicit expression of $Y(L_{-4}\cbf,z)$ on $V_\Vir(c,0)$ in terms of the Virasoro operators $L_n$.
\end{exe}

\subsection{}\label{lb47}

$V_\Vir(c,0)$ is not always an irreducible $\Vir$-module. But the irreducible cases are the most interesting one. For instance, every CFT-type unitary VOA is irreducible. (See \cite{CKLW18}.)

The method of getting irreducible examples is quite standard in Lie theory: We shall take the largest quotient of $V_\Vir(c,h)$. To be more precise, note that for any proper $\Vir$-invariant subspace $W$ of $V_\Vir(c,h)$, note that $L_0$ is diagonalizable on $W$,\footnote{In general, if $D$ is a diagonalizable linear operator on a vector space $M$ and $W$ is an $D$-invariant subspace of $M$, then $D|_W$ is diagonalizable. To see this, choose any $w\in M$ which is a finite sum $w_1+\cdots +w_k$ where each summand is an eigenvector of $D$ in $M$, and they have distinct eigenvalues $\lambda_1,\dots,\lambda_k$. Use polynomial interpolation to find a polynomial $p$ such that $p(\lambda_j)=\delta_{1,j}\lambda_1$.  So $w_1=p(D)w\in M$.} i.e., $W$ has a $L_0$-grading, whose lowest weight must not be $0$ since otherwise it contains $\id$ and hence must be $V_\Vir(c,0)$. Let $I$ be the span of all such $W$, then $I$ is the largest proper $\Vir$-subspace since $I$ has no non-zero weight-$0$ vectors. Then
\begin{align*}
L_\Vir(c,0):=V_\Vir(c,0)/I
\end{align*}
is an irreducible $\Vir$-module, which is also a Virasoro VOA of CFT type by Prop. \ref{lb45}.

\subsection{}

One may wonder when $L_\Vir(c,0)$ equals $V_\Vir(c,0)$, i.e., when $I$ is trivial. Indeed, $I$ is non-trivial if and only if
\begin{align}
c=c_{p,q}=1-\frac{6(p-q)^2}{pq}	\label{eq75}
\end{align}
where $p,q\in\{2,3,4,\dots\}$ are relatively prime. (Cf. \cite[Rem. 6.1.13]{LL} and the reference therein.) In this case, $L_\Vir(c,0)$ is called a \textbf{minimal model}. It has finitely many irreducible modules. Minimal models are an important class of ``rational" VOAs. More precisely: \textbf{rational and $C_2$-cofinite} VOAs. We will give precise meanings of these terms in later sections. The theory of conformal blocks for such VOAs is well-established.

It is a deep result that $L_\Vir(c,0)$ is a unitary $\Vir$-module if and only if $c\geq1$ or $c$ satisfies \eqref{eq75} with $|p-q|=1$, namely,
\begin{align}
	c=1-\frac 6{m(m+1)}\label{eq130}
\end{align}
for some integer $m\geq2$. We refer the readers to \cite[Chapter 8]{FMS} and \cite[Chapter IV]{Was10} for details.

\subsection{}

We now turn to  affine VOAs. We fix a finite dimensional complex Lie algebra $\gk$ together with a non-degenerate symmetric invariant bilinear form $(\cdot,\cdot)$. (Indeed, we will not use the non-degeneracy until we define the Virasoro operators.) Recall that invariance means
\begin{align}
([X,Y],Z)=-(Y,[X,Z]).\label{eq74}	
\end{align}
An \textbf{affine Lie algebra}  is $\wht\gk$ with basis $X_n,K$ (where $X\in\gk,n\in\Zbb$) satisfying the Lie bracket relation
\begin{gather*}
\begin{gathered}
[X_m,Y_n]=[X,Y]_{m+n}+m(X,Y)\delta_{m,-n}K,\\
[K,X_m]=0.
\end{gathered}	
\end{gather*}
It is more convenient to add a basis element $D$ (which will be the $L_0$ of our VOA) to $\wht\gk$ to get a slightly larger Lie algebra $\wtd\gk=\wht\gk\rtimes\Cbb D$ such that
\begin{align*}
[D,X_m]=-mX_m,\qquad [D,K]=0.
\end{align*}
$\wtd\gk$ is also called an affine Lie algebra.

\subsection{}\label{lb87}

$\wtd\gk$ decomposes into Lie subalgebras $\wtd\gk=\wtd\gk_-\oplus\wtd\gk_+$ where
\begin{align*}
\wtd\gk_-=\Span\{X_n:X\in\gk,n<0\},\qquad \wtd\gk_+=\Span\{X_n,K,D:X\in\gk,n\geq0\}.	
\end{align*}
Then $U(\wtd\gk)\simeq U(\wtd\gk_-)\otimes U(\wtd\gk_+)$ by Exercise \ref{lb46}. For each $l\in\Cbb$ called the \textbf{level}, we let $\Cbb_l=\Cbb$ be an $\wtd\gk_+$-module such that $K$ acts as $l$ and $X_n,D$ act trivially. We are interested in two types of associated VOAs:
\begin{align}
V_\gk(l,0):=\Ind_{U(\wtd\gk_+)}^{U(\wtd\gk)}\Cbb_l=U(\wtd \gk)_{\otimes U(\wtd\gk_+)}	\Cbb_l
\end{align}
which as a vector space is naturally equivalent to $U(\wtd \gk_-)$. Let $\id$ be the $1\otimes 1$ in $U(\wtd \gk)_{\otimes U(\wtd\gk_+)}	\Cbb_l$. Then $V_\gk(l,0)$ has a basis of vectors
\begin{align*}
X^{i_1}_{-n_1}\cdots X^{i_k}_{-n_k}\id	
\end{align*}
(which has $D$-weight $n_1+\cdots+n_k$) written in the lexicon order where $\{X^1,X^2,\dots\}$ is a basis of $\gk$ and $n_1,\dots,n_k>0$. Thus, $D$ is diagonaizable on $V_\gk(l,0)$ with non-negative spectrum, and each eigenspace is finite dimensional. Similar to the argument in Subsec. \ref{lb47}, we can take a simple quotient
\begin{align}
L_\gk(l,0)=V_\gk(l,0)/I	
\end{align}
where $I$ is the largest proper $\wtd\gk$-submodule.

$V_\gk(0,0)$ and $L_\gk(0,0)$ are never equal, because:
\begin{exe}
Show that $L_\gk(0,0)$ is spanned by $\id$. Equivalently, show that if $l=0$, then $I$ contains all $D$-eigenvectors with eigenvalues $>0$.
\end{exe}

In the following, we discuss how to make $L_\gk(l,0)$ a VOA since $L_\gk(l,0)$ is our main interest. The same method applies to $V_\gk(l,0)$.

For each $X\in\gk$, $X_n$ acts on $L_\gk(l,0)$ in an obvious way. We define $X(z)\in\End(L_\gk(l,0))[[z^{\pm1}]]$ to be
\begin{align*}
	X(z)=\sum_{n\in\Zbb}X_nz^{-n-1}.
\end{align*}
It is a homogeneous field (with respect to $D$) with weight $1$ since $[D,X_n]=-nX_n$. One checks easily that these fields satisfy the creation property and locality, and that they  generate $L_\gk(l,0)$. So it remains to construct $L_{-1}$ and verify the translation property. We shall actually construct all $L_n$ in a uniform way.

\subsection{}

Choose a basis $E$ of $\gk$, which gives a dual basis $\{\wch e:e\in E\}$, namely, for each $e,f\in E$, $(e,\wch f)=\delta_{e,f}$ with respect to the given non-degenerate symmetric bilinear form $(\cdot,\cdot)$. By linear algebra,
\begin{align}
\sum_{e\in E}e\otimes \wch e\in\gk\otimes \gk	
\end{align}
is independent of the choice of basis $E$. As an immediate consequence, we have
\begin{align}
\sum_{e\in E} \wch e\otimes e=\sum_{e\in E} e\otimes\wch e.\label{eq77}
\end{align}
With the help of $\gk\otimes\gk\rightarrow \gk,X\otimes Y\mapsto[X,Y]$, this shows $\sum[\wch e,e]=\sum[e,\wch e]=-\sum[\wch e,e]$, i.e.,
\begin{align}
\sum_{e\in E}[\wch e,e]=0.\label{eq81}
\end{align}
\begin{lm}\label{lb50}
For each $X\in\gk$, we have
\begin{align}
	\sum_{e\in E}\wch e\otimes [e,X]=-\sum_{e\in E}[\wch e,X]\otimes e.
\end{align}
\end{lm}
\begin{proof}
Evaluate both sides by $Y\otimes Z$ using $(\cdot,\cdot)$, and use the invariance condition \eqref{eq74} to show that both sides equal $(Y,[X,Z])$.
\end{proof}
Thus, on each $\gk$-module $V$, we have $\sum \wch e[e,X]+\sum [\wch e,X]e=0$, namely,
\begin{align}
\sum_e [\wch ee,\gk]=0.	
\end{align}
So when $V$ is finite dimensional and  is either irreducible or trivial, $\Omega=\sum \wch ee\in\End(V)$ is a constant by Schur's lemma, called \textbf{Casimir element}. The operator $\Omega$ in general gives the nagative Laplactian of the Lie group action.

\begin{ass}\label{lb52}
We assume that for the adjoint representation $\gk\curvearrowright\gk,X\mapsto [X,\cdot]$, the Casimir element is a constant $2h^\vee\in\Cbb$, i.e., 
\begin{align}
\sum_{e\in E}[\wch e,[e,\cdot]]=2h^\vee\id_\gk.	
\end{align}
This is always true when $\gk$ is abelian (in which case $h^\vee=0$) or simple. We assume
\begin{align*}
	l+h^\vee\neq0.
\end{align*}
\end{ass}

\subsection{}

We define the Virasoro operator ``as if" the conformal vector is
\begin{align}
	\cbf=\gamma^{-1}\sum_e\wch e_{-1}e_{-1}\id\qquad\big(\text{where }\gamma=2(l+h^\vee)\big).\label{eq79}
\end{align}
Thus, using \eqref{eq59} and $L_m=Y(\cbf)_{m+1}$, and noting that $\wch e_ie_j=e_i\wch e_j$ by \eqref{eq77}, we write down the definition
\begin{align}
L_m=\gamma^{-1}\sum_e\Big(\sum_{k\leq-1} \wch e_ke_{m-k}+\sum_{k\geq0} \wch e_{m-k}e_k\Big)
\end{align}
acting on $L_\gk(l,0)$. This is called \textbf{Sugawara construction}. One checks that this sum is finite when acting on any vector.

To use the reconstruction theorem, we need the following crucial fact:
\begin{pp}\label{lb48}
For each $m,n\in\Zbb$ and $X\in\gk$,
\begin{align}
[L_m,X_n]=-nX_{m+n}.	\label{eq78}
\end{align}
\end{pp}
(Note that if we assume the existence of the VOA structure, then \eqref{eq78} can be derived from the conformal Ward identity \eqref{eq69} and the fact that $X_{-1}\id$ is indeed primary.)

\begin{cv}\label{lb49}
In the remaining part of this section, we suppress $\sum_e$ if possible.
\end{cv}
From this proposition, we know that $T(z)=\sum_m L_mz^{-m-2}$ and $X(z)$ are local, and $X(z)$ satisfies the translation property. To use the reconstruction theorem, we need to check the following facts:

\begin{lm}
The following are true.
\begin{enumerate}[label=(\alph*)]
\item $T(z)$ satisfies the creation property, namely, $L_n\id=0$ if $n\geq-1$.
\item $L_0$ agrees with $D$.
\item $\{L_n\}$ satisfy the Virasoro relation.
\end{enumerate}
\end{lm}

\begin{proof}
(a) Assume $m\geq-1$.  $\sum_{k\geq0}\wch e_{m-k}e_k\id$ is $0$ since all $X_0\id$ are zero by our construction. 	$\sum_{k\leq-1}\wch e_k e_{m-k}\id$ is $0$ because $m-k\geq m+1\geq0$.

(b) Since $L_0\id=0$ and $[L_0,X_n]=-nX_n=[D,X_n]$,  $L_0$ and $D$ act the same on any $X^1_{n_1}\cdots X^k_{n_k}\id$. So $L_0=D$.

(c) By the reconstruction theorem, $L_\gk(l,0)$ is a graded vertex algebra. Clearly $L_m=Y(\cbf)_{m+1}$ by our definition of $L_m$ and $\cbf$. We can use \eqref{eq49} or \eqref{eq54} to show
\begin{align}
[L_m,L_n]=Y(L_{-1}\cbf)_{m+n+2}+\sum_{l\geq 0}{m+1\choose l+1}Y(L_l\cbf)_{m+n+1-l}. \label{eq80}
\end{align}
By the expression $\cbf$, clearly $L_0\cbf=D\cbf=2\cbf$. Also, from the Sugawara construction, we clearly have $[D,L_m]=-mL_m$, i.e., $[L_0,L_m]=-mL_m$. So $L_l\cbf=0$ if $l>2$. To find $[L_m,L_n]$, we need to find $L_1\cbf$ and $L_2\cbf$.

Using \eqref{eq78}, we calculate that $\gamma L_1\cbf$ equals
\begin{align*}
L_1\wch e_{-1}e_{-1}\id=[L_1,\wch e_{-1}]e_{-1}\id+\wch e_{-1}[L_1,e_{-1}]\id=\wch e_0e_{-1}\id+\wch e_{-1}e_0\id=\wch e_0 e_{-1}\id.	
\end{align*}
And $\wch e_0e_{-1}\id=[\wch e_0,e_{-1}]\id=[\wch e,e]_{-1}\id$ equals $0$ by \eqref{eq81}. Recall $K$ acts as $l$ on $L_\gk(l,0)$. Then $\gamma L_2\cbf$ equals
\begin{align*}
&L_2\wch e_{-1}e_{-1}\id=[L_2,\wch e_{-1}]e_{-1}\id+\wch e_{-1}[L_2,e_{-1}]\id=\wch e_1e_{-1}\id+\wch e_{-1}e_1\id\\
=&	\wch e_1e_{-1}\id=[\wch e_1,e_{-1}]\id=[\wch e,e]_0\id+l(\wch e,e)\id,
\end{align*}
which equals $l\cdot \dim \gk\cdot \id$. Therefore, using \eqref{eq53}, we find that \eqref{eq80} becomes the Virasoro relation where $\frac c{2}=\gamma^{-1}l\cdot\dim\gk$.
\end{proof}

Thus, by the reconstruction Thm. \ref{lb39}, we conclude:

\begin{thm}
For $l\neq -h^\vee$, $V_\gk(l,0)$ and $L_\gk(l,0)$ are VOAs satisfying $Y(X_{-1}\id,z)=\sum_{n\in\Zbb}X_nz^{-n-1}$ (for all $X\in\gk$) if we define the conformal vector $\cbf$ as in \eqref{eq79}. The central charge is $\frac{l\dim\gk}{l+h^\vee}$.
\end{thm}

\subsection{$\star$}

It remains to prove Prop. \ref{lb48}. Recall Convention \ref{lb49} that we are suppressing $\sum_e$. The following discussions focus on $L_\gk(l,0)$, though the same argument works for $V_\gk(l,0)$.

\begin{lm}\label{lb51}
For all $i,j,n\in\Zbb$, on $L_\gk(l,0)$ we have $[\wch e_ie_j,X_n]=A_{i,j,n}+B_{i,j,n}$ where
\begin{subequations}
\begin{gather}
A_{i,j,n}=\wch e_i[e,X]_{j+n}-\wch e_{i+n}[e,X]_j\label{eq82}\\
B_{i,j,n}=-nl(\delta_{j,-n}X_i+\delta_{i,-n}X_j).	\label{eq83}
\end{gather}
\end{subequations}
In particular, $B_{i,j,n}=B_{j,i,n}$.
\end{lm}

\begin{proof}
We compute
\begin{align*}
[\wch e_ie_j,X_n]=\wch e_i[e_j,X_n]+[\wch e_i,X_n]e_j=A_{i,j,n}+B_{i,j,n}	
\end{align*}	
where
\begin{gather*}
A_{i,j,n}=\wch e_i[e,X]_{j+n}+[\wch e,X]_{i+n}e_j\\
B_{i,j,n}=-nl \delta_{j,-n}\cdot \wch e_i(e,X)-nl\delta_{i,-n}(\wch e,X)e_j.	
\end{gather*}
$B_{i,j,n}$ clearly equals \eqref{eq83} by the basic property of (dual) basis. Note that in general, for all $i,j\in\Zbb$, by Lemma \ref{lb50} and the map $\gk\otimes\gk\rightarrow\End(L_\gk(l,0))$ sending $Y\otimes Z$ to $Y_iZ_j$, we have
\begin{align}
[\wch e,X]_ie_j=-\wch e_i[e,X]_j.\label{eq86}
\end{align}
This proves that $A_{i,j,n}$ equals \eqref{eq82}.
\end{proof}

\begin{proof}[Proof of Prop. \ref{lb48}]
We compute
\begin{align*}
&[\gamma L_m,X_n]=\sum_{k\leq-1}[\wch e_ke_{m-k},X_n]+\sum_{k\geq0}[\wch e_{m-k}e_k,X_n]\\
=&\sum_{k\leq-1}(A_{k,m-k,n}+B_{k,m-k,n})+\sum_{k\geq0}(A_{m-k,k,n}+B_{m-k,k,n}).	
\end{align*}
By Lemma \ref{lb51}, the sum of the two  $B$ is
\begin{align*}
\sum_{k\in\Zbb}B_{k,m-k,n}=-nl\sum_{k\in\Zbb} (\delta_{m-k,-n}X_k+\delta_{k,-n}X_{m-k})=-2nl X_{m+n}.
\end{align*}
Also,
\begin{align*}
\sum_{k\geq 0}A_{m-k,k,n}=\sum_{k\geq0}\wch e_{m-k}[e,X]_{k+n}-\sum_{k\geq 0}\wch e_{m+n-k}[e,X]_k	
\end{align*}
where the two sums are both finite when acting on any vector. But the first summand is just (setting $j=k+n$) $\sum_{j\geq n}\wch e_{m+n-j}[e,X]_j$. So
\begin{align}
\sum_{k\geq 0}A_{m-k,k,n}=-\big(\wch e_{m+n}[e,X]_0+\wch e_{m+n-1}[e,X]_1+\cdots+\wch e_{m+1}[e,X]_{n-1}\big).	\label{eq84}
\end{align}
Simiarly, setting $i=m-k$,
\begin{align}
&\sum_{k\leq-1}A_{k,m-k,n}=\sum_{i\geq m+1}\wch e_{m-i}[e,X]_{i+n}-\sum_{i\geq m+1}\wch e_{m+n-i}[e,X]_i\nonumber\\
=&-\big(\wch e_{n-1}[e,X]_{m+1}+\cdots+\wch e_0[e,X]_{m+n} \big).	\label{eq85}
\end{align}
By Lemma \ref{lb53}, the sum of \eqref{eq84} and \eqref{eq85} is $-2nh^\vee X_{m+n}$. This finishes the proof.
\end{proof}

\begin{lm}\label{lb53}
For each $i,j\in\Zbb$ and $X\in\gk$,
\begin{align}
\wch e_i[e,X]_j+\wch e_j[e,X]_i=2h^\vee X_{i+j}.	
\end{align}
\end{lm}
This is the only place we use the definition of $h^\vee$ (cf. Assumption \ref{lb52}).

\begin{proof}
By \eqref{eq86},
\begin{align*}
\wch e_i[e,X]_j+\wch e_j[e,X]_i=\wch e_i[e,X]_j-[\wch e,X]_je_i,	
\end{align*}
which, according to \eqref{eq77} and the map $\gk\otimes\gk\rightarrow\End(L_\gk(l,0)),Y\otimes Z\mapsto [Y,X]_jZ_i$, is
\begin{align*}
\wch e_i[e,X]_j-[e,X]_j \wch e_i=[\wch e_i,[e,X]_j]=[\wch e,[e,X]]_{i+j}+il\delta_{i,-j}(\wch e,[e,X]).
\end{align*}
Now, by the invariance of $(\cdot,\cdot)$, $(\wch e,[e,X])=([\wch e,e],X)$, which equals $0$ by \eqref{eq81}. By the definition of $h^\vee$, $[\wch e,[e,X]]=2h^\vee X$. We are done with the proof.
\end{proof}

\subsection{}

We now discuss the unitarity problem for affine VOAs. We first look at Heisenberg VOAs, namely, we assume $\gk$ is abelian. We assume that $\gk$ is equipped with an inner product $(\cdot|\cdot)$ (antilinear on the first variable) and an anti-unitary involution $X\in\gk\mapsto X^*\in\gk$. Recall that ``anti-unitary" means that $*$ is conjugate linear, bijective, and satisfies
\begin{align*}
(X^*|Y^*)=(Y|X).
\end{align*}
Involution means $X^{**}=X$.  By considering $\gk$ as an (abelian) \textbf{unitary Lie algebra}, we regard $*$ and $(\cdot|\cdot)$ as part of the data of $\gk$.
\begin{exe}
Show that $\gk$ is unitarily isomorphic to $\Cbb^n$ with the standard inner product, where the involution is $(z_1,\dots,z_n)\mapsto (\ovl{z_1},\dots,\ovl{z_n})$, the unique one fixing $\Rbb^n$. (Hint: First find an real isomorphism from $\{X\in\gk:X^*=X\}$ to $\Rbb^n$  preserving the inner products.)
\end{exe}

It is easy to check that the bilinear form $(\cdot,\cdot)$ on $\gk$ defined by
\begin{align}
(X,Y)=(X^*|Y)	
\end{align}
is symmetric. (It is obviously invariant.)  We define $V_\gk(l,0)$ using this bilinear form. 

\begin{pp}\label{lb54}
$l>0$ if and only if there exists an inner product $\bk{\cdot|\cdot}$ on $V_\gk(l,0)$ satisfying $\bk{\id|\id}=1$ such that the representation of $\wtd\gk$ on $V_\gk(l,0)$ is unitary, namely, for each $X\in\gk,u,v\in V_\gk(l,0),n\in\Zbb$,
\begin{align*}
\bk{u|X_nv}=\bk{(X^*)_{-n}u|v},\qquad \bk{u|Kv}=\bk{Ku|v},\qquad \bk{u|Dv}=\bk{Dv|u},
\end{align*} 
or simply $(X_n)^\dagger=X^*_{-n}$, $K^\dagger=K$, $D^\dagger=D$ for short. Such $\bk{\cdot|\cdot}$ is unique if it exists.
\end{pp}

The if part is easy to explain: We compute that $\bk{X_{-1}\id|X_{-1}\id}=\bk{\id|X^*_1X_{-1}\id}=\bk{\id|[X^*,X]_0\id}+l(X^*,X)=l(X|X)$. So if $\bk{\cdot|\cdot}$ is an inner product, then for each $X\neq 0$, $l(X|X)$ is $>0$. So $l>0$. We now explain the only if part. To simplify discussions, by scaling $(\cdot|\cdot)$ and hence $(\cdot,\cdot)$ by $l$ and $K$ by $l^{-1}$, it suffices to assume $l=1$. (Indeed, people usually just assume $l=1$ when discussing Heisenberg VOAs.)

\subsection{$\star$}
Assume $l=1$. The uniqueness of $\bk{\cdot|\cdot}$ is easy to prove:
\begin{align*}
\bk{X^1_{n_1}\cdots X^k_{n_k}\id|Y^1_{m_1}\cdots Y^l_{m_l}\id}=	\bk{\id|(X^k)^*_{-n_k}\cdots (X^1)^*_{-n_1}Y^1_{m_1}\cdots Y^l_{m_l}\id}=:\bk{\id|w}.
\end{align*}
If $n_1+\cdots+n_k=m_1+\cdots+m_l$, then $w$ has $D$-weight $0$. But the weight-$0$ homogeneous vectors are $\Cbb\id$. So $w=\lambda\id$, and $\lambda$ uniquely determined by the Lie bracket relations. If $n_1+\cdots+n_k\neq m_1+\cdots+m_l$, then the weight of $w$ is not $0$. So $w=0$ since $\bk{D\id|w}=\bk{\id|Dw}$.

The existence part follows from the general construction of symmetric Fock spaces. Let $W$ be a (complex) inner product space together with an antiunitary involution $*$. Note that for each $N\in\Nbb$, $W^{\otimes N}$ is naturally an inner product space. We assume $W$ has an orthonormal basis $\{e_i:i\in\fk I\}$ (which spans $W$ algebraically). Let $\fk S_N$ be the set of permutations on $\{1,\dots,N\}$. For each $v_1,\dots,v_N\in W$, we define
\begin{align*}
v_1\cdots v_N:=\frac 1{\sqrt{N!}}\sum_{\sigma\in\fk S_N}v_{\sigma(1)}\otimes\cdots \otimes v_{\sigma(N)},
\end{align*}
and let $S^N(W)\subset W^{\otimes N}$ be spanned by all such vectors. We understand $S^0(W)$ to be the standard one dimensional inner product space $\Cbb$. In particular, it has a unit vector $1$. $S^N(W)$ has an orthonormal basis consisting of vectors
\begin{align}
\frac {(e_{i_1})^{m_1}\cdots (e_{i_k})^{m_k}}{\sqrt{m_1!\cdots m_k!}}\qquad(\text{where }i_1,\dots,i_k\in\fk I\text{ are distinct and }\sum_{j=1}^k m_j=N).	\label{eq87}
\end{align}

Define an inner product space
\begin{align}
S^\blt(W)=\bigoplus_{N\in\Nbb} S^N(W),	
\end{align}
called the \textbf{symmetric Fock space} associated to $W$. For each $v\in W$, define linear maps $a^+(v),a^-(v)$ on $S^\blt(W)$ determined by
\begin{subequations}
\begin{gather}
a^+(v)1=v,\qquad a^+(v)v_1\cdots v_N=vv_1\cdots v_N.\\
a^-(v)1=0,\qquad a^-(v)v_1\cdots v_N=\sum_{j=1}^N \bk{v^*|v_j}\cdot v_1\cdots v_{j-1}v_{j+1}\cdots v_N.
\end{gather}
\end{subequations}
The maps $a^\pm(v)$ are well-defined, thanks to the basis \eqref{eq87}.

\begin{exe}
Prove the following relations.
\begin{enumerate}
\item $a^+(v)^\dagger=a^-(v^*)$, namely, $\bk{\xi|a^+(v)\nu}=\bk{a^-(v^*)\xi|\nu}$ for all $\xi,\nu\in S^\blt (W)$. (Hint: write  $\xi,\nu,v$ in terms of the previously mentioned orthonormal basis vectors.)
\item $[a^-(u),a^+(v)]=\bk{u^*|v}\id_{S^\blt(W)}$. This is called the \textbf{canonical commutation relation (CCR)}.
\end{enumerate}
\end{exe}

Now let $W=t^{-1}\cdot\gk [t^{-1}]$ with inner product
\begin{align*}
\bk{Xt^{-m}|Yt^{-n}}=m(X|Y)\delta_{m,n}	
\end{align*}
for all $m,n\in\Zbb_+$. The involution is defined to be $(Xt^{-m})^*=X^*t^{-m}$. According to the description of the basis of $S^\blt(W)$, $V_\gk(1,0)$ is linearly equivalent to $S^\blt(W)$ by identifying $\id$ with $1$ and 
\begin{align}
X^1_{-n_1}\cdots X^k_{-n_k}\id\qquad\text{with}\qquad X^1t^{-n_1}\cdots X^kt^{-n_k}.
\end{align}
We use the inner product on $S^\blt(W)$ to define the one on $V_\gk(1,0)$. Using CCR, it is not hard to check that the action of $X_n$ on $V_\gk(1,0)\simeq S^\blt(W)$ is
\begin{align}
X_n=\left\{
\begin{array}{cc}
a^+(Xt^{-|n|}) &\text{if }n<0,\\
0 &\text{if }n=0,\\
a^-(Xt^{-n}) &\text{if }n>0.	
\end{array}
\right.	
\end{align}
Thus, the representation of $\wtd\gk$ on $V_\gk(1,0)$ is unitary.

\subsection{}

When $l>0$, $L_\gk(l,0)$ and $V_\gk(l,0)$ share the same unitarity property, because:

\begin{pp}
If $l\in\Cbb^\times$, then $V_\gk(l,0)$ is an irreducible $\wtd\gk$-module, i.e., $V_\gk(l,0)=L_\gk(l,0)$.
\end{pp}

\begin{proof}
We assume $l>0$ and prove the irreducibility using the unitarity. Choose any non-zero $\wtd\gk$-submodule $W$ of $V_\gk(l,0)$. We shall show $W=V_\gk(l,0)$.

Since $W$ is a $D$-invariant subspace, $D$ is diagonalizable on $W$. So $W$ has $D$-grading $W=\bigoplus_{n\geq a}W(n)$ where $a$ is the smallest eigenvalue of $D$ on $W$. We claim that $a=0$. Then, as the $D$-weight $0$ subspace of $V_\gk(l,0)$ is clearly spanned by $\id$, we must have $\id\in W$. From this one sees that $W=V_\gk(l,0)$.

Suppose $a>0$. We choose a non-zero $w\in W(a)$, which must be a sum of vectors of the form $X^1_{-n_1}\cdots X^k_{-n_k}\id$ where the sum of the positive integers $n_1,\dots,n_k$ is $a$. Then by the unitarity, $\bk{w|w}$ (which is non-zero) is a sum of $\bk{\id|(X^k)^*_{n_k}\cdots (X^1)^*_{n_1}w}$. So for some $X^1_{n_1}$, the vector $v=(X^1)^*_{n_1}w$ must be nonzero. But $v$ has $D$-weight $a-n_1<a$, and clearly $v\in W$. This is a contradiction. 

Now, for a general $l=|l|e^{\im\theta}\in\Cbb^\times$,  we may replace $(\cdot,\cdot)$ by $e^{\im\theta}(\cdot,\cdot)$ and $K$ by $e^{-\im\theta}K$. Then  $(\cdot|\cdot)$ and the new $(\cdot,\cdot)$ are related by $(X|Y)=(e^{\im\theta}X^*,Y)$, and $X\mapsto e^{\im\theta}X^*$ is clearly an antiunitary involution. So $V_\gk (l,0)$ becomes $V_\gk(|l|,0)$ under the new involution and bilinear form, and the latter has been proved irreducible.
\end{proof}

\subsection{}

In general, we say a finite-dimensional (complex) Lie algebra $\gk$ is \textbf{unitary} if it is equipped with an inner product $(\cdot|\cdot)$ and an antiunitary involution $*$ satisfying the following conditions:
\begin{enumerate}
\item $[X,Y]^*=[Y^*,X^*]$. 
\item The inner product is \textbf{invariant}, namely, the adjoint representation of $\gk$ on $\gk$ is unitary:
\begin{align*}
	([X,Y]|Z)=(Y|[X^*,Z]).	
\end{align*}
\end{enumerate}
Then $(X,Y):=(X^*|Y)$ defines a symmetric invariant bilinear form on $\gk$.

\begin{exe}
Let $\fk k$ be an $\gk$-invariant and $*$-invariant (i.e. $[\gk,\fk k]\subset\fk k$, $\fk k^*=\fk k$) subspace of $\gk$.  Let $\fk k^\perp$ be the orthogonal complement of $\fk k$ in $\gk$.
\begin{enumerate}
\item Show that $\fk k^\perp$ is also $\gk$-invariant and $*$-invariant.
\item Show that $[\fk k,\fk k^\perp]=0$ and hence $[\gk,\fk k]=[\fk k,\fk k]$. Consequently, if $\fk k$ is an irreducible $\gk$-submodule, then $\fk k$ is an irreducible $\fk k$-module, which is (by definition) a simple Lie algebra if moreover $[\fk k,\fk k]\neq 0$.
\end{enumerate}
\end{exe}

Let $\zk$ be the center of $\gk$, which is clearly $\gk$- and $*$-invariant. Let $\gk_{\mathrm{ss}}=\zk^\perp$ so that $\gk=\zk\oplus^\perp\gk_{\mathrm{ss}}$. Then the adjoint representation $\gk\curvearrowright\gk_{\mathrm{ss}}$ (equivalently, $\gk_{\mathrm{ss}}\curvearrowright\gk_{\mathrm{ss}}$) has orthogonal irreducible $*$-invariant decomposition $\gk_{\mathrm{ss}}=\gk_1\oplus^\perp\cdots\oplus^\perp\gk_N$. Each $\gk_j$ is a simple unitary Lie algebra, which is classified by the type A-G Dynkin diagrams.

Conversely, suppose $\gk$ is a complex simple Lie algebra, which is the complexification of $\gk_\Rbb$ which is the real Lie algebra of a finite dimensional compact real Lie group $G$. It is well known in Lie theory that the real vector space $\gk_\Rbb$ has a unique up to $\Rbb_{>0}$-scalar multiplication  $G$-invariant (equivalently, $\gk_\Rbb$-invariant) inner product, which extends to a complex invariant inner product $(\cdot|\cdot)$ on $\gk$ thanks to  the real direct sum $\gk=\gk_\Rbb\oplus\im\gk_\Rbb$. The antiunitary involution on $\gk$ is defined to be the unique one fixing $\im\gk_\Rbb$. Thus $\gk$ is unitary. 

Therefore, in general, if $\fk z$ is abelian and $\gk_1,\dots,\gk_N$ are simple, then $\gk=\fk z\oplus\gk_1\oplus\cdots\oplus\gk_N$ is naturally a unitary Lie algebra. So the study of unitary affine VOAs for unitary Lie algebras reduces to that of the abelian case (which we have finished) and the simple case.

\subsection{}

When $\gk$ is simple, the unitarity properties of $V_\gk(l,0)$ and $L_\gk(l,0)$ are very different from the abelian case. Indeed, in the abelian case, scaling the inner product does not change the unitary equivalence class of abelian unitary Lie algebras. (This is because scaling the vectors by a non-zero constant is an isomorphism of abelian Lie algebras.) But this is no longer true for non-abelian Lie algebras. Also, it turns out that for a simple $\gk$, $V_\gk(l,0)$ is never a unitary $\wtd\gk$-module, and $L_\gk(l,0)$ is unitary for a discrete set of levels $l$ if one fixes the invariant inner product, or for a discrete set of invariant inner product if one fixes the level $l$.

Assume $\gk$ is a simple Lie algebra with compact form decomposition $\gk=\gk_\Rbb\oplus\im\gk_\Rbb$. Let $*$ be the unique involution fixing $\im\gk_\Rbb$. As we have said, the invariant bilinear forms on $\gk_\Rbb$ (and hence on $\gk$) are unique up to scalar multiplication. So it would be better to fix one. The one that people usually choose is:

\begin{cv}\label{lb55}
We choose the invariant inner product on $\gk$ (under which $*$ is antiunitary) to be the unique one such that the longest roots of $\gk$ have length $\sqrt 2$.
\end{cv}
It follows from  the invariance of $(\cdot|\cdot)$ that  $h^\vee$ (defined in Assumption \ref{lb52}) is a positive number. (To see this, one may choose $E$ to be an orthonormal basis of $\gk$, and check that its dual basis $\{\wch e:e\in E\}$ satisfies $\wch e=e^*$.) The $h^\vee$ corresponding to the inner product in  Convention \ref{lb55} is called the \textbf{dual Coxeter number}. We have said that $L_\gk(l,0)$ and $V_\gk(l,0)$ are VOAs if $l\neq -h^\vee$. So this is true when $l\geq0$.

\begin{thm}
$L_\gk(l,0)$ is unitary if and only if $l\in\Nbb$. For such $l$, $L_\gk(l,0)$ is called a \textbf{Weiss-Zumino-Witten (WZW)} model. 
\end{thm}

This is a highly non-trivial result whose proof relies on deep Lie theory. We refer the readers to \cite[Chapter III, Sec. 2 and 10]{Was10} for a proof. Moreover, just like minimal models, WZW models are $C_2$-cofinite and rational. So their representation categories are extremely nice. Due to these properties, WZW models are central objects in the study of CFT and VOAs. (However, Heisenberg VOAs are neither $C_2$-cofinite nor rational.)

\subsection{}\label{lb58}

We have shown the existence of affine VOAs when the unitary Lie algebra $\gk$ is abelian or simple. The general case can be addressed by  tensor product VOAs.

Let $\Vbb_1,\Vbb_2$ be VOAs. We use the diagonalizable operator $L_0\otimes\id_{\Vbb_2}+\id_{\Vbb_1}\otimes L_0$ to define the grading on $\Vbb_1\otimes\Vbb_2$. The vacuum vector is $\id\otimes\id$. $\Vbb_1\otimes\Vbb_2$ is clearly generated by $Y(v_1)_m\otimes\id_{\Vbb_2}$ and $\id_{\Vbb_1}\otimes Y(v_2)_n$ where $v_j\in\Vbb_j$, and $Y(v_1,z)\otimes\id_{\Vbb_2}$ is clearly local to  $Y(u_1,z)\otimes\id_{\Vbb_2}$ (where $u_1\in\Vbb_1$) and $\id_{\Vbb_1}\otimes Y(v_2,z)$. One checks that $L_{-1}\otimes\id_{\Vbb_2}+\id_{\Vbb_1}\otimes L_{-1}$ satisfies the translation property. So $\Vbb\otimes\Vbb$ is naturally a graded vertex algebra by the reconstruction theorem. Its vertex operator satisfies
\begin{align}
Y(v_1\otimes\id,z)=Y(v_1,z)\otimes\id_{\Vbb_2},\qquad Y(\id\otimes v_2,z)=\id_{\Vbb_1}\otimes Y(v_2,z).
\end{align}

\begin{exe}
Use \eqref{eq59} or \eqref{eq88} to show
\begin{align}
Y(v_1\otimes v_2,z)=Y(v_1,z)\otimes Y(v_2,z).
\end{align}
Equivalently,
\begin{align}
Y(v_1\otimes v_2)_n=\sum_{n\in\Zbb}~ \sum_{n_1+n_2=n-1}Y(v_1)_{n_1}Y(v_2)_{n_2}.	
\end{align}
\end{exe}
 
When $\Vbb_1,\Vbb_2$ are VOAs with conformal vectors $\cbf_1,\cbf_2$ and central charges $c_1,c_2$, it is easy to check that $\Vbb_1\otimes\Vbb_2$ is a VOA with conformal vector $\cbf_1\otimes\id+\id\otimes\cbf_2$. In particular, its Virasoro operators are $Y(\cbf_1\otimes\id+\id\otimes\cbf)_{n+1}=L_n\otimes\id_{\Vbb_2}+\id_{\Vbb_1}\otimes L_n$. We call $\Vbb_1\otimes\Vbb_2$ the \textbf{tensor product VOA} of $\Vbb_1$ and $\Vbb_2$.

\begin{exe}
Show that $\Vbb_1\otimes\Vbb_2$ has central charge $c_1+c_2$.
\end{exe}

We remark that if $\Vbb_1$ and $\Vbb_2$ are unitary, then their tensor product is also unitary (cf. \cite{DL14,CKLW18}).

\begin{exe}
Let $\gk_1,\dots,\gk_N$ be either abelian or simple. Let $\Vbb=L_{\gk_1}(l_1,0)\otimes\cdots\otimes L_{\gk_N}(l_N,0)$. Show that the weight-$1$ subspace $\Vbb(1)$, as a Lie algebra (cf. Subsec. \ref{lb56}), is naturally isomorphic to $\gk:=\gk_1\oplus\cdots\oplus\gk_N$. Show that $\Vbb(1)$ generates $\Vbb$.
\end{exe}

\begin{exe}
Show that $L_{\Cbb^n}(1,0)\simeq\underbrace{L_\Cbb(1,0)\otimes\cdots\otimes L_\Cbb(1,0)}_{n\text{ times}}$.
\end{exe}

\section{Local fields}\label{lb83}

\subsection{}

Having explored some important examples, we now return to the general theory. The goal of this section is to understand the close relationship between the three statements in Subsec. \ref{lb57}. The precise formulation of statement 1 is the Lie bracket version of local fields, as defined in Def. \ref{lb59} or Rem. \ref{lb60}. For statement 2 we give two rigorous descriptions: the complex analytic version and the formal variable version of local fields. We first give the complex analytic version, which is more intuitive. 

We first need to define:

\begin{df}
Let $\Omega$ be a locally compact Hausdorff space. A  series of functions $\sum_n f_n$ is said to \textbf{converge absolutely and locally uniformly (a.l.u.) \index{00@Absolute and locally uniform (a.l.u.) convergence} on $\Omega$} if  each $x_0\in\Omega$ is contained in a neighborhood $U$ such that
\begin{align*}
\sup_{x\in U}\sum_n|f_n(x)|<+\infty.	
\end{align*}
Equivalently, for each compact subset $K\subset \Omega$, we have $\sup_{x\in K}\sum_n|f_n(x)|<+\infty$
\end{df}

Clearly, if each $\sum f_n$ converges a.l.u. and each $f_n$ is continuous (resp. holomorphic), then so is the limit $\sum f_n$.

\subsection{}

Now let $\Vbb=\bigoplus_{n\in\Nbb}\Vbb(n)$ be graded by a diagonalizable $L_0$. Recall the projection $P_n:\Vbb^\cl=\prod_{m\in\Nbb}\Vbb(m)\rightarrow \Vbb(n)$ (cf. \eqref{eq89}). Let $A(z)=\sum A_nz^{-n-1}$, $B(z)=\sum B_nz^{-n-1}$ be homogeneous fields with weights $\Delta_A,\Delta_B$ (cf. Def. \ref{lb61}). For each $n\in\Nbb$ and $v,v'\in\Vbb$, we have
\begin{align}
\bk{v',A(z_1)P_nB(z_2)v} \quad \in \scr O(\Cbb^\times\times\Cbb^\times)
\end{align}
since, when $v,v'$ are homogeneous, this expression equals
\begin{align*}
\bk{v',A_{n_1}B_{n_2}v}z_1^{-n_1-1}z_2^{-n_2-1}	
\end{align*}
where $n_2,n_1$ are determined by $\Delta_B+\wt v-n_2-1=n$ and $\Delta_A+n-n_1-1=\wt v'$.

\begin{df}[\textbf{Local fields} (complex analytic version)]\label{lb62}
We say $A(z)$ and $B(z)$ are \textbf{local} to each other if for each $v\in\Vbb,v'\in\Vbb'$ the following hold.
\begin{enumerate}
\item The series
\begin{subequations}
\begin{gather}
\bk{v',A(z_1)B(z_2)v}:=\sum_{n\in\Nbb}\bk{v',A(z_1)P_nB(z_2)v}\label{eq90}\\
\bk{v',B(z_2)A(z_1)v}:=\sum_{n\in\Nbb}\bk{v',B(z_2)P_nA(z_1)v}\label{eq91}
\end{gather}
\end{subequations}
converge a.l.u. respectively on the open sets $\Omega_1=\{(z_1,z_2)\in\Cbb^2:0<|z_2|<|z_1|\}$ and $\Omega_2=\{(z_1,z_2)\in\Cbb^2:0<|z_1|<|z_2|\}$. So \eqref{eq90} and \eqref{eq91} are automatically holomorphic functions on $\Omega_1$ and $\Omega_2$.
\item \eqref{eq90} and \eqref{eq91} can be analytically continued to the same holomorphic function $f_{v,v'}$ on $\Conf^2(\Cbb^\times)$. Moreover, there exists $N\in\Nbb$ depending only on $A,B$ but not on $v,v'$ such that the function
\begin{align}
(z_1-z_2)^Nf_{v,v'}(z_1,z_2)	
\end{align}
is holomorphic on $\Cbb^\times\times\Cbb^\times$.
\end{enumerate}
\end{df}

Roughly speaking, this definition says that \eqref{eq90} and \eqref{eq91} converge a.l.u on $\Omega_1,\Omega_2$ and extend to the same holomorphic function on $\Conf^2(\Cbb^\times)$ which has poles of order at most $N$ at $z_1=z_2$, where $N$ is independent of $v,v'$.

\subsection{}

The readers will immediately notice that there is another natural convergence condition on $A(z_1)B(z_2)$: that $\bk{v',A(z_1)B(z_2)v}$ as a formal Laurent series of $z_1,z_2$ converges a.l.u. on $\Omega_1$. Or  more precisely, the joint series
\begin{align}
\sum_{m,n\in\Zbb}\bk{v',A_mB_nv}z_1^{-m-1}z_2^{-n-1}\label{eq96}	
\end{align}
converges a.l.u. on $\Omega$. Is this equivalent to the convergence statement in Def. \ref{lb62}? The answer is yes. But people will easily overlook the need to justify this equivalence. And we need both versions of convergence since they are useful in different situations. For instance, to prove that formal variable implies complex analytic, it is easier to prove the a.l.u. convergence of the formal Laurent series; to prove the other direction, it is better to use the a.l.u. convergence of the RHS of \eqref{eq90} and \eqref{eq91}.

There is (unfortunately) one more way to understand the convergence \eqref{eq90}: we regard the RHS as a series of formal Laurent series of $z_1,z_2$, which converges formally to the LHS also as a formal Laurent series in the following sense:

\begin{df}\label{lb66}
	We say that a sequence (indexed by $k$)
	\begin{align*}
		f_k(z_1,\dots,z_M)=\sum_{n_1,\dots,n_M\in\Zbb}f_{k,n_\blt}z_1^{n_1}\cdots z_M^{n_M}	
	\end{align*}
	of elements of $W[[z_1^{\pm1},\dots,z_M^{\pm1}]]$ \textbf{converges formally}\index{00@Formal convergence} to
	\begin{align*}
		f(z_1,\dots,z_M)=\sum_{n_1,\dots,n_M\in\Zbb}f_{n_\blt}z_1^{n_1}\cdots z_M^{n_M}	
	\end{align*}
	if for each $n_\blt$, the coefficient $f_{k,n_\blt}$ equals $f_{n_\blt}$ except for finitely many $k$.
\end{df}

Note that in applications, $k$ can be in any countable set: $\Nbb,\Zbb,\Zbb^2$, etc.

We will show the equivalence of the two a.l.u. convergences with the help of the following obvious lemma.

\begin{lm}\label{lb65}
Let $X$ be a complex manifold. Let $f_k(x,z_\blt)$ be a series of $\scr O(X)$ -coefficients monomials of $z_1^{\pm1},\dots,z_M^{\pm1}$, i.e.,   $f_k(x,z_\blt)=g_k(x)z_1^{n_{k,1}}\cdots z_M^{n_{k,M}}$ where each $g_k\in\scr O(X)$ and $n_{k,j}\in\Zbb$. Assume that if $k\neq k'$ then $n_{k,j}\neq n_{k',j}$ for some $1\leq j\leq M$. Then $\sum_k f_k(x,z_\blt)$ clearly converges formally to some $f\in\scr O(X)[[z_1^{\pm1},\dots,z_M^{\pm1}]]$.  Namely, the following holds formally:
\begin{align}
f(x,z_\blt)=\sum_k f_k(x,z_\blt).\label{eq98}
\end{align}
Moreover, let $\Omega$ be an open subset of $\Cbb^M$. Then $f(x,z_\blt)$ as an $\scr O(X)$-coefficients formal Laurent series of $z_\blt$ (indexed by the powers of $z_\blt$) converges a.l.u. on $X\times\Omega$ if and only if the series $\sum_k f_k(x,z_\blt)$ (indexed by $k$) converges a.l.u. on $X\times\Omega$. If so, then the two limits are equal, i.e., \eqref{eq98} holds as holomorphic functions on $X\times\Omega$.
\end{lm}

\subsection{}\label{lb184}

We now show that \eqref{eq90} as an infinite sum over $n$ converges a.l.u. on $\Omega_1$ iff the LHS of \eqref{eq90} as a formal Laurent series of $z_1,z_2$ converges a.l.u. on $\Omega_1$. Note that both convergences are preserved by taking linear combinations. So it suffices to assume that $v,v'$ are homogeneous.\footnote{We cannot directly apply Lemma \ref{lb65} if $v,v'$ are not homogeneous.} Let us prove our claim by checking that  the sum \eqref{eq90} satisfies the assumption in Lemma \ref{lb65}:

Since $B(z_2)$ is homogeneous, similar to the proof of Prop. \ref{lb63}, we have the translation covariance
\begin{align}
	B(\lambda z_2)=\lambda^{-\Delta_B}\cdot\lambda^{L_0}B(z_2)\lambda^{-L_0}.\label{eq103}	
\end{align}
This shows
\begin{align}
B(z_2)=z_2^{-\Delta_B}\cdot z_2^{L_0}B(1)z_2^{-L_0}.\label{eq117}	
\end{align}
A similar relation holds for $A(z_1)$. So for each $n\in\Nbb$, we have (in the sense of  $\Cbb[z_1^{\pm1},z_2^{\pm1}]$)
\begin{align}
	&\bk{v',A(z_1)P_nB(z_2)v}=	\bk{v',z_1^{L_0-\Delta_A}A(1)z_1^{-L_0}P_nz_2^{-\Delta_B+L_0}B(1)z_2^{-L_0}v}\nonumber\\
	=&z_1^{\wt v'-\Delta_A}z_2^{-\Delta_B-\wt v}\cdot \big(\frac{z_2}{z_1}\big)^n\bigbk{v',A(1)P_nB(1)v},\label{eq97}
\end{align}
noting that $z_1^{-L_0}P_n=z_1^{-n}P_n$ and $P_nz_2^{L_0}=P_nz_2^n$. 

\begin{exe}\label{lb89}
Let $\Vbb$ be a graded vertex algebra. Choose $u,v\in\Vbb$ and $v'\in\Vbb'$.  Use \eqref{eq40} and Lemma \ref{lb65} to show that
\begin{align}
\sum_{n\in\Nbb}\bigbk{v',Y(u,z)P_ne^{-\tau L_{-1}}v}=\sum_{n\in\Nbb}\bigbk{v',e^{-\tau L_{-1}}P_nY(u,z+\tau)v},
\end{align}
where both sides converge a.l.u. on $\{z\neq 0,|\tau|<|z|\}$ to the same function. (Note that the RHS is a finite sum.)
\end{exe}

\subsection{}

\begin{df}[\textbf{Local fields} (formal variable version)]
There exists $N\in\Nbb$ depending only on $A$ and $B$ such that the equation
\begin{align}
(z_1-z_2)^N[A(z_1),B(z_2)]=0\label{eq92}	
\end{align}
holds on the level of $\End(\Vbb)[[z_1^{\pm1},z_2^{\pm1}]]$.
\end{df}

This version of local fields is the most common in the literature, partly because it is the most concise. Indeed, since locality implies Jacobi identity, many people use locality instead of Jacobi identity in the definition of VOAs. We do not take this approach because locality has its own limitation: in the definition of VOA modules and conformal blocks, we need the full Jacobi identity, but not just locality.

\subsection{}\label{lb64}

Almost everyone will have the following question when they first see this definition: doesn't \eqref{eq92} imply $[A(z_1),B(z_2)]=0$? The answer is no: for a vector space $W$ in general, it is possible that $fg=0$ for some $f(z_1,z_2),g(z_1,z_2)\in W[[z_1^{\pm1},z_2^{\pm1}]]$ although $f\neq 0,g\neq 0$. In other words, assuming $W=\Cbb$ for simplicity, then $\Cbb[[z_1^{\pm1},z_2^{\pm1}]]$ (unlike $\Cbb[[z_1,z_2]]$) has ``zero divisors". (We put quotation marks here because $\Cbb[[z_1^{\pm1},z_2^{\pm1}]]$ is actually not a ring.)

Indeed, choose $N>0$. Then $(z_1-z_2)^{-N}$ can be expanded in two ways: $f=\sum_{j\geq 0}{-N\choose j}z_1^j(-z_2)^{-N-j}$ as if $|z_1|<|z_2|$, and $g=\sum_{j\geq 0}{-N\choose j}z_1^{-N-j}(-z_2)^j$ as if $|z_1|>|z_2|$. Then $f\neq g$, but $(z_1-z_2)^Nf=(z_1-z_2)^Ng=1$. So $(z_1-z_2)^N$ is a zero divisor. Similarly, one shows that $(1+z)^N$ (where $N>0$) is a zero divisor in $\Cbb[[z^{\pm1}]]$ by expanding $(1+z)^{-N}$ as if $|z|<1$ and as if $|z|>1$.

This phenomenon is closely related to the fact that $\Cbb[[z_1^{\pm1},z_2^{\pm1}]]$ (and similarly $\Cbb[[z^{\pm1}]]$) is not a ring: the product of two arbitrary elements cannot be defined. This is in contrast to the following basic fact:
\begin{lm}
If $\mbb F$ is a field, then $\mbb F((z))$ is naturally a field. In particular, $\mbb F((z))$ is closed under taking product and inverse (for non-zero elements).
\end{lm}

\begin{exe}
Suppose $f(z)\in\mbb F((z))$ is not zero. Find an algorithm of determining the inverse $1/f(z)$.
\end{exe}

Thus, by taking $\mbb F=\Cbb((z_1))$, we see that $\Cbb((z_1))((z_2))$ is also a field. This implies that $(z_1-z_2)^N$ is not a zero divisor in $\Cbb((z_1))((z_2))$: Suppose that $(z_1-z_2)^Nf(z_1,z_2)=0$,  and that $f\in\Cbb((z_1))((z_2))$, i.e., 
\begin{align*}
f(z_1,z_2)=\sum_{\begin{subarray}{c}
	n_2\geq L\\	
n_1\geq K_{n_2} 
\end{subarray}}f_{n_1,n_2}z_1^{n_1}z_2^{n_2}
\end{align*}
for some $L\in\Zbb$ and $K_{n_2}\in\Zbb$ for each $n_2$. Then $f=0$ because $f=(z_1-z_2)^{-N}(z_1-z_2)^Nf=0$ where $(z_1-z_2)^{-N}$ is the \emph{inverse of $(z_1-z_2)^N$ in $\Cbb((z_1))((z_2))$}, which is  $\sum_{j\geq0}{-N\choose j}z_1^{-N-j}(-z_2)^j$. (If we expand $(z_1-z_2)^{-N}$ as if $|z_1|<|z_2|$, we get the inverse of $(z_1-z_2)^N$ in $\Cbb((z_2))((z_1))$.)

If, however, $f\in\Cbb[[z_1^{\pm1},z_2^{\pm1}]]$ is neither in $\Cbb((z_2))((z_1))$ nor in $\Cbb((z_1))((z_2))$, then $(z_1-z_2)^Nf=0$ does not imply $f=0$ since we cannot multiply both sides by either inverse of $(z_1-z_2)^N$. (There is no associativity law $(fg)h=f(gh)$ in $\Cbb[[z_1^{\pm1},z_2^{\pm1}]]$ even if both sides can be defined.)

\subsection{}

Each of the three versions has its own advantage, and it is the goal of this section to prove the equivalence of them. This is a crucial step for proving the reconstruction theorem. Moreover, note that in each of these three versions there is a number $N$. \emph{We can prove the  equivalence of the three versions for the same $N$.}

The Lie algebraic version is the easiest to verify in concrete examples: we have already seen this in the previous section. In contrast, the complex analytic one is the most difficult to verify. But the complex analytic version is closest to how physicists understand local fields. So it allows us to prove results in a similar fashion as in physics literature. For instance: we will prove the existence of OPE using the complex analytic version of local fields. And with the help of OPE, we can prove that complex analytic implies Lie bracket version in the same way as deriving the algebraic Jacobi identity from the complex analytic one using residue theorem. Finally, to prove the complex analytic version from the Lie algebraic one, we need the help of the formal variable version. Also, using the formal variable version, we can generalize the statements in Def. \ref{lb62} to more than two fields. This generalization is crucial for proving the reconstruction theorem.
\begin{equation*}
\begin{tikzcd}
	\text{Lie algebraic} \arrow[r,"\ref{lb67}"] & \text{Formal variable} \arrow[r, bend left,"\ref{lb68}"] & \text{Complex analytic} \arrow[l, bend left,"\ref{lb69}"'] \arrow[ld, bend left,"\ref{lb70}"'] \\
	& \text{OPE} \arrow[lu, bend left,"\ref{lb72}"']             &                                                                    
\end{tikzcd}
\end{equation*}

From the above chart, we see that a direct proof from complex analytic to formal variable is not necessary for proving the equivalence of the three versions. We will still give such a proof because: In the VOA theory,  many definitions and properties can be stated in both algebraic (i.e., formal variable)  and complex analytic language. It is important to learn how to translate between these two.

\subsection{}\label{lb67}

The proof that Lie algebraic implies formal variable is by brutal force. Assume the homogeneous fields $A(z),B(z)$ satisfies \eqref{eq66}. Let us prove that $(z_1-z_2)^N[A(z_1),B(z_2)]=0$.

\begin{proof}
Showing $(z_1-z_2)^N[A(z_1),B(z_2)]=0$ amounts to showing that  for all $m,n\in\Zbb$, the following expression vanishes:
\begin{align}\label{eq102}
\begin{aligned}
&\Res_{z_1=0}\Res_{z_2=0}~z_1^mz_2^n\cdot (z_1-z_2)^N[A(z_1),B(z_2)]dz_1dz_2\\
=&\sum_{j=0}^N \Res_{z_1=0}\Res_{z_2=0}~{N\choose j}z_1^{m+j}z_2^{n+N-j}(-1)^{N-j}[A(z_1),B(z_2)]dz_1dz_2\\
=&\sum_{j=0}^N{N\choose j}(-1)^{N-j}[A_{m+j},B_{n+N-j}]\\
=&\sum_{j=0}^N{N\choose j}(-1)^{N-j}\sum_{l=0}^{N-1}{m+j\choose l}C^l_{m+n+N-l}.	
\end{aligned}
\end{align}
This expression vanishes because of  the next lemma.
\end{proof}

\begin{lm}
For each $N\in\Zbb_+$, $m\in\Zbb$, and  $l=0,1,\dots,N-1$, we have
\begin{align*}
\sum_{j=0}^N{N\choose j}(-1)^{N-j}{m+j\choose l}=0.	
\end{align*}
\end{lm}

\begin{proof}
The function $f(z)=(1+z)^mz^N$ is holomorphic on $\Dbb_1$, and its power series expansion contains no less-than-$N$ powers of $z$. But we can expand $f(z)$ in the following way:
\begin{align*}
&f(z)=(1+z)^m(-1+1+z)^N=\sum_{j=0}^N(1+z)^m\cdot{N\choose j} (-1)^{N-j}(1+z)^j\\
=&	\sum_{j=0}^N \cdot{N\choose j}(-1)^{N-j}(1+z)^{m+j}=\sum_{j=0}^N\sum_{l\in\Nbb} {N\choose j}(-1)^{N-j}{m+j\choose l}z^l.
\end{align*}
The coefficient before $z^l$ vanishes when $l<N$. This proves our formula.
\end{proof}

\subsection{}\label{lb68}

Let us prove that formal variable implies complex analytic. The method is due to \cite{FHL93}.
\begin{proof}
Assume $(z_1-z_2)^N[A(z_1),B(z_2)]=0$. Choose homogeneous $v\in\Vbb,v'\in\Vbb'$. Let
\begin{align*}
	f(z_1,z_2)=	\bk{v',A(z_1)B(z_2)v},\qquad g(z_1,z_2)=\bk{v',B(z_2)A(z_1)v}
\end{align*}
which are both in $\Cbb[[z_1^{\pm1},z_2^{\pm1}]]$. So is
\begin{align*}
	\phi(z_1,z_2)=(z_1-z_2)^Nf(z_1,z_2)=(z_1-z_2)^Ng(z_1,z_2).	
\end{align*}

Step 1. We claim that $\phi$ is actually in $\Cbb[z_1^{\pm1},z_2^{\pm1}]$. Note that
\begin{align}
f(z_1,z_2)=\sum_{m,n\in\Zbb}\bk{A^\tr_mv',B_nv}z_1^{-m-1}z_2^{-n-1}.\label{eq106}
\end{align}
Since $B_n$ increases the weights by $\Delta_B-n-1$, we have $B_nv=0$ for sufficiently positive $n$. $A_m^\tr$ is the transpose of $A$ sending each $u'\in\Vbb'(k)$ to $u'\circ A_m$. One checks easily that $A_m^\tr$ lowers the weights by $\Delta_A-m-1$. So $A_m^\tr v'$ vanishes for sufficiently negative $m$. Therefore, the coefficients of $f$ vanish the if powers of $z_2$ are sufficiently negative or the powers of $z_1$ are sufficiently positive. The same can be said about $\phi=(z_1-z_2)^Nf$. Similarly, the coefficients of $g$ vanishes when the powers of $z_1$ (resp. $z_2$) are sufficiently negative (resp. positive), and the same can be said about $\phi$. Therefore $\phi$ has finitely many terms: $\phi(z_1,z_2)\in\Cbb[z_1^{\pm1},z_2^{\pm1}]$. In particular, $\phi\in\scr O(\Cbb^\times\times\Cbb^\times)$.

Step 2. From \eqref{eq106}, it is clear that $f(z_1,z_2)$ is in $\Cbb[z_1^{\pm1}]((z_2))\subset\Cbb((z_1))((z_2))$. So $f(z_1,z_2)=(z_1-z_2)^{-N}\phi(z_1,z_2)$ where $(z_1-z_2)^{-N}\in\Cbb((z_1))((z_2))$ is the inverse of $(z_1-z_2)^N$ expanded in $|z_2|<|z_1|$ (cf. Subsec. \ref{lb64}). So the formal Laurent series $f(z_1,z_2)$ converges a.l.u. to the rational function $(z_1-z_2)^{-N}\phi(z_1,z_2)$ on $0<|z_2|<|z_1|$ since the series expansion of $(z_1-z_2)^{-N}\phi(z_1,z_2)$ does. Similarly, $g(z_1,z_2)$ converges a.l.u. on $0<|z_1|<|z_2|$ to $(z_1-z_2)^{-N}\phi(z_1,z_2)$. This finishes the proof.
\end{proof}

\subsection{}\label{lb69}

We now prove that complex analytic implies formal variable. To prepare for the proof, note that for any $k\in\Nbb$, any $m,n\in\Zbb$, and any $R_1,R_2>0$,
\begin{align}
	\oint_{|z_1|=R_1}\oint_{|z_2|=R_2}z_1^mz_2^n\bk{v',A(z_1)P_kB(z_2)v}\frac {dz_1dz_2}{(2\im\pi)^2}=\bk{v',A_mP_kB_nv}.\label{eq101}
\end{align}
Indeed, this is obvious when $\Vbb(k)$ is finite dimensional, in which case $\bk{v',A(z_1)P_kB(z_2)v}=\sum_e \bk{v',A(z_1)e}\bk{\wch e,B(z_2)v}$ where $\{e\}$ is a basis of $\Vbb(k)$ and $\{\wch e\}$ is its dual basis. In the general case, we may first fix $z_2$ and integrate $z_1$ by considering $P_kB(z_2)v$ as a fixed vector in $\Vbb(k)$, and then integrate $z_2$ by considering $\bk{v',A_mP_k \cdot}$ as an element of $\Vbb'(k)=\Vbb(k)^*$.

\begin{proof}
Assume the statements in Def. \ref{lb62} hold. Let $f_{v,v'}\in\scr O(\Conf^2(\Cbb^\times))$ be as in Def. \ref{lb62}. Since $\phi:=(z_1-z_2)^Nf_{v,v'}$ belongs to $\scr O(\Cbb^\times\times\Cbb^\times)$, by complex analysis, for each $m,n\in\Zbb$ the value of
\begin{align*}
\Gamma:=\oint_{|z_1|=R_1}\oint_{|z_2|=R_2}z_1^mz_2^n \phi(z_1,z_2)	\frac{dz_1dz_2}{(2\im\pi)^2}
\end{align*}
is independent of the specific values of $R_1,R_2$. (This is where we use the fact that $\phi$ has no poles at $z_1=z_2$.)

We compute $\Gamma$ in two ways. Assume $R_1>R_2$. Then since $0<|z_2|<|z_1|$, we have
\begin{align*}
\phi(z_1,z_2)=\sum_{k\in\Nbb}(z_1-z_2)^N\bk{v',A(z_1)P_kB(z_2)v}.	
\end{align*}
Thus, using \eqref{eq101}, we can compute
\begin{align*}
&\Gamma=	\oint_{|z_1|=R_1}\oint_{|z_2|=R_2}\sum_{k\in\Nbb} z_1^mz_2^n (z_1-z_2)^N\bk{v',A(z_1)P_kB(z_2)v}	\frac{dz_1dz_2}{(2\im\pi)^2}\\
=&\sum_{k\in\Nbb} \oint_{|z_1|=R_1}\oint_{|z_2|=R_2}\sum_{j=0}^N {N\choose j}z_1^{m+j}z_2^{n+N-j} (-1)^{N-j}\bk{v',A(z_1)P_kB(z_2)v}	\frac{dz_1dz_2}{(2\im\pi)^2}\\
=&\sum_{k\in\Nbb}\sum_{j=0}^N{N\choose j}(-1)^{N-j}\bk{v',A_{m+j}P_kB_{n+N-j}v}=\sum_{j=0}^N{N\choose j}(-1)^{N-j}\bk{v',A_{m+j}B_{n+N-j}v}
\end{align*}
where $\sum_{k\in\Nbb}$ commutes with the two contour integrals thanks to the a.l.u. convergence. Similarly, if we assume $R_1<R_2$, then $\phi(z_1,z_2)=\sum_{k\in\Nbb}(z_1-z_2)^N\bk{v',B(z_2)P_kA(z_1)v}$, and hence
\begin{align*}
	\Gamma=\sum_{j=0}^N{N\choose j}(-1)^{N-j}\bk{v',B_{n+N-j}A_{m+j}v}.
\end{align*}
This shows $\sum_{j=0}^N{N\choose j}(-1)^{N-j}[A_{m+j},B_{n+N-j}]=0$. If we compare this with the first several lines of \eqref{eq102}, we see that this is equivalent to $(z_1-z_2)^N[A(z_1),B(z_2)]=0$ in $\End\Vbb[[z_1^{\pm1},z_2^{\pm1}]]$. 
\end{proof}

\subsection{}\label{lb70}

In this subsection, we assume the statements in Def. \ref{lb62}, and derive the OPE $A(z_1)B(z_2)=\sum_{k\in\Zbb}(z_1-z_2)^{-k-1}(A_kB)(z_2)$ similar to $Y(u,z_1)Y(v,z_2)=\sum_{k\in\Zbb}(z_1-z_2)^{-k-1}Y(Y(u)_kv,z_2)$ for some fields $(A_kB)(z)$. This is simply done by taking Laurent series expansions of $z_1-z_2$ of the function $f_{v,v'}$ in Def. \ref{lb62}. Thus, the existence of OPE simply follows from complex analysis. Since we are treating multivariable holomorphic functions, to be serious about the domain of a.l.u. convergence, we provide some details below.

\begin{df}
For each $k\in\Zbb$ and $z\in\Cbb^\times$, let $f_{v,v'}\in\scr O(\Conf^2(\Cbb^\times))$ be as in Def. \ref{lb62}. We define the linear map \index{AB@$(A_kB)(z)$}
\begin{align*}
(A_kB)(z):\Vbb'\otimes \Vbb\rightarrow\Cbb,\qquad v'\otimes v\mapsto\bk{v',(A_kB)(z)v}	
\end{align*}
to be
\begin{align}
\bk{v',(A_kB)(z_2)v}=\oint_{C(z_2)}	(z_1-z_2)^kf_{v,v'}(z_1,z_2)\frac{dz_1}{2\im\pi}\label{eq105}
\end{align}
where $C(z_2)$ is any circle in $\Cbb^\times$ surrounding $z_2$. Note that $(A_kB)(z)v$ is naturally an element of $(\Vbb')^*=\prod_{n\in\Nbb}\Vbb(n)^{**}$, the (algebraic) dual space of $\Vbb'$. Also, $\bigbk{v',(A_kB)(z_2)v}$ is clearly a holomorphic function of $z_2$ on $\Cbb^\times$.
\end{df}	

\begin{lm}
$A_kB=0$ whenever $k\geq N$.
\end{lm}
\begin{proof}
When $k\geq N$, $(z_1-z_2)^kf_{v,v'}$ has no poles at $z_1=z_2$. So the RHS of \eqref{eq105} vanishes.
\end{proof}

\begin{pp}\label{lb73}
For each $v\in\Vbb,v'\in\Vbb'$, we have
\begin{align}
f_{v,v'}(z_1,z_2)=	\sum_{k\in\Zbb}(z_1-z_2)^{-k-1}\bigbk{v',(A_kB)(z_2)v}\label{eq112}
\end{align}
where the series on the RHS converges a.l.u. on $\Omega_0=\{(z_1,z_2):0<|z_1-z_2|<|z_2|\}$ to the LHS.
\end{pp}

\begin{proof}
It suffices to prove the claim on $\{(z_1,z_2):0<|z_1-z_2|<r,r<|z_2|\}$ for all $r>0$. Then this follows easily from the following basic lemma.
\end{proof}

\begin{lm}\label{lb76}
Let $U$ be an open subset of $\Cbb^m$ and let $f=f(z_1,\dots,z_m,q_1,\dots,q_n)$ be a holomorphic function on $U\times A_{r_1,R_1}\times\cdots\times A_{r_n,R_n}$ where each $0\leq r_i<R_i\leq+\infty$ and $A_{r_i,R_i}=\{q_i\in\Cbb:r_i<|q_i|<R_i\}$. Then $f$ has Laurent series expansion
	\begin{align}
		f(z_\blt,q_\blt)=\sum_{k_1,\dots,k_n\in\Zbb}f_{k_\blt}(z_\blt)q_1^{-k_1-1}\cdots q_n^{-k_n-1}	
	\end{align}
converging a.l.u. on $U\times A_{r_1,R_1}\times\cdots\times A_{r_n,R_n}$, where each 
\begin{align}
f_{k_\blt}(z_\blt)=\oint_{C_n}\cdots \oint_{C_1}f(z_\blt,q_\blt)q_1^{k_1}\cdots q_n^{k_n}\frac{dq_1\cdots dq_n}	{(2\im\pi)^n}\label{eq104}
\end{align}
(where $C_j$ is an anticlockwise circle around $0$) is clearly holomorphic on $U$.
\end{lm}

\begin{proof}
For simplicity, we assume $n=1$ and write $q_1=q,r_1=r,R_1=R$. We shall prove the a.l.u. convergence on $(z_\blt,q)\in U\times A_{\wtd r,\wtd R}$ for all $\wtd r,\wtd R$ such that $r<\wtd r<\wtd R<R$.  Let $C_-=\{q\in\Cbb:|q|=(r+\wtd r)/2\}$ and $C_+=\{q\in\Cbb:|q|=(R+\wtd R)/2\}$. Then on $U\times A_{\wtd r,\wtd R}$,
\begin{align*}
f(z_\blt,q)=\Res_{p=q}~\frac{f(z_\blt,p)}{p-q}dp=\bigg(\oint_{C_+}-\oint_{C-}\bigg)\frac{f(z_\blt,p)}{p-q}\frac{dp}{2\im\pi}.
\end{align*}
We have $\frac{f(z_\blt,p)}{p-q}=\sum_{k\leq-1}q^{-k-1}p^kf(z_\blt,p)$ where the RHS converges on $(z_\blt,q,p)\in (U\times A_{\wtd r,\wtd R}\times C_+)$ to the LHS by basic analysis. The same can be said about $\frac{f(z_\blt,p)}{p-q}=-\sum_{k\geq0}q^{-k-1}p^kf(z_\blt,p)$ if $C_+$ is replaced by $C_-$. So in view of \eqref{eq104}, and noting that integrals commute with infinite sums due to a.l.u. convergence,  the RHS of $\oint_{C_+}\frac{f(z_\blt,p)}{p-q}\frac{dp}{2\im\pi}=\sum_{k\leq-1}f_k(z_\blt)q^{-k-1}$ (resp. $\oint_{C_-}\frac{f(z_\blt,p)}{p-q}\frac{dp}{2\im\pi}=-\sum_{k\geq0}f_k(z_\blt)q^{-k-1}$) converges a.l.u. on $U\times A_{\wtd r,\wtd R}$ to the RHS. This completes the proof.
\end{proof}

\subsection{}\label{lb72}

We continue our discussion in the previous section. Let $(A_nB)_k:\Vbb'\otimes\Vbb\rightarrow\Cbb$ such that
\begin{align*}
\bk{v',(A_nB)_kv}=\Res_{z=0}\bk{v',(A_nB)(z)v}z^kdz.	
\end{align*}
In other words, $(A_nB)_k$ is a linear map $\Vbb\rightarrow(\Vbb')^*=\prod_{n\in\Nbb}\Vbb(n)^{**}$.

\begin{pp}\label{lb71}
Assume that $A(z),B(z)$ satisfy Def. \ref{lb62}. Then the following Jacobi identity holds:
\begin{align}\label{eq107}
\begin{aligned}
&\sum_{l\in\Nbb}{m\choose l}(A_{n+l}B)_{m+k-l}\\
=&\sum_{l\in\Nbb}(-1)^l{n\choose l}A_{m+n-l}B_{k+l}-\sum_{l\in\Nbb}(-1)^{n+l}{n\choose l}B_{n+k-l} A_{m+l}.
\end{aligned}
\end{align}
\end{pp}

\begin{rem}
There are two immediate consequences of this proposition. First, by setting $m=0$, we get a formula to express $(A_nB)_k$ in terms of the modes of $A(z)$ and $B(z)$. From that expression, one easily checks that $(A_nB)_k$ sends each $\Vbb(a)$ to $\Vbb(b)$ where $b=a+\Delta_A+\Delta_B-n-k-2$. This shows that $(A_nB)_k$ is a linear operator on $\Vbb$, and that $(A_nB)(z)$ is a homogeneous field with weight $\Delta_A+\Delta_B-n-1$. Second, by setting $n=0$, we see that $A(z)$ is local to $B(z)$ in the Lie algebraic sense.
\end{rem}

\begin{proof}[Proof of Prop. \ref{lb71}]
The idea is the same as the proof of VOA Jacobi identity in Subsec. \ref{lb38}. (Note that roles of $z_1,z_2$ in Subsec. \ref{lb38} are switched here.) For each $z_2\in\Cbb^\times$, we choose a large circle $C_+$ and a small one $C_-$ centered at $0$, and a small one $C_0$ centered at $z_2$. Choose $\mu=z_1^m(z_1-z_2)^ndz_1$. Set $f=f_{v,v'}$. Then 
\begin{align}\label{eq108}
	\oint_{C_+}\frac{f\mu}{2\im\pi} -\oint_{C_-}\frac{f\mu}{2\im\pi}=\oint_{C_0} \frac{f\mu}{2\im\pi}.	
\end{align}
When $z_1$ is on $C_+$, $f$ takes the form \eqref{eq90}. Moreover, the RHS of
\begin{align*}
z_1^m(z_1-z_2)^ndz_1f(z_1,z_2)=\sum_{l,s\in\Nbb}{n\choose l}(-z_2)^lz_1^{m+n-l}\bk{v',A(z_1)P_sB(z_2)v}	
\end{align*}
converges a.l.u. on $0<|z_2|<|z_1|$ to the LHS. So
\begin{align*}
&\oint_{C_+}\frac{f\mu}{2\im\pi}=\oint_{C_+}\sum_{l,s\in\Nbb}{n\choose l}(-z_2)^lz_1^{m+n-l}\bk{v',A(z_1)P_sB(z_2)v}\frac{dz_1}{2\im\pi}\\
=&\sum_{l,s\in\Nbb}{n\choose l}\oint_{C_+}(-z_2)^lz_1^{m+n-l}\bk{v',A(z_1)P_sB(z_2)v}\frac{dz_1}{2\im\pi}\\
=&\sum_{l,s\in\Nbb}{n\choose l}(-z_2)^l\bk{(A_{m+n-l})^\tr v',P_sB(z_2)v}\frac{dz_1}{2\im\pi}
\end{align*}
where the contour integral commutes with the infinite sum due to the a.l.u. convergence; $(A_{m+n-l})^\tr$ is the transpose of $A_{m+n-l}$, sending $v'$ to a vector of $\Vbb(s)$ where $s=\wt v'-\Delta_A+m+n-l+1$. So when $s$ is not this weight, the above summand vanishes. We can thus write the above expression as
\begin{align*}
\sum_{l\in\Nbb}{n\choose l}(-z_2)^l\bk{(A_{m+n-l})^\tr v',B(z_2)v}\frac{dz_1}{2\im\pi}.
\end{align*}

The integral on $C_-$ can be treated in a similar way. And by Prop. \ref{lb73},
\begin{align*}
\oint_{C_0}\frac{f\mu}{2\im\pi}=\oint_{C_0}	\sum_{l,s\in\Nbb}{m\choose l}z_2^{m-l}(z_1-z_2)^{n+l}\cdot (z_1-z_2)^{-s-1}\bk{v',(A_sB)(z_2)v}\frac{dz_1}{2\im\pi}
\end{align*}
where series inside the integrand converge a.l.u. on $0<|z_1-z_2|<|z_2|$. So we can exchange the integral and the sum to compute the result
\begin{align*}
\sum_{l\in\Nbb}{m\choose l}z_2^{m-l}\bk{v',(A_{n+l}B)(z_2)v}.
\end{align*}
This computes \eqref{eq108}. Now all three terms are clearly holomorphic functions of $z_2$ on $\Cbb^\times$. Multiply them by $z_2^kdz_2$ and evaluate the residue at $z_2=0$, we get \eqref{eq107}.
\end{proof}

\subsection{}

We are now ready to prove the equivalence of the complex analytic version and the algebraic version of Jacobi identity.

\begin{df}[\textbf{Jacobi identity} (complex-analytic version)]\label{lb194} \index{00@Jacobi identity, complex-analytic version}
For each $u,v,w\in\Vbb$ and $w'\in\Vbb$, the following series
\begin{subequations}
\begin{gather}
\bigbk{w',Y(u,z_1)Y(v,z_2)w}:=\sum_{n\in\Nbb}\bigbk{w',Y(u,z_1)P_nY(v,z_2)w},\label{eq109}\\
\bigbk{w',Y(v,z_2)Y(u,z_1)w}:=\sum_{n\in\Nbb}\bigbk{w',Y(v,z_2)P_nY(u,z_1)w},\label{eq110}\\
\bigbk{w',Y\big(Y(u,z_1-z_2)v,z_2\big)w}:=\sum_{n\in\Nbb}\bigbk{w',Y\big(P_nY(u,z_1-z_2)v,z_2\big)}\label{eq111}
\end{gather}	
\end{subequations}
converge a.l.u. respectively on
\begin{subequations}
\begin{gather}
\{(z_1,z_2)\in\Cbb^2:0<|z_2|<|z_1|\},\\
\{(z_1,z_2)\in\Cbb^2:0<|z_2|<|z_2|\},\\
\{(z_1,z_2)\in\Cbb^2:0<|z_1-z_2|<|z_2|\}
\end{gather}	
\end{subequations}
and can be extended to the same holomorphic function $f_{w,u,v,w'}$ on $\Conf^2(\Cbb^\times)$.
\end{df}

\begin{thm}
The complex analytic and the algebraic versions of Jacobi identity are equivalent.
\end{thm}

\begin{proof}
Complex analytic implies algebraic: This follows from the argument in Subsec. \ref{lb38} or the proof of Prop. \ref{lb71}.

Algebraic implies complex analytic: Assume that $u,v,w,w'$ are homogeneous. Let $A(z)=Y(u,z)$ and $B(z)=Y(v,z)$. Then $A$ and $B$ are local. Moreover, the VOA Jacobi identity expresses $Y(Y(u)_nv,z)$ in terms of $Y(u,z),Y(v,z)$, and \eqref{eq107} expresses $(A_nB)(z)$ in terms of $A(z),B(z)$. From these two expressions, it is clear that
\begin{align}
Y(Y(u)_nv,z)=(A_nB)(z).	\label{eq118}
\end{align}
Thus, the complex analytic locality of $A$ and $B$ proves the complex analytic Jacobi identity. Note that the a.l.u. convergence of $\eqref{eq111}=\sum_{m\in\Zbb}\bigbk{w',Y(Y(u)_mv)w}(z_1-z_2)^{-m-1}$ (note that $P_n Y(u,z_1-z_2)v=Y(u)_m(z_1-z_2)^{-m-1}v$ where $n=\wt u+\wt v-m-1$) follows from that of \eqref{eq112}.
\end{proof}

\section{$n$-point functions for vertex operators; proof of reconstruction theorem} \label{lb84}

\subsection{}\label{lb85}

The goals of this section are twofold. We first prove two analytic properties for $n$-point functions generalizing Def. \ref{lb62}. Then we use these results to prove the reconstruction theorem.

\begin{thm}\label{lb74}
Assume that the homogeneous fields $A^1(z),\dots,A^M(z)\in(\End\Vbb)[[z^{\pm1}]]$ are mutually local. Then for each $v\in\Vbb,v'\in\Vbb'$ and each permutation $\sigma$ of $\{1,\dots,M\}$, the series
\begin{align}
\bigbk{v',A^{\sigma(1)}(z_{\sigma(1)})\cdots A^{\sigma(M)}(z_{\sigma(M)})v}\qquad\in(\End\Vbb)[[z_1^{\pm1},\dots,z_M^{\pm1}]]\label{eq116}
\end{align}
converges a.l.u. on
\begin{align}
\Omega_\sigma=\{z_\blt\in\Cbb^M:0<|z_{\sigma(M)}|<\cdots<|z_{\sigma(1)}|\}	\label{eq115}
\end{align}
and can be extended to some $f_{v,v'}\in\scr O(\Conf^M(\Cbb^\times))$ independent of $\sigma$. Moreover, there exists $N\in\Nbb$ for all $v,v'$ such that
\begin{align}
f_{v,v'}(z_\blt)\cdot \prod_{1\leq i<j\leq M}(z_i-z_j)^N	
\end{align}
is holomorphic on $(\Cbb^\times)^M$. (Indeed, it is an element of $\Cbb[z_1^{\pm1},\dots,z_M^{\pm1}]$.)
\end{thm}

$f_{v,v'}$ is called the \textbf{$(M+2)$-point (genus $0$ correlation) function} associated to the fields $A^\blt(z)$. In case each $A^i(z)$ is a vertex operator $Y(u_i,z)$, $f_{v,v'}$ is the correlation function associated to (setting $\zeta$ to be the standard coordinate of $\Cbb$)
\begin{align}
(\Pbb^1;0,z_1,\dots,z_M,\infty;\zeta,\zeta-z_1,\dots,\zeta-z_M,\zeta^{-1}),	
\end{align}
where $v,u_1,\dots,u_M,v'$ are going into the punctures $0,z_1,\dots,z_M,\infty$ respectively.

\begin{rem}
The a.l.u. convergence on $\Omega_\sigma$ of the formal Laurent series \eqref{eq116} is equivalent to that  of the series of functions
\begin{align}
\begin{aligned}
&\bigbk{v',A^{\sigma(1)}(z_{\sigma(1)})\cdots A^{\sigma(M)}(z_{\sigma(M)})v}	\\
:=&\sum_{n_2,\dots,n_M\in\Nbb}\bigbk{v',A^{\sigma(1)}(z_{\sigma(1)})P_{n_2}A^{\sigma(2)}(z_{\sigma(2)})P_{n_3}\cdots P_{n_M}A^{\sigma(M)}(z_{\sigma(M)})v}.
\end{aligned}	
\end{align}
Indeed, assume for simplicity that $\sigma=1$ and $v,v'$ are homogeneous. Then by scale covariance \eqref{eq117}, the RHS of the above formula equals
\begin{align}\label{eq99}
\begin{aligned}
\sum_{n_2,\dots,n_M\in\Nbb}&\bigbk{v',A^1(1)P_{n_2}A^2(1)P_{n_3}\cdots P_{n_M}A^M(1)v}\\
&\cdot \Big(\frac {z_2}{z_1}\Big)^{n_2}\Big(\frac {z_3}{z_2}\Big)^{n_3}\cdots \Big(\frac {z_M}{z_{M-1}}\Big)^{n_M}\cdot z_1^{\wt v'}z_M^{-\wt v}\cdot \prod_{i=1}^Mz_i^{-\Delta_{A^i}},
\end{aligned}	
\end{align}
which together with Lemma \ref{lb65} proves the claim.
\end{rem}

\subsection{}

\begin{proof}[Proof of Thm. \ref{lb74}] (Cf. \cite{FHL93}.)
The method is the same as in Subsec. \ref{lb68}. Choose $N$ such that $(z_i-z_j)^N[A^i(z_i),A^j(z_j)]=0$ for all $i,j$. Set
\begin{align}
f^\sigma(z_\blt)=\bigbk{v',A^{\sigma(1)}(z_{\sigma(1)})\cdots A^{\sigma(M)}(z_{\sigma(M)})v}\qquad\in\Cbb[[z_\blt^{\pm1}]]=\Cbb[[z_1^{\pm1},\dots,z_M^{\pm1}]].	
\end{align}
Then the formal Laurent series
\begin{align}
\phi(z_\blt)=f^\sigma(z_\blt)	\cdot \prod_{1\leq i<j\leq M}(z_i-z_j)^N\label{eq113}
\end{align}
is independent of the permutation $\sigma$. From
\begin{align*}
f^1(z_\blt)=\sum_{m,n\in\Zbb}\bk{(A^1_m)^\tr v',A^2(z_2)\cdots A^{M-1}(z_{m-1})A^M_nv}\cdot z_1^{-m-1}z_M^{-n-1}
\end{align*}
and the lower truncation property, we see that the coefficients of $f^1(z_\blt)$ and hence of $\phi(z_\blt)$ vanish if the powers of $z_M$ (resp. $z_1$) is sufficiently negative (resp. positive). Since we can replace $M$ with $\sigma(M)$ and $1$ with $\sigma(1)$, we see that the coefficients of $\phi$ vanish except when the powers $z_1,\dots,z_M$ are all bounded from below and from above by some fixed constants. Namely, $\phi(z_\blt)\in\Cbb[z_1^{\pm1},\dots,z_M^{\pm1}]$. In particular, $\phi$ can be regarded as a holomorphic function on $(\Cbb^\times)^M$.

By expanding $f^1(z_\blt)$ as a formal Laurent series of $z_1,\dots,z_M$, it is not hard to see (e.g. by induction on $M$ or by \eqref{eq99}) that
\begin{align}
	f^1(z_\blt)\in\mbb F:=\Cbb((z_1))((z_2))\cdots((z_M)).
\end{align}
So, for $\sigma=1$, \eqref{eq113} holds in the field $\mbb F$. So
\begin{align}
f^1(z_\blt)=\phi(z_\blt)\cdot\prod_{1\leq i<j\leq M}(z_i-z_j)^{-N}	\label{eq114}
\end{align}
where $g(z_\blt)=\prod_{1\leq i<j\leq M}(z_i-z_j)^{-N}\in\mbb F$ is expanded in the region $\Omega_1$ (defined by \eqref{eq115}), cf. Subsec. \ref{lb64}. Namely,
\begin{align*}
g(z_\blt)=\prod_{1\leq i<j\leq M}z_i^{-N}\cdot \sum_{k\in\Nbb}{-N\choose k}\Big(-\frac{z_j}{z_i}\Big)^k\qquad\in\mbb F
\end{align*}
We claim that the series $g(z_\blt)$ converges a.l.u. on $\Omega_1$. Then $\phi(z_\blt)g(z_\blt)$ as a formal series converges a.l.u. on $\Omega_1$ to $\phi(z_\blt)g(z_\blt)$ as a rational function (which is holomorphic on $\Omega_1$). We denote this rational function by $f_{v,v'}$. Then, this statement also holds when the permutation $1$ is replaced by any $\sigma$, our proof is therefore completed.

To prove the claim, it suffices to show that $g(q_1,q_1q_2,\dots,q_1q_2\cdots q_M)$, as a formal Laurent series of $q_1=z_1,q_2=z_2/z_1,q_3=z_3/z_2,\dots,q_M=z_M/z_{M-1}$, converges a.l.u. on $\Omega'=\{(q_1,\dots,q_M)\in\Cbb^M:0<|q_1|,0<|q_j|<1\text{ if }2\leq j\leq M\}$. But this is clearly true because
\begin{align*}
g=\prod_{1\leq i<j\leq M}(q_1q_2\cdots q_i)^{-N}\cdot \sum_{k\in\Nbb}{-N\choose k}(-q_{i+1}q_{i+2}\cdots q_j)^k\qquad\in\Cbb((q_1,\dots,q_M))
\end{align*}
is the Laurent series expansion on the polyannulus $\Omega'$ of the holomorphic function
\begin{align*}
\prod_{1\leq i<j\leq M}(q_1q_2\cdots q_i-q_1q_2\cdots q_j)^{-N}\qquad\in\scr O(\Omega')
\end{align*}
\end{proof}

\subsection{}
\begin{df}\label{lb75}
	Let $A^1,\dots,A^M$ be mutually local. For each $z_\blt\in\Conf^M(\Cbb^\times)$,
	\begin{align}
		A^1(z_1)\cdots A^M(z_M):\Vbb'\otimes\Vbb\rightarrow\Cbb	
	\end{align}
is defined to be the linear map sending $v'\otimes v$ to $f_{v,v'}$ in Thm. \ref{lb74}. Equivalently, $A^1(z_1)\cdots A^M(z_M)$ is a linear map from $\Vbb$ to $(\Vbb')^*=\prod_{n\in\Nbb}\Vbb(n)^{**}$.
\end{df}
According to this notation, for local $A,B$ we have
\begin{align}
	A(z_1)B(z_2)=B(z_2)A(z_1),\qquad (A_nB)(z_2)=\Res_{z_1=z_2}A(z_1)B(z_2)(z_1-z_2)^ndz_1.
\end{align}

The following is our second analytic property for $n$-point functions.

\begin{thm}\label{lb77}
Assume $A^1,\cdots,A^m,B^1,\dots,B^n$ are mutually local. Then on
\begin{align*}
\Omega=\{(z_1,\dots,z_m,\zeta_1,\dots,\zeta_n)\in\Conf^{m+n}(\Cbb^\times):|z_i|>|\zeta_j|\text{ for all }i,j\},
\end{align*}
for each $v\in\Vbb,v'\in\Vbb'$ the RHS of
\begin{align}
&\bk{v',A^1(z_1)\cdots A^m(z_m)B^1(\zeta_1)\cdots B^n(\zeta_n)v}\nonumber\\
=&\sum_{k\in\Nbb}	\bk{v',A^1(z_1)\cdots A^m(z_m)P_kB^1(\zeta_1)\cdots B^n(\zeta_n)v}\label{eq122}
\end{align}
converges a.l.u. to the LHS.
\end{thm}

The meaning of the product of $A^1(z_1)\cdots A^m(z_m)$ and $B^1(\zeta_1)\cdots B^n(\zeta_n)$ is clear: it corresponds to the sewing of (setting $\zeta$ to be the standard coordinate of $\Cbb$)
\begin{gather*}
\fk X_1=(\Pbb^1;0,z_1,\dots,z_m,\infty;\zeta,\zeta-z_1,\dots,\zeta-z_m,\zeta^{-1}),\\
\fk X_2=(\Pbb^1;0,\zeta_1,\dots,\zeta_n,\infty;\zeta,\zeta-\zeta_1,\dots,\zeta-\zeta_n,\zeta^{-1})	
\end{gather*}
along $0\in\fk X_1$ and $\infty\in\fk X_2$, in case all these fields are vertex operators.  Moreover, this picture, as well as the theorem, can be easily generalized to the products of several strings of mutually local fields.

\begin{rem}
Note that each summand on the RHS of \eqref{eq122} is holomorphic on $(z_\blt,\zeta_\blt)\in\Conf^m(\Cbb^\times)\times\Conf^n(\Cbb^\times)$. When $\Vbb(k)$ is finite-dimensional, this is due to Thm. \ref{lb74} and that $P_k$ can be written as $\sum_e e\rangle\langle\wch e$ for a basis $\{e\}$ of $\Vbb(k)$ and dual basis $\{\wch e\}$. 

In the general case that $\Vbb(k)$ is not necessarily finite dimensional, $P_k$ is the projection from $(\Vbb')^*$ onto $\Vbb(k)^{**}$. In each series $A^i(z_i)=\sum_{n\in\Nbb}A^i_nz_i^{-n-1}$ in \eqref{eq122}, $A^i_n$ is understood as $(A^i_n)^{\tr\tr}$ sending each $\Vbb(a)^{**}$ to $\Vbb(b)^{**}$ where $b=a+\Delta_A-n-1$. Then, in this sense $A^1,\dots,A^m$ are mutually local. Each summand on the RHS of \eqref{eq122} is continuous over $(z_\blt,\zeta_\blt)\in\Conf^m(\Cbb^{\times})\times\Conf^n(\Cbb^\times)$; for fixed $z_\blt$, it is holomorphic over $\zeta_\blt$ by treating $\bk{v',A^1(z_1)\cdots A^m(z_m)P_k\cdot}$ as an element of $\Vbb(k)^*$; similarly, it is holomorphic over $z_\blt$. So, again, it is holomorphic on $\Conf^m(\Cbb^\times)\times\Conf^n(\Cbb^\times)$.
\end{rem}

\subsection{}

The idea of the proof of Thm. \ref{lb77} is the following. To show that a series of functions $\sum_n f_n(z_\blt)$ converges a.l.u. on a domain $U$: We try to find $r>1$ and a smaller $U'$ such that $\sum_n f_n(z_\blt)q^n$ converges a.l.u. on $z_\blt\in U'$ and $q\in\Dbb_r^\times$ to a function $f$ holomorphic on $U\times\Dbb_r^\times$. Then by Lemma \ref{lb76}, $\sum_n f_n(z_\blt)q^n$ is the series expansion of $f$, which must converge a.l.u. on $U\times\Dbb_r^\times$. 

\begin{proof}[Proof of Thm. \ref{lb77}]
It suffices to prove the proposition when $\Omega$ is replaced by all possible
\begin{align*}
	\Omega_r=\{(z_1,\dots,z_m,\zeta_1,\dots,\zeta_n)\in\Conf^{m+n}(\Cbb^\times):|z_i|>r|\zeta_j|\text{ for all }i,j\}
\end{align*}
where $r>1$. To clarify the subtlety, we let $\wtd A^i_n:\Vbb(a)^{**}\rightarrow\Vbb(b)^{**}$ denote the double transpose of $A^i_n:\Vbb(a)\rightarrow\Vbb(b)$.   To simplify discussions we assume $m=n=2$. Consider the following element of $\scr O(\Conf^2(\Cbb^\times)^2)[[q]]$:
\begin{align}
\sum_{k\in\Nbb}\bk{v',\wtd A^1(z_1)\wtd A^2(z_2)P_kB^1(\zeta_1)B^2(\zeta_2)v}q^k.\label{eq121}
\end{align}
Note that $P_kq^k=P_kq^{L_0}$. By scale covariance, as elements of $\Hom(\Vbb'\otimes\Vbb,\Cbb)$,
\begin{align}
q^{L_0}B^1(\zeta_1)B^2(\zeta_1)=q^{\Delta_{B^1}+\Delta_{B^2}}B^1(q\zeta_1)B^2(q\zeta_2)q^{L_0}	
\end{align}
whenever $0<|\zeta_2|<|\zeta_1|$. So it holds for all $\zeta_\blt\in\Conf^2(\Cbb^\times)$ by holomorphicity. Thus, there is  $d\in\Zbb$ such that \eqref{eq121}, as a series of functions of $(z_\blt,\zeta_\blt,q)$, equals
\begin{align}
\sum_{k\in\Nbb}q^d\bk{v',\wtd A^1(z_1)\wtd A^2(z_2)P_kB^1(q\zeta_1)B^2(q\zeta_2)v}.	\label{eq100}
\end{align}

On $\Omega_r'=\{(z_\blt,\zeta_\blt):0<r|\zeta_2|<r|\zeta_1|<|z_2|<|z_1|\}$ and $q\in\Dbb_r^\times$, \eqref{eq100} equals
\begin{align}\label{eq267}
\begin{aligned}
&\sum_{k\in\Nbb}\sum_{s\in\Nbb}q^d\bk{v',\wtd A^1(z_1)P_s\wtd A^2(z_2)P_kB^1(q\zeta_1)B^2(q\zeta_2)v}\\
\xlongequal{\text{Rem. \ref{lb195}}}&\sum_{k\in\Nbb}\sum_{s\in\Nbb}\sum_{t\in\Nbb}q^d\bk{v',\wtd A^1(z_1)P_s\wtd A^2(z_2)P_kB^1(q\zeta_1)P_tB^2(q\zeta_2)v}\\
=&\sum_{k\in\Nbb}\sum_{s\in\Nbb}\sum_{t\in\Nbb}q^d\bk{v',A^1(z_1)P_s A^2(z_2)P_kB^1(q\zeta_1)P_tB^2(q\zeta_2)v}
\end{aligned}
\end{align}
By Thm. \ref{lb74}, this series (and hence series \eqref{eq121}) converges a.l.u. on $\Omega_r\times\Dbb_r^\times$ to the holomorphic function
\begin{align*}
g(z_\blt,\zeta_\blt,q)=q^d	\bk{v', A^1(z_1) A^2(z_2)B^1(q\zeta_1)B^2(q\zeta_2)v}.
\end{align*}
Therefore, \eqref{eq121} is the Laurent series expansion of $g$ when $(z_\blt,\zeta_\blt)\in\Omega_r'$. Namely: the coefficients of \eqref{eq121} equal those in the expansion of $g$ when $(z_\blt,\zeta_\blt)\in\Omega_r'$. By Lemma \ref{lb76} applied to the holomorphic function $g(z_\blt,\zeta_\blt,q)$ on $\Omega_r\times\Dbb_r^\times$, this statement is true when $(z_\blt,\zeta_\blt)\in\Omega_r$ (since the coefficients are holomorphic on $\Omega_r$), and the series expansion of $g$  converges a.l.u. on $\Omega_r\times\Dbb_r^\times$ to $g$. So \eqref{eq121} converges a.l.u. on $\Omega_r\times\Dbb_r^\times$ to $g$. This finishes the proof if we set $q=1$.
\end{proof}

\begin{rem}\label{lb195}
Note that in \eqref{eq267}, we have used the obvious fact that $\sum_t P_kB^1(q\zeta_1)P_t B^2(q\zeta_2)v$ converges $*$-weakly in $\Vbb(k)^{**}$, in the sense that its evaluation with every element of $\Vbb(k)^*$ converges. We have also used the fact that, assuming $v'\in\Vbb(m)^*$, the linear opeartor $P_m\wtd A^1(z_1)P_s\wtd A^2(z_2)P_k:\Vbb(k)^{**}\rightarrow \Vbb(m)^{**}$, which is the double transpose of $P_m A^1(z_1)P_s A^2(z_2)P_k:\Vbb(k)\rightarrow \Vbb(m)$, is continuous with respect to the weak-$*$ topology. This is justified by the following easy lemma.  
\end{rem}

\begin{lm}
Let $T:U\rightarrow V$ be a linear map of vector spaces. Then $T^\tr:V^*\rightarrow U^*$ is $*$-weakly  continuous, which means if $v^*_n$ is a sequence (or more generally, a net) in $V^*$ which converges weakly-$*$ to $v^*\in V^*$ in the sense that $\lim_n \bk{v_n^*,v}=\bk{v^*,v}$ for all $v\in V$, then $T^\tr v^*_n$ converges $*$-weakly  to $T^\tr v^*$.
\end{lm}

\subsection{}

We now discuss the proof of the reconstruction Thm. \ref{lb39}. Assume that the assumptions for graded vertex algebras in Thm. \ref{lb39} hold. We may extend $\mc E$ to also include the identity field $\id(z)=\id_\Vbb$. Namely, $\id_n=\delta_{n,-1}\id_\Vbb$. Motivated by \eqref{eq118}, for each $A^1,\dots,A^k\in\mc E$ and $n_1,\dots,n_k\in\Zbb$, we define
\begin{align}
	Y\big(A^1_{n_1}\cdots A^k_{n_k}\id,z\big)=\big(A^1_{n_1}\cdots A^k_{n_k}\id\big)(z)\label{eq119}
\end{align}
where the right hand side is defined inductively by
\begin{align*}
	\big(A^1_{n_1}\cdots A^k_{n_k}\id\big)(z)=\big(A^1_{n_1}(A^2_{n_2}\cdots A^k_{n_k}\id)\big)(z).
\end{align*}
By the generating property, we can define $Y(u,z)$ for every $u\in\Vbb$ using \eqref{eq119} and linearity. 

There are two immediate problems with this approach: First, to define the RHS of \eqref{eq119} inductively, we need the fact that $A^1_{n_1}(z)$ is local to $(A^2_{n_2}\cdots A^k_{n_k}\id)(z)$ (``Dong's lemma"). Second, we need to show that the above definition of $Y(u,z)$ is unique, i.e., independent of how $u$ is written as a linear combination of  $A^1_{n_1}\cdots A^k_{n_k}\id$ (``Goddard uniqueness"). Besides these two, we also need to check that such defined $Y(u,z)$ satisfies the translation property. Let us first check the translation property.

\subsection{}

\begin{lm}\label{lb79}
Assume that homogeneous fields $A(z),B(z)$ are local and satisfy the translation property $[L_{-1},A_k]=-kA_{k-1}$, $[L_{-1},B_k]=-kB_{k-1}$. Then so does each $A_nB$:
\begin{align}
[L_{-1},(A_nB)_k]=-k(A_nB)_{k-1}.\label{eq120}	
\end{align}
\end{lm}

\begin{proof}
By the Jacobi identity \eqref{eq107},
\begin{align}
(A_nB)_k=\sum_{l\in\Nbb}(-1)^l{n\choose l}A_{n-l}B_{k+l}-\sum_{l\in\Nbb}(-1)^{n+l}{n\choose l}B_{n+k-l}A_l, \label{eq126}	
\end{align}
and hence
\begin{align*}
-k(A_nB)_{k-1}=\sum_{l\in\Nbb}(-1)^l{n\choose l}(-k)A_{n-l}B_{k+l-1}+\sum_{l\in\Nbb}(-1)^{n+l}{n\choose l}kB_{n+k-l-1}A_l.	
\end{align*}
So by the translation property of $A,B$,
\begin{align*}
&[L_{-1},(A_nB)_k]=\sum_{l\in\Nbb}(-1)^l{n\choose l}(-n+l)A_{n-l-1}B_{k+l}+\sum_{l\in\Nbb}(-1)^l{n\choose l}(-k-l)A_{n-l}B_{k+l-1}\\
&+\sum_{l\in\Nbb}(-1)^{n+l}{n\choose l}(n+k-l)B_{n+k-l-1}A_l+\sum_{l\geq1}(-1)^{n+l}{n\choose l}lB_{n+k-l}A_{l-1}.
\end{align*}
Look at the RHS. In the first sum, notice  $(-1)^l{n\choose l}(-n+l)=(-1)^{l+1}{n\choose l+1}(l+1)$ and replace $l$ by $l-1$; in the fourth sum, notice $(-1)^{n+l}{n\choose l}l=(-1)^{n+l-1}{n\choose l-1}(l-1-n)$ and replace $l$ by $l+1$. Then we see why \eqref{eq120} is true.
\end{proof}

\subsection{}\label{lb86}

\index{00@Dong's Lemma}
\begin{pp}[\textbf{Dong's lemma}]
Let $A(z),B(z),C(z)$ be mutually local homogeneous fields. Then for each $n\in\Zbb$, $C(z)$ is local to $(A_nB)(z)$. 
\end{pp}

We prove that $C(z)$ is complex-analytically local to $(A_nB)(z)$. 
\begin{proof}
Step 1. Choose $v\in\Vbb,v'\in\Vbb'$. Then we have series
\begin{align}
\sum_{k\in\Nbb}\bk{v',(A_nB)(z_2)P_kC(z_3)}=\sum_{k\in\Nbb}\Res_{z_1=z_2}(z_1-z_2)^n\bk{v',A(z_1)B(z_2)P_kC(z_3)v}dz_1.\label{eq123}
\end{align}
On the region  $\Omega_1=\Conf^3(\Cbb^\times)\cap\{|z_1|>|z_3|,|z_2|>|z_3|\}$, the RHS of
\begin{align*}
f(z_1,z_2,z_3):=\bk{v',A(z_1)B(z_2)C(z_3)v}=\sum_{k\in\Nbb}\bk{v',A(z_1)B(z_2)P_kC(z_3)v}	
\end{align*}
converges a.l.u. to the LHS by Thm. \ref{lb77}. Therefore, on $\Omega_1$, the sum and the residue (i.e. contour integral) on the RHS of \eqref{eq123} commute, and \eqref{eq123} converges a.l.u. on $|z_2|>|z_3|>0$ to
\begin{align*}
g(z_2,z_3)=\Res_{z_1=z_2}(z_1-z_2)^nf(z_1,z_2,z_3)dz_1	
\end{align*}
which is holomorphic on $\Conf^2(\Cbb^\times)$ since $f$ is holomorphic on $\Conf^3(\Cbb^\times)$ by Thm. \ref{lb74}. Similarly,  $\sum_{k\in\Nbb}\bk{v',C(z_3)P_k(A_nB)(z_2)}$ converges a.l.u. on $|z_3|>|z_2|>0$ to $g(z_2,z_3)$.\\

Step 2. To complete the proof, we need to show that $(z_2-z_3)^kg$ is holomorphic near $z_2=z_3$ for some $k$. By Thm. \ref{lb74}, $(z_1-z_2)^nf$ is a linear combination of
\begin{align*}
z_1^az_2^bz_3^c(z_1-z_2)^{n-N}(z_1-z_3)^{-N}(z_2-z_3)^{-N}	
\end{align*}
for some $N\in\Nbb$ and $a,b,c\in\Zbb$. To prove the claim, we may assume that $(z_1-z_2)^nf$ is just of this form. Then near $z_1=z_2$, $(z_1-z_2)^nf$ has a.l.u. convergent series expansion
\begin{align*}
(z_1-z_2)^nf=\sum_{i,j\geq0}{a\choose i}(z_1-z_2)^{n-N+i}z_2^{a-i+b}z_3^c{-N\choose j}(z_1-z_2)^j(z_2-z_3)^{-2N-j}.	
\end{align*}
Apply $\Rep_{z_1=z_2}\cdot dz_1$. This means taking the coefficient of $(z_1-z_2)^{n-N+i+j}$ where $n-N+i+j=-1$. So we set $i=N-n-j-1$. Since $i\geq0$, we take $0\leq j\leq N-n-1$. So
\begin{align*}
g(z_2,z_3)=\sum_{j=0}^{N-n-1}	{a\choose N-n-j-1}z_2^{a-(N-n-j-1)+b}z_3^c{-N\choose j}(z_2-z_3)^{-2N-j},
\end{align*}
which clearly has finite poles at $z_2=z_3$.
\end{proof}

\subsection{}
\index{00@Goddard uniqueness}
\begin{pp}[\textbf{Goddard uniqueness}]\label{lb162}
Let $\mc E$ be a set of homogeneous fields satisfying the assumptions for graded vertex algebras in the reconstruction Thm. \ref{lb39}. If $A^1(z),A^2(z)\in\mc E$ satisfy $A^1_{-1}\id=A^2_{-1}\id$, then $A^1(z)=A^2(z)$.
\end{pp}

\begin{proof}
Set $A=A^1-A^2$, and assume without loss of generality that $A\in\mc E$. Then $A_{-1}\id=0$. By the generating property, we can  show $A(z)=0$ by show that for any $v'\in\Vbb'$, $B^1,\dots,B^k\in\mc E$, and $n,n_1,\dots,n_k\in\Zbb$,
\begin{align}
\bk{v',A_nB^1_{n_1}\cdots B^k_{n_k}\id}=0.	\label{eq124}
\end{align} 
Suppose we can show that
\begin{align}
\bk{v',A(z)B^1(z_1)\cdots B^k(z_k)\id}=0	\label{eq125}
\end{align}
as functions on $\Conf^{k+1}(\Cbb^\times)$. Then multiplying it by any Laurent polynomial of $z,z_1,\dots,z_N$ and taking contour integrals over $|z|=R,|z_1|=r_1,\dots,|z_k|=r_k$ where $0<r_k<\cdots<r_1<R$, we will get \eqref{eq124}.

Since the LHS of \eqref{eq125} is holomorphic, it suffices to prove \eqref{eq125} when $0<|z|<|z_1|<\cdots<|z_k|$, i.e., to prove in this domain that
\begin{align*}
\sum_{s\in\Nbb}\bk{v',B^1(z_1)\cdots B^k(z_k)P_sA(z)\id}=0.	
\end{align*}
Therefore, it suffices to prove $A(z)\id=0$. Since $A(z)$ satisfies the translation property and the creation property $\lim_{z\rightarrow0}A(z)\id=A_{-1}\id$, similar to the proof of Cor. \ref{lb78} we have $A(z)\id=e^{zL_{-1}}A_{-1}\id$. So $A(z)\id$ must be $0$.
\end{proof}

\subsection{}

\begin{proof}[Proof of the reconstruction Thm. \ref{lb39}]
Assume that $\mc E$ contains the identity field $\id(z)=\id_\Vbb$. If $A(z),B(z)\in\mc E$, then using \eqref{eq126}, one checks easily that $A_nB$ satisfies the creation property with
\begin{align}
(A_nB)_{-1}\id=A_nB_{-1}\id.
\end{align}
So by Lemma \ref{lb79} and Dong's lemma, if we include $A_nB$ in $\mc E$, then the new $\mc E$ still satisfies the assumptions for graded vertex algebras in Thm. \ref{lb39}. By induction, when $A^1,\dots,A^k\in\mc E$  we have
\begin{align}
(A^1_{n_1}\cdots A^k_{n_k}\id)_{-1}\id=	A^1_{n_1}\cdots A^k_{n_k}\id.
\end{align}
Therefore, by including any linear combination of vectors of the form $A^1_{n_1}\cdots A^k_{n_k}\id$ in $\mc E$, we may assume that for each homogeneous $u\in\Vbb$ there exists  $A(z)\in\mc E$ such that $A_{-1}\id=u$. By Goddard uniqueness, such $A(z)$ is unique and hence can be written as $Y(u,z)$.

We now prove the Jacobi identity for $Y$ since the other axioms of graded vertex algebras are obvious. Choose $A(z)=Y(u,z)$ and $B(z)=Y(v,z)$ in $\mc E$.  Note that $v=B_{-1}\id$. By extending $\mc E$, we may assume that each $A_nB$ is in $\mc E$. To show the VOA Jacobi identity \eqref{eq47},  by \eqref{eq107}, it suffices to show $Y(Y(u)_nv,z)=(A_nB)(z)$. This follows from Goddard uniqueness and
\begin{align*}
(A_nB)_{-1}\id=A_nB_{-1}\id=A_nv=Y(u)_nv.	
\end{align*}
So $\Vbb$ is a graded vertex algebra. The last paragraph of Thm. \ref{lb39} about VOA is obvious.
\end{proof}

\section{VOA modules; contragredient modules}\label{lb192}

\subsection{}

Let $\Vbb$ be a VOA. 

\begin{df}
A vector space $\Wbb$ equipped with a linear map
\begin{gather*}
Y_\Wbb:\Vbb\rightarrow(\End\Wbb)[[z^{\pm1}]],\\
v\mapsto Y_\Wbb(v,z)=\sum_{n\in\Zbb}Y_\Wbb(v)_nz^{-n-1}	
\end{gather*}
(where each $Y_\Wbb(v)_n\in\End\Wbb$) is called a \textbf{weak $\Vbb$-module} if the following hold:
\begin{itemize}
\item (Lower truncation) $Y_\Wbb(v,z)w\in\Wbb((z))$ for each $v\in\Vbb,w\in\Vbb$.
\item $Y_\Wbb(\id,z)=\id_\Wbb$.
\item (Jacobi identity) \index{00@Jacobi identity, algebraic version} For each $u,v\in\Vbb$, 
\begin{align}\label{eq128}
	\begin{aligned}
		&\sum_{l\in\Nbb}{m\choose l}Y_\Wbb\big(Y(u)_{n+l}v\big)_{m+k-l}\\
		=&\sum_{l\in\Nbb}(-1)^l{n\choose l}Y_\Wbb(u)_{m+n-l}Y_\Wbb(v)_{k+l}-\sum_{l\in\Nbb}(-1)^{n+l}{n\choose l}Y_\Wbb(v)_{n+k-l} Y_\Wbb(u)_{m+l}	.
	\end{aligned}
\end{align}
\end{itemize}
\end{df}

\begin{df}\index{00@Weak $\Vbb$-modules and (finitely) admissible $\Vbb$-modules}
An \textbf{admissible $\Vbb$-module} $\Wbb$ (or simply a $\Vbb$-module) is a weak $\Vbb$-module  such that $\Wbb=\bigoplus_{n\in\Nbb}\Wbb(n)$ is graded by a diagonalizable operator $\wtd L_0$ \index{L0@$\wtd L_0$} satisfying the grading property
\begin{align}
[\wtd L_0,Y_\Wbb(v,z)]=Y_\Wbb(L_0v,z)+z\partial_zY_\Wbb(v,z)	
\end{align}
for each $v$.  Equivalently, for homogeneous $v$,
\begin{align}
[\wtd L_0,Y_\Wbb(v)_n]=(\wt v-n-1)Y_\Wbb(v)_n,	\label{eq127}
\end{align} 
i.e., $Y_\Wbb(v,z)$ is $\wtd L_0$-homogeneous with weight $\wt v$. Zero and eigenvectors of $\wtd L_0$ are called \textbf{($\wtd L_0$)-homogeneous vectors}. If $w\in\Wbb(n)$, then $\wtd\wt w:=n$ \index{vw@$\wt v,\wtd\wt w$} is called the \textbf{($\wtd L_0$)-weight} of $w$. If each $\Wbb(n)$ is finite-dimensional, we say $\Wbb$ is \textbf{finitely-admissible}.
\end{df}

The lower-truncation property is redundant in the definition of admissible modules since it follows from the grading property.

\begin{cv}
$\Vbb$ itself is an admissible $\Vbb$-module, called the \textbf{vacuum module}. (It is analogous to the adjoint representations of Lie algebras.) We always choose the operator $\wtd L_0$ on $\Vbb$ to be $L_0$.
\end{cv}

\subsection{}

\begin{pp}
Let $\Wbb$ be a weak $\Vbb$-module. Then for each $u\in\Vbb$, the following translation property holds:
\begin{align}
[L_{-1},Y_\Wbb(v,z)]=Y_\Wbb(L_{-1}v,z)=\partial_zY_\Wbb(v,z).	
\end{align}
\end{pp}
\begin{proof}
Applying the Jacobi identity to $[Y_\Wbb(\cbf)_0,Y_\Wbb(u)_k]$ gives $[L_{-1},Y(u)_k]=Y(L_{-1}u)_k$. By \eqref{eq52}, $L_{-1}u=Y(u)_{-2}\id$. Applying the Jacobi identity to $Y_\Wbb(Y(u)_{-2}\id)_k$ shows that it equals $-kY_\Wbb(u)_{k-1}$.
\end{proof}

\begin{pp}
Let $\Wbb$ be a weak $\Vbb$-module. Define the action of $L_n$ on $\Wbb$ to be
\begin{align}
L_n=Y_\Wbb(\cbf)_{n+1}	
\end{align}
Then $(L_n)_{n\in\Zbb}$ satisfy the Viarsoro relation with the same central charge $c$ as that of $\Vbb$.
\end{pp}
\begin{proof}
Use the Jacobi identity, the translation property, and Rem. \ref{lb80} to compute $[Y_\Wbb(\cbf)_{m+1},Y_\Wbb(\cbf)_{k+1}]$.
\end{proof}

\begin{exe}
Show that $[L_0,Y_\Wbb(v,z)]=Y_\Wbb(L_0v,z)+z\partial_zY_\Wbb(v,z)$.
\end{exe}

\begin{rem}\label{lb81}
The above exercise shows that if $\Wbb$ is admissible, then $A:=\wtd L_0-L_0$ commutes with the action of $\Vbb$ on $\Wbb$, i.e., $A\in\End_\Vbb(\Wbb)$. In particular, it commutes with $L_0=Y_\Wbb(\cbf)_1$ and hence with $\wtd L_0$. Therefore, $\wtd L_0-L_0$ is an endomorphism of the admissible $\Vbb$-module $\Wbb$ commuting with $\wtd L_0$. Note also that by \eqref{eq127}, $L_n$ lowers the $\wtd L_0$-weights by $n$:
\begin{align}
[\wtd L_0,L_n]=-nL_n.\label{eq134}	
\end{align}
\end{rem}

\begin{cv}
	The grading of an admissible $\Vbb$-module always means the $\wtd L_0$-grading, even when $L_0$ is diagonalizable.
\end{cv}

\subsection{}

We discuss some basic properties of irreducible modules. 

\begin{cv}
A homomorphism of weak/admissible/finitely admissible modules $A:\Wbb_1\rightarrow\Wbb_2$ always means a linear map intertwining the actions of $\Vbb$.
\end{cv}

\begin{df}
	An \textbf{irreducible $\Vbb$-module} \index{00@Irreducible VOA modules} is a finitely admissible $\Vbb$-module with no proper graded $\Vbb$-invariant subspaces (i.e., no proper $\Vbb$- and $\wtd L_0$-invariant subspaces). 
\end{df}

\begin{lm}[Schur's lemma]\label{lb82}
Let $\Wbb$ be an irreducible $\Vbb$-module. Let $A\in\End_\Vbb(\Wbb)$ satisfying $[\wtd L_0,A]=0$. Then $A\in\Cbb\id_\Wbb$.
\end{lm}

\begin{proof}
By $[\wtd L_0,A]=0$, $A$ restricts to a linear operator on each $\Wbb(n)$. Choose $n$ such that $\Wbb(n)\neq0$. Since $\Wbb(n)$ is finite-dimensional, $A|_{\Wbb(n)}$ has an eigenvalue $\lambda$. So the (clearly $\Vbb$-invariant) subspace $\Ker(A-\lambda)$ is non-zero. It is also $\wtd L_0$ invariant since $[\wtd L_0,A-\lambda]=0$. So $\Ker(A-\lambda)=\Wbb$.
\end{proof}

\begin{co}\label{lb92}
Let $\Wbb$ be an irreducible $\Vbb$-module. Then $L_0=\wtd L_0+\lambda$ for some $\lambda\in\Cbb$. In particular, $L_0$ is diagonalizable on $\Wbb$.
\end{co}

\begin{proof}
This follows immediately from Rem. \ref{lb81} and Schur's lemma \ref{lb82}.
\end{proof}

From this corollary, we see that the $\wtd L_0$-gradings of an irreducible module are unique up to scalar addition.

\begin{co}
Any irreducible $\Vbb$-module $\Wbb$ has no proper $\Vbb$-invariant subspace.
\end{co}

\begin{proof}
Let $\Mbb$ be a $\Vbb$-invariant subspace of $\Wbb$. Then $\Mbb$ is $L_0$-invariant since $L_0=Y_\Wbb(\cbf)_1$. So $\Mbb$ is $\wtd L_0$-invariant, i.e., a graded subspace. So $\Mbb$ is not proper.
\end{proof}

By the same reasoning, we have:
\begin{co}[Schur's lemma]
Let $\Wbb$ be an irreducible $\Vbb$-module. Then $\End_\Vbb(\Wbb)=\Cbb\id_\Wbb$.
\end{co}

\begin{df}\label{lb164}
We say that $\Vbb$ is \textbf{rational} \index{00@Rational VOAs} if any admissible $\Vbb$-module $\Wbb$ is completely reducible, i.e., $\Wbb$ is a direct sum of irreducible $\Vbb$-modules.
\end{df}

\subsection{}

By replacing $L_0$ with $\wtd L_0$, all the results in Sec. \ref{lb83} and Subsec. \ref{lb85}-\ref{lb86} hold for admissible modules. 
\begin{exe}
Give a complex analytic definition of Jacobi identity (cf. Def. \ref{lb194}) for admissible $\Vbb$-modules that is equivalent to the algebraic Jacobi identity \eqref{eq128}.
\end{exe}
However, due to the lack of vacuum vector $\id$, the Goddard uniqueness and hence the reconstruction theorem do not hold for modules. Therefore, checking locality is not enough to prove the existence of $\Vbb$-module structures. To construct examples of modules, new methods are needed. 

Here is one easy method to construct VOA modules. Suppose $\Vbb$ is a subalgebra of a VOA $\Ubb$ such that the $L_0$ on $\Ubb$ restricts to that of $\Vbb$. (We do not assume $\Vbb$ and $\Ubb$ have the same conformal vector.) If we have constructed an admissible $\Ubb$-module $\Mbb$ (for instance, $\Mbb=\Ubb$), then by regarding $\Mbb$ as a $\Vbb$-module, any $\Vbb$-invariant graded subspace of $\Mbb$ is clearly a $\Vbb$-module.  In particular, if we already know that $\Mbb$ is a unitary $\Ubb$-module, then such constructed $\Vbb$-modules are unitary.

\subsection{}

Here we state some results on the irreducible modules associated to affine and Virasoro VOAs without providing proofs. The readers are referred to \cite[Chapter 6]{LL}, \cite{Was10}, and \cite{FZ92,Wang93} for details.

Let $\gk$ be either abelian or a simple Lie algebras. Let $W$ be a finite dimensional irreducible representation of $\gk$. Fix a level $l\in\Cbb$ such that $l+h^\vee\neq0$.   Recall the decomposition $\gk=\gk_-\oplus\gk_+$ into Lie subalgebras defined in Subsec. \ref{lb87}:
\begin{align}
	\wtd\gk_-=\Span\{X_n:X\in\gk,n<0\},\qquad \wtd\gk_+=\Span\{X_n,K,D:X\in\gk,n\geq0\}.	\label{eq129}
\end{align}
Then the $\gk$-module $W$ extends to a $\wtd\gk_+$-module structure such that $X_n$ acts trivially on $W$ if $n>0$, $X_0$ acts as $X$, and $D=0,K=l$ on $W$. Let
\begin{align*}
V_\gk(l,W)=\Ind_{\wtd\gk_+}^{\wtd\gk}(W)=U(\wtd\gk)\otimes_{U(\wtd\gk_+)}W,\qquad L_\gk(l,W)=V_\gk(l,W)/I 
\end{align*}
where $I$ is the largest proper $D$- and $\wtd\gk$-invariant subspace. Then $V_\gk(l,W)$ and $L_\gk(l,W)$ have  unique finitely admissible $V_\gk(l,0)$-module structures such that $D=\wtd L_0$ and that, letting $\Wbb$ be either of them, $Y_\Wbb(X_{-1}\id)_n$ equals the action of $X_n$ on $\Wbb$ for each $X\in\gk,n\in\Zbb$. $L_\gk(l,W)$ is irreducible. When $\gk$ is abelian, $W$ is unitary, and $l>0$, then $V_\gk(l,W)=L_\gk(l,W)$ are unitary modules.

Assume that $\gk$ is simple and $l\in\Nbb$. Then $W$ is naturally a unitary $\gk$-module. Then all irreducible modules of the WZW model $L_\gk(l,0)$ are unitary, and are given by all $L_\gk(l,W)$ where $W$ is an irreducible $\gk$-module satisfying the following property: (Skip this part if you are not familiar with Lie algebra representations.) Let $\lambda$ be the highest weight of $W$. Let $\theta$ be the highest  root (which is also a longest root) of $\gk$, namely, the highest weight of the adjoint representation of $\gk$. Recall the inner product $(\cdot|\cdot)$ on $\gk$ satisfying $(\theta|\theta)=2$, which restricts an inner product on the Cartan subalgebra $\hk$. It gives canonically an inner product on the dual space $\hk^*$ (i.e. the weight space). Then $(\theta|\lambda)$ (which is always $\geq0$) should be $\leq l$. There are only finitely many equivalence classes of such $W$.

Similarly, $\Vir=\Vir^+\oplus\Vir^-$ where
\begin{align*}
\Vir^-=\Span\{L_n:n\leq-1\},\qquad \Vir^+=\Span\{K,L_n:n\geq0\}.	
\end{align*}
For each $c,h\in\Cbb$, let $\Cbb_{c,h}$ be the one dimensional $\Vir^+$-module such on which $K=c,L_0=h$ and $L_n=0$ for all $n>0$. Let
\begin{align*}
M_\Vir(c,h)=\Ind_{\Vir^+}^{\Vir}\Cbb_{c,h}=U(\Vir)\otimes_{U(\Vir^+)}\Cbb_{c,h},\qquad L_\Vir(c,h)=M_\Vir(c,h)/I	
\end{align*}
where $I$ is again the largest proper submodule. Then there exist unique finitely admissible $V_\Vir(l,0)$-module structure on $\Wbb=M_\Vir(l,h)$ or $\Wbb=L_\Vir(l,h)$ with $\wtd L=L_0-h$ such that $Y_\Wbb(\cbf)_n$ is the action of $L_{n-1}$ on $\Wbb$.  

When $c$ satisfies \eqref{eq75}, the irreducible modules of the minimal model $L_\Vir(c,0)$ are classified by all $L_\Vir(c,h_{m,n})$ where $m,n$ are integers with $0<m<p,0<n<q$ and 
\begin{align}
h_{m,n}=\frac{(np-mq)^2-(p-q)^2}{4pq}.	
\end{align}
When $c$ satisfies \eqref{eq130},  $L_{\Vir}(c,0)$ and all its irreducible modules are unitary.

\subsection{}

The remaining part of this section is devoted to the study of contragredient modules (i.e., dual modules). Let $\Wbb=\bigoplus_{n\in\Nbb}\Wbb(n)$ be an admissible $\Vbb$-module. As usual, for each $n$ we define the projection of algebraic completion to $\Wbb(n)$ in the canonical way: \index{VW@$\Vbb^\cl,\Wbb^\cl$, the algebraic completions} \index{Pn@$P_n$}
\begin{align}
P_n:\Wbb^\cl=\prod_{n\in\Nbb}\Wbb(n)\rightarrow\Wbb(n).	
\end{align}
Define the graded dual space \index{VW@$\Vbb',\Wbb'$, the graded dual spaces}
\begin{align*}
\Wbb'=\bigoplus_{n\in\Nbb}\Wbb'(n):=\bigoplus_{n\in\Nbb}\Wbb(n)^*
\end{align*}
as usual. Then $P_n:\Wbb'\rightarrow\Wbb(n)^*$ is defined in an obvious way.

\subsection{}

Our goal is to define an admissible $\Vbb$-module structure $Y_{\Wbb'}$ on $\Wbb'$. To find the formula of $Y_{\Wbb'}$, consider the data
\begin{align*}
\fk X=(\Pbb^1;0,z,\infty;\zeta,\zeta-z,\zeta^{-1})	
\end{align*}
where $\zeta$ is the standard coordinate of $\Cbb$. If $\mc H^\fin$ contains $\Wbb\otimes\wht\Wbb$ where $\wht\Wbb$ is a $\wht\Vbb$-module, and if $w\otimes\wht w\in\Wbb\otimes\wht\Wbb$, $v\otimes\wht v\in\Vbb\otimes\wht\Vbb$, $w'\otimes\wht w'\in\Wbb'\otimes\wht\Wbb'$ are going into the punctures $0,z,\infty$ respectively, then the correlation function is given by
\begin{align}
\bk{w'\otimes\wht w',Y_\Wbb(v,z)w\otimes W_{\wht\Wbb}(\wht v,\ovl z)\wht w}=	\bk{w',Y_\Wbb(v,z)w}\bk{\wht w',Y_{\wht\Wbb}(\wht v,\ovl z)\wht w}.\label{eq132}
\end{align} 

To simplify discussions, we focus on the chiral halves. The standard conformal block for $\Wbb,\Vbb,\Wbb'$ associated to $0,z,\infty$ is given by $\bk{w',Y_\Wbb(v,z)w}$. Indeed, if we choose $\wht v=\id$, choose $\wht w,\wht w'$ such that $\bk{\wht w',\wht w}=1$, and identify $\Wbb$ with $\Wbb\otimes \wht w$ by identifying $w$ with $w\otimes\wht w$, and similarly identify $\Wbb'$ with $\Wbb'\otimes\wht w'$, then the correlation function \eqref{eq132} becomes exactly the conformal block $\bk{w',Y_\Wbb(v,z)w}$. So we can also view $\bk{w',Y_\Wbb(v,z)w}$ as a (restricted) correlation function.

We wish that the correlation function/standard conformal block associated to
\begin{align*}
\fk Z=(\Pbb^1;0,z^{-1},\infty;\zeta,\zeta-z^{-1},\zeta^{-1})	
\end{align*}
is
\begin{align*}
\psi(w'\otimes v\otimes w)=\bk{Y_{\Wbb'}(v,z^{-1})w',w}	
\end{align*}
where $\Wbb',\Vbb,\Wbb$ are associated to $0,z^{-1},\infty$. Now, the biholomorphism $\gamma\in\Pbb^1\mapsto\gamma^{-1}\in\Pbb^1$ gives almost an equivalence of $\fk X$ and $\fk Y$: the only exception is that the local coordinate $\zeta-z$, pulled back along this map, is $\zeta^{-1}-z$ but not $\zeta-z^{-1}$. So let us consider 
\begin{align*}
\fk Y=(\Pbb^1;0,z^{-1},\infty;\zeta,\zeta^{-1}-z,\zeta^{-1})	
\end{align*}
equivalent to $\fk X$ via $\gamma\mapsto\gamma^{-1}$. Again, we associate $\Wbb',\Vbb,\Wbb$  to $0,z^{-1},\infty$ as for $\fk Z$. Then the standard conformal block for $\fk Y$ is still
\begin{align*}
\phi(w'\otimes v\otimes w)=\bk{w',Y_\Wbb(v,z)w}.	
\end{align*}

Now we relate $\phi$ and $\psi$ using the change of coordinate formula, noting that $\zeta-z^{-1}=\vartheta_z\circ(\zeta^{-1}-z)$ where (for each $t\in\Pbb^1$) \index{zz@$\vartheta_z$}
\begin{align}
\vartheta_z(t)=\frac1{z+t}-\frac 1z.\label{eq195}
\end{align}
Therefore
\begin{align}
\phi(w'\otimes v\otimes w)=\psi(w'\otimes \mc U(\vartheta_z) v\otimes w)\label{eq131}	
\end{align}
where $\mc U(\vartheta_z)$ is the operator on (the Hilbert space completion of) $\Wbb$ associated to $\Vbb$.

\subsection{}

It remains to find $\mc U(\vartheta_z)$. To avoid conflict of notations, we write $z^{n+1}\partial_z$ in Sec. \ref{lb41} as $\zeta^{n+1}\partial_\zeta$. Then by \eqref{eq17}, $\exp(z\zeta^2\partial_\zeta)$ sends $\gamma$ to $(1/\gamma-z)^{-1}$ and hence $-z^{-2}\gamma$ to $\vartheta_z(\gamma)$. This means
\begin{align}
\vartheta_z=\exp(z\zeta^2\partial_\zeta)\circ \exp(\log(-z^{-2})\zeta\partial_\zeta)=\exp(z\zeta^2\partial_\zeta)\circ (-z^{-2})^{\zeta\partial_\zeta}.
\end{align}
Thus, on $\Vbb$,
\begin{align}
\mc U(\vartheta_z)=e^{zL_1}(-z^{-2})^{L_0}.	\label{eq196}
\end{align}
Expanding \eqref{eq131}, we get
\begin{align*}
\bigbk{w',Y_\Wbb(v,z)w}=\bigbk{Y_{\Wbb'}\big(e^{zL_1}(-z^{-2})^{L_0}v,z^{-1}\big)w',w}	
\end{align*}
Exchange the role of $\Wbb$ and $\Wbb'$, and we get our definition:

\begin{df}
Let $\Wbb$ be an admissible $\Vbb$-module. Then $Y_{\Wbb'}:\Vbb\rightarrow(\End\Wbb')[[z^{\pm1}]]$ is defined by
\begin{align}
\bigbk{Y_{\Wbb'}(v,z)w',w}=\bigbk{w',Y_\Wbb\big(e^{zL_1}(-z^{-2})^{L_0}v,z^{-1}\big)w}\label{eq141}		
\end{align}
for each $v\in\Vbb,w\in\Wbb,w'\in\Wbb'$. Assuming $v$ to be homogeneous, this means
\begin{align}
Y_{\Wbb'}(v,z)=\sum_{k\in\Nbb}\frac {z^k}{k!}\cdot (-z^{-2})^{\wt v}\cdot Y_\Wbb\big(L_1^kv,z^{-1}\big)^\tr.
\end{align}
Expanding both sides, we see that for each $n\in\Zbb$,
\begin{align}
Y_{\Wbb'}(v)_n=\sum_{k\in\Nbb}\frac{(-1)^{\wt v}}{k!}\big(Y_\Wbb(L_1^kv)_{-n-k-2+2\wt v}\big)^\tr.\label{eq133}
\end{align}
\end{df}

\begin{exe}
Let $L_n$ be $Y_{\Wbb'}(\cbf)_{n+1}$ on $\Wbb'$. Use \eqref{eq133} to show that for each $w\in\Wbb,w'\in\Wbb'$,
\begin{align}
\bk{L_nw',w}=\bk{w',L_{-n}w}.	
\end{align}
\end{exe}

\subsection{}

The purpose of this subsection is to prove Cor. \ref{lb91}.

\begin{exe}\label{lb99}
	Use $[\wtd L_0,L_1]=-L_1$ to show that when acting $w\in\Wbb$,
	\begin{subequations}
		\begin{gather}
			L_1	\lambda^{\wtd L_0}=\lambda^{\wtd L_0+1}L_1,\label{eq137}\\
			e^{\tau L_1}\lambda^{\wtd L_0}=\lambda^{\wtd L_0}e^{\tau\lambda L_1}	\label{eq138}
		\end{gather}
	\end{subequations}
	in $\Wbb[\lambda]$ and $\Wbb[\lambda,\tau]$ respectively.
	
	(Hint. Method 1: Compute $\partial_\lambda$ for the first equation,  $\partial_{\tau}$ for the second one, and apply Lemma \ref{lb21}. Method 2: Use the fact that $L_1$ lowers the weights by $1$ to verify the equations when $v$ is homogeneous.)
\end{exe}
By taking $\partial_\lambda$ of \eqref{eq138} at $\lambda=1$, we get
\begin{align}
	e^{\tau L_1}\wtd L_0=\wtd L_0e^{\tau L_1}+\tau L_1e^{\tau L_1}.\label{eq139}	
\end{align}

\begin{co}\label{lb91}
	we have
	\begin{align}
		Y_\Wbb(v,z)=Y_{\Wbb'}\big(e^{zL_1}(-z^{-2})^{L_0}v,z^{-1}\big)^\tr.	
	\end{align}
	Thus, if $\Wbb$ is finitely admissible, then $\Wbb''=\Wbb$ and $Y_{\Wbb''}=Y_\Wbb$. 
\end{co}

\begin{proof}
	By \eqref{eq141},
	\begin{align*}
		Y_{\Wbb'}\big(e^{zL_1}(-z^{-2})^{L_0}v,z^{-1}\big)^\tr=Y_\Wbb\big(e^{z^{-1}L_1}(-z^2)^{L_0}e^{zL_1}(-z^{-2})^{L_0}v,z^{-1} \big)^\tr,	
	\end{align*}
	which equals $Y_\Wbb(v,z)$ since $(-z^2)^{L_0}e^{zL_1}=e^{-z^{-1}L_1}(-z^2)^{L_0}$ due to \eqref{eq138}.
\end{proof}

\subsection{}

In the rest of  this section, we prove the following main result of this section.

\begin{thm}\label{lb94}
Let $(\Wbb,Y_\Wbb)$ be an admissible $\Vbb$-module. Then $(\Wbb',Y_{\Wbb'})$ is an admissible $\Vbb$-module, called the \textbf{contragredint $\Vbb$-module} of $\Wbb$.\index{W@$\Wbb'$, the contragredient $\Vbb$-module of $\Wbb$} If $\Wbb$ is finitely-admissible, then so is $\Wbb'$, and under the canonical identification $\Wbb=\Wbb''$ we have $Y_\Wbb=Y_{\Wbb''}$.
\end{thm}

The very last sentence of this theorem is proved. To verify that $\Wbb'$ is an admissible module, we begin with the following simple observation.

\begin{lm}
If $v\in\Vbb$ is homogeneous, then $Y_{\Wbb'}(v,z)$ is $\wtd L_0$-homogeneous with weight $\wt v$.
\end{lm}
\begin{proof}
Using \eqref{eq134} and \eqref{eq133}, one easily computes $[\wtd L_0,Y_{\Wbb'}(v)_n]=(\wt v-n-1)Y_{\Wbb'}(v)_n$.
\end{proof}

It is clear that $Y_{\Wbb'}(\id,z)=\id_{\Wbb'}$. To prove that $Y_{\Wbb'}$ satisfies the axioms of an admissible module, it remains to check the Jacobi identity. We first prove the locality:

\begin{lm}\label{lb88}
Let $u,v\in\Vbb$ be homogeneous. Then $Y_{\Wbb'}(u,z)$ and $Y_{\Wbb'}(v,z)$ are local.
\end{lm}

\begin{proof}
We prove the complex analytic locality. For each $w\in\Wbb,w'\in\Wbb'$,
\begin{align}
&\sum_{n\in\Nbb}\bigbk{Y_{\Wbb'}(u,z_1)P_nY_{\Wbb'}(v,z_2)w',w}\\
=&\sum_{n\in\Nbb}\bigbk{w',Y_\Wbb\big(e^{z_2L_1}(-z_2^{-2})^{L_0}v,z_2^{-1}\big)P_nY_\Wbb\big(e^{z_1L_1}(-z_1^{-2})^{L_0}u,z_1^{-1}\big)w}
\end{align}
which converges a.l.u. on $0<|z_1^{-1}|<|z_2^{-1}|$ by the locality of $Y_{\Wbb}(u,z)$ and $Y_{\Wbb}(v,z)$. Moreover, this locality shows that the above expression and 
\begin{align}
&\sum_{n\in\Nbb}\bigbk{Y_{\Wbb'}(v,z_2)P_nY_{\Wbb'}(u,z_1)w',w}\\
=&\sum_{n\in\Nbb}\bigbk{w',Y_\Wbb\big(e^{z_1L_1}(-z_1^{-2})^{L_0}u,z_1^{-1}\big)P_n Y_\Wbb\big(e^{z_2L_1}(-z_2^{-2})^{L_0}v,z_2^{-1}\big)w}\label{eq142}
\end{align}
(which converges a.l.u. on $0<|z_2^{-1}|<|z_1^{-1}|$) can be extended to the same holomorphic function $f$ on $\Conf^2(\Cbb^\times)$. This function is a $\Cbb[z_1^{\pm1},z_2^{\pm1}]$-linear combination of $4$-point correlations functions of the form $\bk{w',Y_\Wbb(\cdot,z_1)Y_{\Wbb}(\cdot,z_2)w}$ which is holomorphic on $\Cbb^\times\times\Cbb^\times$ when multiplied by $(z_1-z_2)^N$ for some $N$. So $f$ shares the same property. 
\end{proof}

\subsection{}

Write $A(z)=Y_{\Wbb'}(u,z)$ and $B(z)=Y_{\Wbb'}(v,z)$. Since $A$ and $B$ are local,  we have the Jacobi identity \eqref{eq107} for $A(z),B(z),(A_\blt B)(z)$, which implies the Jacobi identity for $Y_{\Wbb'}$ if we can show that for all $k\in\Zbb$ and homogeneous $w\in\Wbb,w'\in\Wbb'$, as elements of $\Cbb[z_2^{\pm1}]$ we have
\begin{align}
\bigbk{(A_kB)(z_2)w',w}=\bigbk{Y_{\Wbb'}(Y(u)_kv,z_2)w',w}.	\label{eq145}
\end{align}

By \eqref{eq105}, the LHS of \eqref{eq145} is $\Res_{z_1=z_2}(z_1-z_2)^kf(z_1,z_2)dz_1$ where $f$ was defined in the proof of Lemma \ref{lb88}. By \eqref{eq142} and the complex analytic Jacobi identity for $Y_\Wbb$,  the RHS of
\begin{align}
f(z_1,z_2)=	\sum_{n\in\Nbb}\bigbk{w', Y_\Wbb\big(P_nY\big(e^{z_1L_1}(-z_1^{-2})^{L_0}u,z_1^{-1}-z_2^{-1}\big)\cdot e^{z_2L_1}(-z_2^{-2})^{L_0}v,z_2^{-1}\big)w}\label{eq143}
\end{align}
converges a.l.u. on $0<|z_1^{-1}-z_2^{-1}|<|z_2^{-1}|$ to the LHS. 

The RHS of \eqref{eq145} is the application of $\Res_{z_1-z_2=0}\cdot (z_1-z_2)^kd(z_1-z_2)$ to the following  elements of $\Cbb[z_2^{\pm1}][[(z_1-z_2)^{\pm1}]]$:
\begin{align}
&\sum_{n\in\Zbb}(z_1-z_2)^{-n-1}\bigbk{Y_{\Wbb'}\big(Y(u)_nv,z_2\big)w',w}\nonumber\\
=&\bigbk{Y_{\Wbb'}\big(Y(u,z_1-z_2)v,z_2\big)w',w}\nonumber\\
=&\bigbk{w',Y_\Wbb\big(e^{z_2L_1}(-z_2^{-2})^{L_0}Y(u,z_1-z_2)v,z_2\big)w}\nonumber\\
=&\bigbk{w',Y_\Wbb\big(e^{z_2L_1}Y\big((-z_2^{-2})^{L_0})u,z_2^{-2}(z_2-z_1)\big)(-z_2^{-2})^{L_0}v,z_2\big)w}.\label{eq144}
\end{align}
where the scale covariance is used in the last equality.
\begin{exe}
Set the following element of $\Cbb[z_2^{\pm1}][[(z_1-z_2)^{\pm1}]]$:
\begin{align*}
g_n(z_2,z_1-z_2)=\bigbk{w',Y_\Wbb\big(P_ne^{z_2L_1}Y\big((-z_2^{-2})^{L_0})u,z_2^{-2}(z_2-z_1)\big)(-z_2^{-2})^{L_0}v,z_2\big)w}.	
\end{align*}
Show that for each $k\in\Zbb$,
\begin{align}
\Res_{z_1-z_2=0}~(z_1-z_2)^kg_n(z_2,z_1-z_2)d(z_1-z_2)	\label{eq149}
\end{align}
is a monomial of $z_2^{\pm1}$ that vanishes when $n>\wt u+\wt v-k-1$. Conclude that the application of $\Res_{z_1-z_2=0}\cdot (z_1-z_2)^kd(z_1-z_2)$ to \eqref{eq144} equals the (automtically finite) sum over all $n$ of \eqref{eq149}.
\end{exe}


It follows that \eqref{eq145} holds 
if we can show: For any $v'\in\Vbb'(n)=\Vbb(n)^*$ (e.g.,  $\bk{v',\cdot}=\bk{w',Y_\Wbb(P_n\cdot,z_2)w}$), as holomorphic functions of $z_2$ on $\Cbb^\times$,
\begin{align}
&\Res_{z_1=z_2}(z_1-z_2)^k\bigbk{v',Y\big(e^{z_1L_1}(-z_1^{-2})^{L_0}u,z_1^{-1}-z_2^{-1}\big)\cdot e^{z_2L_1}(-z_2^{-2})^{L_0}v}dz_1\nonumber\\
=&	\Res_{z_1-z_2=0}(z_1-z_2)^k \bigbk{v',e^{z_2L_1}Y\big((-z_2^{-2})^{L_0}u,z_2^{-2}(z_2-z_1)\big)(-z_2^{-2})^{L_0}v}d(z_1-z_2)\label{eq146}
\end{align}
where the LHS is the residue of a holomorphic function and the RHS is that of a formal Laurent series. This follows if we can show that
\begin{align}
	&\bigbk{v',Y\big(e^{z_1L_1}(-z_1^{-2})^{L_0}u,z_1^{-1}-z_2^{-1}\big)\cdot e^{z_2L_1}(-z_2^{-2})^{L_0}v}\nonumber\\
	=&\bigbk{v',e^{z_2L_1}Y\big((-z_2^{-2})^{L_0}u,z_2^{-2}(z_2-z_1)\big)(-z_2^{-2})^{L_0}v}\label{eq148}
\end{align}
where the RHS as a formal Laurent series of $z_2,z_1-z_2$ converges a.l.u. on $0<|z_1-z_2|<|z_2|$ to the LHS as a holomorphic function of $z_1,z_2$.

Clearly, as elements of $\Cbb[[z_2^{\pm1},(z_1-z_2)^{\pm1}]]$ the sum
\begin{align}
&\bigbk{v',e^{z_2L_1}Y\big((-z_2^{-2})^{L_0}u,z_2^{-2}(z_2-z_1)\big)(-z_2^{-2})^{L_0}v}\nonumber\\
=&\sum_{n\in\Nbb}\bigbk{v',e^{z_2L_1}P_nY\big((-z_2^{-2})^{L_0}u,z_2^{-2}(z_2-z_1)\big)(-z_2^{-2})^{L_0}v}	\label{eq147}
\end{align}
satisfies the conditions in Lemma \ref{lb65}. If we can prove the claim that the RHS converges a.l.u. on $0<|z_1-z_2|<|z_2|$ to the LHS of \eqref{eq148},
then by Lemma \ref{lb65}, we are done with the proof. The claim follows from the following ``$e^{\tau L_1}$-covariance" (where $\tau=z_2,z=z_2^{-2}(z_2-z_1)$), which we prove for $Y_\Wbb$ though we actually just need it for $Y=Y_\Vbb$.

\begin{pp}\label{lb90}
Let $\Wbb$ be admissible. Then for each $v\in\Vbb,w\in\Wbb,w'\in\Wbb'$, the LHS of
\begin{align}\label{eq151}
\begin{aligned}
\sum_{n\in\Nbb}\bigbk{w',e^{\tau L_1}P_nY_\Wbb(v,z)w}=\bigbk{w',Y_\Wbb\big(e^{\tau(1-\tau z)L_1}(1-\tau z)^{-2L_0}v,z/(1-\tau z)\big)e^{\tau L_1}w}
\end{aligned}	
\end{align}
converges a.l.u. on $\{(z,\tau)\in\Cbb^\times\times\Cbb:|\tau|<|z^{-1}|\}$ to the RHS.
\end{pp}

This theorem is a special case of the conformal covariance Thm. \ref{lb95} which will be explained later. However, the proof of Thm. \ref{lb95} is quite involved. So in the following we give an elementary proof of Prop. \ref{lb90}.

\subsection{}

We view the $e^{\tau L_1}$-covariance of $Y_\Wbb$ as the transpose of the translation covariance of $Y_{\Wbb'}$. So we first need to prove the latter.

\begin{lm}
We have $[L_{-1},Y_{\Wbb'}(v,z)]=\partial_zY_{\Wbb'}(v,z)$.\label{eq136}
\end{lm}

\begin{proof}
Assume $v$ is homogeneous. It suffices prove 
\begin{align}
\big[L_1,Y_\Wbb\big(e^{zL_1}z^{-2L_0}v,z^{-1}\big)\big]=-\partial_zY_\Wbb\big(e^{zL_1}z^{-2L_0}v,z^{-1}\big).\label{eq135}	
\end{align}
which is the transpose of the formula we want to prove multiplied by $(-1)^{\wt v}$. The Jacobi identity for $Y_\Wbb$ implies \eqref{eq68} where $Y$ is replaced by $Y_\Wbb$. By \eqref{eq68},
\begin{align}
	[L_1,Y_\Wbb(v,z)]=z^2Y_\Wbb(L_{-1}v,z)+2zY_\Wbb(L_0v,z)+Y_\Wbb(L_1v,z).	
\end{align}
Using this relation, one checks that the LHS of \eqref{eq135} equals
\begin{align}
Y_\Wbb(L_1e^{zL_1}z^{-2L_0}v,z^{-1})+2z^{-1}Y_\Wbb(L_0e^{zL_1}z^{-2L_0}v,z^{-1})+z^{-2}Y_\Wbb(L_{-1}e^{zL_1}z^{-2L_0}v,z^{-1}).\label{eq140}	
\end{align}

It is easy to guess by chain rule and verify rigorously by series expansions that
\begin{align*}
-\partial_z Y_\Wbb(v,z^{-1})=z^{-2}Y_\Wbb(L_{-1}v,z^{-1}).	
\end{align*}
Thus, the RHS of \eqref{eq135} is
\begin{align*}
-Y_\Wbb(L_1e^{zL_1}z^{-2L_0}v,z^{-1})+2z^{-1}Y_\Wbb(e^{zL_1}L_0z^{-2L_0}v,z^{-1})+z^{-2}Y_\Wbb(L_{-1}e^{zL_1}z^{-2L_0}v,z^{-1})
\end{align*}
which equals \eqref{eq140} due to \eqref{eq139}.
\end{proof}

\subsection{}

Now that we have the translation property for $Y_{\Wbb'}$, we have the translation covariance in the form of \eqref{eq40} or (equivalently) Exercise \ref{lb89}. We need the latter form: the LHS of 
\begin{align}
	\sum_{n\in\Nbb}\bigbk{Y_{\Wbb'}(u,z)P_ne^{\tau L_{-1}}w',w}=\bigbk{Y_{\Wbb'}(u,z-\tau)w',e^{\tau L_1}w},\label{eq150}
\end{align}
converges a.l.u. on $|\tau|<|z|$ to the RHS.

\begin{proof}[Proof of Prop. \ref{lb90}]
By Cor. \ref{lb91}, as sums of holomorphic functions we have
\begin{align*}
&\sum_{n\in\Nbb}\bigbk{w',e^{\tau L_1}P_nY_\Wbb(v,z)w}=\sum_{n\in\Nbb}\bigbk{P_ne^{\tau L_{-1}}w',Y_\Wbb(v,z)w}\\
=&\sum_{n\in\Nbb}\bigbk{Y_{\Wbb'}(e^{zL_1}(-z^{-2})^{L_0}v,z^{-1})P_ne^{\tau L_{-1}}w',w},
\end{align*}
which by \eqref{eq150} converges a.l.u. on $|\tau|<|z^{-1}|$ to
\begin{align*}
\bigbk{Y_{\Wbb'}\big(e^{zL_1}(-z^{-2})^{L_0}v,z^{-1}-\tau\big)w',e^{\tau L_1}w}.
\end{align*}
We move $Y_{\Wbb'}$ to the right using \eqref{eq141}. Then the above becomes
\begin{align*}
\bigbk{w',Y_\Wbb\big(e^{(z^{-1}-\tau)L_1}(-(z^{-1}-\tau)^{-2})^{L_0}e^{zL_1}(-z^{-2})^{L_0}v,(z^{-1}-\tau)^{-1}\big)e^{\tau L_1}w}.
\end{align*}
This equals the RHS of \eqref{eq151} since, by \eqref{eq138}, when acting on $\Vbb$,
\begin{align*}
(-(z^{-1}-\tau)^{-2})^{L_0}e^{zL_1}=e^{-z(z^{-1}-\tau)^2L_1}	(-(z^{-1}-\tau)^{-2})^{L_0}.
\end{align*}
\end{proof}

The proof of Thm. \ref{lb94} is complete.

\subsection{}

\begin{df}
	We say that $\Vbb$ is \textbf{self-dual} \index{00@Self-dual VOAs} if the vacuum module $\Vbb$ (with grading $\wtd L_0=L_0$) is isomorphic to its contragredient module $\Vbb'$.
\end{df}

The construction of tensor product modules is much easier:

\begin{pp}
Let $\Vbb_1,\Vbb_2$ be VOAs and $\Wbb_i$ be an admissible $\Vbb_i$-module. Then the vector space $\Wbb_1\otimes\Wbb_2$ has a unique admissible $\Vbb_1\otimes\Vbb_2$-module structure with grading $\wtd L_0\otimes\id_{\Wbb_2}+\id_{\Wbb_1}\otimes \wtd L_0$ such that for each $v_i\in\Vbb_i$,
\begin{align}
Y_{\Wbb_1\otimes\Wbb_2}(v_1\otimes v_2,z)=Y_{\Wbb_1}(v_1,z)\otimes Y_{\Wbb_2}(v_2,z).	
\end{align}
Equivalently, for each $w_i\in \Wbb_i,w_i'\in\Wbb_i'$,
\begin{align}
\bigbk{w_1'\otimes w_2',Y_{\Wbb_1\otimes\Wbb_2}(v_1\otimes v_2,z)(w_1\otimes w_2)}=\bk{w_1',Y_{\Wbb_1}(v_1,z)w_1}\cdot \bk{w_2',Y_{\Wbb_2}(v_2,z)w_2}.	\label{eq152}
\end{align}
\end{pp}

\begin{proof}
Using \eqref{eq152}, it is easy to verify that $Y_{\Wbb_1\otimes\Wbb_2}$ satisfies the complex analytic Jacobi identity.
\end{proof}

\section{Change of coordinate theorems}

\subsection{}

The goal of this section is to study the change of local coordinates in a rigorous way. Due to some convergence issues, it is very difficult to show that a given local coordinate of $\Cbb$ at $0$ can be written as $\exp(f\partial_z)$ for a holomorphic vector field $f\partial_z$. So we first discuss formal coordinates and find the formal vector fields generating them. 

Define the following two subspaces of $z\cdot \Cbb[[z]]$ \index{G@$\scr G,\scr G_+$}
\begin{gather}
\mc G=\Big\{\sum_{n\in\Zbb_+} a_nz^n:a_1\neq 0\Big\},\qquad \mc G_+=\Big\{z+\sum_{n\geq 2} a_nz^n\in\mc G\Big\}.	
\end{gather}
Elements of $\mc G$ are viewed as formal local coordinates of $\Cbb$ at $0$. Likewise, set \index{G@$\Gbb,\Gbb_+$}
\begin{align}
\Gbb=\Big\{\alpha(z)\in\mc G:\sum_n |a_n|r^n<+\infty\text{ for some }r>0\Big\},\qquad \Gbb_+=\Gbb\cap\mc G_+.
\end{align}
Then elements of $\Gbb$ are local coordinates of $\Cbb$ at $0$, or equivalently, transformations of local coordinates.

There is an obvious right action of $\mc G$ on $\Cbb((z))$ defined by composition $f\mapsto f\circ\alpha$ if $f\in\Cbb((z))$ and $\alpha\in\mc G$. We leave it to the readers to check that it is well-defined. So $\mc G$ is a group whose product is the composition and whose identity is $z$.




\subsection{}

According to Sec. \ref{lb41}, to find the change of coordinate operator $\mc U(\alpha)$ for each $\alpha\in\mc G$, we need to write it as $\alpha=\exp(\sum_{n\geq 0}c_n z^{n+1}\partial_z)$. This task is easy if $\alpha\in\mc G_+$. Indeed, write
\begin{align}
\alpha(z)=z+\sum_{n\geq 2}a_nz^n.	
\end{align}
Then we can indeed choose $c_0=0$, and
\begin{align}\label{eq163}
\begin{aligned}
&\alpha(z)=\sum_{k\in\Nbb}\frac 1{k!}\Big(\sum_{n\geq 1}c_n z^{n+1}\partial_z\Big)^k(z)\\
=&z+\sum_{n_1\geq 1}c_{n_1}z^{n_1+1}+\frac 1{2!}\sum_{n_1,n_2\geq 1}(n_1+1)c_{n_1}c_{n_2}z^{n_1+n_2+1}\\
&+\frac 1{3!}\sum_{n_1,n_2,n_3\geq1}(n_1+1)(n_1+n_2+1)c_{n_1}c_{n_2}c_{n_3}z^{n_1+n_2+n_3+1}+\cdots.
\end{aligned}
\end{align}
This means that for each $m\geq 1$,
\begin{align}
a_{m+1}=c_m+\sum_{\begin{subarray}{c}
2\leq l\leq m\\
n_1,\dots,n_l\in\Zbb_+\\
n_1+\cdots+n_l=m		
\end{subarray}}	
\frac 1{l!}(n_1+1)\cdots (n_1+n_2+\cdots +n_{l-1}+1)c_{n_1}\cdots c_{n_l}.\label{eq158}
\end{align}
This shows that one can solve $c_1,c_2,\dots$ given the coefficients $a_2,a_3,\dots$.

For a general $\alpha\in\mc G$, instead of solving $\alpha=\exp(\sum_{n\geq 0}c_n z^{n+1}\partial_z)$, it is easier to solve
\begin{align}
\alpha(z)=\alpha'(0)\cdot \exp\Big(\sum_{n\geq 1} c_nz^{n+1}\partial_z\Big)(z).	
\end{align}
since $\alpha(z)/\alpha'(0)\in\mc G_+$. The first several terms are
\begin{subequations}
\begin{gather}
c_1=\frac 12\frac{\alpha''(0)}{\alpha'(0)},\\
c_2=\frac 16\frac{\alpha'''(0)}{\alpha'(0)}-\frac 14\Big(\frac{\alpha''(0)}{\alpha'(0)}\Big)^2.
\end{gather}
\end{subequations}
The corresponding linear operator on an admissible $\Vbb$-module $\Wbb$ is given by \index{U@$\mc U(\alpha),\mc U(\eta),\mc U(\eta_\blt)$}
\begin{align}
\mc U(\alpha)=\alpha'(0)^{\wtd L_0}\exp\Big(\sum_{n\geq1} c_nL_n\Big)=	\alpha'(0)^{\wtd L_0}\sum_{k\in\Nbb}\frac 1{k!}\Big(\sum_{n\geq 1}c_n L_n\Big)^k.\label{eq192}
\end{align}
Its inverse is $\mc U(\alpha)^{-1}=\exp(-\sum_{n\geq 1}c_nL_n)\alpha'(0)^{-\wtd L_0}$.

The point of replacing $L_0$ with $\wtd L_0$  is to avoid the ambiguity caused by the non-integral eigenvalues of $L_0$. Since (by Cor. \ref{lb92}) $\wtd L_0-L_0$ is a constant if $\Wbb$ is irreducible,  it is not a big deal to make such a replacement.

\begin{rem}
By the fact that $L_n$ lowers the weights by $n$, the  above double sum is finite when $\mc U(\alpha)$ (and similarly $\mc U(\alpha)^{-1}$) is acting on any vector. Moreover, they preserve $\Wbb^{\leq n}$ for each $n\in\Nbb$ where \index{W@$\Wbb^{\leq n}$}
\begin{align}
\Wbb^{\leq n}=\bigoplus_{0\leq j\leq n}\Wbb(j).	
\end{align}
So $\mc U(\alpha)$ restricts to a linear isomorphism on each $\Wbb^{\leq n}$. Note that $\mc U(\alpha)$ does not preserve $\Wbb(n)$.
\end{rem}

\subsection{}\label{lb97}

In applications, we need to consider a \textbf{holomorphic family of (analytic) transformations} \index{00@Holomorphic families of transformations} \index{00@Holomorphic families of transformations} $\rho:X\rightarrow\Gbb$, which means that $\rho=\rho_x(z)$ is a holomorphic function on a neighborhood of $X\times \{0\}$ in $X\times\Cbb$ where $X$ is a complex manifold (here $(x,z)\in X\times \Cbb$), and $\rho_x(0)=0$ and $\rho_x'(0)\equiv\partial_z\rho_x(0)\neq0$ for all $x\in X$.

We now restrict to the case that $X$ is an open subset $U$ of $\Cbb$ and let $\zeta$ be the standard variable of $U$, but consider a slightly more general situation that $\rho=\scr O(U)[[z]]$ with $\rho_\zeta(0)=0$ and $\rho_\zeta'(0)\neq0$ for all $\zeta\in U$. Equivalently,
\begin{align}
	\rho_\zeta(z)=\sum_{n\geq 1}\frac 1{n!}\rho_\zeta^{(n)}(0)z^n
\end{align}
where each $\zeta\mapsto\rho_\zeta^{(n)}(0)$ is an element of $\scr O(U)$ and $\rho_\zeta'(0)\neq0$. Note that when $z\neq0$, $\rho_\zeta(z)$ does not make sense as a value. We call $\rho:U\rightarrow\mc G$ a \textbf{family of formal coordinates}.

\begin{rem}\label{lb96}
We can take limits and derivatives for elements of $\scr O(U)[[z^{\pm1}]]$ by treating each $\scr O(U)$-coefficient. So, for instance, the derivative $\partial_z\rho_\zeta(z)$ at $\zeta_0\in U$ makes sense analytically as the value of the limit $\lim_{\zeta\rightarrow\zeta_0}\frac{\rho_\zeta(z)-\rho_{\zeta_0}(z)}{\zeta-\zeta_0}$.
\end{rem}

By \eqref{eq158},
\begin{align}
	\rho_\zeta=\rho'_\zeta(0)\exp\Big(\sum_{n\geq 1}c_n(\zeta)z^{n+1}\partial_z\Big)	\label{eq165}
\end{align}
where $c_1,c_2,\dots\in\scr O(U)$. So
\begin{align}
\mc U(\rho_\zeta)=\rho_\zeta'(0)^{\wtd L_0}\exp\Big(\sum_{n\geq 1}c_n(\zeta)L_n\Big),\label{eq167}
\end{align}
which shows that
\begin{align}
	\mc U(\rho_\zeta)\big|_{\Wbb^{\leq k}}\in\End(\Wbb^{\leq k})\otimes\scr O(U)\label{eq170}	
\end{align}
for each $k\in\Nbb$.

\subsection{}

Let $\rho:U\rightarrow\mc G$ be a family of formal coordinates.

\begin{pp}\label{lb98}
Suppose $0\in U$ and $\rho_0(z)=z$. Then, when acting on each vector of $\Wbb$, or equivalently, when restricted to each $\Wbb^{\leq k}$,
\begin{align}
	\partial_\zeta \mc U(\rho_\zeta)\Big|_{\zeta=0}=\sum_{n\geq 1}\frac 1{n!}\Big(\partial_\zeta\rho_\zeta^{(n)}(0)\Big|_{\zeta=0}\Big) \wtd L_{n-1}\label{eq166}
\end{align}
where $\wtd L_n=L_n$ when $n\geq 1$, and $\partial_\zeta\rho^{(n)}_\zeta(z)=\partial_\zeta\partial_z^n\rho_\zeta(z)$.
\end{pp}

\begin{rem}\label{lb112}
The geometric meaning of Prop. \ref{lb98} is the following. Assume that $\rho:U\rightarrow\Gbb$ is a holomorphic family with $\rho_0(z)=z$. Then for each $z_0$ near $0$, $\zeta\mapsto \rho_\zeta(z_0)$ is a path in $\Cbb$ whose initial value is $z_0$. So $\partial_\zeta\rho_\zeta(z_0)\partial_z\big|_{\zeta=0}$ is the tangent vector at $z_0$ describing the velocity of the path. By assembling these tangent vectors, we get a holomorphic tangent vector field $\partial_\zeta\rho_\zeta(z)\partial_z\big|_{\zeta=0}$, which equals
\begin{align}
\partial_\zeta\rho_\zeta(z)\partial_z\Big|_{\zeta=0}=\sum_{n\geq 1}\frac 1{n!}\partial_\zeta\rho_\zeta^{(n)}(0)z^n\partial_z\Big|_{\zeta=0}.
\end{align}
In view of the correspondence $z^n\partial_z\leftrightarrow L_{n-1}$, Prop. \ref{lb98} says that $\partial_\zeta\mc U(\rho_\zeta)\big|_{\zeta=0}$ is exactly the linear operator corresponding to the tangent vector field.
\end{rem}

\begin{proof}[Proof of Prop. \ref{lb98}]
From \eqref{eq167}, $\partial_\zeta \mc U(\rho_\zeta)$ is expressed in terms of $c_n$. So we need to express $c_n$ in terms of $\partial_\zeta\rho_\zeta^{(n)}(0)$. By \eqref{eq165} and \eqref{eq163},
\begin{align*}
\rho_\zeta(z)=\rho'_\zeta(0)\Big(z+\sum_{n\geq 1}c_n(\zeta)z^{n+1}\Big)+R_\zeta(z)	
\end{align*}
where $R_\zeta(z)$ is a sum of polynomials of $z$ multiplied by at least two terms of $c_1(\zeta),c_2(\zeta),\dots$. Since $\rho_0(z)=z$, equivalently, $\rho'_0(0)=1$ and $c_1(0)=c_2(0)=\cdots=0$, we have $\partial_\zeta R(z)|_{\zeta=0}=0$. So
\begin{align}
\partial_\zeta\rho_\zeta(z)\Big|_{\zeta=0}=\partial_\zeta\rho'_\zeta(0)z+\sum_{n\geq 1}\partial_\zeta c_n(0) z^{n+1}.\label{eq168}
\end{align}
A similar argument applied to the derivative of \eqref{eq167} shows
\begin{align}
\partial_\zeta\mc U(\rho_\zeta)\Big|_{\zeta=0}=\partial_\zeta\rho'_\zeta(0)\wtd L_0+\sum_{n\geq 1}\partial_\zeta c_n(0)L_n.	\label{eq169}
\end{align}
By \eqref{eq168}, for all $n\geq 2$,
\begin{align*}
\frac 1{n!}\partial_\zeta\rho^{(n)}_\zeta(0)	\Big|_{\zeta=0}=\partial_\zeta c_{n-1}(0).
\end{align*}
Substituting this relation into \eqref{eq169} proves \eqref{eq166}.
\end{proof}

\subsection{}

\begin{thm}[{\cite[Sec. 4.2]{Hua97}}]\label{lb100}
$\mc U:\mc G\rightarrow\End(\Wbb)$ is a group representation. Namely, $\mc U(\alpha\circ\beta)=\mc U(\alpha)\mc U(\beta)$ for all $\alpha,\beta\in\mc G$.
\end{thm}

With the help of this theorem, we can calculate $\partial_\zeta\mc U(\rho_\zeta)$ at $\zeta=0$ without assuming $\rho_0(z)=z$ by computing $\partial_\zeta\mc U(\rho_\zeta\circ\rho^{-1}_0)$ using Prop. \ref{lb98}.

\begin{proof}
It suffices to consider the following two cases: (a) $\alpha,\beta\in\mc G_+$ (b) $\alpha\in\mc G_+$ and $\beta$ is a scaling. Let $l_n=z^{n+1}\partial_z$.

Case (a). We write $\alpha(z)=\exp(\sum_{n\geq 1} a_nl_n)(z)=\exp(X)(z)$ and $\beta(z)=\exp(\sum_{n\geq 1} b_nl_n)(z)=\exp(Y)(z)$. By the Campbell-Hausdorff theorem \cite[Sec. V.5]{Jac}, $\alpha\circ\beta=\exp(Z)$ where
\begin{align*}
Z=&X+Y+\frac 12[X,Y]+\frac 1{12}\big([X,[X,Y]]+[Y,[Y,X]]\big)-\frac 1{24}[X,[Y,[X,Y]]]\\
&+H_5+H_6+\cdots
\end{align*}
where each $H_n$ is a finite sum of $n-1$ iterated brackets of $X$ and $Y$, and hence an infinite linear combination of $l_n,l_{n+1},\dots$. So $H_n$ increases the powers of $z$ by at least $n$. From this we see that $Z$ is also of the form $\sum_{n\geq 1}c_nl_n$ for some $c_1,c_2,\dots\in\Cbb$.

The representation $l_n\mapsto \pi(l_n)=L_n$  is a representation of the Lie subalgebra $\Span_\Cbb\{l_1,l_2,\dots\}$ of the Witt algebra. (There is no central term!) Write $\pi(X)=\sum_{n\geq 1}a_nL_n$ and $\pi(Y),\pi(Z),\pi(H_n)$ in a similar way. Note that each $\pi(H_n)=\blt L_n+\blt L_{n+1}+\cdots$ lowers the $\wtd L_0$-weights by at least $n$. So $\sum_{n\geq 1}\pi(H_n)$ is well defined.  By Campbell-Hausdorff theorem (applied to $\pi(X)$ and $\pi(Y)$), we have
\begin{align*}
\mc U(\alpha)\mc U(\beta)=\exp(\pi(X))\exp(\pi(Y))=\exp\Big(\sum_{n\geq 1}\pi(H_n)\Big)	=\exp(\pi(Z))=\mc U(\alpha\circ\beta).
\end{align*}

Case (b). Write $\alpha(z)=\exp(\sum_{n\geq 1} a_nl_n)(z)$ and $\beta(z)=\lambda z$ where $\lambda\neq0$. One checks easily that
\begin{align*}
\alpha\circ\beta(z)=\lambda\cdot \exp\Big(\sum_{n\geq 1}a_n\lambda^nl_n\Big)(z).	
\end{align*}
Similar to the argument in Exercise \ref{lb99},  $[\wtd L_0,L_n]=-L_n$ implies
\begin{align*}
\exp\Big(\sum_{n\geq 1} a_nL_n\Big)\lambda^{\wtd L_0}=\lambda^{\wtd L_0}\exp\Big(\sum_{n\geq 1} a_n\lambda^nL_n\Big)	
\end{align*}
This finishes the proof.
\end{proof}

\subsection{}

Our goal is to find the covariance formula for $Y_\Wbb$ under the change of local coordinate of $0\in\Cbb$ from the standard one $\zeta$ to any $\alpha\in\Gbb$ defined on $\Dbb_r$. Choose $z\in\Dbb_r^\times$, and consider
\begin{align}
\fk A=(\Pbb^1;0,\infty;\alpha^{-1},1/\zeta),\qquad \fk P=(\Pbb^1;0,z,\infty;\zeta,\zeta-z,1/\zeta).	
\end{align}
where $\alpha^{-1}$ is the inverse function of $\alpha$, not to be confused with $1/\alpha$.

We associate $\Wbb,\Wbb',\Wbb,\Vbb,\Wbb'$ to the five marked points in the order listed above. By the change of coordinates formula in Sec. \ref{lb41}, the standard conformal blocks associated to these two are
\begin{align}
\bk{w',\mc U(\alpha)w},\qquad \bk{w',Y_\Wbb(v,z)w}.	\label{eq155}
\end{align}
We sew $\fk A$ and $\fk P$ along $0\in\fk A$ and $\infty\in\fk P$. We follow Rem. \ref{lb34} to change the $\alpha^{-1}$ of $\fk A$ to $\alpha^{-1}/r$ and the $1/\zeta$ of $\fk P$ to $r/\zeta$. Replace $r$ by a slightly smaller number $>|z|$. Then the range of $\alpha^{-1}/r$ contains $\Dbb_1^\cl$ (which is pulled back to $\alpha(\Dbb_r^\cl)$ in $\fk A$), and the pullback of the unit  disk under $r/\zeta$ is $\Pbb^1\setminus\Dbb_r$, which is disjoint from $z$ and $0$. So Assumption \ref{lb32} is satisfied.

This sewing identifies the following  parts of $\fk A,\fk P$ respectively
\begin{gather*}
A_1=\{\gamma:0<|\alpha^{-1}(\gamma)|<r\}\\
A_2=\{\gamma:1/r<|1/\gamma|<+\infty\}=\{\gamma:0<|\gamma|<r\}	
\end{gather*}
(cf. \eqref{eq154}) via the rule $\alpha^{-1}(\gamma_1)\cdot 1/\gamma_2=1$, or more precisely,
\begin{align}
\gamma_1\in A_1 \text{ is glued to }\gamma_2\in A_2\qquad\Longleftrightarrow\qquad \gamma_1=\alpha(\gamma_2). 	
\end{align}
The point $0$ of $\fk A$ and the part $\{\gamma:|1/\gamma|\leq 1/r\}=\{\gamma:r\leq |\gamma|\leq+\infty\}$ of $\fk P$ are discarded. We thus have an isomorphism
\begin{align}
\fk A\#\fk P\qquad\xlongrightarrow{\simeq}\qquad\fk X=\big(\Pbb^1;0,\alpha(z),\infty;\alpha^{-1},\alpha^{-1}-z,1/\zeta\big)	
\end{align}
where any $\gamma_1\in\Pbb^1\setminus\{0\}$ of $\fk A$ is identified with $\gamma_1\in\fk X$, and any $\gamma_2\in\Dbb_r$ of $\fk P$ is identified with $\alpha(\gamma_2)$ of $\fk P$.

\subsection{}

On the one hand, the standard conformal block for $\fk A\#\fk P$ is the contraction of the two in \eqref{eq154}, which is
\begin{align}
\bk{w',\mc U(\alpha)Y_\Wbb(v,z)w}.\label{eq156}	
\end{align}
On the other hand, since $\bk{w',Y_\Wbb(v,\alpha(z))w}$ is the standard conformal block for $(\Pbb^1;0,\alpha(z),\infty;\zeta,\zeta-\alpha(z),1/\zeta)$, by the change or coordinate formula in Sec. \ref{lb41}, the standard conformal block  of $\fk P$ should be
\begin{align}
\bigbk{w',Y_\Wbb\big(\mc U(\varrho(\alpha|\id)_z)v,\alpha(z)\big)\mc U(\alpha)w}	\label{eq157}
\end{align}
where $\varrho(\alpha|\id)_z\in\Gbb$ is the change from $\alpha^{-1}-z$ to $\zeta-\alpha(z)$, namely, \index{zz@$\varrho(\alpha\lvert\id),\varrho(\eta\lvert\mu)$}
\begin{align}
\varrho(\alpha|\id)_z(t)=\alpha(z+t)-\alpha(z).\label{eq161}
\end{align}
(The meaning of the notation $\varrho(\alpha|\id)$ will be explained in \eqref{eq177}.) So \eqref{eq156} and \eqref{eq157} should be equal. That this result is a rigorous mathematical theorem is due to Huang. 


\begin{thm}[\cite{Hua97}]\label{lb93}
Let $\Wbb$ be an admissible $\Vbb$-module. Then for each $w\in\Wbb,w'\in\Wbb',v\in\Vbb$ and $\alpha\in\Gbb$, the following equation holds in $\Cbb((z))$
\begin{align}
\bk{w',\mc U(\alpha)Y_\Wbb(v,z)w}=\bigbk{w',Y_\Wbb\big(\mc U(\varrho(\alpha|\id)_z)v,\alpha(z)\big)\mc U(\alpha)w}.\label{eq159}
\end{align}
Equivalently, in $\Cbb((z))$,
\begin{align}
\bk{w',\mc U(\alpha)Y_\Wbb(v,z)\mc U(\alpha)^{-1}w}=\bigbk{w',Y_\Wbb\big(\mc U(\varrho(\alpha|\id)_z)v,\alpha(z)\big)w}.\label{eq160}
\end{align}
\end{thm}

\subsection{}

We explain the meanings of  both sides of \eqref{eq159}; \eqref{eq160} is understood in the similar way. 

The meaning of the LHS of \eqref{eq159} is clear. Suppose $\alpha\in\scr O(\Dbb_r)$. Then
$\bk{w',Y(v,\alpha(z))w}$ is  a Laurent polynomial of $\alpha(z)$, which is clearly holomorphic on $\Dbb_r^\times$ with finite poles at $0$. $z\mapsto \varrho(\alpha|\id)_z$ is a holomorphic family of transformations. So $\mc U(\varrho(\alpha|\id))v$ is in $\Vbb\otimes\scr O(\Dbb_r)$ by \eqref{eq170}. By linearity, the holomorphicity of $\bk{w',Y(v,\alpha(z))w}\in \Cbb[z^{\pm1}]$ implies that the RHS of \eqref{eq159} is also holomrophic on  $\Dbb_r^\times$ with finite poles at $0$. So, \emph{the RHS of \eqref{eq159} is understood as an element of $\Cbb((z))$ by taking Laurent series expansion of the holomorphic function}.

More generally, let $\alpha:X\rightarrow\Gbb$ be a holomorphic family of transformations over a Riemann surface $X$. If $\alpha$ is holomorphic on $X\times\Dbb_r$, then the RHS of \eqref{eq159} is naturally a holomorphic function on $X\times \Dbb_r^\times$ with finite poles at $z=0$. Thus, as an element of $\scr O(X)((z))$ obtained by taking Laurent series expansion, it converges a.l.u. on $X\times \Dbb_r^{\times}$ by Lemma \ref{lb76}. So is the LHS. We conclude:

\begin{thm}\label{lb95}
Suppose $\alpha:X\rightarrow\Gbb$ is a holomorphic family of transformations that is holomorphic on $X\times\Dbb_r$. Then both sides of \eqref{eq159} and \eqref{eq160} are elements of $\scr O(X)((z))$ and converge a.l.u. on $X\times\Dbb_r^\times$ to the same function. Moreover, the following series 
\begin{align}
\sum_{n\in\Nbb}\bk{w',\mc U(\alpha)P_nY_\Wbb(v,z)w}	
\end{align}
of elements of $\scr O(X\times\Cbb^{\times})$ converges a.l.u. on $X\times\Dbb_r^\times$ to \eqref{eq159}.
\end{thm}

\begin{proof}
The last statement is due to Lemma \ref{lb65} when $v,w,w'$ are homogeneous.
\end{proof}

\subsection{$\star$}

We present the proof of \eqref{eq160} below. The idea is the same as in the proofs of scale and translation covariance. Also, it is not hard to see that the following proof works for all $\alpha$ in $\mc G$.

\begin{proof}[Proof of Thm. \ref{lb93}]
Step 1. Let us first assume $\alpha\in\Gbb_+$ so that $\alpha'(0)=1$. Choose $c_1,c_2,\dots\in\Cbb$ such that
\begin{align}
	\alpha(z)=\exp\Big(\sum_{n\geq 1} c_nz^{n+1}\partial_z\Big)(z),\label{eq172}	
\end{align}
and set
\begin{align*}
	\alpha_\tau(z)=\exp\Big(\sum_{n\geq 1}\tau c_nz^{n+1}\partial_z\Big)(z)	\qquad\in\Cbb[\tau][[z]]
\end{align*}
so that $\alpha_1(z)=\alpha(z)$. Note that we can write
\begin{align}
	\alpha_\tau(z)=z+\sum_{n\geq 2}p_n(\tau)z^n	\label{eq164}
\end{align}
where $p_n(\tau)\in\Cbb[\tau]$. So we can view $\alpha_\tau(z)$ as a $\Cbb[[z]]$-valued holomorphic function. The limit $\partial_\tau\alpha_\tau(z)=\lim_{\gamma\rightarrow\tau}\frac{\alpha_\gamma(z)-\alpha_\tau(z)}{\gamma-\tau}$ makes sense analytically as in Rem. \ref{lb96}.

\eqref{eq164} shows that
\begin{align*}
	1/\alpha_\tau(z)\in z^{-1}\Cbb[\tau][[z]].
\end{align*}
Therefore, $\bk{w',Y_\Wbb(v,\alpha_\tau(z))w}$, which is a Laurent polynomial of $\alpha_\tau(z)$, must also be in $\Cbb[\tau]((z))$. It is not hard to verify that $\partial_\tau  \alpha_\tau(z)|_{\tau=0}=\sum c_nz^{n+1}$ and that  $\alpha_\gamma\circ\alpha_\tau(z)=\alpha_{\gamma+\tau}(z)$ for each $\gamma,\tau\in\Cbb$. By taking derivative in the sense of Rem. \ref{lb96}, we obtain
\begin{align*}
	\partial_\tau\alpha_\tau(z)=\sum_{n\geq 1}c_n\alpha_\tau(z)^{n+1}.	
\end{align*}
From this and the translation property, we obtain in $\Cbb[\tau]((z))$ that
\begin{align}
	\partial_\tau\bk{w',Y_\Wbb(v,\alpha_\tau(z))w}=\sum_{n\geq 1}c_n\alpha_\tau(z)^{n+1}\cdot \bk{w',Y_\Wbb(L_{-1}v,\alpha_\tau(z))w}\label{eq173}
\end{align}
as $\Cbb((z))$-valued holomorphic functions of $\tau\in\Cbb$.  \\

Step 2. Let us calculate $\partial_\tau \mc U(\varrho(\alpha_\tau|\id)_z)v$. Note that any formal power series composed with $z+t$ is an element of $\Cbb[[z,t]]$. So, even though $\alpha_\tau$ is a formal coordinate, $\alpha_\tau(z+t)$ still makes sense, and we can use \eqref{eq161} again to define $\varrho(\alpha_\tau|\id)_z$. Namely, in view of \eqref{eq164},
\begin{align*}
\varrho(\alpha_\tau|\id)_z(t)=t+\sum_{n\geq 2}p_n(\tau)\sum_{j=1}^n{n\choose j}z^{n-j}t^j\qquad \in\Cbb[\tau][[z]][[t]].	
\end{align*}
Similarly, 
\begin{align}\label{eq171}
	\begin{aligned}
&\varrho(\alpha_\zeta|\id)_{\alpha_\tau(z)}(t):=\alpha_\zeta(\alpha_\tau(z)+t)-\alpha_\zeta(\alpha_\tau(z))	\\
=&t+\sum_{n\geq 2}p_n(\zeta)\sum_{j=1}^n{n\choose j}\alpha_\tau(z)^{n-j}t^j
	\end{aligned}
\end{align}
makes sense as an element of $\Cbb[\zeta,\tau][[z]][[t]]$. Using $\alpha_\zeta(\alpha_\tau(z))=\alpha_{\zeta+\tau}(z)$, one checks easily that
\begin{align*}
\varrho(\alpha_\zeta|\id)_{\alpha_\tau(z)}\circ \varrho(\alpha_\tau|\id)_z(t)=\varrho(\alpha_{\zeta+\tau}|\id)_z(t).	
\end{align*}
Apply Thm. \ref{lb100} to the above relation and take $\partial_\zeta$ at $\zeta=0$, we obtain
\begin{align}
\partial_\tau \mc U(\varrho(\alpha_\tau|\id)_z)v=\partial_\zeta\mc U(\varrho(\alpha_\zeta|\id)_{\alpha_\tau(z)})\big|_{\zeta=0}\cdot \mc U(\varrho(\alpha_\tau|\id_z))v.	
\end{align}

Clearly $\varrho(\alpha_0|\id)_{\alpha_\tau(z)}(t)=t$. By going through the proof of Prop. \ref{lb98}, we see that Prop. \ref{lb98} also applies to the present situation: acting on $\Vbb$ we have
\begin{align}
\partial_\zeta\mc U(\varrho(\alpha_\zeta|\id)_{\alpha_\tau(z)})\Big|_{\zeta=0}=\sum_{k\geq 1}\frac 1{k!}\Big(\partial_\zeta\varrho(\alpha_\zeta|\id)_{\alpha_\tau(z)}^{(k)}(0)\Big|_{\zeta=0} \Big)L_{k-1}.	
\end{align}
By \eqref{eq171}, it is clear that
\begin{align*}
\partial_\zeta\varrho(\alpha_\zeta|\id)_{\alpha_\tau(z)}^{(k)}(0)=\partial_\zeta\alpha_\zeta^{(k)}(\alpha_\tau(z)).	
\end{align*}
Since, by \eqref{eq172}, we have $\partial_\zeta\alpha_\zeta(z)\big|_{\zeta=0}=\sum_{n\ge 1}c_nz^{n+1}$ and hence
\begin{align*}
\frac 1{k!}\partial_\zeta\alpha_\zeta^{(k)}(z)\big|_{\zeta=0}=\sum_{n\geq 1}{n+1\choose k}c_nz^{n-k+1},	
\end{align*}
we obtain
\begin{align}
&\partial_\zeta\mc U(\varrho(\alpha_\zeta|\id)_{\alpha_\tau(z)})\Big|_{\zeta=0}=\sum_{k,n\geq 1}{n+1\choose k}c_n\alpha_\tau(z)^{n-k+1}L_{k-1}\nonumber\\
=&\sum_{n\geq 1}c_n\sum_{l\geq 0}{n+1\choose l+1}\alpha_\tau(z)^{n-l}L_l.
\end{align}
To sum up, we get
\begin{align}
&\partial_\tau\bigbk{w',Y_\Wbb(\mc U(\varrho(\alpha_\tau|\id)_z)v,z)w}\nonumber\\
=&	\sum_{n\geq 1}c_n\sum_{l\geq 0}{n+1\choose l+1}\alpha_\tau(z)^{n-l}\bigbk{w',Y_\Wbb(L_l\mc U(\varrho(\alpha_\tau|\id)_z)v,z)w}.
\end{align}
Combining this relation with \eqref{eq173} and \eqref{eq68} yields
\begin{align}
\partial_\tau\bigbk{w',Y_\Wbb(\mc U(\varrho(\alpha_\tau|\id)_z)v,\alpha_\tau(z))w}=\sum_{n\geq 1}c_n	\bigbk{w',[L_n,Y_\Wbb(\mc U(\varrho(\alpha_\tau|\id)_z)v,\alpha_\tau(z))]w}.
\end{align}
(We leave it to the readers to check that this infinite sum is well-defined.) A similar calculation shows
\begin{align}
\partial_\tau\bigbk{w',\mc U(\alpha_\tau)Y_\Wbb(v,z)\mc U(\alpha_\tau)^{-1}w}=\sum_{n\geq 1}c_n	\bigbk{w',[L_n,\mc U(\alpha_\tau)Y_\Wbb(v,z)\mc U(\alpha_\tau)^{-1}]w}.
\end{align}
Thus, by Lemma \ref{lb21}, we get \eqref{eq160} for all $\alpha\in\Gbb_+$. We have also proved \eqref{eq160} when $\alpha$ is a scaling. The general case follows from the combination of these two cases. We leave the details the readers. 
\end{proof}

\section{Definitions of conformal blocks and sheaves of VOAs}\label{lb155}

\subsection{}

The goal of this section is to give two equivalent definitions of conformal blocks, both due to \cite{FB04}.

\begin{ass}\label{lb101}
Starting from this section, we assume $\dim\Vbb(n)<+\infty$ for each $n$, and write $Y_\Wbb$ as $Y$ when possible. By ``$\Vbb$-modules", we mean admissible $\Vbb$-modules.
\end{ass}

Let \index{00@$N$-pointed compact Riemann surfaces}
\begin{align}
\fk X=(C;x_1,\dots,x_N;\eta_1,\dots,\eta_N)	\label{eq181}
\end{align}
be an $N$-pointed compact Riemann surface with local coordinates.  Assume that $\eta_j$ is holomorphic (and injective) on an neighborhood $U_j$ of $x_j$. Assume that $x_j\notin U_i$ if $i\neq j$. 

\begin{ass}\label{lb102}
Unless otherwise stated, by an $N$-pointed compact Riemann surface, we assume that each connected component contains at least one marked point.
\end{ass}

Recall that in Segal's picture, we have decomposition $\mc H^\fin=\bigoplus \Wbb_i\otimes\wht\Wbb_i$, and the correlation function decomposes to $\Vbb$- and $\wht\Vbb$-conformal blocks $T_{\fk X}=\sum_{i_1,\dots,i_N\in\fk I}\upphi_{\fk X,i_\blt}\otimes\uppsi_{\ovl{\fk X},i_\blt}$ as in \eqref{eq174}, where each $\upphi_{\fk X,i_\blt}$ is a linear functional on $\Wbb_{i_\blt}:=\Wbb_{i_1}\otimes\cdots\otimes \Wbb_{i_N}$. 

In the following discussions, we fix a vector $\wht w_i$ in each $\wht\Wbb_i$, and identify each $\Wbb_i$ with $\Wbb_i\otimes \wht w_i$ so that we can restrict the correlation function $T_{\fk X}$ onto $\Wbb_{i_\blt}$ to get a conformal block. Thus, we shall not distinguish between conformal blocks and (restrictions of) correlation functions.


\subsection{}\label{lb160}
We write $\Wbb_{i_k}=\Wbb_k$ and $\upphi_{\fk X_{i_\blt}}=\upphi$ for simplicity. So the $\Vbb$-modules $\Wbb_1,\dots,\Wbb_N$ are associated to $x_1,\dots,x_N$. Recall the notation $\Wbb_\blt=\Wbb_1\otimes\cdots\otimes\Wbb_N$.

We add a point $x$ to $\fk X$ different from $x_1,\dots,x_N$. Then we get a new $(N+1)$-pointed compact Riemann surface $\wr\fk X_x$. We insert vectors of $\Vbb\simeq \Vbb\otimes\id$ to $x$. Then we get a new conformal block $\wr\upphi_x:\Vbb\otimes\Wbb_\blt\rightarrow\Cbb$, which is the restriction of the correlation function $T_{\wr\fk X_x}$ to $\Vbb\otimes\Wbb_\blt$. $\wr\upphi_x$ has the following two features. (Let $\zeta$ be the standard coordinate of $\Cbb$.)

First, assume  $\eta_j(U_j)\supset\Dbb_{r_j}$.  Let $x\in\eta_j^{-1}(\Dbb_{r_j})$. We assign local coordinate $\eta_j-\eta_j(x)$ to $x$ so that every marked point of $\wr\fk X_x$ has an associated local coordinate. Let 
\begin{align}
	\fk P_{\eta_j(x)}=(\Pbb^1;0,\eta_j(x),\infty;\zeta,\zeta-\eta_j(x),1/\zeta).	
\end{align}
Consider the sewing $\fk P_{\eta_j(x)}\#\fk X$ along $\infty\in\fk P_{\eta_j(x)}$ and $x_j\in\fk X$.  We have an equivalence
\begin{align}
	\fk P_{\eta_j(x)}\#\fk X \simeq \wr\fk X_x	
\end{align}
where the parts $\Pbb^1\setminus\Dbb_{r_j}$ and $x_j$ of $\fk P_{\eta_j(x)}$ and $\fk X$ are discarded; any $\gamma\in\Dbb_{r_j}$ is equivalent to $\eta_j^{-1}(\gamma)$ of $\wr\fk X_x$, and is glued with $\eta_j^{-1}(\gamma)$ of $\fk X$ when $\gamma\in\Dbb_{r_j}^\times$; in particular, the marked points $0,\eta_j(x)$ of $\fk P_{\eta_j(x)}$ (which are not discarded) are identified respectively with $x_j,x$ of $\wr\fk X_x$. 
\begin{align}
	\begin{aligned}
		&\vcenter{\hbox{{
					\includegraphics[height=2.3cm]{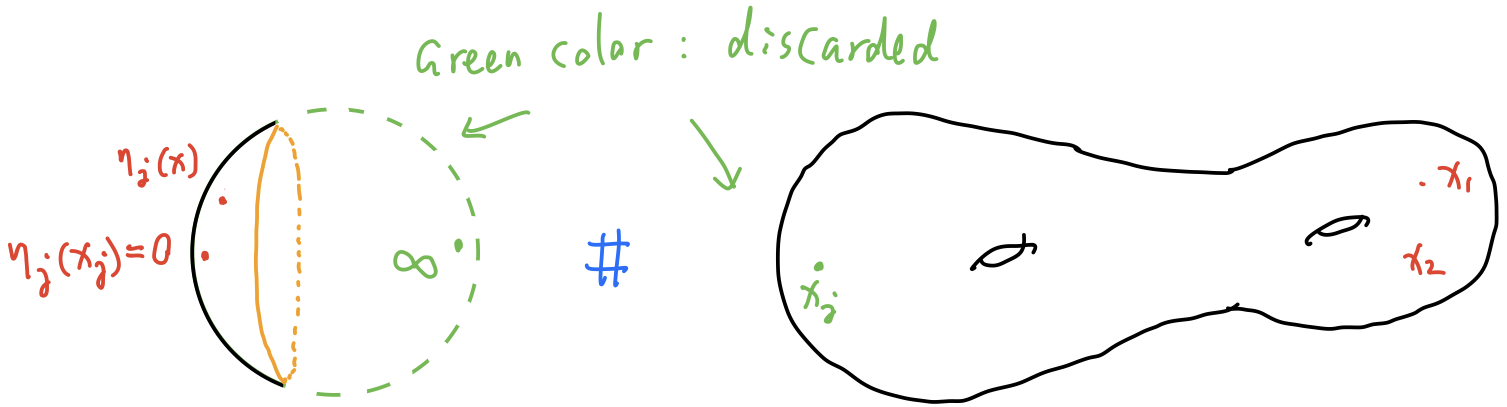}}}}\\[1ex]
		\simeq&\qquad\vcenter{\hbox{{
					\includegraphics[height=1.9cm]{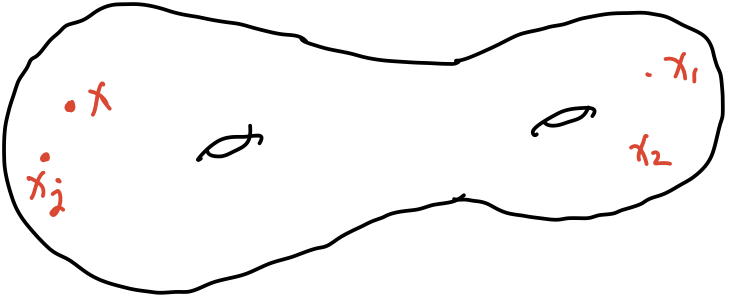}}}}
	\end{aligned}	
\end{align}
Therefore, by the sewing-contraction correspondence, the conformal block $\wr\upphi_x$ associated to $\wr\fk X_x$ (where the local coordinate at $x$ is $\eta_j-\eta_j(x)$) is
\begin{align}
	\wr\upphi_x(v\otimes w_\blt)=\upphi\big(w_1\otimes\cdots \otimes Y(v,\eta_j(x))w_j\otimes\cdots \otimes w_N\big)	\label{eq175}
\end{align}
where the RHS is short for the following two equivalent series (cf. Lemma \ref{lb65}) and is converging a.l.u. to the LHS of \eqref{eq175}:
\begin{align}\label{eq191}
\begin{aligned}
&\text{RHS of }\eqref{eq175}=\sum_{n\in\Zbb}\upphi\big(w_1\otimes\cdots \otimes Y(v)_nw_j\otimes\cdots \otimes w_N\big)z^{-n-1}\big|_{z=\eta_j(x)}\\
=&\sum_{n\in\Nbb}\upphi\big(w_1\otimes\cdots \otimes P_nY(v,\eta_j(x))w_j\otimes\cdots \otimes w_N\big).
\end{aligned}	
\end{align}

\subsection{}

The second feature is: according to \eqref{eq176}, for any $x$ on $C$ not necessarily close to any of $x_\blt$, $\wr\upphi_x(v\otimes w_\blt)$ is holomorphic with respect to the motion of $x$. A downside of this description is that it depends on a particular choice of local coordinates at $x$: if in one local coordinate $v$ is a constant, then in another one $v$ will vary. So let us give an coordinate-independent descrption: 

Besides the translation of $x$, we also allow $v$ to vary holomorphically with respect to $x$. Namely, let $U\subset C$ be open, choose a sufficiently large $n\in\Nbb$, and assume $v$ is a $\Vbb^{\leq n}$-valued holomorphic function on $U$. (Recall that $\Vbb^{\leq n}$ is finite dimensional by Convention \ref{lb101}.) Namely,
\begin{align}
v\in\Vbb^{\leq n}\otimes_\Cbb \scr O(U).	
\end{align}
Assume that there is a \textbf{univalent} \index{00@Univalent functions} (i.e., holomorphic+injective) function $\mu:U\rightarrow\Cbb$.\footnote{Indeed, one only needs to assume that $d\mu$ is nowhere zero on $U$. Then $\mu$ must be locally univalent, which is sufficient for applications.} (It is helpful to think of $\mu$ vanishing at some point $y\in U$, i.e., $\mu$ is a local coordinate at $y$. But technically this is not necessary.) Then at each $x\in U$ there is a natural local coordinate $\mu-\mu(x)$. If we let $\wr\upphi_x$ act on abstract vectors instead of concrete ones, then for each $v$ as above (so that each $\mc U(\mu-\mu(x))^{-1}v(x)$ is an abstract vector)
\begin{align}
x\in U\mapsto \wr\upphi_x\big(\mc U(\mu-\mu(x))^{-1}v(x)\otimes w_\blt\big)	\label{eq178}
\end{align}
is a holomorphic function. The choice of local coordinate $\mu-\mu(x)$ is in accordance with $(\zeta-z)/r$ in \eqref{eq224} if we assume $r=1$ and identify $\mu$ with the standard coordinate $\zeta$ of $\Cbb$.

\subsection{}

We explain why this description is independent of the choice of $\mu$. Let $\eta\in\scr O(U)$ be also univalent. Let $\varrho(\eta|\mu)_x\in\Gbb$ be the change of coordinate from $\mu-\mu(x)$ to $\eta-\eta(x)$. \index{zz@$\varrho(\alpha\lvert\id),\varrho(\eta\lvert\mu)$} Namely
\begin{align}
\varrho(\eta|\mu)_x\big(\mu(y)-\mu(x)\big)=\eta(y)-\eta(x)	\label{eq177}
\end{align}
for any $y\in C$ close to $x$. Equivalently,
\begin{align}
\varrho(\eta|\mu)_x(z)=\eta\circ\mu^{-1}(z+\mu(x))-\eta(x),\label{eq205}
\end{align}
from which we see that $\varrho(\eta|\mu):U\rightarrow\Gbb,x\mapsto\varrho(\eta|\mu)_x$ is a holomorphic family of transformations. Thus, by \eqref{eq170}, $\mc U(\varrho(\eta|\mu))\big|_{\Vbb^{\leq n}}$ is in $\End\Vbb^{\leq n}\otimes\scr O(U)$. Thus, by $\scr O(U)$-linearity, $\mc U(\varrho(\eta|\mu))$ sends each section of $\Vbb^{\leq n}\otimes\scr O(U)$ to $\Vbb^{\leq n}\otimes\scr O(U)$ such that its valued at each $x$ is an automorphism of $\Vbb^{\leq n}$. 

This property can be summarized in the following way: Let $\scr O_U$ be the trivial holomorphic line (i.e. $1$-dimensional vector bundle) over $U$. So $\Vbb^{\leq n}\otimes_\Cbb\scr O_U$ is the trivial (holomorphic) vector bundle\footnote{In our notes, all vector bundles are holomorphic with finite ranks unless otherwise stated.} with fiber $\Vbb^{\leq n}$. Then  we have an automorphism of vector bundle (equivalently, an automorphism of $\scr O_U$-module)
\begin{align*}
\mc U(\varrho(\eta|\mu)):\Vbb^{\leq n}\otimes_\Cbb \scr O_U\xrightarrow{\simeq} \Vbb^{\leq n}\otimes_\Cbb \scr O_U.	
\end{align*}

By Subsec. \ref{lb15},
\begin{align*}
\wr\upphi_x\big(\mc U(\mu-\mu(x))^{-1}v(x)\otimes w_\blt\big)=	\wr\upphi_x\big(\mc U(\eta-\eta(x))^{-1}u(x)\otimes w_\blt\big)	
\end{align*}
where $u(x)=\mc U\mc(\varrho(\eta|\mu)_x)v(x)$. Thus, the function $v$ on $U$ is holomorphic iff $u$ is so. This implies that the holomorphicity of \eqref{eq178} is independent of the choice of $\mu$.

\begin{eg}
Let $\zeta$ be the standard coordinate of $\Cbb^\times$. Then for each $\gamma\in\Cbb^\times$,
\begin{align}
\varrho\big(1/\zeta\big|\zeta\big)_\gamma=\varrho\big(\zeta\big|1/\zeta\big)_{1/\gamma}=\vartheta_\gamma
\end{align}
where $\vartheta_\gamma(z)=\frac 1{\gamma+z}-\frac 1\gamma$ (cf. \eqref{eq195}). Therefore, by \eqref{eq196},
\begin{align}
\mc U\big(\varrho\big(1/\zeta\big|\zeta\big)_\gamma\big)=\mc U\big(\varrho\big(\zeta\big|1/\zeta\big)_{1/\gamma}\big)=e^{\gamma L_1}(-\gamma^{-2})^{L_0}.\label{eq197}
\end{align}
\end{eg}

\subsection{}\label{lb119}

The combination of these two features gives the definition of conformal blocks. To simplify the definition and make it more precise, let us introduce some new notions.

We define a vector bundle $\scr V^{\leq n}_C$ \index{V@$\scr V_C^{\leq n},\scr V_C$} over $C$ whose fibers are equivalent to $\Vbb^{\leq n}$ as follows. Recall that holomorphic vector bundles can be constructed once we have holomorphic transation functions. By \eqref{eq117}, for univalent $\eta_i\in\scr O(U)$, $i=1,2,3$, we have
\begin{align}
\varrho(\eta_1|\eta_2)_x\circ\varrho(\eta_2|\eta_3)_x=\varrho(\eta_1|\eta_3)_x	
\end{align}
and hence the cocycle condition
\begin{align}
\mc U(\varrho(\eta_1|\eta_2))\mc U(\varrho(\eta_2|\eta_3))=\mc U(\varrho(\eta_1|\eta_3))	
\end{align}
due to Thm. \ref{lb100}. Thus, we have a unique (up to equivalence) vector bundle $\scr V^{\leq n}_C$ whose transition functions are of the form $\mc U(\varrho(\eta|\mu))$. More precisely, for any open $U\subset C$ with a univalent $\eta\in\scr O(U)$ is associated with a trivialization (i.e., an equivalence of vector bundles/$\scr O_U$-modules) \index{U@$\mc U_\varrho(\eta)$}
\begin{align}
\mc U_\varrho(\eta):\scr V^{\leq n}_C\big|_U\xrightarrow{\simeq}	\Vbb^{\leq n}\otimes_\Cbb\scr O_U\label{eq179}
\end{align}
compatible with the restriction of $\eta$ to open subsets (i.e., if $V\subset U$ is open then $\mc U_\varrho(\eta|_V)=\mc U_\varrho(\eta)|_V$) such that if $\mu\in\scr O(U)$ is also univalent, then
\begin{align}
\mc U_\varrho(\eta)\mc U_\varrho(\mu)^{-1}=\mc U(\varrho(\eta|\mu)):\Vbb^{\leq n}\otimes_\Cbb\scr O_U\xrightarrow{\simeq}	\Vbb^{\leq n}\otimes_\Cbb\scr O_U.\label{eq180}
\end{align}

\begin{rem}\label{lb106}
Intuitively, the fiber of $\scr V^{\leq n}_C$ at each $x\in C$ is the vector space $\scr W(\Vbb^{\leq n})$ of abstract VOA vectors whose energies are $\leq n$. The trivialization $\mc U_\varrho(\eta)$ sends each fiber $\scr V^{\leq n}_C|_x$ at $x$  to $\Vbb^{\leq n}$ via the isomorphism $\mc U(\eta-\eta(x))$, and sends each abstract VOA vector to its $(\eta-\eta(x))$-coordinate representation.  If $v\in\Vbb^{\leq n}\otimes\scr O(U)$, then the map $x\mapsto \mc U(\eta-\eta(x))^{-1}v(x)$ is just the section $\mc U_\varrho(\eta)^{-1}v$ of $\scr V^{\leq n}_C$ on $U$, and any section on $U$ is of this form. $\scr V^{\leq n}_C(U)$, the space of all sections of $\scr V^{\leq n}_C$ on $U$, \emph{is the space of all VOA vectors with energies $\leq n$ varying and moving holomorphically on $U$}. 
\end{rem}

\begin{rem}
The vacuum vector $\id$ is fixed by any change of coordinate operator $\mc U(\varrho(\eta|\mu))$ since it is killed by $L_{\geq0}$. So we let $\id$ denote also the element of $\scr V^{\leq n}_C(C)$ whose trivialization under any local univalent map $\eta$ is the vacuum vector $\id$. We call $\id$ the \textbf{vacuum section}. \index{1@$\id$, the vacuum section}
\end{rem}

\subsection{}

Now, the property that  $\wr\upphi$ is holomorphic with respect to the motion and variation of the inserted VOA vectors can be expressed in the following form:
\begin{enumerate}
\item For each open subset $U$ of $C\setminus\{x_\blt\}$,
\begin{align}
\wr\upphi(\cdot\otimes w_\blt):\scr V^{\leq n}_C(U)\rightarrow \scr O(U),\qquad \vbf\mapsto \wr\upphi(\vbf\otimes w_\blt)	
\end{align}
is an $\scr O(U)$-module (homo)morphism. (The reason that it intertwines the actions of $\scr O(U)$ is clear.)
\item $\wr\upphi(\cdot\otimes w_\blt)$ is compatible with the restriction to open subsets. Namely, if $V\subset U$ is open, then $\wr\upphi(\vbf|_V\otimes w_\blt)=\wr\upphi(\vbf\otimes w_\blt)|_V$.
\end{enumerate}

The above two points can be summarized using the sheaf theoretic language: $\wr\upphi(\cdot\otimes w_\blt)$ is a morphism of $\scr O_{C\setminus\{x_\blt\}}$-modules $\scr V^{\leq n}_{C\setminus\{x_\blt\}}\rightarrow\scr O_{C\setminus\{x_\blt\}}$. Equivalently,
\begin{align*}
\wr\upphi(\cdot\otimes w_\blt)\in H^0\big(C\setminus\{x_\blt\},(\scr V_C^{\leq n})^\vee\big).	
\end{align*}

\subsection{}

To simplify the formulation of definitions and theorems, we consider the direct limit sheaf \index{V@$\scr V_C^{\leq n},\scr V_C$}
\begin{align*}
\scr V_C=\varinjlim_{n\in\Nbb}\scr V^{\leq n}_C	
\end{align*}
whose space of sections on any connected open $U\subset C$ (or more generally, any open $U$ with finitely many connected component) is
\begin{align*}
\scr V_C(U)=\varinjlim_{n\in\Nbb}\scr V^{\leq n}_C(U).
\end{align*}
This is possible since for each $n_1\leq n_2$ we have an obvious injective $\scr O_C$-module morphism (i.e., morphism of vector bundles) $\scr V^{\leq n_1}_C\rightarrow\scr V^{\leq n_2}_C$ which under any trivialization as in \eqref{eq179} is the obvious inclusion $\Vbb^{\leq n_1}\otimes\scr O_U\hookrightarrow\Vbb^{\leq n_2}\otimes\scr O_U$. Both $\scr V_C$ and $\scr V_C^{\leq n}$ are called \textbf{sheaves of VOAs} associated to $C$ and $\Vbb$.

Equivalently, $\scr V_C$ is an infinite-rank vector bundle such that for each connected open $U\subset C$ with a univalent $\eta$, we have a trivialization
\begin{align*}
\mc U_\varrho(\eta)	:\scr V_C|_U\xrightarrow{\simeq}\Vbb\otimes\scr O_U
\end{align*}
compatible with the restriction of $\eta$ to connected open subsets, such that for any another univalent $\mu\in\scr O(U)$ we also have $\mc U_\varrho(\eta)\mc U_\varrho(\mu)^{-1}=\mc U(\varrho(\eta|\mu))$ as an automorphism of the $\scr O_U$-module $\Vbb\otimes\scr O_U$.

Thus, roughly speaking, $\scr V_C(U)$ is the set of all sections $v$ belonging to $\scr V_C^{\leq n}(U)$ for some $n\in\Nbb$.

In the rest of these notes, the readers may replace $\scr V_C$ by  $\scr V_C^{\leq n}$ for all possible $n$ if they are not comfortable with  locally free sheaves of infinite ranks.

\subsection{}

We are now ready to state the definition of conformal blocks. Recall the data $\fk X$ in \eqref{eq181} and that each $\eta_i$ is defined on $U_i\ni x_i$. Let $\Vbb$ be a VOA, and let $\Wbb_1,\dots,\Wbb_N$ be admissible $\Vbb$-modules associated respectively to the marked points $x_1,\dots,x_N$.

\begin{df}[Complex analytic version]\label{lb103}
A linear functional $\upphi:\Wbb_\blt=\Wbb_1\otimes\cdots\otimes\Wbb_N\rightarrow\Cbb$ is called a \textbf{conformal block} associated to $\fk X$ and $\Wbb_\blt$ if the following holds: For each $w_\blt\in\Wbb_\blt$, there exists a (necessarily unique) $\scr O_{C\setminus\{x_\blt\}}$-module morphism
\begin{align*}
\wr\upphi(\cdot,w_\blt):\scr V_{C\setminus\{x_\blt\}}\rightarrow\scr O_{C\setminus\{x_\blt\}}
\end{align*}
(equivalently,  $\wr\upphi(\cdot,w_\blt)\in H^0(C\setminus\{x_\blt\},\scr V_C^\vee)$) such that  for each $1\leq i\leq N$, by identifying
\begin{align}
\scr V_C|_{U_i}=\Vbb\otimes\scr O_{U_i}\qquad\text{via }\mc U_\varrho(\eta_i)\label{eq186}
\end{align}
and identifying
\begin{align}
U_i=\eta_i(U_i)\qquad\text{via }\eta_i	\label{eq187}
\end{align}
so that $\eta_i$ becomes the standard coordinate $z$, for each $v\in \scr V_C(U_i)=\Vbb\otimes\scr O(U_i)$ (restricted to $U_i\setminus\{x_i\}=\eta_i(U_i)\setminus\{0\}$), the equality
\begin{align}
\wr\upphi(v,w_\blt)_z=\upphi\big(w_1\otimes\cdots\otimes Y(v(z),z)w_i\otimes\cdots\otimes w_N\big) \label{eq182}
\end{align}
holds in $\Cbb[[z^{\pm1}]]$. \hfill\qedsymbol
\end{df}

Note that the LHS of \eqref{eq182} is an element of $\scr O(\eta_i(U_i)\setminus\{0\})$, regarded as one in $\Cbb[[z^{\pm1}]]$ by taking Laurent series expansions. The RHS is understood as
\begin{align*}
\sum_{m\in\Nbb,n\in\Zbb}\upphi(\cdots \otimes Y(v_m)_nw_i\otimes\cdots)z^{m-n-1}	
\end{align*}
if $v$ has expansion $v(z)=\sum_{m\geq 0}v_mz^m$ where each $v_m\in\Vbb$. In particular, \eqref{eq182} is in $\Cbb((z))$.

\subsection{}

Let us make some comments on this definition.

\begin{rem}
By Lemma \ref{lb76}, we see that if $\eta_i(U_i)\supset\Dbb_{r_i}$, then the formal Laurent  series of $z$ on the RHS of \eqref{eq182}, and equivalently (cf. \eqref{eq191}), the series of functions
\begin{align*}
\sum_{n\in\Nbb}\upphi\big(w_1\cdots\otimes P_nY(v(z),z)w_i\otimes\cdots\otimes w_N\big)	
\end{align*}
converge a.l.u. on $z\in\Dbb_{r_i}^\times$ to the LHS of \eqref{eq182}. This explains why  Def. \ref{lb103} is viewed as a complex analytic definition.
\end{rem}

\begin{rem}
The uniqueness of $\wr\upphi(\cdot,w_\blt)$ is due to the following reason. It suffices to restrict $\omega=\wr\upphi(\cdot,w_\blt)$ to $\scr V^{\leq n}_C$ for each $n\geq 0$. Suppose $\omega'=\wr'\upphi(\cdot,w_\blt)$ is another morphism satisfying the descriptions in Def. \ref{lb103}.  Then $\omega$ and $\omega'$ are sections of the vector bundle $(\scr V_C^{\leq n})^\vee$ over $C\setminus\{x_\blt\}$. Moreover, by \eqref{eq182}, $\omega-\omega'$ vanishes on each $U_i\setminus\{x_i\}$.     Thus, if we let $\Omega\subset C\setminus\{x_\blt\}$ be the set of all points $y$ such that $\omega-\omega'$ vanishes on a neighborhood of $y$, then by Assumption \ref{lb102}, $\Omega$ intersects each connected component of $C$. By complex analysis, $\Omega$ is both open and closed. So $\Omega=C\setminus\{x_\blt\}$.
\end{rem}

\begin{rem}\label{lb191}
By this uniqueness, we may define $\wr\upphi(\cdot,w)$ for all $w\in\Wbb_\blt$ such that $\wr\upphi(\cdot,w)$ is linear over $w$.
\end{rem}

\begin{rem}
By complex analysis, it is clear that the definition of conformal blocks is independent of the sizes and shapes of the neighborhoods $U_1,U_2,\dots$ of $x_\blt$.
\end{rem}

\begin{rem}
By $\scr O(U_i)$-linearity, to verify \eqref{eq182} for all $v\in\Vbb\otimes\scr O(U_i)$, it suffices to verify it for all constant $v\in\Vbb\simeq \Vbb\otimes \id$.
\end{rem}

\subsection{}

\begin{eg}\label{lb173}
Fix $\gamma\in\Cbb^\times$, and let $\fk P=(\Pbb^1;0,\gamma,\infty;\zeta,\zeta-\gamma,1/\zeta)$ where $\zeta$ is the standard coordinate of $\Cbb$. Let $\Wbb$ be an admissible $\Vbb$-module, and associate $\Wbb,\Vbb,\Wbb'$ to $0,\gamma,\infty$. Then the following linear functional is a conformal block, called the \textbf{conformal block associated to the vertex operation $Y_\Wbb$}.
\begin{gather}
\upomega:\Wbb\otimes\Vbb\otimes\Wbb'\rightarrow\Cbb,\qquad w_\blt=w\otimes v\otimes w'\mapsto \bk{w',Y(v,\gamma)w}
\end{gather}

\end{eg}

\begin{proof}
We construct the $\scr O_{\Cbb^\times\setminus\{\gamma\}}$-module morphism $\wr\upomega(\cdot,w_\blt):\scr V_{\Cbb^\times\setminus\{\gamma\}}\rightarrow\scr O_{\Cbb^\times\setminus\{\gamma\}}$ as follows. For every open $U\subset \Cbb^\times\setminus\{\gamma\}$, set
\begin{gather*}
\wr\upomega(\cdot,w_\blt):\scr V_{\Cbb^\times\setminus\{\gamma\}}(U)\rightarrow\scr O(U),\\
\mc U_\varrho(\zeta)^{-1}u\mapsto \bk{w',Y(u(z),z)Y(v,\gamma)w}
\end{gather*}
where $u\in\Vbb\otimes\scr O(U)$, and we have used the convention in Def. \ref{lb75} so that the above termed is defined and holomorphic when $z\neq 0,\gamma,\infty$ and $u$ is holomorphic.

Assume without loss of generality that $u$ is a constant section, i.e. $u\in\Vbb$. By the complex analytic Jacobi identity for $Y_\Wbb$,  \eqref{eq182} holds for $\upomega$ when $\gamma$ is close to $0$ or $\gamma$. When $z$ is close to $\infty$,  $\wr\upomega(\cdot,w_\blt)$ sends $\mc U_\varrho(\zeta)^{-1}u$ to $\bk{w',Y(u,z)Y(v,\gamma)w}$. Thus, it sends
\begin{align*}
\mc U_\varrho(1/\zeta)^{-1}u=\mc U_\varrho(\zeta)^{-1}\mc U(\varrho(\zeta|1/\zeta))u
\end{align*}
to
\begin{align*}
\bigbk{w',Y\big(\mc U(\varrho(\zeta|1/\zeta)_z)u,z\big)Y(v,\gamma)w}\xlongequal{\eqref{eq197}} \bigbk{w',Y\big(e^{z^{-1}L_1}(-z^2)^{L_0}u,z\big)Y(v,\gamma)w},
\end{align*}
which by \eqref{eq141} equals
\begin{align*}
\bigbk{Y\big(u,z^{-1}\big)w',Y(v,\gamma)w}=\bigbk{Y\big(u,\eta_\infty(z)\big)w',Y(v,\gamma)w}
\end{align*}
where $\eta_\infty=1/\zeta$ is the local coordinate at $\infty$. This proves \eqref{eq182} when $z$ is near $\infty$. 
\end{proof}

\begin{exe}
Let $\Wbb_1,\Wbb_2$ be admissible $\Vbb$-modules, and let $T:\Wbb_1\rightarrow\Wbb_2$ be a $\Vbb$-module homomorphism, i.e., a linear map intertwines the $\Vbb$-actions. Let $\fk P=(\Pbb^1;0,\infty;\zeta,1/\zeta)$, and associate $\Wbb_1,\Wbb_2'$ to $0,\infty$ respectively. Show that the following linear functional is a conformal block associated to $\fk P$ and $\Wbb_1,\Wbb_2'$.
\begin{gather}
\Wbb_1\otimes\Wbb_2'\rightarrow\Cbb\qquad w_1\otimes w_2'\mapsto \bk{Tw_1,w_2'} 
\end{gather}
\end{exe}

\subsection{}\label{lb193}

Due to the fact that \eqref{eq182} belongs to $\Cbb((z))$, we may regard $\wr\upphi(\cdot,w_\blt)$ as a section of $(\scr V_C^{\leq n})^\vee$ that has finite poles at $x_\blt$:
\begin{align}
\wr\upphi(\cdot,w_\blt)\in H^0(C,(\scr V_C^{\leq n})^\vee(\star x_\blt)).
\end{align}
The meaning of the notation is the following. Let $\scr E$ be a vector bundle over $C$. Then for each $k_1,\dots,k_N\in\Zbb$,
\begin{align*}
\scr E(k_1x_1+\cdots+k_Nx_N)
\end{align*}
denotes the $\scr O_C$-module whose space of sections on each open $U\subset C$ are all $s\in\scr E(U\setminus\{x_\blt\})$ such that for each $1\leq i\leq N$ the function $\eta_i^{k_i}\cdot s:x\mapsto \eta_i(x)^{k_i}s(x)$ is holomorphic on a neighborhood of $x_i$ (equivalently, on $U_i$). Thus, when $k_1,\dots,k_N\geq 0$, it is the sheaf of sections of $\scr E_{C\setminus\{x_\blt\}}$ that have poles of order at most $k_i$ at $x_i$. Then \index{E@$\scr E(\star x_\blt)$}
\begin{align}\label{eq255}
\scr E(\star x_\blt)=\varinjlim_{k_1,\dots,k_N\in\Nbb}\scr E(k_1x_1+\cdots+k_Nx_N)
\end{align} 
is the sheaf of sections of $\scr E_{C\setminus\{x_\blt\}}$ that have finite poles at $x_1,\dots,x_N$.

This viewpoint allows us to use the strong residue theorem to obtain the algebraic definition of conformal blocks. Let $\omega_C$ \index{zz@$\omega_C$} be holomorphic cotangent line bundle of $C$, i.e., the sheaf of holomorphic $1$-forms on the open subsets of $C$.  The by residue theorem/Stokes' theorem,
\begin{align}
\sum_{i=1}^N \Res_{x_i}\lambda=0
\end{align}
for all $\lambda\in H^0(C\setminus\{x_\blt\},\omega_C)$, and hence for all $\lambda\in H^0(C,\omega_C(\star x_\blt))$.

\begin{thm}[\textbf{Strong residue theorem}]\label{lb104}
Let $\scr E$ be a vector bundle on $C$. For each $1\leq i\leq N$, use a trivialization of $\scr E$ and the corresponding dual trivialization for the dual vector bundle $\scr E^\vee$ to fix an identification
\begin{align}
\scr E|_{U_i}=E_i\otimes\scr O_{U_i},\qquad \scr E^\vee|_{U_i}=E_i^*\otimes\scr O_{U_i}\label{eq184}
\end{align}
where $E_i$ is a finite dimensional vector space and $E_i^*$ is its dual space. Choose
\begin{align}
s_i=\sum_{n\in\Zbb} e_{i,n}\eta_i^n\qquad\in E_i((\eta_i)). \label{eq185}
\end{align}
Then the following are equivalent.
\begin{enumerate}[label=(\alph*)]
\item There exists $s\in H^0(C,\scr E(\star x_\blt))$ whose Laurent series expansion at each $x_i$ is $s_i$.
\item For each $\sigma\in H^0(C,\scr E^\vee\otimes\omega_C(\star x_{\blt}))$,
\begin{align}
\sum_{i=1}^N \Res_{x_i}\bk{s_i,\sigma}=0. \label{eq183}
\end{align}
\end{enumerate}
\end{thm}
Here, $\scr E^\vee\otimes\omega_C$ is the tensor product of the two vector bundles. Recall that in general, if $\scr E$ and $\scr F$ are vector bundles over a complex manifold $X$, then $\scr E\otimes\scr F$ (or more precisely, $\scr E\otimes_{\scr O_X}\scr F$) is the one whose transition functions are given by the tensor products of those of $\scr E$ and $\scr F$. Equivalently, $\scr E\otimes\scr F$ is the sheafification of the presheaf whose space of sections over any open $U\subset X$ is $\scr E(U)\otimes_{\scr O(U)}\scr F(U)$. $(\scr E\otimes\scr F)(U)$ equals $\scr E(U)\otimes_{\scr O(U)}\scr F(U)$ when $\scr E_U$ and $\scr F_U$ are trivializable (i.e. equivalent to free $\scr O_U$-modules). (To see this, simply assume $\scr E_U=\scr O_U^{\oplus m}$ and $\scr F_U=\scr O_U^{\oplus n}$.)

\subsection{}

The LHS of \eqref{eq183} is understood in the following way. In view of \eqref{eq184}, at each $x_i$, $\sigma$ has expansion $\sigma=\sum_{n\in\Zbb}\varepsilon_{i,n} \eta_i^nd\eta_i$ where $\varepsilon_{i,n}\in E_i^*$. Then $\bk{s_i,\sigma}=\sum_{m,n\in\Zbb}\bk{e_{i,m},\varepsilon_{i,n}}\eta_i^{m+n}d\eta_i$. So \eqref{eq183} reads
\begin{align*}
\sum_{i=1}^N\sum_{m+n=-1}\bk{e_{i,m},\varepsilon_{i,n}}=0
\end{align*}
where the sum over all  $m,n\in\Zbb$ satisfying $m+n=-1$ is finite. 

\begin{rem}
Suppose $\eta_i(U_i)\supset\Dbb_{r_i}$. Then it is clear that if (a) or (b) holds, then the series $s_i=\sum_{n\in\Zbb} e_{i,n}\eta_i^n$ converges a.l.u. on $\eta_i\in\Dbb_{r_i}^\times$.  It is remarkable that this analytic property follows from the algebraic condition \eqref{eq183}. This is analogous to that the formal variable version of local fields implies the complex analytic one, and that the algebraic Jacobi identity for VOAs implies the complex analytic one.
\end{rem}

That $(a)\Rightarrow(b)$ follows from the residue theorem, since $\bk{s,\sigma}$ is an element of $H^0(C,\omega_C(\star x_\blt))$. The other direction is more difficult. To prove it one needs more advance tools  such as sheaf cohomology and Serre duality, which we are not able to present here due to page limitations. We refer the readers to \cite[Sec. 1.2.2]{Muk10}\footnote{Though \cite{Muk10} only discusses the case that $\scr E=\scr O_C$, its proof applies to all vector bundles.}, \cite[Sec. 1.4]{Gui}, or \cite[Sec. 7]{Gui21} for details.

\subsection{}

We now apply the strong residue theorem to the case that $\scr E=(\scr V_C^{\leq n})^\vee$. The trivialization \eqref{eq184} is given by $\mc U_\varrho(\eta_i)$ (cf. \eqref{eq185}) and its dual. In particular, $E_i=(\Vbb^{\leq n})^*$. The series $s_i$  we choose is the RHS of \eqref{eq182}, namely,
\begin{align*}
s_i=\sum_{n\in\Zbb}s_{i,n}\eta_i^n\qquad \in (\Vbb^{\leq n})^*((\eta_i))
\end{align*}
where $s_{i,n}\in (\Vbb^{\leq n})^*$ sends each $v\in\Vbb^{\leq n}$ to
\begin{align*}
s_{i,n}(v)=\upphi\big(w_1\otimes\cdots\otimes Y(v)_{-n-1}w_i\otimes\cdots\otimes w_N\big).
\end{align*}

Now, Def. \ref{lb103} says simply that (for all $n$) all $s_1,\dots,s_N$ are series expansions at $x_1,\dots,x_N$ of the same section of $H^0(C,\scr E(\star x_\blt))$, namely $\wr\upphi(\cdot,w_\blt)$. Thus, by the strong residue Thm. \ref{lb104}, the statements in Def. \ref{lb103} (when restricted to $\scr V^{\leq n}_C$) are equivalent to $\sum_{i=1}^N \Res_{x_i}\bk{s_i,\sigma}=0$ for all $\sigma\in H^0(C,\scr V_C^{\leq n}\otimes\omega_C(\star x_\blt))$. Namely, $\upphi$ vanishes on
\begin{align}
\sigma\cdot w_\blt=\sum_{i=1}^N w_1\otimes\cdots \otimes \sigma\cdot w_i\otimes\cdots\otimes w_N\label{eq188}
\end{align}
where for each $i$,
\begin{gather}
\sigma\cdot w_i=\Res_{z=0} ~Y(v_i(z),z)w_idz \qquad\in\Wbb_i\label{eq190}
\end{gather}
and $\sigma|_{U_i}=v_i(z)dz$ under the identifications \eqref{eq186} and \eqref{eq187}.

For instance, if $\sigma|_{U_i}=uz^kdz$ where $u\in\Vbb$, then
\begin{align}
(uz^kdz)\cdot w_i=Y(u)_kw_i.
\end{align}

\begin{df}\label{lb120}
We define a linear action of $H^0\big(C,\scr V_C\otimes\omega_C(\star x_\blt)\big)$ on $\Wbb_\blt$ \index{00@The action of $H^0\big(C,\scr V_C\otimes\omega_C(\star x_\blt)\big)$ on $\Wbb_\blt$} such that for each $\sigma,w_\blt$ in the two vector spaces respectively, $\sigma\cdot w_\blt$ is defined by \eqref{eq188} and \eqref{eq190}. We call it the \textbf{residue action}.
\end{df}

Thus, taking all $n\in\Nbb$ into account, we see that the complex analytic Def. \ref{lb103} of conformal blocks is equivalent to the following algebraic one:

\begin{df}[Algebraic version]
A linear functional $\upphi:\Wbb_\blt\rightarrow\Cbb$ is called a \textbf{conformal block} associated to $\fk X$ and $\Wbb_\blt$ if it vanishes on the following subspace
\begin{align}
H^0\big(C,\scr V_C\otimes\omega_C(\star x_\blt)\big)\cdot\Wbb_\blt\label{eq189}
\end{align}
of $\Wbb_\blt$, where we have suppressed $\Span_\Cbb$ in \eqref{eq189}.
\end{df}

\begin{df}
The vector space \index{T@$\scr T_{\fk X}(\Wbb_\blt),\scr T_{\fk X}^*(\Wbb_\blt)$}
\begin{align}
\scr T_{\fk X}(\Wbb_\blt)=\frac{\Wbb_\blt}{H^0\big(C,\scr V_C\otimes\omega_C(\star x_\blt)\big)\cdot\Wbb_\blt}
\end{align}
is called the \textbf{space of coinvariants} (also called space of covacua) associated to $\fk X$ and $\Wbb_\blt$. Its dual space is denoted by $\scr T_{\fk X}^*(\Wbb_\blt)$ and called the \textbf{space of conformal blocks} (or space of vacua, space of invariants). 
\end{df}

\section{Pushforward and Lie derivatives in sheaves of VOAs}

\subsection{}\label{lb105}

We continue our discussions in the previous section. The residue action of $\sigma$ on $w_i$ is crucial in the theory conformal blocks. Let us present its definition in a form that indicates the choice of local coordinate $\eta_i$.

We now only assume that $\sigma$ is a section of $\scr V_C\otimes\omega_C(\star x_\blt)$ defined on a neighborhood of $x_i$, say on $U_i$. (Namely, $\sigma$ is a section of $\scr V_C\otimes\omega_C$ on $U_i\setminus\{x_i\}$ with finite poles at $x_i$.) Let $\mc V_\varrho(\eta_i)\sigma$ be $v_i(z)dz$ in \eqref{eq190}. Then \eqref{eq190} reads
\begin{gather}
\sigma\cdot w_i=\Res_{z=0} ~Y(\mc V_\varrho(\eta_i)\sigma,z)w_i. \label{eq193}
\end{gather}

Let us  describe $\mc V_\varrho(\eta_i)$ in a more geometric way. Notice that we
have an obvious equivalence
\begin{align*}
(\eta_i)_*:\scr O_{U_i}\xrightarrow{\simeq}\scr O_{\eta_i(U_i)}
\end{align*}
sending $f$ to $f\circ\eta_i^{-1}$. Then $\id_\Vbb\otimes(\eta_i)_*:\Vbb\otimes_\Cbb\scr O_{U_i}\xrightarrow{\simeq}\Vbb\otimes_\Cbb\scr O_{\eta_i(U_i)}$. We define the \textbf{pushforward} \index{V@$\mc V_\varrho(\eta_i),\mc V_\varrho(\varphi)$}
\begin{gather}\label{eq218}
\begin{gathered}
\mc V_\varrho(\eta_i):\scr V_{U_i}\xrightarrow{\simeq}\Vbb\otimes\scr O_{\eta_i(U_i)}\\
\mc V_\varrho(\eta_i)=(\id_\Vbb\otimes(\eta_i)_*)\mc U_\varrho(\eta_i)
\end{gathered}
\end{gather}
Its tensor product with $(\eta_i)_*=(\eta_i^{-1})^*:\omega_{U_i}\xrightarrow{\simeq}\omega_{\eta_i(U_i)}$ is also denoted by $\mc V_\varrho(\eta_i)$:
\begin{gather}
\begin{gathered}
\mc V_\varrho(\eta_i)\equiv \mc V_\varrho(\eta_i)\otimes(\eta_i)_*:\scr V_{U_i}\otimes\omega_{U_i}(\star x_i)\xrightarrow{\simeq}\Vbb\otimes\omega_{\eta_i(U_i)}(\star 0).
\end{gathered}
\end{gather}

\subsection{}

The above geometric description is convenient when treating simultaneously more than one local coordinate at $x_i$ and the corresponding trivializations. As an application, let us show that the action of $\sigma$ on $\Wbb_i$ can be formulated in a coordinate-independent way.

From now on, we do not fix the local coordinates of $\fk X=(C;x_1,\dots,x_N)$. Let $\scr W(\Wbb_i)$ be an abstract vector space isomorphic to $\Wbb_i$. To be more precise, we consider $\scr W(\Wbb_i)$ \index{W@$\scr W(\Wbb_i)$, $\scr W_{\fk X}(\Wbb_\blt)$} as a (infinite rank) vector bundle over a single point with trivialization \index{U@$\mc U(\alpha),\mc U(\eta),\mc U(\eta_\blt)$}
\begin{align}
\mc U(\eta_i):\scr W(\Wbb_i)\xrightarrow{\simeq} \Wbb_i
\end{align}
for any local coordinate $\eta_i$ of $C$ at $x_i$, such that if $\mu_i$ is also a local coodinate at $x_i$, then the transition function  is
\begin{align}
\mc U(\eta_i)\mc U(\mu_i)^{-1}=\mc U(\eta_i\circ\mu_i^{-1}):\Wbb_i\xrightarrow{\simeq}\Wbb_i.
\end{align}
Note that $\eta_i\circ\mu_i^{-1}\in\Gbb$ is the change of coordinate from $\mu_i$ to $\eta_i$, and $\mc U(\eta_i\circ\mu_i^{-1})$ is the corresponding invertible operator on $\Wbb_i$ defined by \eqref{eq192}.

For each $\sigma\in H^0(U_i,\scr V_{U_i}\otimes\omega_{U_i}(\star x_i))$ and $\wbf\in\scr W(\Wbb_i)$, define
\begin{align}
\sigma\cdot \wbf=\mc U(\eta_i)^{-1}\cdot \sigma\cdot \mc U(\eta_i)\wbf\label{eq194}
\end{align}
where the action of $\sigma$ on $\mc U(\eta_i)\wbf$ is defined by \eqref{eq193}.

\subsection{}

\begin{pp}\label{lb123}
The definition of residue action $\sigma\cdot \wbf$ is independent of the choice of local coordinates $\eta_i$ at $x_i$.
\end{pp}

The proof of this proposition is a good exercise of computing $\mc V_\varrho(\eta)\sigma$ when $\eta,\scr V_{U_i},\scr W(\Wbb_i)$ are not identified with the standard ones using the trivializations.

\begin{proof}
Write $x_i=x,U_i=U,\Wbb_i=\Wbb$ for simplicity. Let $\mu,\eta\in\scr O(U)$ be coordinates of $U$ at $x$. (So $\eta(x)=\mu(x)=0$.) Identify $U$ with $\mu(U)$ via $\mu$ so that $\mu$ is identified with the standard coordinate $\id_\Cbb$ of $\Cbb$. We have $\eta\in\Gbb$. Identify $\scr W(\Wbb)$ with $\Wbb$ via $\mc U(\mu)$. So $\mc U(\mu)=\id$, and $\mc U(\eta):\scr W(\Wbb)\rightarrow\Wbb$ agrees with the operator associated with the transformation $\eta$. We  write $\wbf\in\scr W(\Wbb)$ as $w\in\Wbb$.

Due to the above identifications, we have $\mu_*=\id$ and hence $\mc V_\varrho(\mu)=\mc U_\varrho(\mu)$. Write
\begin{align*}
\mc V_\varrho(\mu)\sigma=\mc U_\varrho(\mu)\sigma=u(z)dz
\end{align*}
where $u=u(z)$ belongs to $H^0(U,\Vbb\otimes\scr O_U(\star 0))$. So the action $\sigma\cdot w$ defined by $\mu$ is simply $\Res_{z=0}Y(u(z),z)wdz$.

Let us compute $\sigma\cdot w$ using $\eta$. In view of \eqref{eq193} and \eqref{eq194}, we compute $\mc V_\varrho(\eta)\sigma$. First,
\begin{align*}
\mc U_\varrho(\eta)\sigma=\mc U_\varrho(\eta)\mc U_\varrho(\mu)^{-1}u(z)dz=\mc U(\varrho(\eta|\mu))u(z)dz=\mc U(\varrho(\eta|\id_\Cbb)_z)u(z)dz.
\end{align*}
Here $z$ is the standard variable of $\Cbb$. Applying $\eta_*=(\eta^{-1})^*$, we get
\begin{align*}
\mc V_\varrho(\eta)\sigma=\mc U(\varrho(\eta|\id_\Cbb)_{\eta^{-1}(z)})u(\eta^{-1}(z))d{\eta^{-1}(z)}
\end{align*}
defined on $\eta(U)\subset\Cbb$. Thus, when evaluated with any vector  $w'\in\Wbb'$, we have
\begin{align*}
&\mc U(\eta)^{-1}\cdot \sigma\cdot \mc U(\eta)w=\sum_{n\in\Nbb}\mc U(\eta)^{-1}P_n\cdot \sigma\cdot \mc U(\eta)w\\
=&\sum_{n\in\Nbb}\Res_{z=0}~\mc U(\eta)^{-1}P_n Y\big(\mc V_\varrho(\eta)\sigma,z\big)\mc U(\eta) w\\
=&\sum_{n\in\Nbb}\Res_{z=0}~\underbrace {\mc U(\eta)^{-1}P_n Y\big(\mc U(\varrho(\eta|\id_\Cbb)_{\eta^{-1}(z)})u(\eta^{-1}(z)),z\big)\mc U(\eta) w}_{A_n}\cdot d{\eta^{-1}(z)}.
\end{align*}
By the change of coordiante Thm. \ref{lb95}, $\sum_n\bk{w',A_n}$ converges a.l.u. when $z\neq 0$ is small. Thus we can move the infinite sum into the residue, and by Thm. \ref{lb95} again, the above equals
\begin{align*}
\Res_{z=0}~Y(u(\eta^{-1}(z)),\eta^{-1}(z))w\cdot d\eta^{-1}(z)
\xlongequal{\zeta=\eta^{-1}(z)}\Res_{\zeta=0}~Y(u(\zeta),\zeta)w\cdot d\zeta.
\end{align*}
This finishes the proof.
\end{proof}

\subsection{}\label{lb124}

We are now ready to give a coordinate independent definition of conformal blocks. Let $\fk X=(C;x_\blt)$ be an $N$-pointed compact Riemann surface, for which we do not fix local coordinates. Again, we associate admissible $\Vbb$-modules $\Wbb_\blt$ to the markd points $x_\blt$. Let  \index{W@$\scr W(\Wbb_i)$, $\scr W_{\fk X}(\Wbb_\blt)$}
\begin{align}
\scr W_{\fk X}(\Wbb_\blt)=\scr W(\Wbb_1)\otimes\cdots\otimes\scr W(\Wbb_N).
\end{align}
Then for each choice of local coordinates $\eta_\blt$, we have trivialization \index{U@$\mc U(\alpha),\mc U(\eta),\mc U(\eta_\blt)$}
\begin{align}
\mc U(\eta_\blt):=\mc U(\eta_1)\otimes\cdots\otimes\mc U(\eta_N):\scr W_{\fk X}(\Wbb_\blt)\xrightarrow{\simeq}\Wbb_\blt.
\end{align}
If $\mu_\blt$ is another set of local coordinates, then we have transition function
\begin{align}\label{eq225}
\begin{aligned}
\mc U(\eta_\blt)\mc U(\mu_\blt)^{-1}=\mc U(\eta_\blt\circ\mu_\blt^{-1}):=\mc U(\eta_1\circ\mu_1^{-1})\otimes\cdots\otimes\mc U(\eta_N\circ\mu_N^{-1}).
\end{aligned}
\end{align}

For each $\wbf_\blt=\wbf_1\otimes\cdots\wbf_N\in\scr W_{\fk X}(\Wbb_\blt)$ and $\sigma\in H^0(C;\scr V_C\otimes\omega_C(\star x_\blt))$, define the residue action
\begin{align}
\sigma\cdot\wbf_\blt=\sum_{i=1}^N \wbf_1\otimes\cdots\otimes \sigma\cdot \wbf_i\otimes\cdots\otimes \wbf_N
\end{align}
(where each $\sigma\cdot \wbf_i$ is defined by \eqref{eq194}). This gives a linear action of $H^0(C;\scr V_C\otimes\omega_C(\star x_\blt))$ on $\scr W_{\fk X}(\Wbb_\blt)$.

\begin{df}
The vector space \index{T@$\scr T_{\fk X}(\Wbb_\blt),\scr T_{\fk X}^*(\Wbb_\blt)$}
\begin{align}
\scr T_{\fk X}(\Wbb_\blt)=\frac{\scr W_{\fk X}(\Wbb_\blt)}{H^0\big(C,\scr V_C\otimes\omega_C(\star x_\blt)\big)\cdot\scr W_{\fk X}(\Wbb_\blt)}
\end{align}
and its dual space $\scr T^*_{\fk X}(\Wbb_\blt)$ are called respectively  the \textbf{space of coinvariants} and the \textbf{space of conformal blocks} associated to $\fk X$ and $\Wbb_\blt$.
\end{df}

\subsection{}\label{lb128}

Let us generalize the pushforward in Subsec. \ref{lb105} to a more general geometric setting. Let $X$ and $Y$ be (non-necessarily compact)  Riemann surfaces, and let $\varphi:X\xrightarrow{\simeq}Y$ be a bi-holomorphism. Let
\begin{align}
\varphi_*:\scr O_X\rightarrow\scr O_Y,\qquad f\mapsto f\circ\varphi^{-1}
\end{align}
be the pushforward of the structure sheaves. We let $\varphi_*$ also denote
\begin{align}
\varphi_*\equiv\id_\Vbb\otimes\varphi_*:\Vbb\otimes\scr O_X\xrightarrow{\simeq}\Vbb\otimes\scr O_Y.\label{eq198}
\end{align}

Let $U\subset X$ and $V\subset Y$ be open and connected such that $V=\varphi(U)$. Suppose there is a univalent $\eta\in\scr O(Y)$. Recall that we have an equivalence
\begin{align}
\mc V_\varrho(\eta)=\eta_*\cdot\mc U_\varrho(\eta):\scr V_V\xrightarrow{\simeq}\Vbb\otimes\scr O_{\eta(V)}\label{eq201}
\end{align}
where the pushforward $\eta_*:\Vbb\otimes\scr O_V\rightarrow\Vbb\otimes\scr O_{\eta(V)}$ is similar to \eqref{eq198}. We define \index{V@$\mc V_\varrho(\eta_i),\mc V_\varrho(\varphi)$}
\begin{gather}\label{eq200}
\begin{gathered}
\mc V_\varrho(\varphi):\scr V_U\xrightarrow{\simeq}\scr V_V,\qquad \mc V_\varrho(\eta)\mc V_\varrho(\varphi)=\mc V_\varrho(\eta\circ\varphi).
\end{gathered}
\end{gather}
Equivalently,
\begin{align}
\mc U_\varrho(\eta)\mc V_\varrho(\varphi)=\varphi_*\cdot \mc U_\varrho(\eta\circ\varphi).\label{eq199}
\end{align}
\begin{proof}
Note that $\mc V_\varrho(\eta)=\eta_*\cdot \mc U_\varrho(\eta)$, $\mc V_\varrho(\eta\circ\varphi)=(\eta\circ\varphi)_*\cdot\mc U_\varrho(\eta\circ\varphi)$, and $(\eta\circ\varphi)_*=\eta_*\cdot \varphi_*$.
\end{proof}

\begin{lm}
The definition of $\mc V_\varrho(\varphi)$ is independent of the choice of univalent map $\eta$.
\end{lm}

\begin{proof}
Let $\mu\in\scr O(V)$ be univalent. Using \eqref{eq177}, one checks easily that
\begin{align*}
\mc U(\varrho(\eta\circ\varphi|\mu\circ\varphi))=\varphi^{-1}_*\cdot\mc U(\varrho(\eta|\mu))\cdot\varphi_*
\end{align*}
as morphisms $\Vbb\otimes\scr O_U\rightarrow\Vbb\otimes\scr O_U$. This means
\begin{align}
\mc U_\varrho(\eta\circ\varphi)\mc U_\varrho(\mu\circ\varphi)^{-1}=\varphi^{-1}_*\cdot\mc U_\varrho(\eta)\mc U_\varrho(\mu)^{-1}\cdot\varphi_*.
\end{align}
The independence follows immediately from the above formula and \eqref{eq199}.
\end{proof}

By this lemma, we have a global equivalence
\begin{align}
\mc V_\varrho(\varphi):\scr V_X\xrightarrow{\simeq}\scr V_Y
\end{align}
defined locally by \eqref{eq200} or \eqref{eq199}. We call $\mc V_\varrho(\varphi)$ the \textbf{pushforward} associated to $\varphi$. We also use $\mc V_\varrho(\varphi)$ to denote
\begin{align}
\mc V_\varrho(\varphi)\equiv \mc V_\varrho(\varphi)\otimes\varphi_*:\scr V_X\otimes\omega_X\xrightarrow{\simeq}\scr V_Y\otimes\omega_Y\label{eq206}
\end{align}
where $\varphi_*$ is  $(\varphi^*)^{-1}=(\varphi^{-1})^*:\omega_X\rightarrow\omega_Y$.

\subsection{}

\begin{rem}
From \eqref{eq200}, it is clear that $\mc V_\varrho(\psi\circ\varphi)=\mc V_\varrho(\psi)\mc V_\varrho(\varphi)$ if $\psi:Y\rightarrow Z$ is a bi-holomorphism of complex manifolds.
\end{rem}

\begin{rem}\label{lb107}
The geometric meanings of $\mc V_\varrho(\varphi):\scr V_X\rightarrow\scr V_Y$ and the formula \eqref{eq199} are as follows. Let $x\in X$. Choose a vector $\ubf$ in the fiber $\scr V_X|x$, considered an abstract VOA vector. Let $\vbf=\mc V_\varrho(\varphi)\ubf$. Then by \eqref{eq199} and the geometric meanings of $\mc U_\varrho(\eta)$ and $\mc U_\varrho(\mu)$ (cf. Rem. \ref{lb106}), $\ubf$ and $\vbf$ are related by the property that for any univalent $\eta$ holomorphic on a neighborhood $y$, if we set $\mu=\eta\circ\varphi$, then  the coordinate representation of $\ubf$ under $\mu-\mu(x)$ is the same as that of $\vbf$ under $\eta-\eta(y)$. 

We will simply say that the $\mu$-trivialization of $\ubf$ and the $\eta$-trvialization of $\vbf$ are equal.
\end{rem}

\begin{rem}
Now $\mc V_\varrho(\eta)$ has two meanings: as an equivalence $\scr V_V\rightarrow\Vbb\otimes\scr O_{\eta(V)}$ defined by \eqref{eq201}, and as an equivalence $\scr V_V\rightarrow\scr V_{\eta(V)}$ defined similar to $\mc V_\varrho(\varphi)$. These two meanings agree if we identify  $\scr V_{\eta(V)}$ with $\Vbb\otimes\scr O_{\eta(V)}$ via the trivialization $\mc U_\varrho(\zeta)$ where $\zeta$ is the standard coordinate of $\Cbb$.
\end{rem}

\subsection{}\label{lb150}

That one can define pushforward for (co)tangent bundles as well as for sheaves of VOAs implies that these two classes of objects are closely related. Indeed, one can view $\scr V_C$ as a twisted direct sum of tensor products of the holomorphic tangent line bundle $\Theta_C$ of $C$. \index{zz@$\Theta_C$} (Note that $\omega_C$ is the dual of $\Theta_C$.)

To see this, let us look at the transition function $\mc U(\varrho(\eta|\mu)):\Vbb^{\leq n}\otimes\scr O_U\xrightarrow{\simeq}\Vbb^{\leq n}\otimes\scr O_U$ where $\mu,\eta\in\scr O(U)$ are univalent. By \eqref{eq192}, $\mc U(\varrho(\eta|\mu))_x=\varrho(\eta|\mu)_x'(0)^{L_0}(1+\text{products of }L_{>0})$ on $\Vbb$. From \eqref{eq177}, $\varrho(\eta|\mu)_x'(0)=\frac{\partial\eta}{\partial\mu}(x)$. Thus, as $L_{>0}$ lowers weights, we conclude that for each $v=v(x)\in\Vbb^{\leq n}\otimes\scr O_U$,
\begin{align}
\mc U_\varrho(\eta)\mc U_\varrho(\mu)^{-1}v=\mc U(\varrho(\eta|\mu))v=(\partial\eta/\partial\mu)^nv\qquad\mod~\Vbb^{\leq n-1}\otimes\scr O_U.
\end{align}
Thus, the transition function $\mc U(\varrho(\eta|\mu))$ from the $\mu$-coordinate to the $\eta$-coordinate for the quotient bundle $\scr V^{\leq n}_C/\scr V^{\leq n-1}_C$ is $(\partial\eta/\partial\mu)^n$, which agrees that of $\Vbb(n)\otimes_\Cbb\Theta_C^{\otimes n}$. We conclude:
\begin{pp}\label{lb178}
There is an equivalence of $\scr O_C$-modules
\begin{align}
\scr V^{\leq n}_C/\scr V^{\leq n-1}_C\simeq \Vbb(n)\otimes_\Cbb\Theta_C^{\otimes n}\label{eq247}
\end{align}
such that if $U\subset C$ is open and $\eta\in\scr O(U)$ is univalent, then for each $v\in\Vbb(n)$, $v\otimes \partial_\eta^n$ (which is an element in the RHS of \eqref{eq247}) is equivalent to the equivalence class of $\mc U_\varrho(\eta)^{-1}v$ in the LHS of \eqref{eq247}.
\end{pp}

Thus, in general, an element of $\Vbb(n)\otimes\Theta_C^{\otimes n}(U)$ is a sum of those of the form $v\otimes f\partial_\eta^n$ where $v\in\Vbb(n)$ and $f\in\scr O(U)$. It is identified with $f\cdot \mc U_\varrho(\eta)^{-1}v$ in the LHS of \eqref{eq247}.

\subsection{}

If we focus on only primary vectors, we can get subbundles of $\scr V_C$ naturally equivalent to direct sums of tensor products of $\Theta_C$ without taking quotient. Recall that a primary vector $v$ in $\Vbb(n)$ is  one killed by $L_{>0}$.  So the change of coordinate formula for $v$ is $\mc U(\varrho(\eta|\mu))v=(\partial\eta/\partial\mu)^nv$. Thus, if we let $\mbf P(n)$ be the subspace of weight $n$ primary vectors of $\Vbb$, then $\scr V_C$ has a vector subbundle $\scr P^n_C$ with local trivialization $\mc U_\varrho(\eta):\scr P^n_C|_U\xrightarrow{\simeq} \mbf P(n)\otimes_\Cbb\scr O_U$ for any univalent $\eta\in\scr O(U)$. Moreover, $\scr P_C^n$ has the same transition functions as $\Theta_C^{\otimes n}$. So $\scr P_C^n\simeq \mbf P(n)\otimes_\Cbb\Theta_C^{\otimes n}$.

Since the basic properties of line bundles $\Theta_C^{\otimes n}$ are well known, in the early development of the mathematical theory of conformal blocks, sheaves of VOAs were not yet defined, and the sheaves $\scr P^n_C$ were sometimes used instead to define and study conformal blocks. Specifically, in the landmark paper \cite{TUY89}, conformal blocks for a WZW model $\Vbb=L_l(\gk,0)$ (where $\gk$ is simple and $l\in\Nbb$) was defined using
\begin{align*}
\scr P^1_C\otimes \omega_C(\star x_\blt)\simeq\mbf P(1)\otimes_\Cbb \Theta_C\otimes\omega_C(\star x_\blt)=\gk\otimes_\Cbb\scr O_C(\star x_\blt).
\end{align*}
(Note that $\Theta_C\otimes\omega_C\simeq\scr O_C$ since $\omega$ is dual to $\Theta_C$.) Thus, for WZW models, the space of coinvariants was defined (for $\fk X$ with local coordinates) in \cite{TUY89} to be
\begin{align*}
\frac{\Wbb_\blt}{H^0\big(C,\gk\otimes_\Cbb\scr O_C(\star x_\blt)\big)\cdot\Wbb_\blt}
\end{align*}
Fortunately, this definition agrees with the one defined using $H^0(C,\scr V_C\otimes\omega_C(\star x_\blt))$. See \cite[Sec. 9.3]{FB04}.

\subsection{}

In differential geometry, the Lie derivatives of sections of (tensor products of) tangent and cotangent bundles are defined using the pushforward or the pullback maps associated to flows. Likewise, we can define Lie derivatives for sections of $\scr V_C$.

Let $W\subset C$ be an open subset, and choose $\fk x\in\Theta_C(W)$, namely, $\xk$ is a holomorphic tangent field on $W$. Note that for any precompact open subset $V\subset W$ (i.e., the closure of $V$ in $W$ is compact), there is a neighborhood $T\subset\Cbb$ of $0$ (with variable $\zeta$) such that the holomorphic flow $\exp(\zeta \fk x)$ is holomorphic on $T\times V$ and is injective as a function on $V$ for each $\zeta\in T$. (Cf. Subsec. \ref{lb12}.)

In the following, we write $\exp(\zeta\xk)(x)$ as $\exp_{\zeta\xk}(x)$.

\begin{df}\label{lb129}
For any $\vbf\in\scr V_C^{\leq n}(W)$ and $\xk\in\Theta_C(W)$, define the \textbf{Lie derivative} $\mc L_\xk\vbf$  \index{L@$\mc L_\xk$, the Lie derivative} to be an element of $\scr V_C^{\leq n}(W)$ (if the limit exists) as follows . Choose any precompact open subset $V$ in $W$. Then
\begin{align}\label{eq207}
\mc L_\xk\vbf\big|_V=\lim_{\zeta\rightarrow0} \frac{\mc V_\varrho(\exp_{\zeta\xk})^{-1}\big(\vbf\big|_{\exp_{\zeta\xk}(V)}\big)-\vbf\big|_V}{\zeta}
\end{align}
\end{df}

Intuition: For each $p\in V$, $\vbf(p)\in\scr V_C^{\leq n}|p$ is an abstract VOA vector at $p$. Let $q=\exp_{\zeta\xk}(p)$. Then  $\vbf(q)\in\scr V_C^{\leq n}|q$ is an abstract VOA vector at $\vbf(q)$,  which is pulled back to the vector $\mc V_\varrho(\exp_{\zeta\xk})^{-1}\vbf(q)\in\scr V_C^{\leq n}|p$ via the map $\exp_{\zeta\xk}$.
\begin{align}\label{eq202}
\vcenter{\hbox{{
\includegraphics[height=2cm]{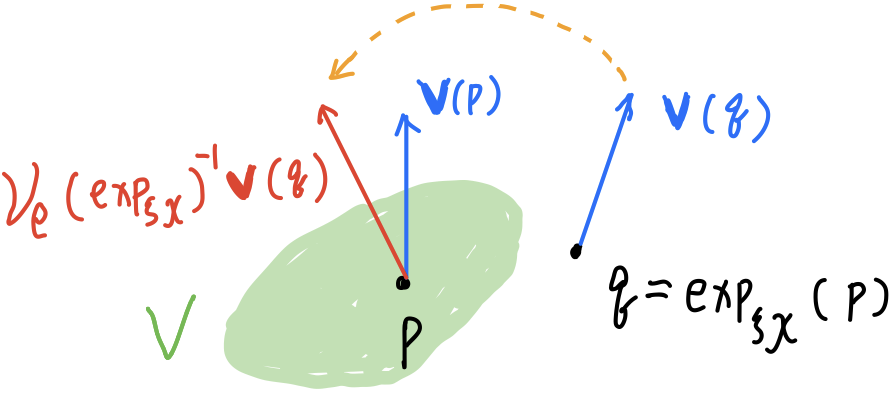}}}}
\end{align}
Then for small $\zeta$,
\begin{align}
(\mc L_\xk\vbf)(p)\approx \frac{\mc V_\varrho(\exp_{\zeta\xk})^{-1}\vbf(q)-\vbf(p)}{\zeta}
\end{align}

\subsection{}

\begin{pp}\label{lb108}
Assume that $\eta\in\scr O(W)$ is univalent, and set
\begin{align*}
u=\mc U_\varrho(\eta)\vbf\qquad\in\Vbb^{\leq n}\otimes_\Cbb\scr O(W).
\end{align*}
Write $\xk=h\partial_\eta$ where $h\in\scr O(W)$. Then $\mc L_\xk\vbf$ exists (i.e. the limit on the RHS of \eqref{eq207} exists) as an element of $\scr V^{\leq n}(W)$, and its $\eta$-trivialization is
\begin{align}
\mc U_\varrho(\eta)\mc L_\xk\vbf=h\partial_\eta u-\sum_{k\geq1} \frac 1{k!}\partial_\eta^k h\cdot L_{k-1}u.\label{eq204}
\end{align}
\end{pp}

\begin{proof}
We need to find the $\eta$-trivialization of $\mc V_\varrho(\exp_{\zeta\xk})^{-1}\big(\vbf\big|_{\exp_{\zeta\xk}(V)}\big)$ at any $p\in V$, namely, the $\eta$-trivialization of  the red vector in  \eqref{eq202}. Since the $\eta$-trivialization of $\vbf(q)$ is $u(q)$, by \eqref{eq199} or Rem. \ref{lb107}, the $\eta\circ\exp_{\zeta\xk}$-trivialization of the red vector is also $u(q)$. So the $\eta$-trivialization of the red vector (which is at $p$) is
\begin{align}
\mc U(\varrho(\eta|\eta\circ\exp_{\zeta\xk})_p)u(\exp_{\zeta\xk}(p))\label{eq203}.
\end{align}
Its derivative over $\zeta$ at $\zeta=0$ gives $\mc L_\xk\vbf(p)$ under the $\eta$-trivialization. (The readers can check \cite[Sec. 2.6]{Gui} if they are not satisfied with the rigorousness of the proof here.)

The derivative at $\zeta=0$ of $u(\exp_{\zeta\xk}(p))$ is just the action of the vector field $\xk$ on $u$, namely $h\partial_\eta u$ at $p$. (Notice \eqref{eq7}.) The derivative of $\mc U(\varrho(\eta|\eta\circ\exp_{\zeta\xk})_p)$ at $0$ can be calculated using Prop. \ref{lb98}: if we identify $\eta$ with the standard coordinate of $\Cbb$, then 
\begin{align*}
&\partial_\zeta\mc U(\varrho(\eta|\eta\circ\exp_{\zeta\xk})_p)(t)\big|_{\zeta=0}\xlongequal{\eqref{eq205}}\partial_\zeta\big(\exp_{-\zeta\xk}(t+\exp_{\zeta\xk}(p))-p\big)\big|_{\zeta=0}\\
=&-h(t+p)+h(p).
\end{align*}
Its $k$-th derivative over $t$ at $t=0$ is then $-\partial_\eta^kh(p)$. Thus, by Prop. \ref{lb98},
\begin{align*}
\partial_\zeta \mc U(\varrho(\eta|\eta\circ\exp_{\zeta\xk})_p)\big|_{\zeta=0}=-\sum_{k\geq1} \frac 1{k!}\partial_\eta^k h\cdot L_{k-1}.
\end{align*}
\end{proof}

\subsection{}

In Prop. \ref{lb108}, if we assume that  $u\in\mbf P(n)\otimes_\Cbb\scr O(W)$, i.e., the values of $u$ are primary with weights $n$, then the Lie derivative formula is $h\partial_\eta u-n\partial_\eta h\cdot u$. Not surprisingly, this result agrees with the formula of Lie derivatives in $\Theta_C^{\otimes n}$, including the case $n=-m<0$ where we understand $\Theta_C^{\otimes (-m)}=\omega_C^{\otimes m}$.  

Since we have pushforward for sections of $\scr V_C^{\leq n}\otimes\omega_C$ (cf. \eqref{eq206}), we can also define Lie derivatives in this bundle using the same formula \eqref{eq207}. The result is easy to guess by Leibniz rule and prove rigorously:

\begin{co}\label{lb130}
Let $\sigma\in\scr V^{\leq n}_C\otimes\omega_C(W)$, and set
\begin{align*}
u\cdot d\eta=\mc U_\varrho(\eta)\sigma\qquad\in\Vbb^{\leq n}\otimes_\Cbb\omega_C(W)
\end{align*}
where $u\in\Vbb^{\leq n}\in\scr O(W)$. Write $\xk=h\partial_\eta$ where $h\in\scr O(W)$. Then
\begin{align}
\mc U_\varrho(\eta)\mc L_\xk\sigma=h\partial_\eta u\cdot d\eta-\sum_{k\geq1} \frac 1{k!}\partial_\eta^k h\cdot L_{k-1}u\cdot d\eta+\partial_\eta h\cdot u\cdot d\eta.
\end{align}
\end{co}

\section{Families of compact Riemann surfaces and parallel sections of conformal blocks}

\subsection{}

\begin{df}
A \textbf{family of compact Riemann surfaces} is the data $\pi:\mc C\rightarrow\mc B$ where $\mc B,\mc C$ are Riemann surfaces,  the surjective holomorphic map $\pi$ is proper (i.e. $\pi^{-1}(\text{compact})$ is compact) and a submersion (i.e. the linear map $d\pi$ between holomorphic tangent spaces is everywhere surjective), and for each $b\in\mc B$ the fiber $\mc C_b=\pi^{-1}(b)$ \index{Cb@$\mc C_b=\pi^{-1}(b)$, the fiber of $\mc C$ at $b$} is a compact Riemann surface. Clearly, $\pi$ is an open map.
\end{df}

By Ehresmann's fibration theorem, if $\mc B$ is connected, then all fibers of the family are diffeomorphic; moreover,  as a family of differential manifolds, $\pi:\mc C\rightarrow\mc B$ is locally trivial, i.e. as a projection of $C\times V\rightarrow V$ when $V\subset \mc B$ is open and $C$ is a surface. However, it is not locally trivial as a family of complex manifolds.                                                                                                                                                                                                                                                                                                                                                                                                                                                                                                                                                                                                                                                                                                                                                                                                                                                                                                                                                                                                                                                                                                                                                              

\begin{df}\index{00@Families of $N$-pointed compact Riemann surfaces}
A \textbf{family of $N$-pointed compact Riemann surfaces} is the data $\fk X=(\pi:\mc C\rightarrow\mc B;\sgm_1,\dots,\sgm_N)$ where $\pi:\mc C\rightarrow\mc B$ is a family of compact Riemann surfaces, and the following conditions hold:
\begin{enumerate}[label=(\alph*)]
\item Each $\sgm_i:\mc B\rightarrow\mc C$ is a section, i.e., a holomorphic map such that $\pi\circ\sgm_i=\id_{\mc B}$. (So $\sgm_i(b)$ is are points on the fiber $\mc C_b$.)
\item $\sgm_1(b),\dots,\sgm_N(b)$ are distinct, considered as marked points of each fiber $\mc C_b$.
\item Each connected component of each fiber $\mc C_b$ contains at least one of $\sgm_1(b),\dots,\sgm_N(b)$.
\end{enumerate}
The following is a hypersurface in $\mc C$. \index{SX@$\SX=\bigcup_{j=1}^N\sgm_j(\mc B)$}
\begin{align}\label{eq212}
\SX=\bigcup_{j=1}^N\sgm_j(\mc B)
\end{align}

A \textbf{local coordinate} $\eta_i$ of the family at $\sgm_i$ is a holomorphic function on a neighborhood $U_i$ of $\sgm_i(\mc B)$ that restricts to a local coordinate $\eta_i|_{\mc C_b\cap U_i}$ of $\mc C_b$ at $\sgm_i(b)$ for each $b\in\mc B$, i.e., $\eta_i(\sgm_i(b))=0$ and $\eta_i$ is injective on the fiber
\begin{align*}
U_{i,b}=\mc C_b\cap U_i.
\end{align*}
We call the data $\fk X=(\pi:\mc C\rightarrow\mc B;\sgm_1,\dots,\sgm_N;\eta_1,\dots,\eta_N)$ a \textbf{family of $N$-pointed compact Riemann surfaces with local coordinates.} We define the fiber \index{Xb@$\fk X_b$}
\begin{align}
\fk X_b=(\mc C_b;\sgm_i(b),\dots,\sgm_N(b);\eta_1|_{\mc C_b},\dots,\eta_N|_{\mc C_b})
\end{align}
which is an $N$-pointed compact Riemann surface with local coordinates. \hfill\qedsymbol
\end{df}

\subsection{}

Since $\pi$ is a submersion, on a neighborhood of $p\in\sgm_i(\mc B)$, $\pi$ is equivalent to the projection $D\times V\rightarrow V$ where $D\subset\Cbb,V\in\Cbb^m$ are open. So $\sgm_i$ restricted to $V\subset\mc B$ is written as
 $\sgm_i(b)=(\sigma_i(b),b)$ where $\sigma_i:V\rightarrow D$ is holomorphic. Namely, $\sgm_i|_V$ is the graph of $\sigma_i$.
\begin{align}\label{eq220}
\vcenter{\hbox{{
\includegraphics[height=1.8cm]{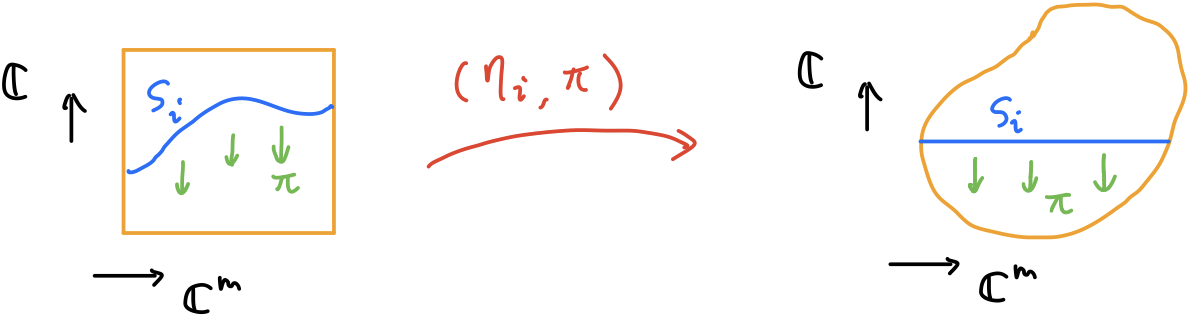}}}}
\end{align}
By the fact that $\eta_i$ is injective on each fiber, $\partial_{z_1}\eta_i$ is nowhere zero where $z_1$ is the coordinate for $D$. So the Jacobian of $(\eta_i,\pi)$ is nowhere zero. Thus, by the inverse mapping theorem, together with the easy fact that $(\eta_i,\pi)$ is injective on $U_i$, we see that \emph{$(\eta_i,\pi)$ is a biholomorphism from $U_i$ to a neighborhood of $\{0\}\times\mc B$ in $\Cbb\times\mc B$.} \eqref{eq220} shows a picture in the case that $V$ is identified with an open subset of $\Cbb^m$.

Thus, by identifying $U_i$ with its image $W$ (which is a neighborhood of $\{0\}\times\mc B$) under $(\eta_i,\pi)$, we may assume that $\pi$ is the projection of $W$  onto $\mc B$, $\sgm_i$ is the canonical map $\mc B\rightarrow\{0\}\times\mc B$, and $\eta_i$ is the projection of $W\subset\Cbb\times\mc B$ onto the $\Cbb$-axis.

\subsection{}

\begin{eg}
Let $C$ be a connected compact Riemann surface. Then
\begin{align*}
\fk X=(\pi:C\times\Conf^N(C)\rightarrow \Conf^N(C);\sgm_1,\dots,\sgm_N)
\end{align*}
is a family of $N$-pointed compact Riemann surface, where $\pi$ is the projection onto the second component, and $\sgm_i:\Conf^N(C)\rightarrow C\times \Conf^N(C)$ sends each $(x_1,\dots,x_N)$ to $(x_i,x_1,\dots,x_N)$. The fibers are $\fk X_{x_\blt}=(C;x_1,\dots,x_N)$.
\end{eg}

\begin{eg}\label{lb118}
Let $\fk P^N=(\pi:\Pbb^1\times\Conf^N(\Cbb^\times)\rightarrow\Conf^N(\Cbb^\times);0,\sgm_1,\dots,\sgm_N,\infty)$ where $0,\infty$ as sections sending $x_\blt$ to  $(0,x_\blt)$ and $(\infty,x_\blt)$ respectively, and $\sgm_i$ is as in the previous example. Then $\fk P^N$ is $(N+2)$-pointed. Moreover, $\fk P^N$ can be equipped with local coordinates $\zeta,\eta_1,\dots,\eta_N,1/\zeta$ at $0,\sgm_1,\dots,\sgm_N,\infty$ respectively, where $\zeta$ sends $(z,z_\blt)$ to $z$, $1/\zeta$ sends $(z,z_\blt)$ to $1/z$, and each $\eta_i$ sends $(z,z_\blt)$ to $z-z_i$. The fibers are
\begin{align*}
\fk P^N_{z_\blt}=(\Pbb^1;0,z_1,\dots,z_N,\infty;\zeta,\zeta-z_1,\dots,\zeta-z_N,1/\zeta)
\end{align*}
where $\zeta$ is now the standard coordinate of $\Cbb$.
\end{eg}

\subsection{}\label{lb170}

\begin{eg}\label{lb154}
Let
\begin{align*}
\wtd{\fk X}=(\wtd C;x_1,\dots,x_N,x',x'';\eta_1,\dots,\eta_N;\xi,\varpi)
\end{align*}
be an $(N+2)$-pointed compact Riemann surface with local coordinates such that each connected component contains one of $x_1,\dots,x_N$. Let $U',U''$ be respectively open  disks centered at $x',x''$ with radii $r,\rho$. More precisely, we assume $\xi,\varpi$ are defined on $U',U''$, and we have biholomorphisms
\begin{align*}
\xi:U'\xrightarrow{\simeq}\Dbb_r,\qquad \varpi:U''\xrightarrow{\simeq} \Dbb_\rho.
\end{align*}
We assume moreover that $U',U'',x_1,\dots,x_N$ are mutually disjoint.

For each $q\in\Dbb^\times_{r\rho}$ we can define an $N$-pointed $\fk X_q$ by the sewing operation as follows. We glue the following annuli
\begin{gather}
	\begin{gathered}
\xi^{-1}(A_{|q|/\rho,r})=\{x\in U': |q|/\rho<|\xi(x)|<r\}\\
\Big\updownarrow\text{identify}\\	
\varpi^{-1}(A_{|q|/r,\rho})=\{y\in U'':|q|/r<|\varpi(y)|<\rho\}
	\end{gathered}	
\end{gather}
where the rule for identification is
\begin{align}
x=y\qquad\text{iff}\qquad\xi(x)\varpi(y)=q.
\end{align}
The parts $Z'_q=\{x\in U':|\xi(x)|\leq |q|/\rho\}$ and $Z''_q=\{y\in U'':|\varpi(y)|\leq |q|/r\}$ are discarded. By this gluing procedure we obtain the sewn Riemann surface $\mc C_q$ with marked points $x_1,\dots,x_N$ (the same as the first $N$ marked points of $\wtd{\fk X}$). The local coordinate at $x_i$ is also chosen to be $\eta_i$. This defines $\fk X_q=(\mc C_q;x_1,\dots,x_N;\eta_1,\dots,\eta_N)$.

One can assemble all $\fk X_q$ to form a family
\begin{align*}
\fk X=(\pi:\mc C\rightarrow\Dbb_{r\rho}^\times;x_1,\dots,x_N;\eta_1,\dots,\eta_N)
\end{align*}
whose fiber at each $q\in\Dbb_{r\rho}^\times$ is $\fk X_q$. (We have abused notations here to let $x_i$ denote a section and $\eta_i$ a local coordinate at the section $x_i$.) It could be obtained in the following way:
\begin{itemize}
\item We have closed subsets $E'=\bigcup_{q\in\Dbb_{r\rho}^\times} Z_q'\times \{q\}$ and $E''=\bigcup_{q\in\Dbb_{r\rho}^\times} Z_q''\times\{q\}$ of $\wtd C\times\Dbb_{r\rho}^\times$. Consider the projection
\begin{align*}
\pi: (\wtd C\times\Dbb_{r\rho}^\times)\setminus (E'\cup E'')\rightarrow \Dbb_{r\rho}^\times.
\end{align*}
Each $x_i$ is the section sending $q\in\Dbb_{\rho}^\times$ to $(x_i,q)$, and $\eta_i$ sends $(x,q)$ to $\eta_i(x)$ when $x$ is close to $x_i$. Modding this data by a suitable holomorphic relation gives the family $\fk X$.
\end{itemize}
\hfill \qedsymbol
\end{eg}

In the above example, we can in fact extend $\fk X$ to a family over $\Dbb_{r\rho}$ where $\fk X_0=(\mc C_0;x_\blt;\eta_\blt)$ is the ``limit" of $\fk X_q$ as $q\rightarrow0$. As a topological space, $\mc C_0$ is obtained by gluing $x'$ and $x''$ of $\wtd C$. $\mc C_0$ is not a smooth manifold, and hence cannot be a Riemann surface. However, one can make $\mc C_0$ a singular complex manifold (more precisely: a complex space) by defining a suitable structure sheaf $\scr O_{\mc C_0}$. $\mc C_0$ is called a \textbf{nodal curve}. Nodal curves are crucial to the proof of sewing and factorization of conformal blocks. However, this topic is out of the scope of our notes. We refer the readers to \cite{Gui} for a detailed discussion of this topic.

\subsection{}

\begin{eg}\label{lb109}
Let $\fk X_0=(C;x_1,\dots,x_N;\eta_1,\dots,\eta_N)$ be an $N$-pointed compact Riemann surface with local coordinates. Write $x_1=x$ and $\eta_1=\eta$ for simplicity. Let $\eta$ be defined on a neighborhood $U=U_1\ni x_1$ disjoint from $x_2,\dots,x_N$. Assume that $\eta(U)$ is an open  disk centered at $0$ with radius $>1$.

Let $h$ be a holomorphic function on a neighborhood of $\Sbb^1$. Then $\xk=h\partial_z$ is a holomorphic tangent field near $\Sbb^1$. We choose $0<r<1<R$ such that $h$ is defined on an open set containing the closure of $A_{r,R}=\{z\in\Cbb:r<|z|<R\}$.  Moreover, we choose a connected neighborhood $\Delta\subset\Cbb$ of $0$ such that the following hold.
\begin{enumerate}
\item There is a neighborhood $\Delta\subset\Cbb$ of $0$ such that the holomorphic flow  $\tau\in\Delta\mapsto\exp(\tau \xk)=\exp_{\tau\xk}$ is defined on $(z,\tau)\in A_{r,R}\times\Delta$ and is injective on $z\in A_{r,R}$ for any fixed $\tau$.  (Cf. Subsec. \ref{lb12}.)
\item For each $\tau\in\Delta$, we have $0\notin\exp_{\tau\xk}(\Sbb^1)$.
\end{enumerate}

Let $\Gamma_\tau$ be the simple closed curve $\exp_{\tau\xk}:\Sbb^1\rightarrow\Cbb$. Then by the Jordan curve theorem, for each $\tau\in\Delta$, $\Pbb^1\setminus\Gamma_\tau$ has two connected components
\begin{align*}
\Pbb^1\setminus\Gamma_\tau=\Omega_\tau\sqcup\Omega_\tau' 
\end{align*}
where $\Omega_\tau'$ is the one containing $\infty$. In the following, we give some technical remarks which can be skipped on first reading:
\begin{itemize}
\item By Stokes' theorem, for each $z\in\Pbb^1$, $z\in \Omega_\tau$ (resp. $z\in\Omega_\tau'$) iff $\oint_{\Gamma_\tau}\frac {d\zeta}{\zeta-z}$ equals $2\im\pi$ (resp. 0). This implies that
\begin{align*}
O=\{(z,\tau)\in\Pbb^1\times\Delta:z\in\Omega_\tau\}\qquad O'=\{(z,\tau)\in\Pbb^1\times\Delta:z\in\Omega_\tau'\}
\end{align*}
are both closed and open inside $\Pbb^1\times\Delta$. In summary: the property that $z$ is inside (resp. outside) $\Gamma_\tau$ is continuous with respect to the variation of $\tau$ and $z$.
\item Consequently, for each $z\in A_{r,R}\setminus\Sbb^1$, the subset of all $\tau\in\Delta$ such that  $\exp_{\tau\xk}(z)$ belongs to $\Omega_\tau$ (resp. $\Omega_\tau'$) is an open subset of $\Delta$, and hence also closed, and hence must be $\emptyset$ or $\Delta$. This shows that for each $z\in A_{r,R}$,
\begin{gather}\label{eq208}
\begin{gathered}
|z|<1\qquad\Longleftrightarrow\qquad\exp_{\tau\xk}(z)\in\Omega_\tau\text{ for all }\tau\in \Delta\\
|z|>1\qquad\Longleftrightarrow\qquad\exp_{\tau\xk}(z)\in\Omega_\tau'\text{ for all }\tau\in \Delta
\end{gathered}
\end{gather}
A similar argument shows that if $z\in\Pbb^1$ and $z\notin\exp_{\tau\xk}(\Sbb^1)$ for all $\tau\in\Delta$, then
\begin{gather}\label{eq209}
\begin{gathered}
|z|<1\qquad\Longleftrightarrow\qquad z\in\Omega_\tau\text{ for all }\tau\in \Delta\\
|z|>1\qquad\Longleftrightarrow\qquad z\in\Omega_\tau'\text{ for all }\tau\in \Delta
\end{gathered}
\end{gather}
In particular, $0\in\Omega_\tau$ for all $\tau\in\Delta$.
\end{itemize}

The  family $\fk X$ we shall construct has base manifold $\Delta$. For each $\tau\in\Delta$, let
\begin{align*}
\mc R_\tau=\exp_{\tau\xk}(A_{r,R})\cup \Omega_\tau.
\end{align*}
Then the fiber $\mc C_\tau$ is obtained by gluing $C\setminus\eta^{-1}(\Dbb_r^\cl)$ with $\mc R_\tau$ by identifying the subsets $\eta^{-1}(A_{r,R})$ and $\exp_{\tau\xk}(A_{r,R})$ via the biholomorphism $\exp_{\tau\xk}\circ\eta$. (We leave it to the readers to check that $\mc C_\tau$ is a compact Riemann surface. \eqref{eq208} is needed when checking the sequential compactness.) 

\begin{align}
\vcenter{\hbox{{
\includegraphics[height=2cm]{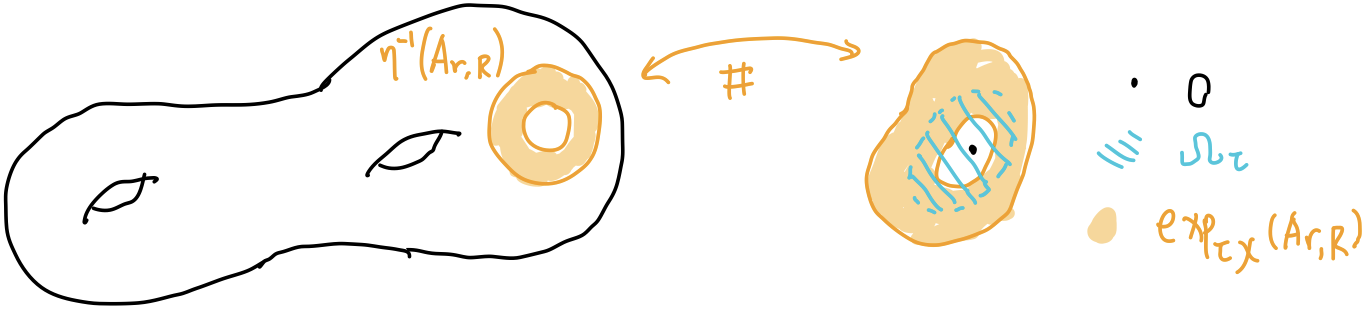}}}}
\end{align}

The marked points of $\mc C_\tau$, together with local coordinates, are chosen to be $0\in\mc R_\tau$ with the standard coordinate $\zeta$ of $\mc R_\tau\subset\Cbb$, and $x_2,\dots,x_N\in C\setminus\eta^{-1}(\Dbb_r^\cl)$ together with $\eta_2,\dots,\eta_N$. This gives an $N$-pointed compact Riemann surface with local coordinates $\fk X_\tau$. We leave it to the readers to construct a family $\fk X$ over $\Delta$ whose fibers are $\fk X_\tau$. \hfill\qedsymbol
\end{eg}

\subsection{}\label{lb110}
In the previous example, suppose we associate $\Vbb$-modules $\Wbb_1,\dots,\Wbb_N$ to $0\in\mc R_\tau,x_2,\dots,x_N$ respectively, and let $\upphi_\tau$ denote a conformal block associated to $\fk X_\tau$. Let $\xk=\sum_{n\in\Zbb} c_nz^{n+1}\partial_z$.  $\fk X_0$ is changed to $\fk X_\tau$ by changing the local coordinate $\eta$ of $C$ at $x_1$ to the first one of $\fk X_\tau$, which is $\exp_{\tau\xk}\circ\eta$ when restricted to $C\setminus\eta^{-1}(\Dbb_r^\cl)$. Thus, intuitively,  for each given $\upphi_0$, one can construct $\upphi_\tau$ using the formal expression
\begin{align}
\upphi_\tau(w_\blt)=\upphi_0(e^{-\tau\sum_n c_nL_n}w_1\otimes w_2\otimes\cdots\otimes w_N)\label{eq245}
\end{align}
thanks to the change of boundary parametrization formula. This expression actually converges in certain good cases, e.g. when $c_n=0$ for sufficiently negative $n$. (See Example \ref{lb149}.) In the case where the expression converges, the map $\upphi_0\mapsto\upphi_\tau$ defines a linear map between the spaces of conformal blocks $\scr T^*_{\fk X_0}(\Wbb_\blt)\rightarrow\scr T^*_{\fk X_\tau}(\Wbb_\blt)$, which is an isomorphism since the operator $e^{-\tau\sum_n c_nL_n}$ is invertible. In particular, the dimensions of $\scr T^*_{\fk X_0}(\Wbb_\blt)$ and $\scr T^*_{\fk X_\tau}(\Wbb_\blt)$ are equal. As we will see in Sec. \ref{lb153}, for an arbitrary family, we will prove the equidimensionality of spaces of conformal blocks for fibers as well as the local freeness of sheaves of conformal blocks by this method, and a crucial step is to prove the convergence of $\upphi_\tau$.

$\upphi_\tau$ satisfies the differential equation $\partial_\tau\upphi_\tau+\sum c_n\upphi_\tau\circ(L_n\otimes\id_{\Wbb_2}\otimes\cdots\otimes \id_{\Wbb_N})=0$. This fact can be rephrased by saying that
\begin{align}
\nabla_{\partial_\tau}:=\partial_\tau+\sum_n c_n (L_n\otimes\id_{\Wbb_2}\otimes\cdots\otimes \id_{\Wbb_N})^\tr
\end{align}
defines a natural connection $\nabla$  on the sheaf of conformal blocks (associated to $\fk X$) over $\Delta$, and $\tau\mapsto\upphi_\tau$ is a parallel section under this connection.

\subsection{}

Example \ref{lb109} can be easily generalized to the case that on each neighborhood of $\Sbb^1$ around $x_i$ a holomorphic vector field $\xk_i$ is associated. The flows generated by these fields define a family.

We now consider another important generalization of Example \ref{lb109}:

\begin{eg}\label{lb111}
Let $\fk X_0,U$ be as in Example \ref{lb109}. Let $\Delta$ be a connected neighborhood of $0\in\Cbb$. We choose a neighborhood $\Delta\subset\Cbb$ of $0$ and an annulus $A_{r,R}$ where $0<r<1<R$, and choose a holomophic function $\beta=\beta_\zeta(z)$ on $(z,\zeta)\in A_{r,R}\times\Delta$ such that the following hold:
\begin{enumerate}
\item $\beta_0(z)=z$.
\item For each $\zeta\in\Delta$, $\beta_\zeta$ is injective on $A_{r,R}$.
\item For each $\zeta\in\Delta$, we have $0\notin\beta_\zeta(\Sbb^1)$.
\end{enumerate}

As in Example \ref{lb109}, we let $\Gamma_\zeta=\beta_\zeta(\Sbb^1)$, which divides $\Pbb^1$ into two connected components $\Omega_\zeta\ni 0$ and $\Omega_\zeta'\ni\infty$. Let
\begin{align*}
\mc R_\zeta=\beta_\zeta(A_{r,R})\cup\Omega_\zeta.
\end{align*}
Then one can construct a family $\fk X$ with base manifold $\Delta$ such that each fiber $\mc C_\zeta$ is obtained by gluing $C\setminus\eta^{-1}(\Dbb_r^\cl)$ with $\mc R_\zeta$ by identifying $\eta^{-1}(A_{r,R})$ and $\beta_\zeta(A_{r,R})$ via the biholomorphism $\beta_\zeta\circ\eta$. The marked points and the local coordinates of $\mc C_\zeta$ are chosen in the same way as at the end of Example \ref{lb109}. 

As in Example \ref{lb109}, $O=\{(z,\zeta)\in\Pbb^1\times\Delta:z\in\Omega_\zeta\}$ is an open set. Let
\begin{align*}
\mc A=\{(\beta_\zeta(z),\zeta)\in \Cbb\times\Delta:z\in A_{r,R}\}.
\end{align*}
Then, as a family, $\mc C$ is obtained by gluing $\Delta\times(C\setminus\eta^{-1}(\Dbb_r^\cl))$ and $\mc A\cup O$ such that on each fiber the gluing is as in the previous paragraph. \hfill\qedsymbol
\end{eg}

In the above example, the standard local coordinate at $0\in\mc R_\zeta$ is the boundary parametrization $\beta_\zeta\circ\eta$ on $C$. So $\fk X_0$ is changed to $\fk X_\zeta$ by changing $\eta$ to $\beta_\zeta\circ\eta$. Thus, we make the following definition:
\begin{df}\label{lb115}
We say that $\fk X_\zeta$ is the $N$-pointed compact Riemann surface with local coordinates obtained by changing the local coordinate $\eta=\eta_1$ of $\fk X_0$ at $x=x_1$ to the boundary parametrization $\beta_\zeta\circ\eta$.
\end{df}

\subsection{}\label{lb116}

Let $\upphi_\zeta$ be a conformal block associated to $\fk X_\zeta$ and $\Wbb_\blt$ where $\fk X$ is constructed in Example \ref{lb111}. As in Subsec. \ref{lb110}, let us find the differential equation that $\upphi_\zeta$ satisfies.

Let $\mc U(\beta_\zeta)$ be the (not yet rigorously defined) operator associated to the change of parametrization $\beta_\zeta$. Then according to the change of boundary parametrization formula in Sec. \ref{lb41}, 
\begin{align}
\upphi_\zeta\big(\mc U(\beta_\zeta)w_1\otimes w_2\otimes\cdots\otimes w_N\big)=\upphi_0(w_\blt).\label{eq211}
\end{align}

Let us find the formula for $\partial_\zeta\mc U(\beta_\zeta)$. Choose a holomorphic function $h=h(z,\zeta)$ on a neighborhood of $\Sbb^1\times\Delta$ in $A_{r,R}\times\Delta$ such that 
\begin{align}
\partial_\zeta\beta_\zeta(z)=h\big(\beta_\zeta(z),\zeta\big).\label{eq210}
\end{align}
Since $\beta$ is not necessarily the flow generated by a vector field, $h$ depends on $\zeta$. One may view $\beta_\zeta$ as the path generated by the time-dependent vector field $h\partial_z$.

We first consider $\partial_\zeta\mc U(\beta_\zeta)$ at $\zeta=0$. Recall $\beta_0(z)=z$. Similar to the explanation in Rem. \ref{lb112}, the velocity of $\beta_\zeta$ at $\zeta=0$ is the vector field $\partial_\zeta\beta_\zeta(z)\partial_z\big|_{\zeta=0}$, which according to \eqref{eq210} is $h(z,0)\partial_z$. Writing $h$ as
\begin{align}
h(z,\zeta)=\sum_{n\in\Zbb}h_n(\zeta)z^n\label{eq216}
\end{align}
(where $h_n\in\scr O(\Delta)$). Then $h(z,0)\partial_z=\sum_{n\in\Zbb}h_n(0)z^n\partial_z$, which corresponds to $\sum_n h_n(0)L_{n-1}$. This should be the formula for $\partial_\zeta \mc U(\beta_\zeta)$ at $\zeta=0$.

For an arbitrary $\zeta\in\Delta$, we find the formula of $\partial_\zeta\mc U(\beta_\zeta)=\partial_\lambda\mc U(\beta_{\lambda+\zeta})\big|_{\lambda=0}$ using the same method. Write $\mc U(\beta_{\lambda+\zeta})=\mc U(\beta_{\lambda+\zeta}\circ\beta_\zeta^{-1})\circ\mc U(\beta_\zeta)$. Then the $\partial_\lambda$ of $\beta_{\lambda+\zeta}\circ\beta_\zeta^{-1}$ at $\lambda=0$ is $(\partial_\zeta\beta)\circ\beta_\zeta^{-1}$, which by \eqref{eq210} equals $h(z,\zeta)$. Thus
\begin{align}
\partial_\zeta\mc U(\beta_\zeta)=\sum_{n\in\Zbb} h_n(\zeta)L_{n-1}\mc U(\beta_\zeta).
\end{align}
Thus, since the derivative of the LHS of \eqref{eq211} is zero, we conclude that $\upphi_\zeta$ is killed by $\partial_\zeta+\sum_n h_n(\zeta)(L_{n-1}\otimes\id_{\Wbb_2}\otimes\cdots\otimes \id_{\Wbb_N})^\tr$. Equivalently, $\upphi_\zeta$ is parallel under the connection $\nabla$ defined by
\begin{align}
\nabla_{\partial_\zeta}=\partial_\zeta+\sum_{n\in\Zbb} h_n(\zeta)\big(L_{n-1}\otimes\id_{\Wbb_2}\otimes\cdots\otimes \id_{\Wbb_N}\big)^\tr.
\end{align}

\subsection{}\label{lb144}

The importance of Example \ref{lb111} (or its generalization to the case that around each $x_i$ there is a $\beta$) is that any family with $1$-dimensional base manifold is locally of this form. Let us explain this fact in more details.

Let $\fk X=(\pi:\mc C\rightarrow\mc B;\sgm_\blt;\eta_\blt)$ be a family of $N$-pointed compact Riemann surfaces with local coordinates. Recall that by our convention, $\Theta_{\mc C}$ and $\Theta_{\mc B}$ are respectively holomorphic tangent bundle of $\mc C$ and $\mc B$. Let $U\subset \mc C$ be open, and $\xk\in\Theta_{\mc C}(U)$. Note that $W=\pi(U)$ is open. Choose $\yk\in\Theta_{\mc B}(W)$. We say that $\xk$ is a \textbf{lift} \index{00@A lift $\xk$ of $\yk\in H^0(\mc B,\Theta_{\mc B})$} of $\yk$ if for each $x\in\mc C$, the differential map $d\pi:\Theta_{\mc C}|_x\rightarrow\Theta_{\mc B}|_{\pi(x)}$ between tangent spaces sends $\xk(x)$ to $\yk(\pi(x))$. 

If we have $\eta\in\scr O(U)$ univalent on each fiber $U_b=U\cap\mc C_b$ (where $b\in\mc B$) \index{Ub@$U_b=\mc C_b\cap U$} of $U$, then the relationship between $\xk$ and $\yk$ can be written in an explicit way. 

\begin{ass}\label{lb113}
Assume $W$ is biholomorphic to an open subset of $\Cbb^m$ via a map $\tau_\blt=(\tau_1,\dots,\tau_m):W\rightarrow\Cbb^m$. Identify $W$ with $\tau_\blt(W)$ so that $\tau_\blt$ are identified with the standard coordinates of $\Cbb^m$. Note that $(\eta,\pi)$ is a biholomorphism between $U$ and an open subset of $\Cbb\times\Cbb^{m+1}$. We identify $U$ with $(\eta,\pi)(U)$ so that $\eta$ becomes the standard coordinate $z$ of $\Cbb$, and $\pi$ becomes the standard coordinates $\tau_\blt$ of $\Cbb^m$.
\end{ass}

Then we can write $\xk$ and $\yk$ as
\begin{subequations}
\begin{gather}
\yk=\sum_{j=1}^m g_j(\tau_\blt)\partial_{\tau_j}\qquad \xk=h(z,\tau_\blt)\partial_z+\sum_{j=1}^m g_j(\tau_\blt)\partial_{\tau_j}
\end{gather}
\end{subequations}
where $g_j\in\scr O(W),h\in\scr O(U)$. From this formula, it is clear that if $\exp_{\zeta\yk}$ sends $b$ to $b'$, then $\exp_{\zeta\xk}$ sends points of $W_b=W\cap\mc C_b$ to those of $W_{b'}$ provided that the flows can be defined on the points. Namely, $\exp_{\zeta\xk}$ preserves fibers.

Recall $\SX=\bigcup_i\sgm_i(\mc B)$. For each vector bundle $\scr E$ on $\mc C$ and each $k\in\Zbb$, we let $\scr E(k\SX)$ \index{ESX@$\scr E(k\SX),\scr E(\star\SX)$} be the sheaf whose sections on any open $U\subset\mc C$ are all $s\in\mc E(U\setminus\SX)$ such that for each $i$, $\eta_i^k\cdot s$ can be extended to a section of $\scr E$ on a neighborhood of $\sgm_i(\mc B)$. Then $\scr E(k\SX)$ is a locally free $\scr O_{\mc C}$-module. We let
\begin{align*}
\scr E(\star\SX)=\varinjlim_{k\in\Nbb}\scr E(k\SX).
\end{align*}
So for $k\geq0$, $\scr E(k\SX)$ is the sheaf of sections of $\scr E$ with poles of order at most $k$ at $\SX$, and  $\scr E(\star\SX)$ is the sheaf of sections of $\scr E$ with finite poles at $\SX$.

\begin{pp}\label{lb114}
Assume that $\mc B$ is a Stein manifold. Then each $\yk\in H^0(\mc B,\Theta_{\mc B})$ has a lift $\xk$ in $H^0(\mc C,\Theta_{\mc C}(\star\SX))$. 
\end{pp}

We do not explain the meaning of Stein manifolds in our notes, but refer the interested readers to \cite[Sec. I.4]{GR-a} or \cite[Sec. III.3]{GR-b}  for details. Here, we only give some examples, which are sufficient for applications. Stein manifolds are complex manifolds including the following examples: \index{00@Stein manifolds }
\begin{itemize}
\item Every non-compact connected Riemann surface.
\item A finite product of Stein manifolds.
\item A finite intersection of Stein open subsets of a complex manifold.
\item A closed complex submanifold of a Stein manifold.
\item If $X$ is Stein and $f\in\scr O(X)$ then $X\setminus\{x\in X:f(x)=0\}$ is Stein.
\end{itemize}
From these examples, it is clear that the Stein open subsets of a  complex manifold $X$ form a basis of the topology of $X$.

Stein manifolds are those that many local problems related to vector bundles have global solutions. In Prop. \ref{lb114}, if $\mc B$ is not necessarily Stein, then a lift of $\xk$ always exists locally (i.e., after shrinking $\mc B$). The global existence is due to the Stein property. We refer the readers to \cite{Gui}[Sec. 3.6] for a detailed explanation of Prop. \ref{lb114}.

\subsection{}\label{lb126}

We now assume that $\mc B$ is a connected Stein open subset of $\Cbb^m$ containing $0$, and let $\tau_\blt$ be the standard coordinates of $\Cbb^m$. Choose
\begin{align}
\yk=\sum_j g_j(\tau_\blt)\partial_{\tau_j}\qquad \in H^0(\mc B,\Theta_{\mc B})\label{eq215}
\end{align}
where $g_j\in\scr O(\mc B)$, and let $\xk\in H^0(\mc C,\Theta_{\mc C}(\star\SX))$ be a lift. 

For each $1\leq i\leq N$, let $U_i\subset\mc C$ be a neighborhood of $\sgm_i(\mc B)$ on which $\eta_i$ is defined, and assume $U_i$ intersects only $\sgm_i(\mc B)$ among $\sgm_1(\mc B),\dots\sgm_N(\mc B)$. Then, after identifying $U_i$ with its image under $(\eta_i,\pi)$ as in Assumption \ref{lb113} (which is a neighborhood of $\{0\}\times\mc B$), we can write
\begin{align}
\xk|_{U_i}=h_i(z,\tau_\blt)\partial_z+\sum_j g_j(\tau_\blt)\partial_{\tau_j}\label{eq213}
\end{align}
where $h_i$ is holomorphic on $U_i\setminus\sgm_i(\mc B)=(\eta_i,\pi)(U_i)\setminus(\{0\}\times\mc B)$ and has finite poles on  $\{0\}\times\mc B$, i.e., $z^nh_i(z,\tau_\blt)$ is holomorphic on $U_i$ for some $n\in\Nbb$. We then have Laurent series expansion
\begin{align}
h_i(z,\tau_\blt)=\sum_{n\in\Zbb}h_{i,n}(\tau_\blt)z^n\label{eq214}
\end{align}
converging a.l.u. on $U_i\setminus\sgm_i(\mc B)$, where $h_{i,n}$ is a zero function for sufficiently negative $n$.

\subsection{}

We continue our discussion from the previous subsection. We claim that if we restrict $\mc B$ to the complex curve $\zeta\mapsto \exp_{\zeta\yk}(0)$ so that the base manifold of $X$ is $1$-dimensional, then $\fk X$ can be described by Example \ref{lb111}. 

To see this, let us assume for simplicity that $\eta_i(U_i\cap\mc C_0)\supset \Dbb_1^\cl$ for all $i$, and choose $0<r<1$. Let
\begin{align*}
\mc C_0^+=\mc C_0\setminus\bigcup_i \eta_i^{-1}(\Dbb_r^\cl),
\end{align*}
which plays the same role as $C\setminus\eta^{-1}(\Dbb_r^\cl)$ in Example \ref{lb111}. Consider the flow $\exp_{\zeta\xk}$ on $\mc C\setminus\SX$ generated by $\xk$. We choose a neighborhood $\Delta$ of $0\in\Cbb$ such that  $\exp_{\zeta\xk}$ is defined and injective on $\mc C_0^+$ for all $\zeta\in\Delta$. Let 
\begin{align*}
b_\zeta=\exp_{\zeta\yk}(0)\qquad\in\Cbb^m.
\end{align*}
Then $\exp_{\zeta\xk}(\mc C_0^+)$ is inside $\mc C_{b_\zeta}$. So $\mc C_{b_\zeta}$ has an open submanifold $\exp_{\zeta\xk}(\mc C_0^+)$ equivalent to $\mc C_0^+$. Note that $b_0=0$.
\begin{align}\label{eq237}
\vcenter{\hbox{{
\includegraphics[height=3cm]{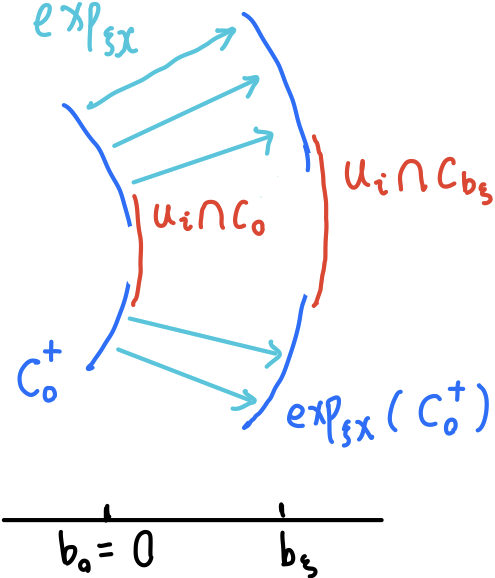}}}}
\end{align}


$\mc C_{b_\zeta}$ can be viewed as gluing $\exp_{\zeta\xk}(\mc C_0^+)$ with $U_1\cap \mc C_{b_\zeta},\dots,U_N\cap\mc C_{b_\zeta}$, and clearly the function $\eta_i$ on $U_i\cap\mc C_{b_\zeta}$ becomes $\eta_i$ on an annulus in $\exp_{\zeta\xk}(\mc C_0^+)$. Equivalently, $\mc C_{b_\zeta}$ is the gluing of $\mc C_{b_\zeta}$ with all $U_i\cap\mc C_{b_\zeta}$ such that the $\eta_i$ on $U_i\cap\mc C_{b_\zeta}$ becomes the function $\eta_i\big|_{\mc C_{b_\zeta}}\circ\exp_{\zeta\xk}$ on an annulus inside $\mc C_0^+$. It is not hard to see that on that annulus,
\begin{align}
\eta_i\big|_{\mc C_{b_\zeta}}\circ\exp_{\zeta\xk}=\beta^i_\zeta\circ\eta_i\big|_{\mc C_0}
\end{align}
where $\beta^i_\zeta(z)=\alpha^i_\zeta(z,0)$ and $(\eta_i,\pi)\circ\exp_{\zeta\xk}\circ (\eta_i,\pi)^{-1}(z,\tau_\blt)$ equals $(\alpha^i_\zeta(z,\tau_\blt),\exp_{\zeta\yk}(\tau_\blt))$. Namely, $\alpha^i$ is determined by the fact that under the identification of $U_i$ with $(\eta_i,\pi)(U_i)$ via $(\eta_i,\pi)$, 
\begin{align}
\exp_{\zeta\xk}(z,\tau_\blt)=(\alpha^i_\zeta(z,\tau_\blt),\exp_{\zeta\yk}(\tau_\blt)).\label{eq236}
\end{align}

\begin{ccl}
$\fk X_{b_\zeta}$ is obtained by changing the local coordinates $\eta_1\big|_{\mc C_0},\dots,\eta_N\big|_{\mc C_0}$ of $\fk X_{b_0}=\fk X_0$ to the boundary parametrizations $\beta^1_\zeta\circ\eta_1\big|_{\mc C_0},\dots,\beta^N_\zeta\circ\eta_N\big|_{\mc C_0}$. (Cf. Def. \ref{lb115}.)
\end{ccl}

\subsection{}\label{lb117}

That $\exp_{\zeta\xk}$ is the flow generated by $\xk$ means that $\partial_\zeta \big(f\circ\exp_{\zeta\xk}\big)$ equals  $(\xk f)\circ\exp_{\zeta\xk}$. Take $f=\eta_i$, and identify $U$ with its image under $(\eta_i,\pi)$ to simplify the situation. Then by \eqref{eq213},
\begin{align}
\partial_\zeta\alpha^i_\zeta(z,\tau_\blt)=h_i\big(\alpha^i_\zeta(z,\tau_\blt),\exp_{\zeta\yk}(\tau_\blt)\big),
\end{align}
and hence
\begin{align}
\partial_\zeta\beta^i_\zeta(z)=h_i\big(\beta^i_\zeta(z),b_\zeta\big).
\end{align}

Let $\upphi_{\tau_\blt}$ be a conformal block associated to $\fk X_{\tau_\blt}$ for each $\tau_\blt\in\mc B$. Recall the Laurent series expansion \eqref{eq214}. Similar to the reasoning in Subsec. \ref{lb116}, we have
\begin{align}
\partial_\zeta\upphi_{b_\zeta}(w_\blt)+\sum_{i=1}^N\sum_{n\in\Zbb}h_{i,n}(b_\zeta)(w_1\otimes\cdots\otimes L_{n-1}w_i\otimes\cdots\otimes w_N)=0.
\end{align}
By \eqref{eq215}, we have $\partial_\zeta b_\zeta=(g_1(b_\zeta),\dots,g_m(b_\zeta))$. Thus
\begin{align*}
\partial_\zeta\upphi_{b_\zeta}(w_\blt)=\sum_{j=1}^m g_j(\tau_\blt)\partial_{\tau_j}\upphi_{\tau_\blt}(w_\blt)\big|_{\tau_\blt=b_\zeta}=\yk\upphi_{\tau_\blt}(w_\blt)\big|_{\tau_\blt=b_\zeta}.
\end{align*}
We conclude that on the complex path $\zeta\in\Delta\mapsto b_\zeta=\exp_{\zeta\yk}(0)$,
\begin{align}
\sum_{j=1}^m g_j(\tau_\blt)\partial_{\tau_j}\upphi_{\tau_\blt}(w_\blt)+ \sum_{i=1}^N\sum_{n\in\Zbb}h_{i,n}(\tau_\blt)\upphi_{\tau_\blt}(w_1\otimes\cdots\otimes L_{n-1}w_i\otimes\cdots\otimes w_N)=0.
\end{align}
This fact can be rephrased as follows: on the complex path $\zeta\mapsto b_\zeta$, $\upphi_{\tau_\blt}$ is parallel under the connection $\nabla_\yk$ defined by
\begin{align}
\nabla_{\yk}=\underbrace{\sum_{j=1}^m g_j\partial_{\tau_j}}_{\yk}+\sum_{i=1}^N\sum_{n\in\Zbb} h_{i,n}\cdot\big(\id_{\Wbb_1}\otimes\cdots\otimes L_{n-1}\big|_{\Wbb_i}\otimes\cdots\otimes\id_{\Wbb_N}\big)^\tr.\label{eq217}
\end{align}

\subsection{}

We close this section by giving some examples of lifts. 

\begin{eg}\label{lb148}
Let $\fk X$ the family in Example \ref{lb111}. Let $h(z,\tau)$ be defined by \eqref{eq210} whose Laurent series expansion with respect to $z$ (cf. \eqref{eq216}) has only finitely many negative powers of $z$.  

Let $\yk\in H^0(\Delta,\Theta_\Delta)$ be $\partial_\tau$ where $\tau$ is the standard coordinate of $\Cbb$. Recall (cf. the end of Example \ref{lb111}) that $\mc C$ is the gluing of $\Delta\times(C\setminus\eta^{-1}(\Dbb_r^\cl))$ and $\mc A\cup O$, where the latter is an open subset of $\Cbb\times\Cbb$. We define the lift $\xk$ to be the canonical one $\partial_\tau$ on $\Delta\times(C\setminus\eta^{-1}(\Dbb_r^\cl))$, i.e., the one parallel to the $\Delta$-component  and hence orthogonal to the $(C\setminus\eta^{-1}(\Dbb_r^\cl))$-component. Then on $\mc A\cup O$, using the standard coordinates $(z,\tau)$ of $\Cbb\times\Cbb$, $\xk$ is $h(z,\tau)\partial_z+\partial_\tau$, which has finite poles at $z=0$. This shows that $\xk\in H^0(\mc C,\Theta_{\mc C}(\star\SX))$, and that (not surprisingly) the  $\nabla_\yk$ defined as in Subsec. \ref{lb117} agrees with that in Subsec. \ref{lb116}.  \hfill\qedsymbol
\end{eg}

\begin{eg}\label{lb180}
In Example \ref{lb118}, let $(\tau_1,\dots,\tau_N)$ be the standard coordinates of the base manifold $\Conf^N(\Cbb^\times)$ inherited from $\Cbb^N$. Let $\yk=\partial_{\tau_k}$ where $1\leq k\leq N$. Then the lift $\xk$ can be chosen to the standard one $\partial_{\tau_k}$, i.e., the one orthogonal to the $\Pbb^1$-component in the Cartesian product $\mc C=\Pbb^1\times\Conf^N(\Cbb^\times)$. Then $\xk\in H^0(\mc C,\Theta_{\mc C})$.

Using the notations and the identification in \eqref{eq213}, we have $\xk|_{U_i}=\partial_{\tau_k}$ if $i\neq k$, and
\begin{align}
\xk|_{U_k}=-\partial_z+\partial_{\tau_k}.
\end{align}
Associate $\Wbb_0,\Wbb_1,\dots,\Wbb_N,\Wbb_\infty$ to the marked points $0,\sgm_1,\dots,\sgm_N,\infty$ of $\fk P^N$. Then by \eqref{eq217}, the conformal blocks are parallel under
\begin{align}
\nabla_{\partial_{\tau_k}}=\partial_{\tau_k}-\big(\id_{\Wbb_0}\otimes\id_{\Wbb_1}\otimes\cdots\otimes L_{-1}\big|_{\Wbb_k}\otimes\cdots\otimes\id_{\Wbb_N}\otimes\id_{\Wbb_\infty}\big)^\tr.
\end{align}
\end{eg}

\section{Sheaves of coinvariants and conformal blocks, and their connections}\label{lb136}

\subsection{}\label{lb127}

We study conformal blocks for families of compact Riemann surfaces in a rigorous way. In this section and the next one, we let
\begin{align*}
\fk X=(\pi:\mc C\rightarrow\mc B;\sgm_1,\dots,\sgm_N)
\end{align*}
be a family of $N$-pointed compact Riemann surface. Associate admissible $\Vbb$-modules $\Wbb_1\dots,\Wbb_N$ to $\sgm_1,\dots,\sgm_N$ respectively. Choose a  neighborhood $U_i\subset\mc C$ of $\sgm_i(\mc B)$ disjoint from $\sgm_j(\mc B)$ if $i\neq j$. If we choose local coordinate $\eta_i$ at $\sgm_i$, we assume $\eta_i$ is defined on $U_i$.

\subsection{}

For each $n\in\Nbb$, let us define a vector bundle $\scr V^{\leq n}_{\fk X}$ on $\mc C$ whose restriction to each fiber $\mc C_b$ is the bundle $\scr V^{\leq n}_{\mc C_b}$ defined in Subsec. \ref{lb119}. Let $U\subset\mc C$ be open, and choose $\eta,\mu\in\scr O(U)$ univalent on each fiber $U_b=U\cap\mc C_b$ of $U$. For each $p\in U$, we define $\varrho(\eta|\mu)_p\in\Gbb$ to be
\begin{align}
\varrho(\eta|\mu)_p=\varrho\Big(\eta|_{\mc C_{\pi(p)}}\Big|\mu|_{\mc C_{\pi(p)}}\Big).
\end{align}
Namely, for each $x\in U_{\pi(p)}$, \index{zz@$\varrho(\alpha\lvert\id),\varrho(\eta\lvert\mu)$}
\begin{align}
\eta(x)-\eta(p)=\varrho(\eta|\mu)_p\big(\mu(x)-\mu(p)\big).
\end{align}
The map $\varrho(\eta|\mu):p\in U\rightarrow\varrho(\eta|\mu)_p\in\Gbb$ is clearly a holomorphic family of transformations. Thus, by Rem. \ref{lb96}, we have an equivalence of $\scr O_U$-modules
\begin{align}
\mc U(\varrho(\eta|\mu)):\Vbb^{\leq n}\otimes\scr O_U\xrightarrow{\simeq}\Vbb^{\leq n}\otimes\scr O_U.
\end{align}

Similar to the case of a single compact Riemann surface, we define $\scr V^{\leq n}_{\fk X}$ \index{VX@$\scr V^{\leq n}_{\fk X},\scr V_{\fk X}$} to be the locally free $\scr O_{\mc C}$-module such that each open $U\subset\mc C$ with $\eta\in\scr O(U)$ univalent on each fiber is associated with a trivialization \index{U@$\mc U_\varrho(\eta)$}
\begin{align}
\mc U_\varrho(\eta):\scr V^{\leq n}_{\fk X}|_U\xrightarrow{\simeq}\Vbb^{\leq n}\otimes\scr O_U
\end{align}
such that $\mc U_\varrho(\eta|_V)=\mc U_\varrho(\eta)|_V$ for any open $V\subset U$, and that for any $\mu\in\scr O(U)$ univalent on each fiber of $U$, the transition function is given by
\begin{align}
\mc U_\varrho(\eta)\mc U_\varrho(\mu)^{-1}=\mc U(\varrho(\eta|\mu)).
\end{align}
We let $\scr V_{\fk X}=\varinjlim_{n\in\Nbb}\scr V_{\fk X}^{\leq n}$. Both $\scr V_{\fk X}$ and $\scr V_{\fk X}^{\leq n}$ are called \textbf{sheaves of VOAs associated to $\fk X$ and $\Vbb$}.

\subsection{}

$\Theta_{\mc C}$ and $\omega_{\mc C}$ have ranks $\dim\mc B+1$. So their restrictions to each fiber $\mc C_b$ are not $\Theta_{\mc C_b}$ and $\omega_{\mc C_b}$. We consider instead the line bundle $\Theta_{\mc C/\mc B}$ \index{zz@$\Theta_{\mc C/\mc B}$} of sections of $\Theta_{\mc C}$ killed by $d\pi$ (i.e., tangent to each fiber), called the \textbf{relative tangent sheaf}. It's dual bundle is denoted by $\omega_{\mc C/\mc B}$ \index{zz@$\omega_{\mc C/\mc B}$} and called the \textbf{relative dualizing sheaf}. Then we have natural equivalences
\begin{align}
\Theta_{\mc C/\mc B}|_{\mc C_b}\simeq\Theta_{\mc C_b},\qquad \omega_{\mc C/\mc B}|_{\mc C_b}\simeq\omega_{\mc C_b}.
\end{align}
Sections of $\omega_{\mc C/\mc B}(U)$ are of the form $fd\eta$ where $f\in\scr O(U)$ and $\eta\in\scr O(U)$ is univalent on each fiber. For another $\mu\in\scr O(U)$ univalent on each fiber, we have transformation rule
\begin{align}
fd\eta=f\cdot\frac{\partial\eta}{\partial\mu}d\mu
\end{align}
where the tangent field $\frac{\partial}{\partial\mu}$ of $\mc C$ is perpendicular to $d\pi$, i.e. tangent to the fibers. Similar to Prop. \ref{lb178}, we have a natural equivalence
\begin{align}
\scr V_{\fk X}^{\leq n}/\scr V_{\fk X}^{\leq n-1}\simeq \Vbb(n)\otimes_\Cbb\Theta_{\mc C/\mc B}^{\otimes n}.
\end{align}

For each $b\in\mc B$, let \index{SXb@$\SXb=\{\sgm_1(b),\dots,\sgm_N(b)\}$}
\begin{align}
\SXb=\{\sgm_\blt(b)\}=\{\sgm_1(b),\dots,\sgm_N(b)\}.
\end{align}
Then for each $k\in\Zbb$, we have an obvious equivalence of vector bundles.
\begin{align}
\scr V_{\fk X}^{\leq n}\otimes\omega_{\mc C/\mc B}(k\SX)\big|{\mc C_b}\simeq\scr V_{\fk X_b}^{\leq n}\otimes\omega_{\mc C_b}(k\SXb).
\end{align}
(If the readers know how to define the restrictions of sheaves that are not necessarily (finite rank) vector bundles, they can easily check that the above equation holds if the superscript $\leq n$ is removed and $k$ is replaced by $\star$.)

\subsection{}

Given $\eta\in\scr O(U)$ univalent on each fiber, we have an obvious equivalence
\begin{align}
(\eta,\pi)_*:\scr O_U\xrightarrow{\simeq}\scr O_{(\eta,\pi)(U)}
\end{align}
defined by pulling back functions using $(\eta,\pi)^{-1}$. We define the \textbf{pushforward}\footnote{A better notation would be $\mc V_\varrho(\eta,\pi)$. However, we use $\mc V_\varrho(\eta)$ to make the notation shorter.} \index{V@$\mc V_\varrho(\eta_i),\mc V_\varrho(\varphi)$}
\begin{gather}\label{eq233}
\begin{gathered}
\mc V_\varrho(\eta):\scr V_{\fk X}|_U\xrightarrow{\simeq}\Vbb\otimes\scr O_{(\eta,\pi)(U)}\\
\mc V_\varrho(\eta)=(\id_\Vbb\otimes(\eta,\pi)_*)\mc U_\varrho(\eta).
\end{gathered}
\end{gather}
Its restriction to each fiber $U_b=U\cap\mc C_b$ equals the pushforward $\mc V_\varrho(\eta|_{\mc C_b}):\scr V_{U_b}\xrightarrow{\simeq}\Vbb\otimes\scr O_{\eta(U_b)}$ defined by \eqref{eq218}.

We have an equivalence $(\eta,\pi)_*=((\eta,\pi)^{-1})^*:\omega_{\mc C/\mc B}|_U\rightarrow \omega_{(\eta,\pi)(U)/\pi(U)}$. Note that $(\eta,\pi)(U)\subset \Cbb\times\mc B$. $\omega_{(\eta,\pi)(U)/\pi(U)}$ is the relative dualizing sheaf associated to the family $(\eta,\pi)(U)\rightarrow\pi(U)$ inherited from the projection $\Cbb\times\mc B\rightarrow\mc B$. If we let $z$ be the standard coordinate of $\Cbb$, then for each section $fd\eta\in\omega_{\mc C/\mc B}|_U$ where $f\in\scr O_{\mc C}$,
\begin{align*}
(\eta,\pi)_*fd\eta=\big(f\circ(\eta,\pi)^{-1}\big)dz.
\end{align*}
We let $\mc V_\varrho(\eta)$ also denote
\begin{align}
\mc V_\varrho(\eta)\equiv\mc V_\varrho(\eta)\otimes (\eta,\pi)_*:\scr V_{\fk X}\otimes\omega_{\mc C/\mc B}\big|_U\xrightarrow{\simeq} \Vbb\otimes_\Cbb\omega_{(\eta,\pi)(U)/\pi(U)}.\label{eq219}
\end{align}

\subsection{}


To define sheaves of coinvariants and conformal blocks, we first consider the case that local coordinates $\eta_1,\dots,\eta_N$ at $\sgm_1,\dots,\sgm_N$ are chosen and defined on $U_1,\dots,U_N$.

For each open $V\subset\mc B$, let \index{CV@$\mc C_V=\pi^{-1}(V)$}
\begin{align}
\mc C_V=\pi^{-1}(V),
\end{align}
and we have an $\scr O(V)$-linear action of $H^0(\mc C_V,\scr V_{\fk X}\otimes\omega_{\mc C/\mc B}(\star\SX))$ on $\Wbb_\blt\otimes\scr O(V)$ whose restriction to each fiber is the residue action of $H^0(\mc C_b,\scr V_{\fk X_b}\otimes\omega_{\mc C_b}(\star\SXb))$ on $\Wbb_\blt$ defined by Def. \ref{lb120}. So this action is compatible with the restriction to open subsets of $V$.

Let us describe this action in more details. Suppose $\sigma\in H^0(U_i\cap\mc C_V,\scr V_{\fk X}\otimes\omega_{\mc C/\mc B}(\star\SX))$. Note that by \eqref{eq219} we have (noting $\pi(U_i)=\mc B$)
\begin{align}
\mc V_\varrho(\eta_i):\scr V_{\fk X}\otimes\omega_{\mc C/\mc B}(\star\SX)\big|_{U_i}\xrightarrow{\simeq} \Vbb\otimes_\Cbb\omega_{(\eta_i,\pi)(U_i)/\mc B}(\star \{0\}\times\mc B)
\end{align}
since $\sgm_i(\mc B)$ is the only one of $\sgm_\blt(\mc B)$ intersecting (and also inside) $U_i$, and $(\eta_i,\pi)$ sends $\sgm_i(\mc B)$ to $\{0\}\times\mc B$ (cf. \eqref{eq220}). Then for each $w_i\in\Wbb_i\otimes\scr O(V)$ (we regard $w_i=w_i(b)$ as a $\Wbb_i$-valued holomorphic functions on $V$), we define \textbf{residue action}
\begin{align}
\sigma\cdot w_i=\Res_{z=0}~Y\big(\mc V_\varrho(\eta_i)\sigma,z \big)w_i.\label{eq226}
\end{align}

More precisely, write
\begin{align}
\mc V_\varrho(\eta_i)\sigma=v(z,b)dz=\sum_{n\in\Zbb}v_n(b)z^ndz
\end{align}
where $v=v(z,b)$ is a $\Vbb$-valued holomorphic function on $(\eta_i,\pi)(U_i\cap \mc C_V)$ (which is a neighborhood of $\{0\}\times V$ in $\Cbb\times V$), $v_n=v_n(b)$ is in $\Vbb\otimes\scr O(V)$, and $v_n=0$ for sufficiently negative $n$. (So $\sigma$ equals $vdz$ if we identify $U_i\cap\mc C_V$ with its image under $(\eta_i,\pi)$, and identify $\scr V_{\fk X}|_{U_i\cap\mc C_V}$ with $\Vbb\otimes\scr O_{U_i\cap\mc C_V}$ via $\mc U_\varrho(\eta_i)$.) Then
\begin{align}
(\sigma\cdot w_i)(b)=\Res_{z=0}~Y(v(z,b),z)w_i(b)dz=\sum_{n\in\Zbb} Y\big(v_n(b)\big)_nw_i(b).
\end{align}

Now, any element of $\Wbb_\blt\otimes_\Cbb\scr O(V)$ is a ($\Cbb$-)linear combination of $\Wbb_\blt$-valued holomorphic functions $w_\blt$ where (for each $b\in V$)
\begin{align}
w_\blt(b)=w_1(b)\otimes_\Cbb\cdots\otimes_\Cbb w_N(b)\qquad\in\Wbb_\blt\label{eq228}
\end{align}
and each $w_i$ is an $\Wbb_i$-valued holomorphic function on $V$. Alternatively, \index{w@$w_\blt=w_1\otimes\cdots\otimes w_N$ as an element $\Wbb_\blt\otimes_\Cbb\scr O(V)$} 
\begin{align}
w_\blt=w_1\otimes_{\scr O(V)}\cdots\otimes_{\scr O(V)}w_N\label{eq229}
\end{align}
is in
\begin{align*}
(\Wbb_1\otimes_\Cbb\scr O(V))\otimes_{\scr O(V)}\cdots\otimes_{\scr O(V)}(\Wbb_N\otimes_\Cbb\scr O(V))\simeq \Wbb_\blt\otimes_\Cbb\scr O(V).
\end{align*}
So the expression $w_\blt=w_1\otimes\cdots\otimes w_N$ can be understood in an unambiguous way. The \textbf{residue action} of any $\sigma\in H^0(\mc C_V,\scr V_{\fk X}\otimes\omega_{\mc C/\mc B}(\star\SX))$ on $w_\blt$ is given by
\begin{align}
\sigma\cdot w_\blt=\sum_{i=1}^N w_1\otimes\cdots\otimes \sigma\cdot w_i\otimes\cdots\otimes w_N.
\end{align}
(It is sufficient to understand this action when $w_\blt$ is a constant function, i.e., $w_\blt\in\Wbb_\blt$.)

\subsection{}\label{lb125}
Define an infinite-rank vector bundle over $\mc B$: \index{W@$\scr W(\Wbb_i)$, $\scr W_{\fk X}(\Wbb_\blt)$}
\begin{align}
\scr W_{\fk X}(\Wbb_\blt)=\Wbb_\blt\otimes_\Cbb\scr O_{\mc B}.
\end{align}
Define an $\scr O(V)$-module \index{JX@$\scr J_{\fk X}(\Wbb_\blt),\scr J_{\fk X_b}(\Wbb_\blt)$}
\begin{align}
\scr J_{\fk X}(\Wbb_\blt)(V)=H^0\big(\mc C_V,\scr V_{\fk X}\otimes\omega_{\mc C/\mc B}(\star\SX)\big)\cdot H^0\big(V,\scr W_{\fk X}(\Wbb_\blt)\big).\label{eq221}
\end{align}
where we have suppressed $\Span_\Cbb$. Then we have a presheaf of $\scr O_{\mc B}$-modules whose space of sections on any open $V\subset\mc B$ is $\scr J_{\fk X}(\Wbb_\blt)(V)$. This is a sub-presheaf of $\scr W_{\fk X}(\Wbb_\blt)$.

\begin{df}
The $\scr O_{\mc B}$-module \index{T@$\scr T_{\fk X}(\Wbb_\blt),\scr T_{\fk X}^*(\Wbb_\blt)$}
\begin{align}
\scr T_{\fk X}(\Wbb_\blt)=\frac{\scr W_{\fk X}(\Wbb_\blt)}{\scr J_{\fk X}(\Wbb_\blt)}
\end{align}
(defined by sheafifying the presheaf $V\mapsto \frac{\scr W_{\fk X}(\Wbb_\blt)(V)}{\scr J_{\fk X}(\Wbb_\blt)(V)}$) and its dual $\scr O_{\mc B}$-module $\scr T_{\fk X}^*(\Wbb_\blt)$ are called respectively the \textbf{sheaf of coinvariants} and the \textbf{sheaf of conformal blocks} associated to $\fk X$ and $\Wbb_\blt$. Sections of $\scr T_{\fk X}^*(\Wbb_\blt)(\mc B)$ are called \textbf{conformal blocks associated to $\fk X$ and $\Wbb_\blt$}.
\end{df}

\subsection{}

Let us give an explicit description of $\scr T_{\fk X}^*(\Wbb_\blt)$. The following is easy to see:
\begin{rem}
Sections of $\scr T_{\fk X}^*(\Wbb_\blt)$ over $V$ are all morphisms $\upphi:\scr W_{\fk X}(\Wbb_\blt)|_V\rightarrow\scr O_V$ that vanish when evaluated with any section of $\scr J_{\fk X}(\Wbb_\blt)|_V$, i.e., $\upphi(s)=0$ for all $s\in\scr J_{\fk X}(\Wbb_\blt)(V_1)$ where $V_1\subset V$ is open.
\end{rem}

\begin{rem}
A morphism $\upphi:\scr W_{\fk X}(\Wbb_\blt)|_V=\Wbb_\blt\otimes_\Cbb\scr O_V\rightarrow\scr O_V$ is equivalently a linear map $\Phi:\Wbb_\blt\rightarrow\scr O(V)$. Indeed, given $\upphi$, we define $\Phi$ to be $\Phi(w)=\upphi(w)\in\scr O(V)$ where each $w\in\Wbb_\blt$ is identified with the constant section $w\otimes 1\in\Wbb_\blt\otimes_\Cbb\scr O(V)$. Conversely, given $\Phi$, we define $\upphi$ sending each $w\otimes f\in\Wbb_\blt\otimes\scr O(V_1)$ (where $V_1\subset V$ is open) to $f\cdot \Phi(w)|_{V_1}$.

Thus, for each $b\in V$, the fiber map $\upphi|_b:\scr W_{\fk X}(\Wbb_\blt)|_b\simeq\Wbb_\blt\rightarrow \scr O_V|_b\simeq\Cbb$ is given by $w\in\Wbb_\blt\simeq\Wbb_\blt\otimes 1\mapsto \upphi(w)(b)$ where $\upphi(w)(b)$ is the value of $\upphi(w)\in\scr O(V)$ at $b$.  \hfill\qedsymbol
\end{rem}

\subsection{}

We can now relate conformal blocks for families and for single complex Riemann surfaces. For simplicity, we assume $V=\mc B$; otherwise we just need to restrict $\fk X$ to the subfamily $\fk X_V$ with base manifold $V$.

\begin{pp}\label{lb121}
Choose an $\scr O_{\mc B}$-module morphism $\upphi:\scr W_{\fk X}(\Wbb_\blt)\rightarrow\scr O_{\mc B}$. If $\mc B$ is a Stein manifold, then $\upphi$ vanishes on $\scr J_{\fk X}(\Wbb_\blt)(\mc B)$ if and only if the restriction $\upphi|_b$ to the fiber $\mc C_b$ is a conformal block for each $b\in\mc B$, i.e., $\upphi|_b$ vanishes on \index{JX@$\scr J_{\fk X}(\Wbb_\blt),\scr J_{\fk X_b}(\Wbb_\blt)$}
\begin{align}
\scr J_{\fk X_b}(\Wbb_\blt)=H^0\big(\mc C_b,\scr V_{\mc C_b}\otimes\omega_{\mc C_b}(\star\SXb)\big)\cdot \Wbb_\blt
\end{align}
\end{pp}

\begin{proof}

This follows from the fact that any element of $\scr J_{\fk X_b}(\Wbb_\blt)$ is the restriction of an element of $\scr J_{\fk X}(\Wbb_\blt)(\mc B)$ due to the next proposition.
\end{proof}

\begin{pp}\label{lb122}
Let $V$ be a Stein open subset of $\mc B$. Then every element of $H^0\big(\mc C_b,\scr V_{\mc C_b}\otimes\omega_{\mc C_b}(\star\SXb)\big)$ is the restriction of some $\sigma\in H^0\big(\mc C_V,\scr V_{\fk X}\otimes\omega_{\mc C/\mc B}(\star\SX)\big)$ to the fiber $\mc C_b$.
\end{pp}

In this proposition, we do not assume that $\fk X$ has local coordinates $\eta_\blt$. To prove this proposition one needs the base change theorem of Grauert \cite[Sec. III.4.2]{GR-b}. See  \cite[Sec. 2.5]{Gui} for a detailed explanation. It is in general true that if $\scr E$ is a vector bundle on $\mc C$, then for any precompact Stein open subset $V\subset\mc B$, there exists $k_0\in\Nbb$ such that for all $k\geq k_0$, every element of $H^0(\mc C,\scr E(k\SX))$ is the restriction of some $\sigma \in H^0(\mc C_V,\scr E(k\SX)|_V)$ to the fiber $\mc C_b$.

From Prop. \ref{lb121} we immediately get:

\begin{thm}\label{lb143}
Choose an $\scr O_{\mc B}$-module morphism $\upphi:\scr W_{\fk X}(\Wbb_\blt)\rightarrow\scr O_{\mc B}$. Then $\upphi$ is a conformal block iff $\upphi|_b$ is a conformal block for each $b\in\mc B$. If $\mc B$ is Stein, then these two conditions are also equivalent to that $\upphi$ vanishes on $\scr J_{\fk X}(\Wbb_\blt)(\mc B)$.
\end{thm}

We give an application of Thm. \ref{lb143}. We remark that Thm. \ref{lb143} and Cor. \ref{lb158} hold without assuming that $\fk X$ has local coordinates (after we define sheaves of conformal blocks in this general case, cf. Subsec. \ref{lb190}), since Prop. \ref{lb122} does.

\begin{co}\label{lb158}
Assume that $\mc B$ is connected. Let $\upphi:\scr W_{\fk X}(\Wbb_\blt)\rightarrow\scr O_{\mc B}$ be an $\scr O_{\mc B}$-module morphism. Assume that $\mc B$ contains a non-empty open subset $V$ such that the restriction $\upphi|_V$ is a conformal block associated to $\fk X_V$ (i.e., $\upphi|_V\in H^0(V,\scr T_{\fk X}^*(\Wbb_\blt))$). Then $\upphi$ is a conformal block associated to $\fk X$.
\end{co}

\begin{proof}
First, assume $\mc B$ is Stein. Then the evaluation of $\upphi$ with any element of $H^0(\mc B,\scr J_{\fk X}(\Wbb_\blt))$ (which is an element of $\scr O(\mc B)$) vanishes on $V$, and hence vanishes on $\mc B$ by complex analysis. So, by Cor. \ref{lb158}, $\upphi$ is a conformal block.

Now, in the general case, we let $\mc B_0$ be the (obviously open) subset of $\mc B$ consisting all $b\in\mc B$ such that $\upphi$ restricts to a conformal block on a neighborhood of $b$. If $b\in\mc B\setminus\mc B_0$, then  every connected Stein neighborhood $W$ of $b$ is disjoint from $\mc B_0$: Otherwise, since $\upphi|_{W\cap\mc B_0}$ is a conformal block, by the first paragraph, $\upphi|_{\mc B_0}$ is a conformal block, which implies $b\in\mc B_0$ and gives a contradiction. So $\mc B_0$ is a non-empty open and closed subset of $\mc B$, which must be $\mc B$. 
\end{proof}

\subsection{}

There are two advantages of working with sheaves of coinvariants instead of sheaves of conformal blocks. First, it is easier to relate the fibers of $\scr T_{\fk X}(\Wbb_\blt)$ and spaces of coinvariants than to do so for sheaves and spaces of conformal blocks. Second, though our ultimate interest lies in  the local freeness of $\scr T_{\fk X}^*(\Wbb_\blt)$, it is easier to first study the local freeness of $\scr T_{\fk X}(\Wbb_\blt)$.

For each $b\in\mc B$, note that
\begin{align*}
\scr T_{\fk X}(\Wbb_\blt)_b=\frac{\scr W_{\fk X}(\Wbb_\blt)_b}{\scr J_{\fk X}(\Wbb_\blt)_b}.
\end{align*}
Let $\fk m_b=\fk m_{\mc B,b}=\{g\in\scr O_{\mc B,b}:g(b)=0\}$. Then we have an obvious equivalence
\begin{align}
\scr T_{\fk X}(\Wbb_\blt)\big|_b=\frac{\scr T_{\fk X}(\Wbb_\blt)_b}{\fk m_b\cdot \scr T_{\fk X}(\Wbb_\blt)_b}\simeq \frac{\scr W_{\fk X}(\Wbb_\blt)_b}{\fk m_b\cdot \scr W_{\fk X}(\Wbb_\blt)_b+ \scr J_{\fk X}(\Wbb_\blt)_b}.
\end{align}
Recall also that
\begin{align}
\scr T_{\fk X_b}(\Wbb_\blt)=\frac{\Wbb_\blt}{\scr J_{\fk X_b}(\Wbb_\blt)}.
\end{align}

\begin{pp}\label{lb134}
The linear map 
\begin{gather}\label{eq222}
\begin{gathered}
\scr W_{\fk X}(\Wbb_\blt)_b=\Wbb_\blt\otimes_\Cbb\scr O_{\mc B,b}\rightarrow \Wbb_\blt\\
w\mapsto w(b)
\end{gathered}
\end{gather}
descends to an isomorphism of vector spaces
\begin{align}\label{eq223}
\scr T_{\fk X}(\Wbb_\blt)\big|_b\xlongrightarrow{\simeq} \scr T_{\fk X_b}(\Wbb_\blt).
\end{align}
\end{pp}

\begin{proof}
The map \eqref{eq222} sends $\fk m_b\cdot \scr W_{\fk X}(\Wbb_\blt)_b=\Wbb_\blt\otimes \fk m_b$ to $0$ and sends $\scr J_{\fk X}(\Wbb_\blt)_b$ into $\scr J_{\fk X_b}(\Wbb_\blt)$ (indeed onto by Prop. \ref{lb122}). So \eqref{eq222} descends to a linear map \eqref{eq223} which is clearly surjective. If $w(b)\in\scr J_{\fk X_b}(\Wbb_\blt)$, then by Prop. \ref{lb122}, $w(b)$ equals $s(b)$ for some $s\in\scr J_{\fk X}(\Wbb_\blt)_b$. So $w-s\in\Wbb_\blt\otimes \scr O_{\mc B,b}$ vanishes at $b$. So clearly $w-s\in\Wbb_\blt\otimes \fk m_b$. Therefore \eqref{eq223} is injective.
\end{proof}

\subsection{}\label{lb190}

Now we do not assume that the local coordinates of $\fk X$ are chosen. We shall define sheaves of coinvariants and conformal blocks associated to $\fk X$ and $\Wbb_\blt$.

Let $\scr W_{\fk X}(\Wbb_\blt)$ be an infinite rank locally free $\scr O_{\mc B}$-module \index{W@$\scr W(\Wbb_i)$, $\scr W_{\fk X}(\Wbb_\blt)$} determined by the following conditions. For any open subset $V\subset\mc B$ together with local coordinates $\eta_1,\dots,\eta_N$ of the restricted family \index{XV@$\fk X_V$}
\begin{align}
\fk X_V=(\pi:\mc C_V=\pi^{-1}(V)\rightarrow V;\sgm_1|_V,\dots,\sgm_N|_V)
\end{align}
defined near $\sgm_1(V),\dots,\sgm_N(V)$ respectively, we have a trivialization
\begin{align}
\mc U(\eta_\blt)\equiv \mc U(\eta_1)\otimes\cdots\mc U(\eta_N):\scr W_{\fk X}(\Wbb_\blt)|_V\xrightarrow{\simeq}\Wbb_\blt\otimes_\Cbb\scr O_V
\end{align}
compatible with the restriction of $\eta_\blt$ and $\fk X_V$ to open subsets of $V$, such that if $\mu_\blt$ is another set of local coordinates, then \index{U@$\mc U(\alpha),\mc U(\eta),\mc U(\eta_\blt)$}
\begin{align*}
\mc U(\eta_\blt)\mc U(\mu_\blt)^{-1}:\Wbb_\blt\otimes\scr O_V\xrightarrow{\simeq}\Wbb_\blt\otimes\scr O_V
\end{align*}
is defined by the transition function
\begin{align}
\mc U(\eta_\blt)\mc U(\mu_\blt)^{-1}\equiv \mc U(\eta_\blt|\mu_\blt)=\mc U((\eta_1|\mu_1))\otimes\cdots\otimes\mc U((\eta_N|\mu_N)).
\end{align}
Here, each $(\eta_i|\mu_i):V\rightarrow\Gbb$ is a holomorphic family of transformations such that for each $b\in V$, $(\eta_i|\mu_i)_b$ changes $\mu_i|_{\mc C_b}$ to $\eta_i|_{\mc C_b}$, i.e., 
\begin{align*}
\eta_i|_{\mc C_b}=(\eta_i|\mu_i)_b\circ \mu_i|_{\mc C_b}
\end{align*}
holds on a neighborhood of $\sgm_i(b)$ in $\mc C_b$. So we have an isomorphism
\begin{align*}
\mc U((\eta_i|\mu_i)):\Wbb_\blt\otimes\scr O_V\xrightarrow{\simeq}\Wbb_\blt\otimes\scr O_V.
\end{align*}

The restriction of $\mc U((\eta_i|\mu_i))$ to each fiber at $b$ is clearly the transition function for $\scr W_{\fk X_b}(\Wbb_\blt)$ (cf. \eqref{eq225}). Thus, we have an obvious equivalence
\begin{align}
\scr W_{\fk X}(\Wbb_\blt)\big|_b\simeq \scr W_{\fk X_b}(\Wbb_\blt).
\end{align}

We define the (obviously $\scr O(V)$-linear) \textbf{residue action} of $\sigma\in H^0(\mc C_V,\scr V_{\fk X}\otimes\omega_{\mc C/\mc B}(\star\SX))$ on $\wbf=H^0(V,\scr W_{\fk X}(\Wbb_\blt))$ to be
\begin{align}
\sigma\cdot\wbf=\mc U(\eta_\blt)\cdot\sigma\cdot \mc U(\eta_\blt)^{-1}\wbf\label{eq227}
\end{align}
where the action of $\sigma$ on $\mc U(\eta_\blt)^{-1}\wbf$ is defined by \eqref{eq226}.
When restricted to each fiber,  \eqref{eq227} is equivalent to the residue action of $H^0(\mc C_b,\scr V_{\mc C_b}\otimes\omega_{\mc C_b}(\star\SXb))$ on $\scr W_{\fk X_b}(\Wbb_\blt)$ defined as in \eqref{eq194}. Since the later is coordinate-independent (cf. Prop. \ref{lb123}), so is \eqref{eq227}.

Thus, using the residue action, we can define the presheaf $\scr J_{\fk X}(\Wbb_\blt)$, the sheaf of coinvariants $\scr T_{\fk X}(\Wbb_\blt)$, and the sheaf of conformal blocks $\scr T_{\fk X}^*(\Wbb_\blt)$ \index{JX@$\scr J_{\fk X}(\Wbb_\blt),\scr J_{\fk X_b}(\Wbb_\blt)$} \index{T@$\scr T_{\fk X}(\Wbb_\blt),\scr T_{\fk X}^*(\Wbb_\blt)$} in the exact same way as in Subsec. \ref{lb125}.

\subsection{}

Our next goal is to define connections on $\scr T_{\fk X}(\Wbb_\blt)$ and $\scr T_{\fk X}^*(\Wbb_\blt)$. We begin with the following general definition:

\begin{df}
Let $\scr E$ be an $\scr O_X$-module where $X$ is a complex manifold with holomorphic tangent line bundle $\Theta_X$. A \textbf{connection} $\nabla$ on $\scr E$ associates to each open $U\subset X$ a bilinear map
\begin{align*}
\nabla:\Theta_X(U)\times\scr E(U)\rightarrow\scr E(U),\qquad (\yk,s)\mapsto\nabla_\yk s
\end{align*}
satisfying the following conditions.
\begin{enumerate}[label=(\alph*)]
\item If $V$ is an open subset of $U$ then $\nabla_{\yk|_V}s|_V=(\nabla_{\yk}s)|_V$.
\item If $f\in\scr O(U)$ then
\begin{subequations}
\begin{gather*}
\nabla_{f\yk}s=f\nabla_{\yk}s\\
\nabla_\yk(fs)=\yk(f)s+f\nabla_\yk s
\end{gather*}
\end{subequations}
\end{enumerate}
If a connection $\nabla$ on $\scr E$ is chosen, the corresponding \textbf{dual connection} $\nabla$ on the dual sheaf $\scr E^\vee$ is defined by
\begin{align}
\bk{\nabla_\yk\varphi,s}=\yk\bk{\varphi,s}-\bk{\varphi,\nabla_\yk s}\label{eq230}
\end{align}
for each $\varphi\in \scr E^\vee(U)=\Hom_{\scr O_U}(\scr E_U,\scr O_U)$, each $\yk\in\Theta(U)$, and each $s\in\scr E_U$.
\end{df}

Note that $\xk\bk{\varphi,s}$ is the action of the vector field $\xk$ on the holomorphic function $\bk{\varphi,s}$.

\subsection{}

We now suppose that the local coordinates $\eta_\blt$ are chosen for $\fk X$, and identify
\begin{align*}
\scr W_{\fk X}(\Wbb_\blt)=\Wbb_\blt\otimes_\Cbb\scr O_{\mc B}\qquad\text{via }\mc U(\eta_\blt).
\end{align*}
We assume that $\mc B$ is a Stein manifold. Choose $\yk\in\Theta_{\mc B}(\mc B)$, together with a lift $\xk\in H^0(\mc C,\Theta_{\mc C}(\star\SX))$. (Cf. Prop. \ref{lb114}).   We first define the \textbf{differential operator $\nabla_\yk$ on $\scr W_{\fk X}(\Wbb_\blt)$}. 

\emph{$\boxed{\text{Assume the setting of Subsec. \ref{lb126}.}}$} Then for each open $V\in\mc B$, $\nabla_\yk$ is the linear operator on $\Wbb_\blt\otimes_\Cbb\scr O(V)$ such that for each $w_i\in\Wbb_i\otimes_\Cbb\scr O(V)$ and $w_\blt=w_1\otimes\cdots\otimes w_N$ in $\Wbb_\blt\otimes_\Cbb\scr O(V)$ (cf. \eqref{eq228} or \eqref{eq229}),
\begin{align}
\nabla_\yk w_\blt=\sum_{j=1}^m  g_j(\tau_\blt)\partial_{\tau_j} w_\blt-\sum_{i=1}^N\sum_{n\in\Zbb}h_{i,n}(\tau_\blt)w_1\otimes\cdots\otimes L_{n-1}w_i\otimes\cdots\otimes w_N.\label{eq231}
\end{align}
Using this formula and \eqref{eq230}, we can define $\nabla_\yk$ on the dual sheaf of $\scr W_{\fk X}(\Wbb_\blt)$, i.e., define $\nabla_\yk\upphi$ for each $\scr O_V$-module morphism $\upphi:\scr W_{\fk X}(\Wbb_\blt)|_V\rightarrow\scr O_V$. This definition of $\nabla_\yk\upphi$ clearly agrees with \eqref{eq217} when $w_1,\dots,w_N$ are constant sections.

Warning: we are using $L_0$ instead of $\wtd L_0$ to define $\nabla_\yk$.

\begin{rem}
In Subsec. \ref{lb126} we assumed that $\mc B$ is inside $\Cbb^m$. In other words, when defining $\nabla_\yk$ using \eqref{eq231}, we have fixed an embedding of the abstract complex manifold $\mc B$ into $\Cbb^m$ as an open subset. However, it is easy to check that this definition is independent of the embedding. Thus, to define $\nabla_\yk$, we assume only that $\mc B$ is Stein, but not necessarily that $\mc B$ can be embedded into $\Cbb^m$.
\end{rem}

\subsection{}

\eqref{eq231} can be written in a more compact way. Recall the neighborhood $U_i$  of $\sgm_i(\mc B)$ on which $\eta_i$ is defined (cf. Subsec. \ref{lb127}). Define
\begin{gather}
\begin{gathered}
\upnu(\xk)\in H^0\big(U_1\cup\cdots\cup U_N,\scr V_{\fk X}\otimes\omega_{\mc C/\mc B}(\star\SX)\big)\\
\mc V_\varrho(\eta_i)\upnu(\xk)|_{U_i}=h_i(z,\tau_\blt)\cbf dz.
\end{gathered}
\end{gather}
Namely, under the given trivialization, $\upnu$ kills $\partial_{\tau_j}$ and sends $\partial_z$ to $\cbf dz$. (Note that $\cbf\in\Vbb(2)$ and $\scr V^{\leq 2}_{\fk X}/\scr V^{\leq 1}_{\fk X}\simeq\Theta_{\mc C/\mc B}^{\otimes 2}$.) Then it is easy to verify that
\begin{align}
\nabla_\yk w_\blt=\sum_{j=1}^m  g_j(\tau_\blt)\partial_{\tau_j} w_\blt-\upnu(\xk)\cdot w_\blt.\label{eq232}
\end{align}
where $\upnu(\xk)\cdot w_\blt$ is the residue action.

\subsection{}

\begin{thm}\label{lb132}
$\nabla_\yk$ preserves $\scr J_{\fk X}(\Wbb_\blt)(V)$ for each open $V\subset\mc B$. So $\nabla_\yk$ is a linear operator on $\scr T_{\fk X}(\Wbb_\blt)$ and (via the formula \eqref{eq230}) on $\scr T^*_{\fk X}(\Wbb_\blt)$.

More precisely, for each $\sigma\in H^0(\mc C_V,\scr V_{\fk X}\otimes\omega_{\mc C/\mc B}(\star\SX))$ and $\wbf=H^0(V,\scr W_{\fk X}(\Wbb_\blt))$, we have
\begin{align}
[\nabla_\yk,\sigma]\wbf=(\mc L_\xk \sigma)\cdot\wbf
\end{align}
where $\mc L_\xk\sigma\in H^0(\mc C_V,\scr V_{\fk X}\otimes\omega_{\mc C/\mc B}(\star\SX))$ is the Lie derivative of $\sigma$ under $\xk$.
\end{thm}

Thus, when $\mc B$ is a Stein open subset of $\Cbb^m$, we may define a connection $\nabla$ on $\scr T_{\fk X}(\Wbb_\blt)$ and $\scr T_{\fk X}^*(\Wbb_\blt)$ by choosing lifts of $\partial_{\tau_1},\dots,\partial_{\tau_m}$, defining $\nabla_{\tau_1},\dots,\nabla_{\tau_m}$, and then extending $\nabla$ to a connection using $\scr O_{\mc B}$-linearity.

We refer the readers to \cite[Sec. 3.6]{Gui} for the proof of this theorem. Here, we explain the meaning of Lie derivative. 

\subsection{}

Let $U,W\subset\mc C$ be open, and let $\varphi:U\rightarrow W$ be a biholomorphism from $U$ onto $W$. We assume that $\varphi$ preserves fibers, i.e. $\varphi(U_{\pi(p)})=W_{\pi\circ\varphi(p)}$ for each $p\in U$. (recall our notation that $W_b=U\cap\mc C_b,W_b=W\cap\mc C_b$ for each $b\in\mc B$). For instance, if $U\subset\mc C\setminus\SX$ is open and precompact, then for sufficiently small $\zeta$,  $\exp_{\zeta\xk}$ from $U$ to its image preserves fibers. (See \eqref{eq237} for the figure.)

The \textbf{pushforward} \index{V@$\mc V_\varrho(\eta_i),\mc V_\varrho(\varphi)$}
\begin{align*}
\mc V_\varrho(\varphi):\scr V_{\fk X}|_U\xrightarrow{\simeq}\scr V_{\fk X}|_W
\end{align*}
is defined such that for each $\eta\in\scr O(W)$ univalent on fibers, noting the pushforward $\mc V_\varrho(\eta):\scr V_{\fk X}|_W\xrightarrow{\simeq}\Vbb\otimes_\Cbb\scr O_{(\eta,\pi)(W)}$ defined by \eqref{eq233}, we have
\begin{align}
\mc V_\varrho(\eta)\mc V_\varrho(\varphi)=\mc V_\varrho(\eta\circ\varphi).
\end{align}
Then for each $b\in\mc B$, the restriction of $\mc V_\varrho(\varphi)$ to  $\scr V_{\fk X}|_{U_b}\xrightarrow{\simeq}\scr V_{\fk X}|_{V_{\varphi(b)}}$ is equivalent to the pushforward $\mc V_\varrho(\varphi):\scr V_{U_b}\xrightarrow{\simeq}\scr V_{V_{\varphi(b)}}$ defined in Subsec. \ref{lb128}. 

By tensoring $\mc V_\varrho(\varphi)$ with $\varphi_*=(\varphi^{-1})^*:\omega_{\mc C/\mc B}|_U\xrightarrow{\simeq}\omega_{\mc C/\mc B}|_W$ sending $(f\circ\varphi)d(\eta\circ\varphi)$ to $fd\eta$ where $f\in\scr O_W$, we get a pushforward which we also denote by $\mc V_\varrho(\varphi)$:
\begin{align}
\mc V_\varrho(\varphi)\equiv \mc V_\varrho(\varphi)\otimes\varphi_*:\scr V_{\fk X}\otimes\omega_{\mc C/\mc B}\big|_U\xrightarrow{\simeq}\scr V_{\fk X}\otimes\omega_{\mc C/\mc B}\big|_W.
\end{align}

We can define the \textbf{Lie derivative} in the same way as Def. \ref{lb129}. Let $\xk$ be  as in Subsec. \ref{lb126}. Suppose $U\subset \mc C\setminus\SX$ is open and precompact, and $\sigma\in H^0(U,\scr V_{\fk X}\otimes\omega_{\mc C/\mc B})$. Define \index{L@$\mc L_\xk$, the Lie derivative}
\begin{align}\label{eq234}
\mc L_\xk\sigma\big|_U=\lim_{\zeta\rightarrow0} \frac{\mc V_\varrho(\exp_{\zeta\xk})^{-1}\big(\sigma\big|_{\exp_{\zeta\xk}(U)}\big)-\sigma\big|_U}{\zeta}
\end{align}
Of course, if we can show that the limit exists for all precompact $U$, then $\mc L_{\xk}\sigma$ exists for all open $U\subset\mc C\setminus\SX$.

The following Proposition can be proved in the same way as Cor. \ref{lb130}. (Or see \cite[Sec. 2.6]{Gui} for details.) Formula \eqref{eq235} is necessary for the proof of Theorem \ref{lb132}.
\begin{pp}\label{lb131}
Let $\eta\in\scr O(U)$ be univalent on fibers. Choose $u\in H^0(U,\Vbb\otimes_\Cbb\scr O_{\mc C})$ such that
\begin{align*}
u\cdot d\eta=\mc U_\varrho(\eta)\sigma\qquad\in H^0(U,\Vbb\otimes_\Cbb\omega_{\mc C/\mc B}).
\end{align*}
Choose $h\in\scr O(U)$ such that if $U$ is identified with $(\eta,\pi)(U)\subset\Cbb\times\Cbb^m$ via $(\eta,\pi)$, then 
\begin{align*}
\xk|_U=h\partial_z+\sum_{j=1}^mg_j(\tau_\blt)\partial_{\tau_j}.
\end{align*}
Then $\mc L_\xk\sigma$ exists as an element of $H^0(U,\scr V_{\fk X}\otimes\omega_{\mc C/\mc B})$, and
\begin{align}
\mc U_\varrho(\eta)\mc L_\xk\sigma=h\partial_\eta u\cdot d\eta+\sum_{j=1}^m g_j\partial_{\tau_j}u\cdot d\eta-\sum_{k\geq1} \frac 1{k!}\partial_\eta^k h\cdot L_{k-1}u\cdot d\eta+\partial_\eta h\cdot u\cdot d\eta.\label{eq235}
\end{align}
\end{pp}

\begin{rem}
If we set $U=U_i\setminus\SX=U_i\setminus\sgm_i(\mc B)$, choose $\sigma\in H^0(\mc C_V,\scr V_{\fk X}\otimes\omega_{\mc C/\mc B}(\star\SX))$, and let $\eta$ be the local coordinate $\eta_i$, then the $u$ in Prop. \ref{lb131} has finite poles at $\SX$, i.e., $u\in H^0(U_i,\Vbb\otimes_\Cbb\omega_{\mc C/\mc B}(\star\SX))$. The $h$ in Prop. \ref{lb131} should be the $h_i$ in Subsec. \ref{lb126}, which has finite poles at $\SX$. Therefore, by \eqref{eq235},  the Lie derivative $\mc L_\xk\sigma$, as a section of $\scr V_{\fk X}\otimes\omega_{\mc C/\mc B}$ defined on $\mc C_V\setminus\SX$, has finite poles at $\SX$. So $\mc L_\xk\sigma\in H^0(\mc C_V,\scr V_{\fk X}\otimes\omega_{\mc C/\mc B}(\star\SX))$, as claimed at the end of Thm. \ref{lb132}.
\end{rem}

\subsection{}

Recall that we are assuming $\mc B$ is Stein (but not necessarily open inside $\Cbb^m$) and local coordinates $\eta_\blt$ are given to $\fk X$. As we have seen, the definition of $\nabla_\yk$ depends not only on $\eta_\blt$ but also on the lift $\xk$ of $\yk\in\Theta_{\mc B}(\mc B)$.

\begin{pp}\label{lb152}
Let $\nabla_\yk$ and $\nabla'_\yk$ be defined by $\eta_\blt$ and two lifts $\xk,\xk'\in H^0(\mc C,\Theta_{\mc C}(\star\SX))$ of $\yk$. Then there exists $f\in \scr O(\mc B)$ depending only on $\fk X$, the local coordinates $\eta_\blt$, $\xk$ and $\xk'$, and the central charge $c$ of $\Vbb$, such that
\begin{align}
\nabla_\yk'=\nabla_\yk+f\id\qquad\text{on }\scr T_{\fk X}(\Wbb_\blt).
\end{align}
\end{pp}

See the end of Sec. 3.6 (and also Sec. 4.2) of \cite{Gui} for the formula of $f$.

In the next subsection, we shall discuss an important case where the projective term $f$ in Prop. \ref{lb152} equals $0$.

\begin{pp}[Projective flatness]
Suppose $\yk_1,\yk_2\in\Theta_{\mc B}(\mc B)$, and $\nabla_{\yk_1}$ and $\nabla_{\yk_2}$ are defined using a set of local coordinates $\eta_\blt$ and the lifts $\xk_1,\xk_2$ of $\yk_1,\yk_2$ respectively. Then  there exists $f\in\scr O(\mc B)$ depending only on $\fk X$, $\eta_\blt$, $\xk_1$ and $\xk_2$, and $c$, such that the curvature
\begin{align*}
[\nabla_{\yk_1},\nabla_{\yk_2}]-\nabla_{[\yk_1,\yk_2]}=f\id\qquad \text{on }\scr T_{\fk X}(\Wbb_\blt).
\end{align*}
\end{pp}

\begin{pp}
Suppose  that on each $\Wbb_i$, $L_0-\wtd L_0$ is a constant $\Delta_i$ (for instance, when $\Wbb_i$ is irreducible). Suppose also that $\nabla_\yk,\nabla_\yk'$ are defined by a lift $\xk$ and two sets of local coordinates $\eta_\blt,\eta'_\blt$. Then  there exists $f\in\scr O(\mc B)$ depending only on $\fk X$, $\eta_\blt$ and $\eta_\blt'$, $\xk$, $c$, and $\Delta_1,\dots,\Delta_N$ such that
\begin{align}
\nabla_\yk'=\nabla_\yk+f\id\qquad\text{on }\scr T_{\fk X}(\Wbb_\blt).
\end{align}
\end{pp}
Clearly, similar results hold on $\scr T_{\fk X}^*(\Wbb_\blt)$.

We refer the readers to Sections 5.1 and 5.2 of \cite{Gui} for details of these two propositions.

\subsection{}

\begin{df}
Let $C$ be a Riemann surface, and let $(U_{\alpha},\eta_\alpha)_{\alpha\in\fk A}$ be a chart, i.e., $(U_\alpha)_{\alpha\in\fk A}$ is an open covering of $C$ and $\eta_\alpha\in\scr O(U_\alpha)$ is univalent. We say that $(U_{\alpha},\eta_\alpha)_{\alpha\in\fk A}$ is a \textbf{projective chart} \index{00@Projective charts/structures} if for each $\alpha,\beta\in\fk A$, the function $\eta_\alpha\circ\eta_\beta^{-1}$ on $\eta_\beta(U_\alpha\cap U_\beta)\subset\Cbb$ is a M\"obius transformation $z\mapsto \frac{az+b}{cz+d}$.
\end{df}

\begin{df}
For the family $\fk X$, let $(U_{\alpha},\eta_\alpha)_{\alpha\in\fk A}$ where $(U_\alpha)_{\alpha\in\fk A}$ is an open cover of $\mc C$ and each $\eta_\alpha\in\scr O(U_\alpha)$ is univalent. We say that  $(U_{\alpha},\eta_\alpha)_{\alpha\in\fk A}$ is a \textbf{projective chart} of $\fk X$ if its restriction to each fiber $\mc C_b$ is a projective chart. A maximal projective chart is called a \textbf{projective structure}.
\end{df}

\begin{eg}
$\Pbb^1$ has an obvious projective structure consisting of all M\"obius transformations. It is the unique projective structure containing the standard coordinate $\zeta$ of $\Cbb$. Indeed, it is the unique projective structure of $\Pbb^1$ \cite[8.2.12]{FB04}.
\end{eg}

\begin{thm}
For any family $\fk X$ of $N$-pointed compact Riemann surfaces, if $\mc B$ is Stein then $\fk X$ has a projective structure.
\end{thm}
This theorem is due to \cite{Hub80}. See also \cite[Sec. 4.1]{Gui} or \cite[Sec. B]{Gui}. According to this theorem, for any $N$-pointed family $\fk X$, by shrinking the base manifold $\mc B$, we may find local coordinates $\eta_\blt$ of $\fk X$ that are contained in a projective structure of $\fk X$.

\begin{pp}\label{lb151}
Suppose that $(U_1,\eta_1),\dots,(U_N,\eta_N)$ belong to a projective structure of $\fk X$. Then the operator $\nabla_\yk$ on $\scr T_{\fk X}(\Wbb_\blt)$ defined by $\eta_\blt$ is independent of the lift $\xk$ of $\yk$.
\end{pp}

\begin{rem}
The rough reason for this Proposition is the following: Let $\xk$ and $\xk'$ be two lifts of $\yk$. Then $d\pi$ sends $\zk=\xk-\xk'$ to $0$. This means that $\zk\in H^0(\mc C,\Theta_{\mc C/\mc B}(\star\SX))$. In view of \eqref{eq232}, we need to show that the residue action of $\upnu(\zk)=\upnu(\xk)-\upnu(\xk')$ is $0$ on $\scr T_{\fk X}(\Wbb_\blt)$, or equivalently, $\upnu(\zk)w_\blt\in\scr J_{\fk X}(\Wbb_\blt)(\mc B)$ for each $w_\blt\in\Wbb_\blt$. The map $\zk\mapsto \upnu(\zk)$ sends a section of $\Theta_{\mc C/\mc B}(\star\SX)$ on $U_1\cup\cdots\cup U_N$ to one of $\scr V_{\fk X}^{\leq 2}\otimes\omega_{\mc C/\mc B}(\star\SX)$ whose trivialization is described by  $\cbf dz$. Locally and under reasonable trivializations, this map sends $h(z,\tau_\blt)\partial_z$ to $h(z,\tau_\blt)\cbf dz$. Since $\cbf$ has weight $2$, the coordinate transformation formula for $\partial_z$ in $\Theta_{\mc C/\mc B}$ equals that of $\cbf dz$ mod a section of $\scr V_{\fk X}^{\leq 1}\otimes\omega_{\mc C/\mc B}(\star\SX)$ (cf. Subsec. \ref{lb150}).  The expression of section is determined by $L_2\cbf=\frac c{2}\id$.

Here comes the crucial point: Since all $(U_i,\eta_i)$ are contained in a projective structure $(U_\alpha,\eta_\alpha)_{\alpha\in\fk A}$, and since in the change of coordinate formula for M\"obius transformations only $L_0,L_{\pm1}$ are involved but $L_2$ is not, the change of coordinate formulas for $\cbf dz$ and for $\partial_z$ are equal. Therefore, as $\zk$ is a global section of $\Theta_{\mc C/\mc B}(\star\SX)$, $\upnu(\zk)$ can be extended to a global section of $\scr V_{\fk X}^{\leq 2}\otimes\omega_{\mc C/\mc B}(\star\SX)$. So $\upnu(\zk)w_\blt\in\scr J_{\fk X}(\Wbb_\blt)(\mc B)$. \hfill\qedsymbol
\end{rem}

Due to Prop. \ref{lb151}, if $\mc B$ is Stein and $\eta_\blt$ belong to a projective structure of $\fk X$, then we can define a connection $\nabla$ on $\scr T_{\fk X}(\Wbb_\blt)$ and hence on $\scr T_{\fk X}^*(\Wbb_\blt)$ such that for each $\yk\in\Theta_{\mc B}$, $\nabla_\yk$ is the one defined by \eqref{eq232} using any lift $\xk$ of $\yk$.

\begin{eg}
For each $\tau\in\Hbb=\{z\in\Cbb:\Imag z>0\}$, the torus $\mbb T_\tau$ defined by $\Cbb$ mod the rank $2$ lattice $\Zbb+\tau\Zbb$ has a standard projective structure: the one inherited from the standard projective structure of $\Cbb$. This projective structure is \textbf{modular invariant}: Let $g\in\mathrm{PSL}(2,\Zbb)$ be $g(\tau)=\frac{a\tau+b}{c\tau+d}$ where $a,b,c,d\in\Zbb$ and $ad-bc=1$. Then the biholomorphism
\begin{align*}
\mbb T_\tau\rightarrow\mbb T_{g(\tau)},\qquad z\mapsto\frac{z}{c\tau+d}
\end{align*}
sends the standard projective structure of $\mbb T_\tau$ to that of $\mbb T_{g(\tau)}$.

Thus, for sheaves of conformal blocks associated to a family of $N$-pointed tori, standard connections are those defined by the local coordinates inside this modular invariant projective structure.  \hfill  \qedsymbol
\end{eg}

\section{Local freeness of sheaves of coinvariants and conformal blocks}\label{lb153}

\subsection{}

As in Subsec. \ref{lb127}, we associate admissible $\Vbb$-modules $\Wbb_1,\dots,\Wbb_N$ to the marked points $\sgm_1,\dots,\sgm_N$ of $\fk X=(\pi:\mc C\rightarrow\mc B;\sgm_\blt)$. We do not assume that $\mc B$ can be embedded as an open subset of $\Cbb^m$ or the local coordinates are chosen.
 
\begin{df}
We say that $\Vbb$ is \textbf{$C_2$-cofinite} \index{00@$C_2$-cofinite VOAs} if $\Vbb/C_2(\Vbb)$ is finite-dimensional where  $C_2(\Vbb)=\Span_\Cbb\{Y(u)_{-2}v:u,v\in\Vbb\}$.
\end{df}

The $C_2$-cofinite condition was introduced by Zhu \cite{Zhu96} in the study of genus-$1$ conformal blocks.

\begin{df}\label{lb165}
We say that the weak $\Vbb$-module $\Wbb$ is \textbf{generated by} \index{00@Generating subsets of $\Wbb$} a subset $\fk S$ if the smallest $\Vbb$-invariant subspace of $\Wbb$ containing $\fk S$ is $\Wbb$.  We say that $\Wbb$ is \textbf{finitely generated} \index{00@Finitely generated weak $\Vbb$-modules} if it is generated by finitely many vectors. When $\Wbb$ is admissible, this is clearly equivalent to saying that $\Wbb$ is generated by finitely many homogeneous vectors.
\end{df}

\begin{rem}
Note that in the case $\fk S\subset \Vbb$, that $\fk S$ generates the vacuum module $\Vbb$ is not the same as that $\fk S$ generates the VOA $\Vbb$ (cf. Def. \ref{lb163}). For instance, the vacuum vector $\id$ generates the $\Vbb$-module $\Vbb$, but not the VOA $\Vbb$.
\end{rem}

The following important result is due to \cite[Lemma 2.4]{Miy04}. Some weaker versions of this result are due to \cite{GN03,Buhl02}.

\begin{thm}\label{lb133}
Assume that $\Vbb$ is $C_2$-cofinite. Let $E\subset \Vbb$ be a finite subset such that $\Vbb=\Span(E)+C_2(\Vbb)$. If $\Wbb$ is a weak $\Vbb$-module generated by a finite set $\fk S$ of vectors, then $\Wbb$ is spanned by vectors of the form
\begin{align}
Y(v_k)_{-n_k}\cdots Y(v_1)_{-n_1}w
\end{align}
where $k\in\Nbb$, $w\in\fk S$, $v_1,\dots,v_k\in E$, and $n_1,\dots,n_k\in\Zbb$ satisfy $n_1<n_2<\cdots<n_k$.
\end{thm}

\begin{exe}\label{lb139}
Use Thm. \ref{lb133} to show that if $\Vbb$ is $C_2$-cofinite, then every finitely-generated admissible $\Vbb$-module is finitely-admissible.
\end{exe}

\subsection{}

\begin{ass}
In this section, we assume that $\Vbb$ is $C_2$-cofinite and $\Wbb_1,\dots,\Wbb_N$ are finitely-generated (finitely-)admissible modules.
\end{ass}

In our notes, we do not use Thm. \ref{lb133} directly. Instead, we use the following consequence of Thm. \ref{lb133}. See \cite[Sec. 7]{Gui20} or \cite[Sec. 3.7]{Gui} for the proof.

\begin{thm}\label{lb138}
For each Stein open subset $V\subset\mc B$, the $\scr O(V)$-module
\begin{align}
\frac{\scr W_{\fk X}(\Wbb_\blt)(V)}{\scr J_{\fk X}(\Wbb_\blt)(V)}\label{eq244}
\end{align}
is generated by finitely many elements.
\end{thm}

\begin{co}
$\scr T_{\fk X}(\Wbb_\blt)$ is a finite type $\scr O_{\mc B}$-module.
\end{co}

\begin{proof}
Assume without loss of generality that $\mc B$ is a Stein open subset of $\Cbb^m$. Then $\scr W_{\fk X}(\Wbb_\blt)=\Wbb_\blt\otimes\scr O_{\mc B}$ is generated by constant sections, i.e., elements of $\Wbb_\blt\simeq\Wbb_\blt\otimes 1$. So $\scr T_{\fk X}(\Wbb_\blt)$ is generated by $\Wbb_\blt$. Choose $w_1,\dots,w_n\in \scr W_{\fk X}(\Wbb_\blt)(\mc B)$ generating the $\scr O({\mc B})$-module \eqref{eq244} (setting $V=\mc B$). So each element of $\Wbb_\blt$ is an $\scr O(\mc B)$-linear combination of $w_1,\dots,w_N$ in the quotient \eqref{eq244}. So $\scr T_{\fk X}(\Wbb_\blt)$ is generated by $w_1,\dots,w_N$.
\end{proof}

By the basic properties of finite type sheaves (cf. Thm. \ref{lba1}), each fiber $\scr T_{\fk X}(\Wbb_\blt)|_b$ (which is equivalent to $\scr T_{\fk X_b}(\Wbb_\blt)\simeq\scr T_{\fk X_b}^*(\Wbb_\blt)$ by Prop. \ref{lb134}) is finite-dimensional; the following rank function $\Rbf:\mc B\rightarrow\Nbb$,
\begin{align}
\Rbf(b)=\dim \scr T_{\fk X}(\Wbb_\blt)|_b=\dim\scr T_{\fk X_b}(\Wbb_\blt)=\dim\scr T_{\fk X_b}^*(\Wbb_\blt) \label{eq246}
\end{align}
is upper semicontinuous; if $\Rbf$ is also lower semicontinuous and hence locally constant, then $\scr T_{\fk X}(\Wbb_\blt)$ is locally free and so is its dual sheaf $\scr T_{\fk X}^*(\Wbb_\blt)$. Then we will have a natural equivalence $\scr T_{\fk X}^*(\Wbb_\blt)|_b\simeq\scr T_{\fk X_b}^*(\Wbb_\blt)$. Namely, if we can show that $\Rbf$ is locally constant, then \emph{the spaces of conformal blocks for all fibers $\fk X_b$ of $\fk X$ form a vector bundle over $\mc B$}.

\subsection{}

\begin{thm}\label{lb135}
$\scr T_{\fk X}(\Wbb_\blt)$ and hence $\scr T_{\fk X}^*(\Wbb_\blt)$ are locally free $\scr O_{\mc B}$-modules. In particular, the rank function $\Rbf$ defined by \eqref{eq246} is locally constant.
\end{thm}

As discussed above, to prove Thm. \ref{lb135}, it suffices to prove that $\Rbf$ is locally constant. Suppose we can show that $\Rbf|_{\mc B_0}$ is lower semicontinuous for any one-dimensional complex submanifold $\mc B_0$ of $\mc B$ biholomorphic to an open disc, then $\Rbf|_{\mc B_0}$ is constant since it is also upper semicontinuous. It then follows that $\Rbf$ is locally constant

Therefore, we may just assume that $\mc B$ is a simply-connected open subset of $\Cbb$ containing $0$, and $\fk X$ admits a set of local coordinates $\eta_\blt$. Then either $\mc B=\Cbb$ or $\mc B$ is not closed. So, as any connected non-compact Riemann surface is Stein, $\mc B$ is Stein.  Identify
\begin{align*}
\scr W_{\fk X}(\Wbb_\blt)=\Wbb_\blt\otimes_\Cbb\scr O_{\mc B}\qquad\text{via }\mc U(\eta_\blt).
\end{align*}
It suffices to show: 

\begin{lm}\label{lb137}
$\Rbf$ is lower semicontinuous at $0$.
\end{lm}

\subsection{}

Let $\tau$ be the standard coordinate of $\mc B\subset\Cbb$. By Sec. \ref{lb136}, we can define a differential operator $\nabla_{\partial_\tau}$ on $\Wbb_\blt\otimes_\Cbb\scr O_{\mc B}$ which preserves $\scr J_{\fk X}(\Wbb_\blt)(\mc B)$ due to Thm. \ref{lb132}. We shall prove Lemma \ref{lb137} using this fact and Thm. \ref{lb138}.

We fix an element $\upphi_0\in\scr W_{\fk X_0}(\Wbb_\blt)$, i.e.  a linear functional on $\Wbb_\blt$ vanishing on $\scr J_{\fk X_0}(\Wbb_\blt)$. Let us prove Lemma \ref{lb137} by constructing a conformal block $\upphi_\tau\in\scr W_{\fk X_\tau}(\Wbb_\blt)$  for each $\tau\in\mc B$ such that  the map $\upphi_0\mapsto \upphi_\tau$ is linear and injective.

\begin{cv}
Let $\Wbb_\blt^{\leq k}$ \index{W@$\Wbb_\blt^{\leq k}$} be the subspace of $\Wbb_\blt$ spanned by homogeneous $w_1\otimes\cdots \otimes w_N$ satisfying $\wtd\wt(w_1)+\cdots+\wtd\wt(w_N)\leq k$. Note that $\Wbb_\blt^{\leq k}$ is finite-dimensional since each $\Wbb_i$ is finitely-admissible.
\end{cv}

In view of \eqref{eq231}, for each $w_\blt\in \Wbb_\blt\otimes\scr O_{\mc B}$, we can write
\begin{align}
\nabla_{\partial_\tau}w_\blt=\partial_\tau w_\blt+A(\tau)w_\blt\label{eq238}
\end{align}
where 
\begin{align}
A(\tau)w_\blt=-\sum_{i=1}^N\sum_{k\in\Zbb}h_{i,k}(\tau)w_1\cdots\otimes L_{k-1}w_i\otimes\cdots\otimes w_N
\end{align}
We take power series expansion
\begin{align*}
A(\tau)=\sum_{n\in\Nbb}A_n\tau^n
\end{align*}
where $A_n\in\End(\Wbb_\blt)$ is given by
\begin{align*}
A_n=-\sum_{i=1}^N\sum_{k\in\Zbb}h_{i,k,n}\cdot\id_{\Wbb_1}\otimes\cdots\otimes L_{k-1}|_{\Wbb_k}\otimes\cdots\otimes\id_{\Wbb_N}
\end{align*}
where  $h_{i,k,n}\in\Cbb$ is determined by $h_{i,k}(\tau)=\sum_{n\in\Nbb}h_{i,k,n}\tau^n$ and vanishes for all $i,n$ and $k\leq K$ for some $K\in\Zbb$. So $A(\tau)\in\End(\Wbb_\blt)[[\tau]]$.

\begin{df}\label{lb147}
Define a linear map
\begin{align*}
\upphi:\Wbb_\blt\rightarrow \Cbb[[\tau]],\qquad w\mapsto \upphi_\tau(w)
\end{align*}
such that for each $w\in\Wbb_\blt$, $\upphi_\tau(w)$ is determined by the formal differential equation
\begin{align}
\partial_\tau\upphi_\tau(w)=\upphi_\tau (A(\tau)w)\label{eq242}
\end{align}
whose initial value $\upphi_\tau|_{\tau=0}$ is the conformal block $\upphi_0$ chosen at the beginning.
\end{df}

More precisely, if we write $\upphi_\tau(w)=\sum_{n\in\Nbb}\upphi_n(w)\tau^n$ where each $\upphi_n:\Wbb_\blt\rightarrow\Cbb$ is linear and $\upphi_0$ is just the previously chosen conformal block, then
\begin{align*}
\sum_{n\in\Nbb}n\upphi_n(w)\tau^{n-1}=\sum_{m,n\in\Nbb}\upphi_n(A_m w)\tau^{m+n}.
\end{align*}
So for each $n\in\Zbb_+$,
\begin{align}
n\upphi_n=\sum_{l=0}^{n-1}\upphi_l\circ A_{n-l-1}.\label{eq241}
\end{align}
This determines all $\upphi_n$ inductively. Our goal is to show that $\upphi_\tau(w)$ is the series expansion of an analytic function on $\mc B$.

\subsection{}

By $\Cbb[[\tau]]$-linearity, we can extend $\upphi$ to a linear map from $\Wbb_\blt\otimes\Cbb[[\tau]]$ to $\Cbb[[\tau]]$. (Note that the RHS of \eqref{eq242} is understood in this way.) Then $\Wbb\otimes\scr O(\mc B)$ is an $\scr O(\mc B)$-submodule of $\Wbb_\blt\otimes[[\tau]]$ by taking power series expansions. We are interested in the restriction
\begin{align*}
\upphi:\Wbb_\blt\otimes\scr O(\mc B)\rightarrow\Cbb[[\tau]],\qquad w\mapsto \upphi_\tau(w).
\end{align*}
It clearly satisfies the differential equation $\partial_\tau\upphi_\tau(w)=\upphi_\tau (\partial_\tau w+A(\tau)w)$, namely (cf. \eqref{eq238})
\begin{align}
\partial_\tau\upphi_\tau(w)=\upphi_\tau(\nabla_{\partial_\tau} w).\label{eq239}
\end{align}

\begin{lm}\label{lb140}
$\upphi_\tau$ vanishes on $\scr J_{\fk X}(\Wbb_\blt)(\mc B)$.
\end{lm}

\begin{proof}
Choose any $w\in \scr J_{\fk X}(\Wbb_\blt)(\mc B)$, which by power series expansion is an element of $\Wbb_\blt\otimes\Cbb[[\tau]]$. Then by \eqref{eq239}, $\upphi_\tau(w)$ has series expansion
\begin{align*}
\upphi_\tau(w)=\sum_{n\in\Nbb}\frac {\tau^n}{n!}\partial^n_\tau\upphi_\tau(w)\big|_{\tau=0}=\sum_{n\in\Nbb}\frac {\tau^n}{n!}\upphi_\tau(\nabla^n_{\partial_\tau} w)\big|_{\tau=0}
\end{align*}
where $\upphi_\tau(\nabla^n_{\partial_\tau} w)\big|_{\tau=0}$ denotes the constant term of the series $\upphi_\tau(\nabla^n_{\partial_\tau} w)\in\Cbb[[\tau]]$. By Thm. \ref{lb132},  $s_n=\nabla^n_{\partial_\tau}w$ belongs to $\scr J_{\fk X}(\Wbb_\blt)(\mc B)$. In particular, $s_n(\tau)|_{\tau=0}\in\scr J_{\fk X_0}(\Wbb_\blt)$. Clearly $\upphi_\tau(s_n)|_{\tau=0}=\upphi_0(s_n(0))$, which equals $0$ because $\upphi_0$ is a conformal block associated to $\fk X_0$. This proves the lemma.
\end{proof}

\subsection{}

To prove that $\upphi_\tau$ is analytic, we need a basic fact about differential equations:

\begin{lm}\label{lb141}
Let $W$ be a finite dimensional vector space. Suppose $f(\tau)=\sum_{n\in\Nbb}f_n\tau^n\in W[[\tau]]$ satisfies a formal differential equation
\begin{align}
\partial_\tau f(\tau)=A(\tau)f(\tau)\label{eq240}
\end{align}
for some $A\in\End(W)\otimes_\Cbb\scr O(\mc B)$, then $f(\tau)$ is the power series expansion of an element of $W\otimes\scr O(\mc B)$ which we also denote by $f(\tau)$.
\end{lm}

\begin{proof}
It is clear that any formal solution $f(\tau)$ of \eqref{eq240} is uniquely determined by its constant term $f_0\in W$. (Cf. the argument for \eqref{eq241}.) By the basic theory of differential equations (e.g. \cite[Thm. B.1]{Kna}), \eqref{eq240} must have a solution in $W\otimes\scr O(\mc B)$ with initial value $f_0$. So this solution must equal $f$ because their constant terms are equal.
\end{proof}

\begin{lm}\label{lb142}
$\upphi$ is an $\scr O(\mc B)$-module morphism from $\Wbb_\blt\otimes\scr O(\mc B)$ to $\scr O(\mc B)$. Thus, it is automatically an $\scr O_{\mc B}$-module morphism $\Wbb_\blt\otimes\scr O_{\mc B}\rightarrow\scr O_{\mc B}$.
\end{lm}

\begin{proof}
By $\scr O(\mc B)$-linearity, it suffices to prove that $\upphi$ sends each constant section $w\in\Wbb_\blt$ to $\upphi_\tau(w)\in\Wbb_\blt\otimes\scr O(\mc B)$.

By Thm. \ref{lb138}, we can find finitely many elements $s_1,s_2,\dots\in\Wbb_\blt\otimes\scr O(\mc B)$ generating $\Wbb_\blt\otimes\scr O(\mc B)$ mod $\scr J_{\fk X}(\Wbb_\blt)(\mc B)$. We fix $k_0\in\Nbb$ such that $s_1,s_2,\dots\in\Wbb_\blt^{\leq k_0}\otimes\scr O(\mc B)$. Consider the restriction of $\upphi$ to $\Wbb_\blt^{\leq k}\rightarrow \Cbb[[\tau]]$ for all $k\geq k_0$, which we denote by $\upphi^{\leq k}$. Recall that $\Wbb_\blt^{\leq k}$ is finite-dimensional.  So $\upphi^{\leq k}$ is an element of $(\Wbb_\blt^{\leq k})^*\otimes\Cbb[[\tau]]$.

Let $\{e_j\}_{j\in J}$ be a basis of $\Wbb_\blt^{\leq k}$. By \eqref{eq242} or \eqref{eq239}, $\partial_\tau\upphi_\tau(e_j)=\upphi_\tau(\nabla_{\partial_\tau} e_j)$ where $\nabla_{\partial_\tau} e_j\in\Wbb_\blt\otimes\scr O(\mc B)$. Since $\nabla_{\partial_\tau} e_j$ is an $\scr O(\mc B)$-linear combination of $s_1,s_2,\dots$ mod $\scr J_{\fk X}(\Wbb_\blt)(\mc B)$, we can find $\Upomega_{i,j}(\tau)\in\scr O(\mc B)$ for all $i,j\in J$ such that
\begin{align*}
\nabla_{\partial_\tau} e_i=\sum_{j\in J}\Upomega_{i,j}(\tau)e_j\qquad\text{mod }\scr J_{\fk X}(\Wbb_\blt)(\mc B).
\end{align*}
Thus, by \eqref{eq239} and Lemma \ref{lb140}, we have
\begin{align}
\partial_\tau\upphi_\tau^{\leq k}(e_i)=\sum_{j\in J}\Upomega_{i,j}(\tau)\upphi_\tau^{\leq k}(e_j).\label{eq243}
\end{align}
Therefore,  $\upphi_\tau^{\leq k}$ as an element of $((\Wbb_\blt)^{\leq k})^*\otimes\Cbb[[\tau]]$ satisfies a linear holomorphic differential equation similar to \eqref{eq240}. So by Lemma \ref{lb141}, this series is an element of $(\Wbb_\blt^{\leq k})^*\otimes\scr O(\mc B)$. This finishes the proof.
\end{proof}

\begin{rem}
The differential equation \eqref{eq243} has a significant role in conformal field theory. Take $\Vbb$ to be a WZW model $L_l(\gk,0)$ and let $\Wbb_1,\dots,\Wbb_N$ be irreducible, and assume that the lowest $\wtd L_0$-eigenvalue for each $\Wbb_i$ is $0$. Take $\fk X$ to be the genus-$0$ family in Example \ref{lb117}. Then we can choose the $k_0$ in the proof of Lemma \ref{lb142} to be $0$. By restricting the base manifold $\Conf^N(\Cbb^\times)$ of $\fk X$ to any complex line parallel to the $z_j$-axis, then \eqref{eq243} shows that $\upphi^{\leq 0}$ satisfies a linear holomorphic $\partial_{z_j}$-differential equation. This is the celebrated \textbf{Knizhnik–Zamolodchikov (KZ) equation}.
\end{rem}

\subsection{}

To summarize the results proved so far, we have:
\begin{thm}\label{lb145}
Let $\mc B$ be a simply-connected open subset of $\Cbb$ containing $0$, and choose local coordinates $\eta_\blt$ for $\fk X$. Define $\nabla_{\partial_\tau}$ using a lift of $\partial_\tau$. Then for each $\upphi_0\in\scr T_{\fk X_0}^*(\Wbb_\blt)$, the $\upphi_\tau$ defined by Def. \ref{lb147} is an element of $\scr T^*_{\fk X}(\Wbb_\blt)(\mc B)$ whose value at $\tau=0$ is $\upphi_0$, and which is annihilated by $\nabla_{\partial_\tau}$.
\end{thm}

\begin{proof}
By Lemma \ref{lb142}, we can define $\upphi$ to be an $\scr O_{\mc B}$-module morphism $\Wbb_\blt\otimes\scr O_{\mc B}\rightarrow\scr O_{\mc B}$. It is a conformal block by Lemma \ref{lb140} and Thm. \ref{lb143}. It is annihilated by $\nabla_{\partial_\tau}$ due to \eqref{eq239} and \eqref{eq230}.
\end{proof}

\begin{proof}[Proof of Lemma \ref{lb137}]
For each $\tau_0\in\mc B$, the map $\upphi_0\mapsto\upphi_{\tau_0}$ is linear. Moreover, for sufficiently large $k$, $\upphi_\tau^{\leq k}$ satisfies a linear holomorphic differential equation \eqref{eq243} whose solutions are determined by their values at any fixed point of $\mc B$, say $\tau_0$. So the function $\upphi^{\leq k}_\tau$ of $\tau$ is uniquely determined by $\upphi^{\leq k}_{\tau_0}$. So $\upphi_0^{\leq k}$ is determined by $\upphi_{\tau_0}^{\leq k}$ for all large $k$. So the linear map $\upphi_0\mapsto\upphi_{\tau_0}$ is injective.
\end{proof}

The proof of Thm. \ref{lb135} is complete.

\begin{eg}\label{lb149}
Assume the setting of Example \ref{lb109}. Assume moreover that $\Delta\subset\Cbb$ is an open  disk centered at $0$, and that  the holomorphic function $h$ defined near $\Sbb^1$ is holomorphic on $\Dbb_r^\times$ for some $r>1$ with finite poles at $0$. So $h(z)=\sum_{n\in\Zbb}c_nz^{n+1}$ where $c_n=0$ for sufficiently negative $n$. Using Example \ref{lb148}, it is easy to see that for each $\upphi_0\in\scr T_{\fk X_0}^*(\Wbb_\blt)$, the $\upphi_\tau$ defined by \eqref{eq245} as a formal power series of $\tau$ satisfies Def. \ref{lb147}. So $\upphi_\tau\in\scr T_{\fk X}^*(\Wbb_\blt)(\Delta)$. In particular, for each $w\in\Wbb_\blt$, $\upphi_\tau(w)$ converges a.l.u. on $\tau\in\Delta$.
\end{eg}

\subsection{}

\begin{co}\label{lb146}
Assume the setting of Thm. \ref{lb145}. Then for each $\tau\in\mc B$, the linear map
\begin{align*}
\scr T_{\fk X_0}^*(\Wbb_\blt)\rightarrow\scr T_{\fk X_\tau}^*(\Wbb_\blt),\qquad \upphi_0\mapsto\upphi_\tau
\end{align*}
is bijective.
\end{co}

\begin{proof}
The injectivity follows from the proof of Lemma \ref{lb135}. The bijectivity follows from the fact that the two vector spaces have the same dimension (due to Thm. \ref{lb135}). Alternatively, it follows from that by switching the role of $\tau$ and $0$, we have a similar injective linear map $\scr T^*_{\fk X_\tau}(\Wbb_\blt)\rightarrow\scr T^*_{\fk X_0}(\Wbb_\blt)$.
\end{proof}

\begin{co}
Assume the setting of Thm. \ref{lb145}. Then $\scr T_{\fk X}^*(\Wbb_\blt)$ and hence $\scr T_{\fk X}(\Wbb_\blt)$ are  trivial vector bundles on $\mc B$.
\end{co}

\begin{proof}
The $\scr O_{\mc B}$-module morphism
\begin{align*}
\scr T_{\fk X_0}^*(\Wbb_\blt)\otimes_\Cbb\scr O_{\mc B}\rightarrow\scr T_{\fk X}^*(\Wbb_\blt)
\end{align*}
sending each constant section $\upphi_0$ to $\upphi_\tau$ (and hence each $\upphi_0\otimes f$ to $f\upphi_\tau$ where $f\in\scr O_{\mc B}$) is an isomorphism due to Cor. \ref{lb146}.
\end{proof}

\begin{co}
Let $\fk Y=(C;x_1,\dots,x_N)$ be an $N$-pointed compact Riemann surface where $C$ is connected with genus $g$, and associate $\Wbb_i$ to $x_i$. Then the dimension of space of conformal blocks $\dim\scr T_{\fk Y}^*(\Wbb_\blt)$ depends only on $g$, $N$, and the (finitely-)admissible $\Vbb$-modules $\Wbb_1,\dots,\Wbb_N$.
\end{co}

So, $\dim\scr T_{\fk Y}^*(\Wbb_\blt)$ does not depend on the complex structure of $C$, the position of $x_\blt$, or the choice of local coordinates. 

\begin{proof}
There is a family $\fk T_{g,N}$ of $N$-pointed compact connected genus-$g$ Riemann surfaces whose base manifold is the Teichm\"uller space $\mc T_{g,N}$ (which is connected), and any $\fk Y$ is equivalent to some fiber of $\fk T_{g,N}$. (See for instance \cite[Chapter XV]{ACG}.) Thus, the corollary follows immediately from Thm. \ref{lb135}.
\end{proof}

\section{Sewing, propagation, and factorization of conformal blocks}

\subsection{}\label{lb169}

Let $\fk X=(\pi:\mc C\rightarrow\Dbb_{r\rho}^\times;x_\blt;\eta_\blt)$ be the family obtained by sewing an $N$-pointed compact Riemann surface with local coordinates $\wtd{\fk X}=(\wtd C;x_\blt,x',x'';\eta_\blt,\xi,\varpi)$ as in Example \ref{lb154}. Recall that we assume, unless otherwise stated, that:
\begin{ass}\label{lb157}
Each connected component of $\wtd C$ contains one of $x_1,\dots,x_N$.
\end{ass}
It follows that each connected component of $\mc C_b$ also contains one of $x_1,\dots,x_N$. 

\begin{cv}\label{lb175}
In this section, by ``$\Vbb$-modules" we mean finitely admissible $\Vbb$-modules. 
\end{cv}
Let $\Wbb_1,\dots,\Wbb_N,\Mbb$ be $\Vbb$-modules. We associate $\Wbb_1,\dots,\Wbb_N,\Mbb,\Mbb'$ to the marked points $x_\blt,x',x''$ of $\wtd{\fk X}$ and $\Wbb_1,\dots,\Wbb_N$ to $x_1,\dots,x_N$ of $\fk X$. Recall that $\Mbb'$ is the contragredient of $\Mbb$. Identify
\begin{gather*}
\scr W_{\wtd{\fk X}}(\Wbb_\blt\otimes\Mbb\otimes\Mbb')=\Wbb_\blt\otimes\Mbb\otimes\Mbb'\qquad\text{via }\mc U(\eta_\blt,\xi,\varpi),\\
\scr W_{\fk X}(\Wbb_\blt)=\Wbb_\blt\qquad\text{via }\mc U(\eta_\blt).
\end{gather*}

\subsection{}

Let $\upphi:\Wbb_\blt\otimes\Mbb\otimes\Mbb'\rightarrow\Cbb$ be a conformal block associated to $\wtd{\fk X}$ and $\Wbb_\blt,\Mbb,\Mbb$. Let \index{zz@$\bowtie$}
\begin{align}
\bowtie_n=\sum_a m(n,a)\otimes \wch m(n,a)\in \Mbb(n)\otimes\Mbb(n)^*
\end{align}
be the contraction where $\{m(n,a):a\in\fk A_n\}$ is a basis of $\Mbb(n)$ with dual basis $\{\wch m(n,a):a\in\fk A_n\}$. Equivalently, $\bowtie_n$ is the identity operator when viewed as an element of $\End(\Wbb(n))$. Recall that $\Mbb(n)$ and $\Mbb(n)^*$ are respectively the $\wtd L_0$-weight $n$ subspaces of $\Mbb$ and $\Mbb'$ respectively. We define a linear map \index{S@$\wtd{\mc S}\upphi$, the normalized sewing}
\begin{gather}
\begin{gathered}
\wtd{\mc S}\upphi:\Wbb_\blt\rightarrow\Cbb[[q]]\\
\wtd{\mc S}\upphi(w_\blt)=\wtd{\mc S}_q\upphi(w_\blt)=\sum_{n\in\Nbb}\upphi(w_\blt\otimes \bowtie_n)q^n=\sum_{n\in\Nbb}\upphi\big(w_\blt\otimes \Wbb\underbrace{(n)\otimes\Wbb}_{\text{contraction}}(n)^*\big)q^n
\end{gathered}
\end{gather}
called the \textbf{(normalized) sewing} of $\upphi$. 

The meaning of $\wtd{\mc S}\upphi(w_\blt)$ is easy to understand: Informally, 
\begin{align}
\wtd{\mc S}_q\upphi(w_\blt)=\upphi\big(w_\blt\otimes q^{\wtd L_0}\underbrace{\cdot~~\otimes~~\cdot}_{\text{contraction}}\big)=\upphi\big(w_\blt\otimes \underbrace{\cdot~~\otimes~~q^{\wtd L_0}\cdot}_{\text{contraction}}\big)
\end{align}
since we can place the projection $P_n$ on the right of $q^{\wtd L_0}$ and take the sum over all $n$, noting that $q^{\wtd L_0}P_n=q^nP_n$. Suppose that the series $\wtd{\mc S}_q\upphi(w_\blt)$ of $q$ converges a.l.u. on $\Dbb_{r\rho}^\times$. Note that for each  $q$, $\fk X_q$ is obtained by scaling either $\xi$ or $\varpi$ by $q^{-1}$ (or more generally, scaling $\xi$ and $\varpi$ by $q_1^{-1},q_2^{-1}$ such that $q_1q_2=q$) and then perform the sewing as in Subsec. \ref{lb156} along $x'$ and $x''$ using their local coordinates. Then $\wtd{\mc S}_q\upphi(w_\blt)$ is the contraction with respect to this sewing.

We can also use $L_0$ instead of $\wtd L_0$ for scaling. For simplicity, we assume that $\Mbb$ is irreducible (or more generally, that $L_0-\wtd L_0$ is a constant on $\Mbb$), then we define the \textbf{(standard) sewing} of $\upphi$ to be \index{S@$\mc S\upphi$, the standard sewing}
\begin{align}
\mc S\upphi=q^d\cdot\wtd{\mc S}\upphi:\Wbb_\blt\rightarrow\Cbb\{q\} 
\end{align} 
where $d\cdot \id_\Mbb=L_0|_\Mbb-\wtd L_0|_\Mbb$. Here, we have used the notation that for any vector space $W$, \index{Wz@$W\{z\}$}
\begin{align*}
W\{q\}=\Big\{\sum_{n\in\Cbb} w_nq^n:w_n\in W\Big\}.
\end{align*}
By linearity, we can extend the definition of $\mc S\upphi$ to the case that $\Mbb$ is a \textbf{semi-simple $\Vbb$-module}, \index{00@Semisimple VOA modules} i.e. a direct sum of irreducible $\Vbb$-modules.

\subsection{}

A proof of the following theorem can be found in \cite[Sec. 3.3]{Gui} or \cite[Sec. 10, 11]{Gui20}.

\begin{thm}\label{lb159}
Instead of Assumption \ref{lb157}, we assume a weaker condition that for each $q\in\Dbb_{r\rho}^\times$, each connected component of $\mc C_q$ contains one of $x_1,\dots,x_N$. Let $\upphi\in\scr T^*_{\wtd{\fk X}}(\Wbb_\blt\otimes\Mbb\otimes\Mbb')$. Then $\wtd{\mc S}\upphi$ is a conformal block associated to $\fk X$, provided that  $\wtd{\mc S}\upphi(w_\blt)$ converges a.l.u. on $q\in\Dbb_{r\rho}^\times$ (equivalently, converges absolutely on $\Dbb_{r\rho}$ or on $\Dbb_{r\rho}^\times$) for each $w_\blt\in\Wbb_\blt$.
\end{thm}

For instance, suppose that $N>0$, and $\wtd C$ is a disjoint union of two connected Riemann surfaces $\wtd C_1,\wtd C_2$ such that  $x'\in\wtd C_1$ and $x_1,\dots,x_N,x''\in\wtd C_2$. Then the condition in this theorem is satisfied but Assumption \ref{lb157} is not.

By Thm. \ref{lb143}, that $\wtd{\mc S}\upphi$ is a conformal block means the following equivalent conditions:
\begin{itemize}
\item $\wtd{\mc S_q}\upphi\in \scr T_{\fk X_q}^*(\Wbb_\blt)$ for each $q\in\Dbb_{r\rho}^\times$.
\item By extending $\wtd{\mc S}\upphi$  to an $\scr O_{\Dbb_{r\rho}^\times}$-module morphism
\begin{align}
\wtd{\mc S}\upphi:\Wbb_\blt\otimes_\Cbb\scr O_{\Dbb_{r\rho}^\times}\rightarrow \scr O_{\Dbb_{r\rho}^\times},\label{eq248}
\end{align}
$\wtd{\mc S}\upphi$ vanishes on $\scr J_{\fk X}(\Wbb_\blt)(\Dbb_{r\rho}^\times)$.
\item As an $\scr O_{\Dbb_{r\rho}}^\times$-module morphism, $\wtd{\mc S}\upphi$ is an element of $H^0(\Dbb_{r\rho}^\times,\scr T_{\fk X}^*(\Wbb_\blt))$.
\end{itemize}

\subsection{}

We give an application of Thm. \ref{lb159}. Assume only in this subsection and the next one that $\fk X=(C;x_\blt;\eta_\blt)$ is an $N$-pointed compact Riemann surface. Recall that by Assumption \ref{lb102}, each connected component of $C$ contains one of $x_1,\dots,x_N$. Identify
\begin{align}
\scr W_{\fk X}(\Wbb_\blt)=\Wbb_\blt\qquad\text{via }\mc U(\eta_\blt).\label{eq250}
\end{align}
Let $\upphi:\Wbb_\blt\rightarrow\Cbb$ be a conformal block associated to $\fk X$.  We use the notations in Subsec. \ref{lb160}. Recall that \eqref{eq175}  gives an explicit formula for $\wr\upphi_x$ when $x$ is close to $x_j$, and the RHS of \eqref{eq175} converges a.l.u.. for such $x$. It is clear that the RHS of \eqref{eq175} is the sewing of a conformal block associated to $\fk P_{\eta_j(x)}\sqcup \fk X$. Therefore, by Thm. \ref{lb159}, $\wr\upphi_x$ is a conformal block associated to $\wr\fk X_x\simeq \fk P_{\eta_j(x)}\#\fk X$ and $\Vbb,\Wbb_\blt$ when $x$ is close to $x_j$.

Let us be more precise. Recall
\begin{align}
\wr\fk X_x=(C;x,x_\blt)
\end{align}
where $x\neq x_1,\dots,x_N$. By Def. \ref{lb103} and Rem. \ref{lb191}, we have an $\scr O_{C\setminus x_\blt}$-module morphism
\begin{align}
\wr\upphi:\scr V_{C\setminus x_\blt}\otimes_\Cbb\Wbb_\blt\rightarrow\scr O_{C\setminus x_\blt},\qquad v\otimes w_\blt\mapsto\wr\upphi(v,w_\blt).
\end{align}
For each $x\in C\setminus x_\blt$, we have a linear map
\begin{align}
\wr\upphi|_x:\scr V_C|x\otimes_\Cbb\Wbb_\blt\rightarrow\Cbb.
\end{align}
For every neighborhood $U$ of $x$ and a univalent $\mu\in\scr O(U)$, the equivalence $\mc U_\varrho(\mu):\scr V_C|_U\xrightarrow{\simeq}\Vbb\otimes_\Cbb\scr O_U$ restricts to $\mc U_\varrho(\mu):\scr V_C|x\xrightarrow{\simeq}\Vbb$. Note that $\mu-\mu(x),\eta_\blt$ are local coordinates of $\wr\fk X_x$ at $x,x_\blt$. We then have an equivalence
\begin{align}
\scr W_{\wr\fk X_x}(\Vbb\otimes\Wbb_\blt)\xrightarrow[\simeq]{\mc U(\mu-\mu(x),\eta_\blt)} \Vbb\otimes\Wbb_\blt\xrightarrow[\simeq]{\mc U_\varrho(\mu)^{-1}\otimes\id}\scr V_C|x\otimes\Wbb_\blt.\label{eq249}
\end{align}
\begin{exe}
Show that the equivalence \eqref{eq249} is independent of the choice of $\mu$.
\end{exe}
Thus, by identifying $\scr W_{\wr\fk X_x}(\Vbb\otimes\Wbb_\blt)$ with $\scr V_C|x\otimes\Wbb_\blt$ via \eqref{eq249}, we see that $\wr\upphi|_x$ is a linear functional
\begin{align}
\wr\upphi|_x:\scr W_{\wr\fk X_x}(\Vbb\otimes\Wbb_\blt)\rightarrow\Cbb.
\end{align}
(Indeed, one can check that this definition is also independent of the local coordinates $\eta_\blt$ of $\fk X$).

By the discussion at the beginning of this subsection, $\wr\upphi|_x$ is a conformal block when $x$ is near any marked point $x_i$. Thus,  by Cor. \ref{lb158} and the fact that each connected component of $C\setminus x_\blt$ intersects a neighborhood of $x_j$ for some $j$, we conclude that $\wr\upphi|_x$ is a conformal block for every $x\in C\setminus x_\blt$. Note that in order to apply Cor. \ref{lb158}, we should organize all $\wr\fk X_x$ to a family
\begin{align}
\wr\fk X=(C\times (C\setminus x_\blt)\rightarrow C\setminus x_\blt;\sgm,x_1,\dots,x_N)
\end{align}
where $\sgm$ sends each $x\in C\setminus x_\blt$ to $(x,x)$ and $x_j$ sends $x$ to $(x_j,x)$. Clearly the fiber of $\wr\fk X$ at each $x\in C\setminus x_\blt$ is $\wr\fk X_x$. Thus, we can view $\wr\upphi$ as an $\scr O_{C\setminus x_\blt}$-morphism $\scr W_{\wr\fk X}(\Vbb\otimes\Wbb_\blt)\rightarrow\scr O_{C\setminus x_\blt}$. It is a global conformal block since it is so near $x_1,\dots,x_N$. We conclude:

\begin{thm}\label{lb161}
Let $\upphi\in\scr T_{\fk X}^*(\Wbb_\blt)$. Then the $\scr O_{C\setminus x_\blt}$-module morphism $\wr\upphi:\scr W_{\wr\fk X}(\Vbb\otimes\Wbb_\blt)\rightarrow\scr O_{C\setminus x_\blt}$ is a conformal block associated to $\wr\fk X$ and $\Vbb,\Wbb_\blt$, called the \textbf{propagation of $\upphi$}. \index{zz@$\wr\upphi$, propagation of conformal blocs}
\end{thm}

We can consider \textbf{multi-propagations of conformal blocks}. Namely, we let several distinct points $y_1,\dots,y_n$ (instead of a single point $x$) vary on $C\setminus x_\blt$, which gives a family $\wr^n\fk X$ with base manifold $\Conf^n(C\setminus x_\blt)$ and fibers $(C;y_1,\dots,y_n,x_1,\dots,x_N)$. Then one has the  $n$-propagation $\wr^n\upphi$ defined inductively by $\wr(\wr^{n-1}\upphi)$, which is a conformal block associated to $\wr^n\fk X$ and $\Vbb,\dots,\Vbb,\Wbb_1,\dots,\Wbb_N$. For instance, the $(N+2)$-point function $\bk{w',Y(v_1,z_1)\cdots Y(v_N,z_N)w}$ is the $n$-propagation of the conformal block $w\otimes w'\in\Wbb\otimes\Wbb'\mapsto \bk{w,w'}$ associated to $(\Pbb^1;0,\infty;\zeta,1/\zeta)$. See \cite[Sec. 3.4]{Gui} or \cite{Gui21} for details.

\subsection{}

As in the previous subsection, let $\fk X=(C;x_\blt;\eta_\blt)$ be $N$-pointed, and assume the identification \eqref{eq250}. We give two applications of propagation of conformal blocks. The first one uses only the fact that $\wr\upphi(\cdot,w_\blt)$ is an $\scr O_{C\setminus x_\blt}$-module, but not really Thm. \ref{lb161}. Recall the meaning of generating subsets of $\Vbb$-modules in Def. \ref{lb165}.

\begin{pp}\label{lb167}
Assume that $C$ is connected and $N\geq 2$. For each $j=2,\dots,N$, chose a subset $E_j\subset \Wbb_j$ generating $\Wbb_j$. Let $\upphi:\Wbb_\blt\rightarrow\Cbb$ be a conformal block associated to $\fk X$ and $\Wbb_\blt$. Then $\upphi=0$ if $\upphi(w_1\otimes w_2\otimes\cdots\otimes w_N)=0$ for all $w_1\in\Wbb_1$ and $w_2\in E_2,\dots,w_N\in E_N$.
\end{pp}

The proof of this Proposition is similar to that of Goddard uniqueness (Prop. \ref{lb162}).

\begin{proof}
Let $w_\blt\in\Wbb_\blt$ such that $w_j\in E_j$ for all $j\geq 2$. Clearly $\upphi(Y(u,z)w_1\otimes w_2\cdots\otimes \otimes w_N)$ (which converges a.l.u. when $z\neq 0$ is small) is $0$. So $\wr\upphi(\cdot,w_\blt)$, as a section of $(\scr V_C^{\leq k})^\vee$ on $C\setminus x_\blt$, vanishes near $x_1$ for all $k$. So it vanishes globally on $C\setminus x_\blt$, and  in particular vanishes near $x_2$. This shows that $\upphi(w_1\otimes Y(u,z)w_2\otimes \cdots\otimes w_N)$ vanishes when $z$ is small. By taking residue at $z=0$, we see that $\upphi(w_1\otimes Y(u)_nw_2\otimes\cdots\otimes w_N)=0$ for all $u\in\Vbb$ and $n\in\Zbb$. Repeating this argument, we see that $\upphi(w_1\otimes Y(u_1)_{n_1}\cdots Y(u_k)_{n_k}w_2\otimes\cdots\otimes w_N)=0$ for all $u_1,\dots,u_k\in\Vbb$ and $n_1,\dots,n_k\in\Zbb$. Therefore, as $E_2$ generates $\Wbb_2$, we conclude that $\upphi(w_\blt)=0$ for all $w_1\in\Wbb_1,w_2\in\Wbb_2$ and $w_j\in E_j$ (where $3\leq j\leq N$). Repeating this procedure shows $\upphi=0$.
\end{proof}

The second application is the following one. Recall $\wr\fk X_x=(C;x,x_\blt)$ if $x\in C\setminus x_\blt$. Recall that $\id\in H^0(C,\scr V_C)$ is the vacuum section which locally equals the constant vacuum vector under any trivialization.

\begin{thm}\label{lb166}
Choose any $x\in C\setminus x_\blt$ and identify 
\begin{gather*}
\scr W_{\wr\fk X_x}(\Vbb\otimes \Wbb_\blt)=\scr V_C|x\otimes\Wbb_\blt\qquad\text{via }\eqref{eq249}.
\end{gather*}
Then we have an isomorphism of vector spaces
\begin{gather}\label{eq251}
\begin{gathered}
\scr T_{\wr\fk X_x}^*(\Vbb\otimes\Wbb_\blt)\xrightarrow{\simeq}\scr T_{\fk X}^*(\Wbb_\blt)\\
\boxed{v\otimes w_\blt\in\scr V_C|x\otimes\Wbb_\blt\mapsto\uppsi(v\otimes w_\blt)}\qquad\mapsto \qquad \boxed{w_\blt\in \Wbb_\blt\mapsto \uppsi(\id\otimes w_\blt)}  
\end{gathered}
\end{gather}
\end{thm}

Note first of all the easy fact:

\begin{lm}\label{lb168}
For each $\upphi\in\scr T_{\fk X}^*(\Wbb_\blt)$, the following holds in $\scr O(C\setminus x_\blt)$.
\begin{align}
\wr\upphi(\id,w_\blt)=\upphi(w_\blt)\label{eq252}
\end{align}
\end{lm}
\begin{proof}
\eqref{eq252} clearly holds near $x_1,\dots,x_N$ by \eqref{eq175}. So \eqref{eq252} holds on $C\setminus x_\blt$ by complex analysis.
\end{proof}

\begin{proof}[Proof of Thm. \ref{lb166}]
We leave it to the readers to check that for each conformal block $\uppsi$ associated to $\wr\fk X_x$, the linear functional $\upphi:\Wbb_\blt\rightarrow\Cbb$ defined by the RHS of \eqref{eq251} satisfies the definition of conformal blocks (Def. \ref{lb103}). The linear map \eqref{eq251} is injective by Prop. \ref{lb167} and the fact that $\id$ generates the vacuum module $\Vbb$. It is surjective due to Lemma \ref{lb168}, which says that $\wr\upphi$ is a preimage of $\upphi\in\scr T_{\fk X}^*(\Wbb_\blt)$ under the map \eqref{eq251}.
\end{proof}

\subsection{}\label{lb182}

We give some applications of propagation.
\begin{eg}\label{lb177}
Let (finitely-admissible) $\Wbb_1,\Wbb_2'$ be associated to the marked points $0,\infty$ of $\fk P=(\Pbb^1;0,\infty;\zeta,1/\zeta)$ where $\zeta$ is the standard coordinate of $\Cbb$. Then it is not hard to check (cf. Example \ref{lb179}) that there is an isomorphism
\begin{gather}
\begin{gathered}
\Hom_\Vbb(\Wbb_1,\Wbb_2^\cl)\xrightarrow{\simeq} \scr T_{\fk P}^*(\Wbb_1\otimes\Wbb_2')\\
T\qquad\mapsto\qquad \boxed{w_1\otimes w_2'\mapsto\bk{T w_1,w_2'}}
\end{gathered}
\end{gather}
(Note that each $Y(u)_n$ acts on $\Wbb_2^\cl$ in an obvious way.) Therefore, by Thm. \ref{lb166}, for any $N$-pointed ($N\geq 2$) sphere such that $\Wbb_1,\Wbb_2'$ are associated to two marked points and $\Vbb$ is associated to the remaining one, the corresponding space of conformal blocks is isomorphic to $\Hom_\Vbb(\Wbb_1,\Wbb_2^\cl)$.
\end{eg}

\begin{rem}
In many important cases, we have
\begin{align}
\Hom_\Vbb(\Wbb_1,\Wbb_2)=\Hom_\Vbb(\Wbb_1,\Wbb_2^\cl).
\end{align}
For instance, this is true when $L_0$ is diagonalizable on $\Wbb_1,\Wbb_2$, and each $L_0$-weight space of $\Wbb_2$ is finite-dimensional. (E.g. when $\Wbb_1,\Wbb_2$ are semisimple.) 

To see this, choose any $L_0$-eigenvector $w_1\in\Wbb_1$ with $L_0w_1=\lambda w_1$. Choose linear $T:\Wbb_1\rightarrow\Wbb_2^\cl$ intertwining the actions of $\Vbb$. Then as $L_0T=TL_0$, we see that $Tw_1\in\Wbb_2^\cl$ is an $L_0$-eigenvector with eigenvalue $\lambda$. Recall that $[\wtd L_0,L_0]=0$ (cf. Rem. \ref{lb81}). So $L_0$ preserves each ($\wtd L_0$-)weight space $\Wbb_2(n)$ of $\Wbb_2$. But $L_0|_{\Wbb_2(n)}$ has eigenvalue $\lambda$ for only finitely many different $n$, otherwise the $\lambda$-eigenspace of $L_0$ on $\Wbb_2$ would be infinite dimensional. This proves  $Tw_1\in\Wbb_2$.  \hfill\qedsymbol
\end{rem}

\begin{eg}\label{lb189}
In Example \ref{lb177}, we let $\Wbb_1=\Vbb$ and $\Wbb=\Wbb_2$. Then we have
\begin{align}
\Hom_\Vbb(\Vbb,\Wbb^\cl)\simeq\scr T_{(\Pbb^1;0,\infty;\zeta,1/\zeta)}^*(\Vbb\otimes\Wbb')\simeq\scr T_{(\Pbb^1;\infty;1/\zeta)}^*(\Wbb')
\end{align}
where the corresponding element of $T\in\Hom_\Vbb(\Vbb,\Wbb^\cl)$ in $\scr T_{(\Pbb^1;\infty;1/\zeta)}^*(\Wbb')$ is  $T\id\in\Wbb^\cl$ as a linear functional on $\Wbb'$. In particular, taking $\Wbb=\Vbb'$, we have
\begin{align}\label{eq256}
\Hom_\Vbb(\Vbb,\Vbb')=\Hom_\Vbb(\Vbb,(\Vbb')^\cl)\simeq \scr T_{(\Pbb^1;\infty;1/\zeta)}^*(\Vbb).
\end{align}
So $\scr T_{(\Pbb^1;\infty;1/\zeta)}^*(\Vbb)$ is trivial if $\Vbb$ is not self-dual.
\end{eg}

\subsection{}

We return to the setting of Subsec. \ref{lb169}.
\begin{thm}\label{lb171}
Assume that $\Vbb$ is $C_2$-cofinite,  $\Wbb_1,\dots,\Wbb_N,\Mbb$ are finitely-generated, and $\wtd L_0|_\Mbb-L_0|_\Mbb$ is a constant (e.g. when $\Mbb$ is irreducible). Let $\upphi\in\scr T^*_{\wtd{\fk X}}(\Wbb_\blt\otimes\Mbb\otimes\Mbb')$. Then $\wtd{\mc S}\upphi$ and $\mc S\upphi$ converge a.l.u. on $q\in\Dbb_{r\rho}^\times$. Moreover, if we define the connection $\nabla$ on $\scr T_{\fk X}^*(\Wbb_\blt)$ using $\eta_\blt$ and a lift of $\partial_q$, then
\begin{align}
\nabla_{\partial_q}\mc S_q\upphi=f\cdot\mc S_q\upphi
\end{align}
for some $f\in\scr O(\Dbb_{r\rho}^\times)$ depending only on $\wtd{\fk X}$ (including the local coordinates $\eta_\blt,\xi,\varpi$) and $r,\rho$, the lift of $\partial_q$, and the central charge $c$. Moreover, $f=0$  if $\eta_\blt,\xi,\varpi$ belong to a projective structure of $\wtd{\fk X}$.
\end{thm}

The proof of Thm. \ref{lb171} has some similarities to the proof in Sec. \ref{lb153} that  $\upphi_\tau$ is analytic and the sheaves of coinvariants/conformal blocks are locally free. We refer the readers to \cite[Sec. 4.3]{Gui} or \cite[Sec. 11]{Gui20} for details of the proof. In the following, we explain some key ideas.

Suppose we add the nodal curve $\mc C_0=\lim_{q\rightarrow 0}\mc C_q$ to the family $\fk X$ (see the end of Subsec. \ref{lb170}). One can also define sheaves of coinvariants and conformal blocks for $\fk X$. Due to the fact that $d\pi:\Theta_{\mc C}|_p\rightarrow \Theta_{\Dbb_{r\rho}}|_{\pi(p)}$ is not surjective if $p\in\mc C$ is the node of $\mc C_0$, we cannot lift $\partial_q$ to a section of $\Theta_{\mc C}$ near $p$, let alone to  $H^0(\mc C,\Theta_{\mc C}(\star\SX))$. (Near the node $p$, $\pi$ is equivalent to $(\xi,\varpi)\in\Cbb^2\mapsto \xi\varpi$ near $\xi=\varpi=0$.) But we can lift $q\partial_q$ to an element $\xk\in H^0(\mc C,\Theta_{\mc C}(\star\SX))$, and one can check that $\xk$ is actually in $ H^0(\mc C,\Theta_{\mc C}(-\log\mc C_0+\star\SX))$, which means that $\xk$ has finite poles at $\SX$ and that $\xk|_{\mc C_0}$ is tangent to $\mc C_0$ and vanishes at the node. (See \cite[Sec. 3.6]{Gui} or \cite[Sec. 11]{Gui20}). 
\begin{align}
\vcenter{\hbox{{
\includegraphics[height=2.3cm]{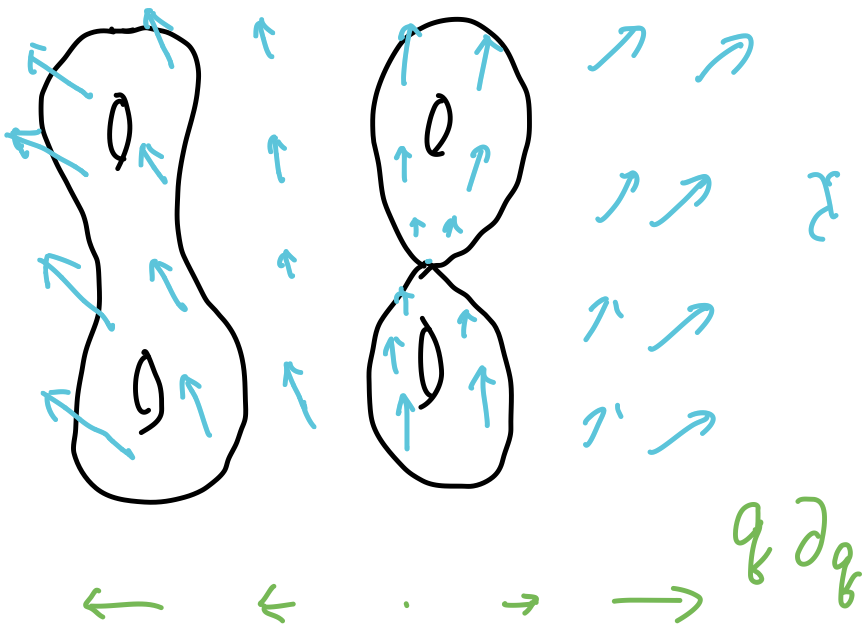}}}}
\end{align}

Using the lift $\xk$, one can define a differential operator $\nabla_{q\partial_q}$ (or more generally, $\nabla_{g\partial_q}$ where $g\in\Theta_{\Dbb_{r\rho}}$ vanishes on $0$) on $\scr T_{\fk X}(\Wbb_\blt)$ and $\scr T_{\fk X}^*(\Wbb_\blt)$. We say that the connection $\nabla$ has \textbf{logarithmic singularity} (or is a \textbf{logarithmic connection} with singularity) at $0$. Then one can show  that
\begin{align}
\nabla_{q\partial_q}\mc S_q\upphi=f\cdot\mc S_q\upphi
\end{align}
where the dependence of $f\in\scr O(\Dbb_{r\rho})$ on the given data is as in Thm. \ref{lb171}. (Note that unlike in Thm. \ref{lb171},  here $f$ is also holomorphic at $0$.) Thm. \ref{lb138} indeed holds in the present case as well. So, similar to the proof of Lemma \ref{lb142}, one shows that for sufficiently large $k$,
\begin{align}
q\partial_q\wtd{\mc S}_q\upphi^{\leq k}(e_i)=\sum_{j\in J}\Upomega_{i,j}(q)\wtd{\mc S}_q\upphi^{\leq k}(e_j).\label{eq253}
\end{align}
where $\Omega_{i,j}\in\scr O(\Dbb_{r\rho})$. Namely, as an $\End(\Wbb_\blt^{\leq k})$-valued formal power series of $q$, $\wtd{\mc S}_q\upphi^{\leq k}$ satisfies a \textbf{linear holomorphic differential equation with simple pole at $q=0$}. It is well known that Lemma \ref{lb141} can be generalized to this case, which asserts that a formal power series satisfying a linear holomorphic differential equation with simple pole must converge a.l.u. on $\Dbb_{r\rho}^\times$ (cf. e.g. \cite[Sec. 1.7]{Gui}). The a.l.u. convergences of $\wtd{\mc S}\upphi$ and $\mc S\upphi$ follow. 

Finally, we remark that Thm. \ref{lb171} can be generalized to the simultaneous sewing along several pairs of points $(y_1',y_1''),\dots,(y_M',y_M'')$ of $\wtd C$, or even more generally, the case that $\wtd {\fk X}$ is a family with base manifold $\wtd{\mc B}$. In this most general case, $\fk X$ is a family over $\wtd{\mc B}\times\Dbb^\times_{r_1\rho_1}\times\cdots\times\Dbb^\times_{r_M\rho_M}$, and the sewing is a.l.u. convergent on this domain. See the references mentioned above.

\subsection{}

Let $\mc E$ be a finite set of mutually inequivalent irreducible $\Vbb$-modules. Then for each $q\in\Dbb_{r\rho}^\times$ with a choice of argument $\arg q$, we have a linear map
\begin{gather}\label{eq254}
\begin{gathered}
\fk S_q:\bigoplus_{\Mbb\in\mc E}\scr T_{\wtd{\fk X}}^*(\Wbb_\blt\otimes\Mbb\otimes\Mbb')\rightarrow\scr T_{\fk X_q}^*(\Wbb_\blt)\\
\bigoplus_\Mbb\upphi_\Mbb\mapsto \sum_\Mbb\mc S_q\upphi_\Mbb
\end{gathered}
\end{gather} 
Note that $\sum_\Mbb\mc S\upphi_\Mbb(w_\blt)$ is a multivalued holomorphic function on $\Dbb_{r\rho}^\times$ (i.e.,  a single-valued holomorphic function of $\log q$ on the universal cover of $\Dbb_{r\rho}^\times$)

\begin{thm}\label{lb172}
Assume that $\Vbb$ is $C_2$-cofinite and $\Wbb_1,\dots,\Wbb_N$ are finitely generated. Then for each $q\in\Dbb_{r\rho}^\times$ with chosen $\arg q$, the linear map $\fk S_q$ is injective.
\end{thm}

See \cite[Sec. 4.4]{Gui} or \cite[Sec. 12]{Gui20} for a proof. The last part of that proof can be simplified thanks to the propagation of conformal blocks. In the following, we present this simplified proof.

\begin{proof}
Suppose $\uppsi_q=\sum_\Mbb\mc S_q\upphi_\Mbb$ is $0$ for one $q$. Then it vanishes for all $q\in\Dbb_{r\rho}^\times$ (and all choices of $\arg q$) since the restriction of $\uppsi$ to $\Wbb_\blt^{\leq k}$ satisfies a linear holomorphic differential equation \eqref{eq253} whose solutions are determined by their (initial) values at any fixed $q$ and $\arg q$. Write $\uppsi_q(w_\blt)=\sum_{n\in\Cbb}\uppsi_n(w_\blt)q^n$, then $\uppsi_n(w_\blt)=0$ for all $n\in\Cbb$.\footnote{Suppose $f(q)=\sum_{n\in\Cbb}a_nq^n\in\Cbb\{q\}$ converges absolutely and equals $0$ when $q\neq 0$ is small. If $f(q)\in\Cbb[[q^{\pm1}]]$, then by taking contour integrals one concludes $a_n=0$ for all $n$. In the general case, one has to be more careful. See the discussions in \cite[Sec. 4.4]{Gui}.}

Let $\mc F$ be the set of all $\Mbb\in\mc E$ such that $\upphi_\Mbb\neq0$. Let us prove that $\mc F=\emptyset$. Note that
\begin{align*}
\uppsi_q(w_\blt)=\sum_{n\in\Cbb}\sum_{\Mbb\in\mc F}q^n\cdot \upphi_\Mbb(w_\blt\otimes \chi_{\Mbb,n})
\end{align*}
where $\chi_{\Mbb,n}\in\Mbb_{(n)}\otimes\Mbb_{(n)}^*$ (where $\Mbb_{(n)}$ is the $L_0$-weight $n$ subspace of $\Mbb$) is the vector for contraction, i.e., $\chi_{\Mbb,n}=\id_{\Mbb_{(n)}}\in\End(\Mbb_{(n)})$. Let $\Xbb=\bigoplus_{\Mbb\in\mc F}\Mbb\otimes\Mbb'$ as a (semisimple) $\Vbb\otimes\Vbb$-module. Then for each $n\in\Cbb$, $\uppsi_n(w_\blt)=0$ means 
\begin{align*}
\Bigbk{\bigoplus_{\Mbb\in\mc F}\upphi_\Mbb,~w_\blt\otimes \big(\bigoplus_{\Mbb\in\mc F} \chi_{\Mbb,n}\big)}=0.
\end{align*}
where $\bigoplus_{\Mbb\in\mc F}\upphi_\Mbb:\Wbb_\blt\otimes\Xbb\rightarrow\Cbb$ is a linear map defined by sending each $w_\blt\otimes m\otimes m'$ (where $m\in\Mbb,m'\in\Mbb'$) to $\upphi_\Mbb(w_\blt\otimes m\otimes m')$.

Let
\begin{align*}
\Abb=\big\{\nu\in\Xbb:\bigbk{\bigoplus_{\Mbb\in\mc F}\upphi_\Mbb,w_\blt\otimes\nu}=0\text{ for all }w_\blt\in\Wbb_\blt\big\}.
\end{align*}
We claim that $\Abb$ is a $\Vbb\otimes\Vbb$-invariant subspace of $\Xbb$. It follows that if $\mc F\neq\emptyset$, then $\Abb$ (which contains all $\chi_{\Mbb,n}$) must be a non-trivial $\Vbb\otimes\Vbb$-submodule of the semisimple module $\Xbb$. Since the irreducible summands $\Mbb\otimes\Mbb'$ of $\Xbb$ are mutually non-isomorphic,  $\Abb$ must contain some $\Mbb\otimes\Mbb'$ where $\Mbb\in\mc F$. Then $\upphi_\Mbb=0$, contradicting the definition of $\mc F$.

That $\Abb$ is $\Vbb\otimes\Vbb$-invariant can be argued in the same way as Prop. \ref{lb167}. Choose any $\nu\in\Abb$. Then when $z\neq 0$ is small, for each $u\in\Vbb$,
\begin{align}
\bigbk{\bigoplus_{\Mbb\in\mc F}\upphi_\Mbb,Y(u,z)w_1\otimes w_2\otimes\cdots\otimes w_N\otimes\nu}
\end{align}
converges a.l.u. and equals $0$. This shows that $\bigoplus_{\Mbb\in\mc F}\wr\upphi_\Mbb(\cdot,w_\blt\otimes\nu)$ vanishes near $x_1$, and hence near $x'$. Therefore, for all $u\in\Vbb$,
\begin{align*}
\bigbk{\bigoplus_{\Mbb\in\mc F}\upphi_\Mbb,w_\blt\otimes(Y(u,z)\otimes \id)\nu}
\end{align*}
equals $0$. So $(Y(u)_n\otimes\id)v\in\Abb$ for all $u\in\Vbb,n\in\Zbb$, which proves that $\Abb$ is $\Vbb\otimes\id$-invariant. Similarly, $\Abb$ is $\id\otimes\Vbb$-invariant.
\end{proof}

\begin{rem}
If we define $\wtd{\fk S}_q$ using the normalized sewing $\wtd{\mc S}_q$,  then using the fact that for each $\Mbb$, $\mc S_q\upphi_\Mbb=q^d\wtd{\mc S}_q\upphi_\Mbb$ for some $d\in\Cbb$, one shows easily that $\wtd{\fk S}_q$ is also injective.
\end{rem}

\subsection{}

\begin{co}\label{lb174}
Assume that $\Vbb$ is $C_2$-cofinite. Then there are only finitely many equivalence classes of irreducible $\Vbb$-modules.
\end{co}

\begin{proof}
For each finite set $\mc E$  as in the previous subsection, we give an upper bound for its cardinality $|\mc E|$. For each $\Mbb\in\mc E$, the vertex operator $Y_\Mbb$ defines a conformal block $\upomega_\Mbb:\Mbb\otimes\Vbb\otimes\Mbb'\rightarrow\Cbb$ for $\fk P=(\Pbb^1;0,1,\infty;\zeta,\zeta-1,1/\zeta)$ as in Example \ref{lb173}. Sewing $\fk P$ along $0,\infty$ with a fixed parameter $q\in\Dbb_1^\times$ gives a $1$-pointed torus $\fk T$. By Thm. \ref{lb172}, $\{\wtd{\mc S}_q\upomega_\Mbb:\Mbb\in\mc E\}$ is a linearly independent subset of $\scr T_{\fk T}^*(\Vbb)$. Therefore $|\mc E|\leq \dim\scr T_{\fk T}^*(\Vbb)$.
\end{proof}

Recall that $\Vbb$ is called rational if every admissible $\Vbb$-module is a direct sum of irreducible $\Vbb$-module.

\begin{thm}[Factorization]\label{lb181}
Assume that $\Vbb$ is $C_2$-cofinite and rational. Assume that $\mc E$ is a maximal set of mutually inequivalent irreducible $\Vbb$-modules. (``Maximal" means that every irreducible $\Vbb$-module is isomorphic to one element of $\mc E$.) Then the linear map $\fk S_q$ defined by \eqref{eq254} is an isomorphism.
\end{thm}

Factorization is equivalent to that for maximal $\mc E$,
\begin{align}
\dim\scr T_{\fk X_q}^*(\Wbb_\blt)=\sum_{\Mbb\in\mc E}\dim\scr T_{\wtd{\fk X}}^*(\Wbb_\blt\otimes\Mbb\otimes\Mbb').
\end{align}
This formula gives an algorithm of calculating the dimensions of spaces of conformal blocks of higher genera or more marked points from those of lower genera or less marked points. Factorization in this form was first proved by \cite{DGT19} using Zhu's algebras. In \cite[Sec. 4.6, 4.7]{Gui}, a different but more analytic and geometric proof was given using (a slight generalization of) double-propagations. The proof of factorization has a long history. In particular, the factorization of WZW conformal blocks was first proved in the landmark paper \cite{TUY89}. See the Introduction of \cite{DGT19} for a discussion of the history.

\begin{co}
Assume that $\Vbb$ is $C_2$-cofinite and rational. Then the number of equivalence classes of irreducible $\Vbb$-modules equals the dimension of the space of conformal blocks associated to any $1$-pointed torus $\fk T$ and the vacuum module $\Vbb$. 
\end{co}
\begin{proof}
This follows immediately from factorization and the proof of Cor. \ref{lb174}.
\end{proof}

\section{Genus 0 conformal blocks and tensor categories of VOA modules}

\subsection{}

In this section, we still follow Convention \ref{lb175}: $\Vbb$-modules mean finitely-admissible $\Vbb$-modules. For each $z_\blt\in\Conf^N(\Cbb)$, let
\begin{align*}
\fk X_{z_\blt}=(\Pbb^1;z_1,\dots,z_N,\infty;\zeta-z_1,\dots,\zeta-z_N,1/\zeta)
\end{align*}
where $\zeta$ is the standard coordiante of $\Cbb$. Choose $\Vbb$-modules $\Wbb_1,\dots,\Wbb_N,\Wbb_{N+1}'$ associated to $z_1,\dots,z_N,\infty$. Write
\begin{align*}
\Wbb_\blt=\Wbb_1\otimes\cdots\otimes\Wbb_N,
\end{align*}
namely, $N+1$ are not included in the $\blt$. Note that a linear functional on $\Wbb_\blt\otimes\Wbb_{N+1}'$ is equivalently a linear map $\Wbb_\blt\rightarrow\Wbb_{N+1}^\cl$.

\subsection{}

We give a criterion to decide whether a linear functional on $\Wbb_\blt\otimes\Wbb_{N+1}'$ is a conformal block.

\begin{pp}\label{lb176}
A linear map $\mc Y:\Wbb_\blt\rightarrow\Wbb_{N+1}^\cl$ belongs to $\scr T_{\fk X_{z_\blt}}^*(\Wbb_\blt\otimes\Wbb_{N+1}')$ if and only if the following condition holds: For each $w_\blt\in\Wbb_\blt,w_{N+1}'\in\Wbb_{N+1}'$ and $u\in\Vbb$, the following formal Laurent series
\begin{gather*}
\bigbk{w_{N+1}',\mc Y(w_1\otimes\cdots\otimes Y(u,z-z_i)w_i\otimes\cdots\otimes w_N)}\qquad \in\Cbb((z-z_i))
\end{gather*}
(for all $i=1,\dots,N$) and the series
\begin{align*}
\bigbk{Y(u,z)^\tr w_{N+1}',\mc Y(w_\blt)}\qquad\in\Cbb((z^{-1}))
\end{align*}
are expansions at $z_1,\dots,z_{N+1},\infty$ of the same function $f\in H^0(\Pbb^1,\scr O_{\Pbb^1}(\star z_\blt+\star\infty))$.
\end{pp}

According to our notation \eqref{eq255}, $f\in H^0(\Pbb^1,\scr O_{\Pbb^1}(\star z_\blt+\star\infty))$ means simply that $f$ is a meromorphic  function on $\Pbb^1$ (i.e. a rational function) with possible (finite) poles only at $z_\blt,\infty$.

\begin{proof}
``If": Denote the function $f$ in the Proposition by $f_{u,w_\blt\otimes w_{N+1}'}$. Then we can define an $\scr O_{\Cbb\setminus x_\blt}$-module morphism
\begin{gather*}
\wr\mc Y(\cdot,w_\blt\otimes w_{N+1}'):\scr V_{\Cbb\setminus x_\blt}\rightarrow\scr O_{\Cbb\setminus x_\blt}\\
\mc U_\varrho(\zeta)^{-1}(u\otimes g)\mapsto g\cdot f_{u,w_\blt\otimes w_{N+1}'}
\end{gather*}
where $u\otimes g\in \Vbb\otimes_\Cbb\scr O_{\Cbb\setminus x_\blt}$. Using the same argument as in Example \ref{lb173}, one checks that $\wr\mc Y$ satisfies the conditions in the complex analytic definition of conformal blocks.

``Only if": Let $v\in H^0(\Pbb^1,\scr V_{\Pbb^1}(\star\infty))$ be $v=\mc U_\varrho(\zeta)^{-1}u$. Let $f=\wr\mc Y(v,w_\blt\otimes w_{N+1}')$ by viewing $\mc Y$ as a linear functional on $\Wbb_\blt\otimes\Wbb_{N+1}'$. One checks that $f$  satisfies the claim in Prop. \ref{lb176}.
\end{proof}

\subsection{}

\begin{eg}\label{lb179}
Consider the case that $N=1$ and $z_1=0$. Then $\mc Y:\Wbb_1\rightarrow\Wbb_2^\cl$ is a conformal block iff $f_1(z)=\bk{w_2,\mc YY(u,z)w_1}$ (which is in $\Cbb((z))$) and $f_2=\bk{Y(u,z)^\tr w_2,\mc Yw_1}$ (in $\Cbb((z^{-1}))$) are expansions at $0,\infty$ of some $f\in H^0(\Pbb^1,\scr O_{\Pbb^1}(\star0+\star\infty))=\Cbb[z^{\pm1}]$. This is equivalent to that  $\Res_{z=0} f_1(z)z^kdz+\Res_{z=\infty}f_2(z)z^kdz=0$, and hence equivalent to that $\bk{w_\blt,[Y(u)_k,\mc Y]w_1}=0$, namely $\mc Y\in\Hom_\Vbb(\Wbb_1,\Wbb_2^\cl)$. We conclude
\begin{align}
\Hom_\Vbb(\Wbb_1,\Wbb_2^\cl)\simeq\scr T_{(\Pbb^1;0,\infty;\zeta,1/\zeta)}^*(\Wbb_1,\Wbb_2').
\end{align} 
\hfill\qedsymbol
\end{eg}

\begin{eg}
Consider the case that $N=0$ and  $\Wbb_1=\Vbb$, which is associated to the only marked point $\infty$. In this case, $\xi:=\mc Y$ belongs to $\Vbb^\cl$. According to Prop. \ref{lb176}, $\mc Y$ is a conformal block iff for each $w'\in\Wbb',u\in\Vbb$, $\bk{w',Y(u,z) \xi}$ belongs to $H^0(\Pbb^1,\scr O_{\Pbb^1}(\star\infty))=\Cbb[z]$. Equivalently, $Y(u)_n\xi=0$ whenever $u\in\Vbb,n\in\Nbb$. In particular,  $L_0\xi=Y(\cbf)_1\xi=0$. We conclude
\begin{align}
\scr T_{(\Pbb^1;\infty;1/\zeta)}^*(\Vbb')=\{v\in\Vbb(0):Y(u)_nv=0\text{ for all }u\in\Vbb,n\in\Nbb\}.
\end{align}
Note also that by Subsec. \ref{lb188}, $Y(u)_nv=0$ for all $n\in\Nbb$ iff $[Y(u,z_1),Y(v,z_2)]=0$ in $\Cbb[[z_1^{\pm1},z_2^{\pm1}]]$. 

It follows that if $\Vbb\simeq\Vbb'$ (self-dual) and $\Vbb(0)=\Cbb\id$ (CFT-type), then
\begin{align}
\scr T_{(\Pbb^1;\infty;1/\zeta)}^*(\Vbb)\simeq\scr T_{(\Pbb^1;\infty;1/\zeta)}^*(\Vbb')=\Cbb\id.
\end{align}
Therefore, by Example \ref{lb189}, $\End_\Vbb(\Vbb)=\Cbb\id_\Vbb$. This implies that if $\Vbb$ is completely reducible, i.e. a sum of irreducible $\Vbb$-modules, then $\Vbb$ must be simple (i.e. an irreducible $\Vbb$-module). We remark that without assuming completely reducibility, one can also deduce that $\Vbb$ is simple from self-dualness and CFT-type. See for instance \cite[Prop. 4.6]{CKLW18}.\hfill\qedsymbol
\end{eg}

Now consider the case $N=2$ and $z_2=0$. Set $\xi=z_1$ which is non-zero. In this case, we write a conformal block $\mc Y(w_1\otimes w_2)$ as $\mc Y(w_1,\xi)w_2$.

\begin{pp}
A linear map $\mc Y(\cdot,\xi):\Wbb_1\otimes\Wbb_2\rightarrow\Wbb_3^\cl$ is an element of $\scr T_{(\Pbb^1;0,\xi,\infty;\zeta,\zeta-\xi,1/\zeta)}^*(\Wbb_1\otimes\Wbb_2\otimes\Wbb_3')$ if and only if for each $w_1\in\Wbb_1$ and $u\in\Vbb$
\begin{align}\label{eq257}
\begin{aligned}
&\sum_{l\in\Nbb}{m\choose l}\xi^{m-l}\cdot \mc Y(Y(u)_{n+l}w_1,\xi)\\
=&\sum_{l\in\Nbb}{n\choose l}(-\xi)^l\cdot Y(u)_{m+n-l}\mc Y(w_1,\xi)-\sum_{l\in\Nbb}{n\choose l}(-\xi)^{n-l}\cdot\mc Y(w_1,\xi)Y(u)_{m+l}
\end{aligned}
\end{align}
\end{pp}

\begin{proof}
With the help of Prop. \ref{lb176}, we can prove the only if part by taking residues as in Prop. \ref{lb38}, and prove the if part using strong residue Thm. \ref{lb104}.
\end{proof}

\subsection{}

Assume that $\Vbb$ is $C_2$-cofinite and $\Wbb_1,\dots,\Wbb_2$ are finitely generated. We assemble $\fk X_{z_\blt}$ to a family
\begin{align}
\fk X=(\Pbb^1\times\Conf^N(\Cbb)\rightarrow\Conf^N(\Cbb);\sgm_1,\dots,\sgm_N,\infty;\eta_1,\dots,\eta_N,1/\zeta)
\end{align}
as in Example \ref{lb118}. Namely, $\sgm_i$ sends $z_\blt\in\Conf^N(\Cbb)$ to $(z_i,z_\blt)$, $\infty$ sends $z_\blt$ to $(\infty,z_\blt)$, $\eta_i$ sends $(z,z_\blt)$ to $z-z_i$, and $1/\zeta$ sends $(z,z_\blt)$ to $1/z$. Identify
\begin{align*}
\scr W_{\fk X}(\Wbb_\blt)=\Wbb_\blt\otimes\scr O_{\Conf^N(\Cbb)}\qquad\text{via }\mc U(\eta_\blt,1/\zeta).
\end{align*}
By Example \ref{lb180}, over the vector bundle $\scr T_{\fk X}^*(\Wbb_\blt)$ one has a (clearly flat) connection $\nabla$ defined by
\begin{align}
\nabla_{\partial_{\tau_k}}=\partial_{\tau_k}-\big(\id_{\Wbb_1}\otimes\cdots\otimes L_{-1}\big|_{\Wbb_k}\otimes\cdots\otimes\id_{\Wbb_N}\otimes\id_{\Wbb_\infty}\big)^\tr
\end{align}
for all $1\leq k\leq N$. 

Thus, if we fix $\gamma_\blt\in\Conf^N(\Cbb)$, then each element $\mc Y(\cdot,\gamma_\blt)\in\scr T_{\fk X}^*(\Wbb_\blt)|\gamma_\blt$ is extended to a parallel section $\mc Y(\cdot,z_\blt):w_\blt\in\Wbb_\blt\mapsto \mc Y(w_\blt,z_\blt)\in\Cbb$ on any simply-connected open subset of $\Conf^N(\Cbb)$ containing $\gamma_\blt$, and furthermore to a multivalued parallel section $\mc Y(\cdot,z_\blt)$ on $\Conf^N(\Cbb)$ (namely, single-valued on the universal cover of $\Conf^N(\Cbb)$).

\subsection{}

As a variant of the above family, we can consider the family
\begin{align*}
\fk P^N=(\pi:\Pbb^1\times\Conf^N(\Cbb^\times)\rightarrow\Conf^N(\Cbb^\times);0,\sgm_1,\dots,\sgm_N,\infty;\zeta,\eta_1,\dots,\eta_N,1/\zeta)
\end{align*}
in Example \ref{lb118}. Then similar properties hold for conformal blocks associated to $\fk P^N$.

\begin{df}
A  parallel section $\mc Y=\mc Y(w_1,\xi)w_2$ of $\scr T_{\fk P^1}^*(\Wbb_2\otimes\Wbb_1\otimes \Wbb_3')$ multivalued on $\xi\in\Cbb^\times$ (and hence single-valued on the universal cover of $\log\xi\in\Cbb^\times$) is called a \textbf{type $\Wbb_3\choose \Wbb_1\Wbb_2$ intertwining operator}. The space of these intertwining operators is denoted by $\mc I{\Wbb_3\choose \Wbb_1\Wbb_2}$. Its dimension is called the \textbf{fusion rule} between $\Wbb_1,\Wbb_2,\Wbb_3$.
\end{df}
Note that $\Wbb_2,\Wbb_1,\Wbb_\infty$ are associated to the sections $0,\sgm=\sgm_1,\infty$ respectively. Also, $\mc Y$ being parallel means that $\mc Y$ satisfies the \textbf{translation property}
\begin{align}
\partial_\xi\mc Y(w_1,\xi)=\mc Y(L_{-1}w_1,\xi).
\end{align}

\subsection{}

We now address a problem overlooked previously: is the vector space $\Wbb'$ independent of the operator $\wtd L_0$ that makes $\Wbb$ finitely-admissible?

Let us prove that this is true when $L_0$ is diagonalizable and each $L_0$-eigenspace is finite-dimensional (e.g. when $\Wbb$ is semi-simple).  Let $\Wbb=\bigoplus_{n\in\Cbb}\Wbb_{(n)}$ be the $L_0$-grading of $\Wbb$. We can define the graded dual $\Wbb^\vee=\bigoplus_{n\in\Cbb}\Wbb_{(n)}^*$ using the $L_0$-grading. Then the independence of $\Wbb'$ on $\wtd L_0$ follows from:

\begin{pp}
Suppose that $\Wbb$ has $L_0$-grading $\Wbb=\bigoplus_{n\in\Cbb}\Wbb_{(n)}$ where each $\Wbb_{(n)}$ is finite-dimensional. Suppose also that $\Wbb$ has an $\wtd L_0$-grading $\Wbb=\bigoplus_{n\in\Nbb}\Wbb(n)$ making $\Wbb$ finitely-admissible. Then $\Wbb'=\Wbb^\vee$.
\end{pp}

\begin{proof}
Consider the linear operators $\wtd L_0^\tr,L_0^\tr$ defined on $\Wbb^*$, namely,
\begin{align*}
\bk{\wtd L_0^\tr w',w}=\bk{w',\wtd L_0w},\qquad \bk{L_0^\tr w',w}=\bk{w',L_0w}
\end{align*}
for all $w\in \Wbb,w'\in \Wbb^*$. Notice the following facts which are stated for $L_0,W^\vee$ and also hold for $\wtd L_0,W'$ in a similar way:
\begin{itemize}
\item From $\Wbb^*=\prod_{n\in\Cbb}\Wbb_{(n)}^*$, we see that a vector $w'\in \Wbb^*$ belongs to $\Wbb^\vee$ iff $w'$ is a finite sum of eigenvectors of $L_0^\tr$.
\item Any generalized eigenvector $w'\in\Wbb^*$ of $L_0^\tr$ is an eigenvector of $L_0^\tr$. Namely, if $(L_0^\tr-\lambda)^kw'=0$ for some $k\in\Nbb$, then  $(L_0^\tr-\lambda)w'=0$. In particular, $L_0^\tr$ is diagonalizable  on each finite-dimensional $L_0^\tr$-invariant subspace of $\Wbb^*$. 
\end{itemize}
By Rem. \ref{lb81}, $L_0$ and $\wtd L_0$ commute on $\Wbb$. So $L_0^\tr$ and $\wtd L_0^\tr$ commute on $\Wbb^*$. Therefore, since $\Wbb(n)^*$ is the $n$-eigenspace of $\wtd L_0^\tr$ on $\Wbb^*$, we see that $\Wbb(n)^*$ is $L_0^\tr$-invariant, and hence $L_0^\tr|_{\Wbb(n)^*}$ is diagonalizable. This proves that any $\wtd L_0^\tr$-eigenvector is a finite sum of $L_0^\tr$-eigenvectors. Therefore $\Wbb'\subset\Wbb^\vee$. A similar argument shows $\Wbb^\vee\subset\Wbb'$.
\end{proof}

\subsection{}

We now assume for simplicity that $\Wbb_1,\Wbb_2,\Wbb_3$ are semisimple, and show that our definition of intertwining operators agrees with the usual ones in the literature (for instance \cite{FHL93}). 

In \eqref{eq257}, set $n=0,u=\cbf,m=1$, we get
\begin{align*}
&[L_0,\mc Y(w_1,\xi)]=\xi\mc Y(L_{-1}w_1,\xi)+\mc Y(L_0w_1,\xi)\\
=&\xi\partial_\xi \mc Y(w_1,\xi)+\mc Y(L_0w_1,\xi).
\end{align*}
It follows that we have \textbf{scale covariance} (assuming $ z\neq 0$)
\begin{align}
 z^{L_0}\mc Y(w_1,\xi) z^{-L_0}=\mc Y( z^{L_0}w_1, z\xi).\label{eq258}
\end{align}
Here, and in the rest of this section, we adhere to the following convention, which is necessary since both $ z^{L_0}$ and $\mc Y(\cdot,\xi)$ depends on the arguments of the variables $ z,\xi$.
\begin{cv}\label{lb183}
We assume $\arg( z\xi)=\arg z+\arg\xi$ and $\arg z^{-1}=-\arg z$. If $a\in\Rbb$, we assume $\arg z^a=a\arg z$. By a positive variable $r>0$, we assume unless otherwise stated that $\arg r=0$. We assume $\arg 1=0$, $\arg e^{\im\theta}=\theta$ (where $\theta\in\Rbb$).
\end{cv}
Set $\xi=1$. Then \eqref{eq258} shows
\begin{align}
\bk{\mc Y(w_1,z)w_2,w_3'}=\bk{\mc Y(z^{-L_0}w_1,1)z^{-L_0}w_2,z^{L_0}w_3}
\end{align}
which must be a (finite) linear combination of (non-necessarily integral) powers of $z$ since it is so when $w_1,w_2,w_3'$ are $L_0$-homogeneous. Thus we can write
\begin{align*}
\mc Y(w_1,z)=\sum_{n\in\Cbb}\mc Y(w_1)_nz^{-n-1}
\end{align*}
where each $\mc Y(w_1)_n:\Wbb_2\rightarrow\Wbb_3^\cl$ satisfies
\begin{align}
[L_0,\mc Y(w_1)_n]=\mc Y(L_0w_1)_n-(n+1)\mc Y(w_1)_n.
\end{align}
This shows that if $w_1$ is $L_0$-homogeneous with weight $\wt w_1$, then $\mc Y(w_1)_n$ raises the $L_0$-weights by $\wt w_1-n-1$. In particular, $\mc Y(w_1)_n$ is a linear map
\begin{align*}
\mc Y(w_1)_n:\Wbb_2\rightarrow\Wbb_3.
\end{align*}
Thus, by checking the coefficients before each powers of $\xi$ in \eqref{eq257}, we obtain:
\begin{pp}
Let $\Vbb$ be $C_2$-cofinite and $\Wbb_1,\Wbb_2,\Wbb_3$ be semisimple. Then a type $\Wbb_3\choose \Wbb_1\Wbb_2$ intertwining operator is equivalently a linear map
\begin{gather*}
\mc Y:\Wbb_1\rightarrow\Hom(\Wbb_2,\Wbb_3)\{z\}\\
w_1\mapsto \mc Y(w_1,z)=\sum_{n\in\Cbb}\mc Y(w_1)_nz^{-n-1}
\end{gather*}
satisfying the following conditions
\begin{itemize}
\item \textbf{Jacobi identity}: For each $u\in\Vbb$, $w_1\in\Wbb_1$, $m,n\in\Zbb$, and $k\in\Cbb$,
\begin{align}
	\begin{aligned}
&\sum_{l\in\Nbb}{m\choose l}\mc Y\big(Y(u)_{n+l}w_1\big)_{m+k-l}\\
=&\sum_{l\in\Nbb}(-1)^l{n\choose l}Y(u)_{m+n-l}\mc Y(w_1)_{k+l}-\sum_{l\in\Nbb}(-1)^{n+l}{n\choose l}\mc Y(w_1)_{n+k-l} Y(u)_{m+l}	.
	\end{aligned}
\end{align}
\item \textbf{Translation property}: For each $w_1\in\Wbb_1$, we have $[L_{-1},\mc Y(w_1,z)]=\mc Y(L_{-1}w_1,z)$.
\end{itemize}
\end{pp}

Note also that by  setting $n=0,u=\cbf,m=0$ in \eqref{eq257}, we get
\begin{align}
[L_{-1},\mc Y(w_1,z)]=\mc Y(L_{-1}w_1,z).\label{eq262}
\end{align}

\subsection{}

\begin{ass}
In the remaning part of this section, we assume $\Vbb$ is $C_2$-cofinite and rational. 
\end{ass}


We shall construct the braided tensor category $\Rep(\Vbb)$ of semisimple $\Vbb$-modules, due to Huang and Lepowsky. See \cite{BK,EGNO} for the definition of braided tensor categories.

The objects of $\Rep(\Vbb)$ are semisimple $\Vbb$-modules, and the morphism space between two objects $\Wbb_1,\Wbb_2$ is $\Hom_\Vbb(\Wbb_1,\Wbb_2)$. This makes $\Rep(\Vbb)$ a semisimple abelian category.

Fix $\mc E$ to be a maximal set of mutually-inequivalent irreducible $\Vbb$-modules as in the factorization Thm. \ref{lb181}. Recall that $\mc E$ is finite by Cor. \ref{lb174}. For each semisimple $\Wbb_1,\Wbb_2$, define the tensor product (more precisely, \textbf{fusion product})
\begin{align}
\Wbb_1\boxtimes\Wbb_2=\bigoplus_{\Wbb_s\in\mc E}\Wbb_s\otimes\mc I{\Wbb_s\choose \Wbb_1\Wbb_2}^*
\end{align}
where $\mc I{\Wbb_s\choose \Wbb_1\Wbb_2}^*$ is the dual space of the (finite-dimensional) space $\mc I{\Wbb_s\choose \Wbb_1\Wbb_2}$. If $F\in\Hom_\Vbb(\Wbb_1,\Wbb_3)$ and $G\in\Hom_\Vbb(\Wbb_2,\Wbb_4)$, then the transpose of the linear map
\begin{gather}
\begin{gathered}
(F\otimes G)^\tr:\mc I{\Wbb_s\choose \Wbb_3\Wbb_4}\rightarrow \mc I{\Wbb_s\choose \Wbb_1\Wbb_2}\\
\mc Y(\cdot,z)\mapsto\mc Y(F\cdot,z)G 
\end{gathered}
\end{gather}
gives a linear map $\mc I{\Wbb_s\choose \Wbb_1\Wbb_2}^*\rightarrow \mc I{\Wbb_s\choose \Wbb_3\Wbb_4}^*$, whose tensor product with $\id_{\Wbb_s}$, added up over all $\Wbb_s\in\mc E$, gives the definition of fusion product of morphisms
\begin{align}
F\boxtimes G:\Wbb_1\boxtimes\Wbb_2\rightarrow\Wbb_3\boxtimes\Wbb_4.
\end{align}

\subsection{}

We have an obvious equivalence
\begin{align*}
\Hom_\Vbb(\Wbb_1\boxtimes\Wbb_2,\Wbb_3)\simeq \bigoplus_{\Wbb_s\in\mc E}\Hom_\Vbb(\Wbb_s,\Wbb_3)\otimes \mc I{\Wbb_s\choose \Wbb_1\Wbb_2}.
\end{align*}
Through the isomorphism
\begin{gather}
\begin{gathered}
\bigoplus_{\Wbb_s\in\mc E}\Hom_\Vbb(\Wbb_s,\Wbb_3)\otimes \mc I{\Wbb_s\choose \Wbb_1\Wbb_2}\xrightarrow{\simeq} \mc I{\Wbb_3\choose \Wbb_1\Wbb_2}\\
T\otimes\mc Y(\cdot,z)\mapsto T\circ\mc Y(\cdot,z) 
\end{gathered}
\end{gather}
we get an isomorphism
\begin{align}
\Hom_\Vbb(\Wbb_1\boxtimes\Wbb_2,\Wbb_3)\simeq\mc I{\Wbb_3\choose \Wbb_1\Wbb_2}\label{eq266}
\end{align}
which is functorial in the sense that if $F:\Mbb_1\rightarrow\Wbb_1$, $G:\Mbb_2\rightarrow\Wbb_2$, $H:\Wbb_3\rightarrow\Mbb_3$ are morphisms, and if $\mc Y\in\mc I{\Wbb3\choose \Wbb_1\Wbb_2}$ corresponds to $T\in\Hom_\Vbb(\Wbb_1\boxtimes\Wbb_2,\Wbb_3)$, then $H\circ\mc Y(F\cdot,z)G$ corresponds to $HT(F\boxtimes G)$ in $\Hom_\Vbb(\Mbb_1\boxtimes\Mbb_2,\Mbb_3)$.

As a special case, we have
\begin{align}
\Hom_\Vbb(\Vbb\boxtimes\Wbb,\Wbb)\simeq\mc I{\Wbb\choose\Vbb~\Wbb}.
\end{align}
By Thm. \ref{lb166} (propagation of conformal blocks) and Subsec. \ref{lb182}, we have an isomorphism
\begin{gather*}
\mc I{\Wbb\choose\Vbb~\Wbb}\xrightarrow{\simeq}\End_\Vbb(\Wbb)\qquad\mc Y\mapsto \mc Y(\id,z)
\end{gather*}
This shows that $\dim\Hom_\Vbb(\Vbb\boxtimes\Wbb,\Wbb)=1$ of $\Wbb$ is irreducible. Therefore, we have a standard homomorphism (the left unitor)
\begin{align}
\mu_L:\Vbb\boxtimes\Wbb\xrightarrow{\simeq}\Wbb
\end{align}
which corresponds to $\id_\Wbb$ in $\End_\Vbb(\Wbb)$. Clearly, $\mu_L$ corresponds to the vertex operator $Y_\Wbb$ in $\mc I{\Wbb\choose\Vbb~\Wbb}$.

$\mu_L$ is indeed an isomorphism. This is easy to see when $\Wbb\in\mc E$. The general case follows by taking direct sums and applying the functoriality of the isomorphism \eqref{eq266}.

The right unitor $\mu_R:\Wbb\boxtimes\Vbb\xrightarrow{\simeq}\Wbb$ is defined using the braiding $\ss Y_\Wbb\in\mc I{\Wbb\choose\Wbb\Vbb}$ of the vertex operator $Y_\Wbb$, where $\ss$ is defined in Subsec. \ref{lb187}.

\subsection{}

To define the associativity isomorphism
\begin{align*}
\mc A:(\Wbb_1\boxtimes\Wbb_2)\boxtimes\Wbb_3\xrightarrow{\simeq}\Wbb_1\boxtimes(\Wbb_2\boxtimes\Wbb_3),
\end{align*}
we write
\begin{align*}
&\Wbb_1\boxtimes(\Wbb_2\boxtimes\Wbb_3)=\bigoplus_{\Wbb_t\in\mc E}\Wbb_t\otimes\mc I{\Wbb_t\choose \Wbb_1,\Wbb_2\boxtimes\Wbb_3}^*\\
=&\bigoplus_{\Wbb_t\in\mc E}\Wbb_t\otimes\mc I{\Wbb_t\choose \Wbb_1,\bigoplus_{\Wbb_p\in\mc E}\Wbb_p\otimes\mc I{\Wbb_p\choose\Wbb_2\Wbb_3}^*}^*.
\end{align*}
Note that in general, for any finite-dimensional vector space $J$ we have an equivalence
\begin{gather*}
\mc I{\Wbb_t\choose \Wbb_1,\Wbb_p}\otimes J^*\xrightarrow{\simeq}\mc I{\Wbb_t\choose \Wbb_1,\Wbb_p\otimes J}\\
\mc Y\otimes\omega\mapsto \Psi_{\mc Y\otimes\omega}
\end{gather*}
where $\Psi_{\mc Y\otimes\omega}(w_1,z)(w_p\otimes \xi)$ (where $w_p\in\Wbb_p$ and $\xi\in J$) equals $\mc Y(w_1,z)w_p\cdot \bk{\omega,\xi}$. So we have a canonical equivalence
\begin{align}
\Wbb_1\boxtimes(\Wbb_2\boxtimes\Wbb_3)\simeq \bigoplus_{\Wbb_t\in\mc E}\bigoplus_{\Wbb_p\in\mc E}\Wbb_t\otimes \mc I{\Wbb_t\choose \Wbb_1\Wbb_p}^*\otimes \mc I{\Wbb_p\choose\Wbb_2\Wbb_3}^*
\end{align}
Similarly, we have a canonical equivalence
\begin{align}
(\Wbb_1\boxtimes\Wbb_2)\boxtimes\Wbb_3\simeq\bigoplus_{\Wbb_t\in\mc E}\bigoplus_{\Wbb_s\in\mc E}\Wbb_t\otimes\mc I{\Wbb_t\choose \Wbb_s\Wbb_3}^*\otimes\mc I{\Wbb_s\choose \Wbb_1\Wbb_2}^*.
\end{align}
Thus, the associativity map can be defined such that on each component it is  $\id_{\Wbb_t}$ tensoring the transpose of an  isomorphism
\begin{align}
\mc F:\bigoplus_{\Wbb_p\in\mc E}\mc I{\Wbb_t\choose \Wbb_1\Wbb_p}\otimes \mc I{\Wbb_p\choose\Wbb_2\Wbb_3}\xlongrightarrow{\simeq}\bigoplus_{\Wbb_s\in\mc E}\mc I{\Wbb_t\choose \Wbb_s\Wbb_3}\otimes\mc I{\Wbb_s\choose \Wbb_1\Wbb_2}.\label{eq260}
\end{align}
Let us define $\mc F$.

\subsection{}

Choose any $0<r<\rho$. Recall that by Convention \ref{lb183}, $\arg r=\arg\rho=0$. Recall that by the notations in Example \ref{lb118}, $\fk P^2_{r,\rho}=(\Pbb^1;0,r,\rho,\infty;\zeta,\zeta-r,\zeta-\rho,1/\zeta)$. By Thm. \ref{lb181}, we have an isomorphism
\begin{gather*}
\mc F_{1,23}:\bigoplus_{\Wbb_p\in\mc E}\mc I{\Wbb_t\choose \Wbb_1\Wbb_p}\otimes \mc I{\Wbb_p\choose\Wbb_2\Wbb_3}\xrightarrow{\simeq}\scr T_{\fk P^2_{r,R}}^*(\Wbb_3\otimes\Wbb_2\otimes\Wbb_1\otimes\Wbb_t')
\end{gather*}
sending each $\mc Y_\alpha\otimes\mc Y_\beta$ to the linear functional on $\Wbb_3\otimes\Wbb_2\otimes\Wbb_1\otimes\Wbb_t'$ defined by
\begin{align}\label{eq259}
\bigbk{w_t',\mc Y_\alpha(w_1,\rho)\mc Y_\beta(w_2,r)w_3}:=\sum_{n\in\Nbb}\bigbk{w_t',\mc Y_\alpha(w_1,\rho)P_n\mc Y_\beta(w_2,r)w_3}.
\end{align}
\eqref{eq259} corresponds to the sewing as in Subsec. \ref{lb185}. Therefore, the RHS of \eqref{eq259} converges absolutely on the region $0<r<\rho$ thanks to Thm. \ref{lb171}. Now assume without loss of generalities that all the vectors are $L_0$-homogeneous. Then by scale covariance \eqref{eq258} and that $\wtd L_0-L_0$ is a scalar on the irreducible $\Wbb_p$, the RHS of \eqref{eq259} can be written as $\rho^ar^b$ (where $a,b\in\Cbb$) times
\begin{align*}
\sum_{n\in\Nbb}\bigbk{w_t',\mc Y_\alpha(w_1,1)P_n\mc Y_\beta(w_2,1)w_3}\cdot \left(\frac {~r~}{~\rho~}\right)^n.
\end{align*}
(Cf. the argument in Subsec. \ref{lb184}.) This shows that the absolute convergence of \eqref{eq259} implies the a.l.u. convergence on $0<r<\rho$.

Similarly, when $0<\rho-r<r$ we have an isomorphism
\begin{align*}
\mc F_{12,3}:\bigoplus_{\Wbb_s\in\mc E}\mc I{\Wbb_t\choose \Wbb_s\Wbb_3}\otimes\mc I{\Wbb_s\choose \Wbb_1\Wbb_2}\xrightarrow{\simeq}\scr T_{\fk P^2_{r,R}}^*(\Wbb_3\otimes\Wbb_2\otimes\Wbb_1\otimes\Wbb_t')
\end{align*}
sending each $\mc Y_\gamma\otimes\mc Y_\delta$ to the linear functional defined by
\begin{align}
\bigbk{w_t',\mc Y_\gamma\big(\mc Y_\delta(w_1,\rho-r)w_2,r\big)w_3}:=\sum_{n\in\Nbb}\bigbk{w_t',\mc Y_\gamma\big(P_n\mc Y_\delta(w_1,\rho-r)w_2,r\big)w_3}\label{eq261}
\end{align}
which converges a.l.u. on $0<\rho-r<r$ and correspond to the sewing in Subsec. \ref{lb186}.

We define \eqref{eq260} to be
\begin{align}
\mc F=\mc F_{12,3}^{-1}\circ\mc F_{1,23}
\end{align}
for any $r,\rho$ satisfying $0<\rho-r<r<\rho$. Using Thm. \ref{lb171}, one checks easily that both \eqref{eq259} and \eqref{eq261} are parallel sections. Therefore $\mc F$ is independent of the particular choice of $r,\rho$. 

\subsection{}\label{lb187}

It remains to define the braiding isomorphisms. We define an isomorphism
\begin{align}
\ss:\mc I{\Wbb_3\choose \Wbb_1\Wbb_2}\xrightarrow{\simeq}\mc I{\Wbb_3\choose \Wbb_2\Wbb_1}
\end{align}
as follows. For each $\mc Y\in\mc I{\Wbb_3\choose \Wbb_1\Wbb_2}$, note that $\mc Y(\cdot,1)\in\scr T_{(\Pbb^1;0,1,\infty;\zeta,\zeta-1,1/\zeta)}^*(\Wbb_2\otimes\Wbb_1\otimes\Wbb_3')$. Let  $\upgamma:[0,1]\rightarrow\Conf^2(\Cbb)$ be the path which is the anticlockwise rotation around $0.5$ from $(0,1)$ to $(1,0)$ by $\pi$, namely $\upgamma(t)=(0.5-0.5e^{\im\pi t},0.5+0.5e^{\im\pi t})$. Along this path we parallel-transport $\mc Y(\cdot,1)$ to an element of $\scr T_{(\Pbb^1;1,0,\infty;\zeta,\zeta-1,1/\zeta)}^*(\Wbb_2\otimes\Wbb_1\otimes\Wbb_3')$. We then define $\ss\mc Y$ such that $\bk{w_3',\ss\mc Y(w_2,1)w_1}$ is this element.

Now we can define the braiding 
\begin{gather}
\begin{gathered}
\upsigma:\Wbb_1\boxtimes\Wbb_2=\bigoplus_{\Wbb_s\in\mc E}\Wbb_s\otimes\mc I{\Wbb_s\choose \Wbb_1\Wbb_2}^*\xrightarrow{\simeq}\Wbb_2\boxtimes\Wbb_1=\bigoplus_{\Wbb_s\in\mc E}\Wbb_s\otimes\mc I{\Wbb_s\choose \Wbb_2\Wbb_1}^*\\
\upsigma=\bigoplus_{\Wbb_s\in\mc E}\id_{\Wbb_s}\otimes \ss^\tr
\end{gathered}
\end{gather}
where $\ss^\tr$ is the transpose of $\ss:\mc I{\Wbb_s\choose \Wbb_2\Wbb_1}\xrightarrow{\simeq}\mc I{\Wbb_s\choose \Wbb_1\Wbb_2}$.

One can check that with the unitors, associators, and braiding operators defined above, $\Rep(\Vbb)$ is a braided tensor category. Moreover, Huang showed in \cite{Hua08b} that
\begin{thm}
Suppose that $\Vbb$ is $C_2$-cofinite, rational, CFT type (i.e. $\Vbb(1)=\Cbb\id$), and self dual (i.e. $\Vbb\simeq\Vbb'$). Then $\Rep(\Vbb)$ is a rigid modular tensor category.
\end{thm}
Huang's proof relies on the convergence of genus-$1$ sewing and factorization \cite{Zhu96,Hua05} and (generalized) Verlinde formula \cite{Hua08a}.

\subsection{}

We give an explicit formula of $\ss\mc Y$ for any $\mc Y\in\mc I{\Wbb_3\choose\Wbb_1\Wbb_2}$. First assume $z$ is positive and $z>1$; in particular $\arg z=0$. Recall the path $\upgamma$ in Subsec. \ref{lb187}. Note that $\mc Y(\cdot, 1)$ is parallel-transported to $\mc Y(\cdot,z)$ along the rightward path $\upalpha_1(t)=(0,1-t+zt)$ in $\Conf^2(\Cbb)$ from $(0,1)$ to $(0,z)$, and $\ss\mc Y$ similarly along the rightward path $\upalpha_2(t)=(1-t+zt,0)$ from $(1,0)$ to $(z,0)$. Therefore, $\mc Y(\cdot,z)$ is parallel-transported to $\ss\mc Y(\cdot,z)$ along the path $\upalpha^{-1}_1*\upgamma*\upalpha_2$  from  $(0,z)$ to $(z,0)$, which is homotopic to $\upgamma_1*\upgamma_2$ where
\begin{itemize}
\item $\upgamma_1(t)=(0,e^{\im \pi t}z)$ is from $(0,z)$ to $(0,-z)$ where the first component is fixed and the second one is the anticlockwise rotation by $\pi$ around $0$ from $1$ to $-1$.
\item $\upgamma_2(t)=(tz,tz-z)$ is the rightward translation from $(0,-z)$ to $(z,0)$.
\end{itemize}
Thus, along $\upgamma_1$, $\mc Y(\cdot,z)$ is parallel-transported to $\mc Y(\cdot,e^{\im\pi}z)$. Let $\upphi_0:\Wbb_1\otimes\Wbb_2\otimes\Wbb_3'\rightarrow\Cbb$ be the conformal block $\bk{w_3',\mc Y(w_1,e^{\im\pi}z)w_2}$ associated to $(\Pbb^1;-1,0,\infty;\zeta+1,\zeta,1/\zeta)$ and $\Wbb_1,\Wbb_2,\Wbb_3'$. Parallel-transporting $\upphi_0$ along the  $\upgamma_2$  gives $\upphi=\upphi_t(w_1\otimes w_2\otimes w_3'):\Wbb_1\otimes\Wbb_2\otimes\Wbb_3'\rightarrow\scr O(\Cbb)$, considered as an $\Hom(\Wbb_1\otimes\Wbb_2\otimes\Wbb_3',\Cbb)$-valued power series of $\tau$. Then according to the definition of $\nabla$, we have
\begin{align*}
\partial_t\upphi_t(w_1\otimes w_2\otimes w_3')=z\upphi_t(L_{-1}w_1\otimes w_2\otimes w_3')+z\upphi_t(w_1\otimes L_{-1}w_2\otimes w_3').
\end{align*}
Let $\uppsi_t(w_1\otimes w_2\otimes w_3)=\bk{e^{tz L_1}w_3',\mc Y(w_1,e^{\im\pi}z)w_2}$. Then $\uppsi_0=\upphi_0$, and by \eqref{eq262},
\begin{align*}
&\partial_t\uppsi_t(w_1\otimes w_2\otimes w_3')=z\bk{L_1e^{tz L_1}w_3',\mc Y(w_1,e^{\im\pi}z)w_2}\\
=&z\uppsi_t(L_{-1}w_1\otimes w_2\otimes w_3')+z\uppsi_t(w_1\otimes L_{-1}w_2\otimes w_3').
\end{align*}
So by Lemma \ref{lb21}, we must have $\upphi_t=\uppsi_t$. Since $\upphi_1$ is given by $\ss\mc Y(\cdot,z)$, we obtain
\begin{align}
\bk{w_3',\ss\mc Y(w_2,z)w_1}=\bk{e^{zL_1}w_3',\mc Y(w_1,e^{\im\pi}z)w_2}\label{eq263}
\end{align}
when $z>1$, and hence for all $z\in\Cbb^\times$ and $\arg z$ by the uniqueness of analytic continuation. We write \eqref{eq263} for short as
\begin{align}
\ss\mc Y(w_2,z)w_1=e^{zL_{-1}}\mc Y(w_1,e^{\im\pi}z)w_2.\label{eq265}
\end{align}

\begin{rem}
Consider the vertex operator $Y=Y_\Vbb$ for $\Vbb$. Using \eqref{eq263}, one checks easily that for all $v'\in\Vbb'$, $\bk{v',Y(\id,z)\id}=\bk{v',\ss Y(\id,z)\id}$. Thus, by Prop. \ref{lb167}, we see that for all $u,v\in\Vbb$,
\begin{align}
\bk{v',Y(u,z)v}=\bk{v',e^{zL_{-1}}Y(v,-z)u}.
\end{align}
Namely,
\begin{align}
Y_\Vbb=\ss Y_\Vbb.
\end{align}
This fact is called \textbf{skew-symmetry}.
\end{rem}

\appendix

\section{Appendix: basic sheaf theory}

The language of sheaves of modules is inevitable in the theory of conformal blocks for the following reason. The spaces of conformal blocks are expected to form a vector bundle (equivalently, locally free sheaves). This result is highly nontrivial. Moreover, we need to formulate the notion of ``forming a vector bundle" in a precise way. To accomplish this goal, we need to expand the concept of vector bundles to that of sheaves of modules.

The goal of this appendix section is to get familiar with the basic language of sheaves. The key points are the following: The equivalence of holomorphic vector bundles and locally free sheaves, the description of dual vector bundles in terms of $\scr O_X$-module morphisms, the fibers of $\scr O_X$-modules and their relationship to the fibers of vector bundles.

\subsection{(Pre)sheaves and stalks}

By definition, a \textbf{presheaf} of (complex) vector spaces $\scr F$ associated to a topological space $X$ consists of the following data: for each open $U\subset X$ there is a vector space $\scr F(U)$, and for each open $V\subset U$, there is a linear map $\scr F(U)\rightarrow\scr F(V),s\mapsto s|_V$ called the \textbf{restriction map} such that $s|_U=s$, and $(s|_V)|_W=s|_W$ for all $s\in\scr F(U)$ if $W\subset V$ is open. Elements in $\scr F(U)$ are called \textbf{sections}.

A presheaf $\scr F$ is called a \textbf{sheaf} if it satisfies:
\begin{itemize}
	\item (Locality) If  $U\subset X$ is a union $U=\bigcup_{\alpha\in\fk A}U_\alpha$ of open subsets, and if $s\in\scr F(U)$ satisfies that $s|_{U_\alpha}=0$ for each $\alpha\in\fk A$, then $s=0$.
	\item (Gluing) If $U\subset X$ is a union $U=\bigcup_{\alpha\in\fk A}U_\alpha$ of open subsets, and if for each $\alpha$ there is an element $s_\alpha\in\scr F(U_\alpha)$ such that $s_\alpha|_{U_\alpha\cap U_\beta}=s_\beta|_{U_\alpha\cap U_\beta}$ for all $\alpha,\beta\in\fk A$, then there exists $s\in \scr F(U)$ whose restriction to each $U_\alpha$ is $s_\alpha$.
\end{itemize}
We also write \index{H@$H^0(X,\scr F)=\scr F(X)$}
\begin{align}
H^0(X,\scr F)=\scr F(X),	
\end{align}
regarding the space of global sections of $\scr F$ as the $0$-th cohomology group of $\scr F$.

If $Y$ is an open subset of $X$, then the set of all $\scr F(U)$ (where $U\subset Y$) form naturally a presheaf, which we denote by $\scr F_Y$ or $\scr F|_Y$.

Let $\scr F$ be a presheaf. For each $x\in X$, we let $\scr F_x$ be the set of all sections $s\in\scr F$ defined on a neighborhood of $x$, mod the equivalence relation that two elements $s,t$ of $\scr F_x$ are regarded as equal iff $s$ equals $t$ on a possibly smaller neighborhood of $x$ inside the open sets on which $s,t$ are defined. $\scr F_x$ is called the \textbf{stalk} of $\scr F$ at $x$, and elements in $\scr F_x$ are called \textbf{germs}. For each $s\in\scr F$ defined near $x$, the corresponding germ at $x$ is denoted by $s_x$. \index{F@$\scr F_x,s_x$, stalks and germs.}

\begin{rem}
	It is easy to see that the presheaf $\scr F$ satisfies locality iff the following holds: for every open $U\subset X$ and section $s\in\scr F(U)$, $s=0$ iff $s_x=0$ for all $x\in U$.
\end{rem}

\subsection{Sheafification}
We are not interested in presheaves that are not sheaves. And each presheaf $\scr F_0$ can be made a sheaf $\scr F$ through the following procedure called \textbf{sheafification}: 

For each open $U\subset X$, let $\scr F_1(U)$ be the set of all $s:=(s_\alpha)_{\alpha\in\fk A}$ where $(U_\alpha)_{\alpha\in\fk A}$ is an open cover of $U$, and $s_{\alpha_1,x}=s_{\alpha_2,x}$ for all $\alpha_1,\alpha_2\in\fk A$ and $x\in U_{\alpha_1}\cap U_{\alpha_2}$. $\scr F(U)$ is $\scr F_1(U)$ mod the following relation: let $(V_\beta)_{\beta\in\fk B}$ be another open cover. Then $s:=(s_\alpha)_{\alpha\in\fk A}$ and $t:=(t_\beta)_{\beta\in\fk B}$ are regarded equal iff $s_{\alpha,x}=t_{\beta,x}$ for all $\alpha\in\fk A$, $\beta\in\fk B$, and $x\in U_\alpha\cap V_\beta$. The linear combinations of $s$ and $t$ can be defined easily by replacing $(U_\alpha)_{\alpha\in\fk A}$ and $(V_\beta)_{\beta\in\fk B}$ by a common finer open cover, e.g. $(U_\alpha\cap V_\beta)_{\alpha\in\fk A,\beta\in\fk B}$.

Note that the stalk $(\scr F_0)_x$ can be naturally identified with that of the sheafification $\scr F_x$.

\subsection{(Pre)sheaves of modules and morphisms}

We now let $X$ be a complex manifold. Then all $\scr O(U)$ (where $U\subset X$ is open) form the sheaf $\scr O_X$ of holomorphic functions on $X$, called the \textbf{structure sheaf} of $X$.

\begin{eg}
Let $U\subset\Cbb^m$ be open. Then the stalk $\scr O_{U,0}=\scr O_{\Cbb^m,0}$ can be identified with the $\Cbb$-subalgebra of elements of $\Cbb[[z_1,\dots,z_m]]$ converging absolutely on an open ball centered at $0$.
\end{eg}

A \textbf{(pre)sheaf of $\scr O_X$-modules} $\scr F$ is a (pre)sheaf such that each $\scr F(U)$ is an $\scr O(U)$-module, and that for each open $V\subset U$, the restriction map $s\in\scr F(U)\mapsto s|_V\in\scr F(V)$ intertwines the actions of $\scr O(U)$, i.e., $(fs)|_V=f|_V\cdot s|_V$ for all $f\in\scr O(U)$. A sheaf of $\scr O_X$-modules is simply called an \textbf{$\scr O_X$-module}.

A \textbf{morphism of (resp. presheaves of) $\scr O_X$-modules} $\varphi:\scr E\rightarrow\scr F$ gives each open $U\subset X$ an $\scr O(U)$-module morphism $\varphi_U:\scr E(U)\rightarrow\scr F(U)$ that is compatible with the restriction to open subsets: if $V\subset U$ is open and $s\in\scr E(U)$ then $\varphi_U(s)|_V=\varphi_V(s|_V)$.

\begin{cv}
	We abbreviate each $\varphi_U(s)$ to $\varphi(s)$. So $\varphi(s|_V)=\varphi(s)|_V$.
\end{cv}

\begin{rem}
	Note that the stalk $\scr O_{X,x}$ of $\scr O_X$ at $x$ is a $\Cbb$-algebra. A morphism $\varphi:\scr E\rightarrow\scr F$ naturally gives an  $\scr O_{X,x}$-module morphism $\varphi_x:\scr E_x\rightarrow\scr F_x$. 
	
	Also, there is a natural $\scr O_X$-module morphism $\scr E^s\rightarrow\scr F^s$ where $\scr E^s$ and $\scr F^s$ are the sheafifications of $\scr E$ and $\scr F$. The corresponding stalk morphism $\varphi_x:\scr E^s_x\rightarrow\scr F^s_x$ agrees with $\varphi_x:\scr E_x\rightarrow\scr F_x$.
\end{rem}

\begin{eg}
	Any (holomorphic) vector bundle $\scr F$ \footnote{Unless otherwise stated, all vector bundles are holomorphic with finite ranks.} over $X$ is an $\scr O_X$-module. 
\end{eg}

\begin{eg}
	If $W$ is a finite dimensional  vector space, let $W\otimes_\Cbb\scr O_X$ be the presheaf whose space of sections on each open $U\subset X$ is $W\otimes_\Cbb\scr O(U)$. Then $W\otimes_\Cbb\scr O_X$ is naturally a sheaf, and hence an $\scr O_X$-module. It is regarded as the trivial vector bundle with fiber $W$. We often suppress the subscript $\Cbb$ in $W\otimes_\Cbb\scr O_X$.
	
When $W$ is infinite-dimensional, the above defined presheaf is not a sheaf since the gluing property does not hold  when considering an open subset $U\subset X$ that has infinitely many connected components. 	 We let $W\otimes_\Cbb\scr O_X$ denote the sheafification of this presheaf. Then $(W\otimes_\Cbb\scr O_X)(U)$ equals $W\otimes\scr O(U)$ if $U$ is connected, or more generally, iff $U$ has finitely many connected components. Thus, we have a natural isomorphism of $\scr O_{X,x}$-modules
\begin{align*}
(W\otimes\scr O_X)_x\simeq W\otimes\scr O_{X,x}.	
\end{align*}

Note that when $U$ is connected, elements of $W\otimes\scr O(U)$ can be viewed as holomorphic functions from $U$ to a finite-dimensional subspace of $W$. We shall call such sections \textbf{$W$-valued holomorphic functions}. \index{00@$W$-valued holomorphic functions} \hfill\qedsymbol
\end{eg}

\begin{cv}\label{lba103}
	The space of $\scr O_X$-module morphisms $\varphi:\scr E\rightarrow\scr F$ form a vector space, which is clearly an $\scr O(X)$-module such that $f\in\scr O(X)$ times $\varphi$ is $f\varphi$, sending each $s\in \scr E(U)$ (where $U\subset X$ is open) to $f|_U\cdot \varphi(s)$. We denote this $\scr O(X)$-module by $\Hom_{\scr O_X}(\scr E,\scr F)$. \index{HomEF@$\Hom_{\scr O_X}(\scr E,\scr F)$}
\end{cv}

\begin{eg}\label{lba104}
	Let $V,W$ be finite dimensional vector spaces. A morphism 
	\begin{align*}
		\varphi:V\otimes\scr O_X\rightarrow W\otimes\scr O_X	
	\end{align*}
	is equivalently a $\Hom(V,W)$-valued holomorphic function $\Phi$ on $X$. Indeed, choose basis $\{e_i\}$ of $V$ and $\{f_j\}$ of $W^*$. Identify each vector of $W$ as a constant section of $W\otimes\scr O(X)$. Then $\varphi(e_i)\in W\otimes\scr O(X)$, and $\Phi$  is a matrix-valued holomorphic function whose $(j,i)$-component is the function $x\mapsto\bk{f_j,\varphi(e_i)(x)}$.
	
	To summarize, we have a canonical isomorphism of $\scr O(X)$-modules
	\begin{align*}
		\Hom_{\scr O_X}(V\otimes\scr O_X,W\otimes\scr O_X)\simeq \Hom_\Cbb(V,W)\otimes\scr O(X).	
	\end{align*}
	\hfill\qedsymbol
\end{eg}

\subsection{Injectivity, surjectivity, isomorphisms}

An $\scr O_X$-module morphism $\varphi:\scr E\rightarrow\scr F$ is called \textbf{injective} resp. \textbf{surjective} if for each $x\in X$ the corresponding stalk morphism $\varphi_x:\scr E_x\rightarrow\scr F_x$ is injective resp. surjective.

\begin{exe}
	Show that $\varphi$ is injective iff $\varphi:\scr E(U)\rightarrow\scr F(U)$ is injective for all open $U\subset X$. Show that $\varphi$ is surjective iff for each $x\in X$ and each section $t\in\scr F$ on a neighborhood $U$ of $X$, by shrinking $U$ to a smaller neighborhood $V\ni x$, we can find $s\in\scr E(V)$ such that $\varphi(s)=t$ when restricted to $V$.
\end{exe}

(Warning: surjectivity does not mean that each $x$ is contained in a neighborhood $U$ such that $\varphi:\scr E(U)\rightarrow\scr F(U)$ is surjective. Thus, surjectivity of sheaves is defined both locally and sectionwisely!)

\begin{rem}\label{lba102}
	Let $\scr E,\scr F$ be presheaves of $\scr O_X$-modules. Suppose that each $\scr E(U)$ is an $\scr O(U)$-submodule of $\scr F(U)$, and the inclusion maps $\scr E(U)\hookrightarrow\scr F(U)$ are compatible with the restriction maps of sheaves. Then there is a natural  morphism $\iota:\scr E\rightarrow\scr F$ such that $\iota_U$ is the inclusion $\scr E(U)\hookrightarrow\scr F(U)$. We say that $\scr E$ is a \textbf{sub-presheaf of $\scr O_X$-modules}  of $\scr F$. If both $\scr E,\scr F$ are sheaves, we say $\scr E$ is an \textbf{$\scr O_X$-submodule} of $\scr F$.
	
	Now suppose $\scr F$ is an $\scr O_X$-modules and $\scr E$ is a sub-presheaf of $\scr O_X$-modules of $\scr F$. Then the sheafification of $\scr E$ can be viewed as an $\scr O_X$-submodule of $\scr F$. Its spaces of sections are all $s\in\scr F(U)$ such that $s_x\in\scr E_x$ for every $x\in U$. \hfill\qedsymbol
\end{rem}

We say that an $\scr O_X$-module morphism $\varphi:\scr E\rightarrow\scr F$ is an \textbf{isomorphism of $\scr O_X$-modules} if it is bijective (i.e. injective+surjective). 
\begin{exe}
	Show that $\varphi$ is an isomorphism if and only if for each open subset $U\subset X$, $\varphi_U:\scr E(U)\rightarrow\scr F(U)$ is an isomorphism of $\scr O(U)$-modules. (Indeed, the only nontrivial part is to show that $\varphi$ being an isomorphism implies the surjectivity of $\varphi_U$. Surprisingly, to prove this part we need the injectivity!) 
\end{exe}

\subsection{Kernals, cokernels, images}

Let $\varphi:\scr E\rightarrow\scr F$ be an $\scr O_X$-module morphism. The \textbf{kernel} $\Ker(\varphi)$ is the presheaf whose space of sections on any open subset $U$ is the kernel of $\varphi:\scr E(U)\rightarrow\scr F$. It is clear that $\Ker(\varphi)$ is a sheaf and is an $\scr O_X$-module. Clearly $\Ker(\varphi)_x$ is the kernel of the stalk map $\varphi:\scr E_x\rightarrow\scr F_x$.

The \textbf{image} $\varphi(\scr E)=\Imag(\varphi)$ is the sheafification of the presheaf whose space of sections on each $U$ is $\varphi(\scr E(U))$.

The \textbf{cokeral} $\Coker(\varphi)$ is the sheafification of the presheaf whose space of sections on each $U$ is $\scr F(U)/\varphi(\scr E(U))$. Equivalently, $\Coker(\varphi)$ is the sheafification of the presheaf whose space of sections on each $U$ is $\scr F(U)/\varphi(\scr E)(U)$. Thus, we also say that $\Coker(\varphi)$ is the \textbf{quotient} of the sheaves $\scr F$ and $\varphi(\scr E)$, and write
\begin{align}
\scr F/\varphi(\scr E)=\coker(\varphi).	
\end{align}

\begin{exe}
	Show that we have natural equivalences
	\begin{gather}
		\varphi(\scr E)_x\simeq\varphi(\scr E_x),\\
		\Coker(\varphi)_x	\simeq \scr F_x/\varphi(\scr E_x).
	\end{gather}
\end{exe}

\subsection{Locally free sheaves}

Let $I$ be an index set. Let $\Cbb^I$ be the direct sum of  $|I|$ copies of $\Cbb$ indexed by elements of $I$. Then $\Cbb^I$ has basis $\{e_i\}_{i\in I}$ where $e_i$ is the vector whose only non-zero component is the $i$-th one, which is $1$.

Let $\scr E$ be an $\scr O_X$-module. A collection of sections $(s_i)_{i\in I}\subset\scr E(X)$ is said to \textbf{generate} (resp. \textbf{generate freely}) $\scr E$ if the natural $\scr O_X$-module $\psi:\Cbb^I\otimes\scr O_X\rightarrow\scr E$ sending each $e_i$ (regarded as a constant section $e_i\otimes 1$) to $s_i$ is surjective (resp. bijective).

Equivalently, $(s_i)_{i\in I}$ generates (resp. generates freely) $\scr E$ iff for each $x\in X$, each $t\in\scr E_x$ can be written as a (resp. unique) $\scr O_{X,x}$-linear combination of the germs $(s_{i,x})_{i\in I}$.

If $U\subset X$ is open, we say $(s_i)_{i\in I}$ generates (resp. freely) $\scr E_U$ if $(s_i|_U)_{i\in I}$ does.

We say that $\scr E$ is \textbf{locally free} if each $x\in X$ is contained in a neighborhood $U$ such that the following equivalent conditions hold:
\begin{itemize}
	\item $\scr E_U$ is generated freely by finitely many sections $s_1,\dots,s_n\in\scr E(U)$. ($s_\blt$ play the role of basis of a vector space.)
	\item $\scr E_U$ is isomorphic to $\Cbb^n\otimes\scr O_U$ for some $n\in\Nbb$. 
\end{itemize}

\begin{rem}\label{lba105}
	It is an important fact that locally free $\scr O_X$-modules are the same as holomorphic vector bundles. Indeed, the sections of vector bundles clearly form a locally free module. Conversely, suppose $\scr E$ is locally free, then we can get a vector bundle whose transition functions are $\psi\circ\varphi^{-1}:W\otimes\scr O_U\xrightarrow{\simeq}W\otimes\scr O_U$ (considered as $\End W$-valued holomorphic functions) where $\varphi,\psi:\scr E\xrightarrow{\simeq} W\otimes\scr O_U$ are trivializations. Equivalently, if $s_1,\dots,s_n$ and $t_1,\dots,t_n$ both generate freely $\scr E_U$, then there is a unique invertible $M_{n\times n}(\Cbb)$-valued holomorphic function $A$ such that $t_i(x)=\sum_j A_{i,j}(x)s_j(x)$. Then $A$ gives a transition function.
\end{rem}

\subsection{Sheaves of morphisms, dual modules}

If $\scr E,\scr F$ are $\scr O_X$-modules, we have a presheaf $\scr G$ whose space of sections on each open $U\subset X$ is $\Hom_{\scr O_U}(\scr E_U,\scr F_U)$. There is an obvious restriction map from $\Hom_{\scr O_U}(\scr E_U,\scr F_U)$ to $\Hom_{\scr O_V}(\scr E_V,\scr F_V)$ if $V\subset U$ is open. $\scr G$ is clearly a presheaf of $\scr O_X$-modules. It is a routine check that $\scr G$ is a sheaf. We denote this sheaf of $\scr O_X$-modules by
\begin{align*}
	\shom_{\scr O_X}(\scr E,\scr F).	
\end{align*}

\begin{exe}
	Find a natural equivalence $\scr F\xrightarrow{\simeq}\shom_{\scr O_X}(\scr O_X,\scr F)$.
\end{exe}

\begin{eg}
	In the setting of Example \ref{lba104}, we have a natural $\scr O_X$-module isomorphism
	\begin{align}
		\shom_{\scr O_X}(V\otimes\scr O_X,W\otimes\scr O_X)\simeq \Hom_\Cbb(V,W)\otimes\scr O_X.	\label{eqa181}
	\end{align}
\end{eg}

We define
\begin{align*}
	\scr E^\vee:=\shom_{\scr O_X}(\scr E,\scr O_X),
\end{align*}
called the \textbf{dual $\scr O_X$-module} of $\scr E$. Then by \eqref{eqa181}, if $\scr E$ is locally free (i.e., a vector bundle), then so is $\scr E^\vee$, and they have the same rank. We regard $\scr E$ as the \textbf{dual vector bundle} of $\scr E$.

\begin{exe}
Let $\scr E$ be an $\scr O_X$-submodule of $\scr F$. Show that $(\scr F/\scr E)^\vee$ is the sheaf whose sections over any open  $U\subset X$ are the $\scr O_U$-module morphisms $\scr F_U\rightarrow\scr O_U$ vanishing on the stalks of $\scr E_U$.
\end{exe}

\begin{cv}
If $U,V\subset X$ are open, $\varphi\in\Hom_{\scr O_V}(\scr E_V,\scr O_V)$ and $s\in\scr E(U)$, we set
\begin{align*}
\bk{\varphi,s}=\varphi(s|_{U\cap V})\qquad (\in\scr O(U\cap V)).	
\end{align*} 
\end{cv}

\begin{rem}
If $\scr E$ is a vector bundle, then the transition functions of $\scr E^\vee$ are the inverses of those of $\scr E$. To see this, choose $s_1,\dots,s_n\in\scr E(U)$ generating freely $\scr E_U$. Then by $\scr E_U\simeq\Cbb^n\otimes\scr O_U$, we can easily find $\wch s_1,\dots,\wch s_n\in\scr E^\vee(U)$ generating freely $\scr E^\vee_U$ such that $\bk{s_j,s_i}$ is the constant section $\delta_{i,j}$. $\wch s_1,\dots,\wch s_n$ are regarded as the dual basis of  $s_1,\dots,s_n$. 

Now, if $t_1,\dots,t_n\in\scr E(U)$ also generates freely $\scr E_U$, then by Rem. \ref{lba105}, the matrix valued holomorphic function $A\in M_{n\times n}\otimes\scr O(U)$ such that $t_i=\sum_j A_{i,j}s_j$ is a transition function of $\scr E$. Let $A^{-1}\in M_{n\times n}\otimes\scr O(U)$ be the function sending $x\in U$ to $A(x)^{-1}$. Then $\wch t_i=\sum_j (A^{-1})_{i,j}\wch s_j$. \hfill\qedsymbol
\end{rem}

\subsection{Fibers}

One can recover the fibers from a locally free sheaf in the following way. Let us consider a general $\scr O_X$-module $\scr E$. For each $x\in X$, let $\fk m_{X,x}$ (or simply $\fk m_x$) \index{mx@$\fk m_{X,x}\equiv\fk m_x$} be the ideal of $\scr O_{X,x}$ consisting of all $s\in\scr O_{X,x}$ whose values at $x$ vanish. Then $\fk m_x\scr E_x$ is an $\scr O_{X,x}$-submodule of $\scr E_x$, and so is the \textbf{fiber} \index{Ex@$\scr E\lvert x=\scr E\lvert_x,s(x)$}
\begin{align}
\scr E|x\equiv\scr E|_x=\frac{\scr E_x}{\fk m_x\scr E_x}.
\end{align} 
where the $\Span_\Cbb$ is suppressed in the notation $\fk m_x\scr E_x$. The equivalence class of $s\in\scr E_x$ in $\scr E|x$ is denoted by $s(x)$, called the value of $s$ on the fiber $\scr E|x$.

If $\varphi:\scr E\rightarrow\scr F$ is an $\scr O_X$-module morphism and $x\in X$, then $\varphi:\scr E_x\rightarrow\scr F_x$ descends to a linear map
\begin{align}
	\varphi:\scr E|x\rightarrow\scr F|x	
\end{align}
since $\varphi(\fk m_x\scr E_x)=\fk m_x\varphi(\scr E_x)\subset\fk m_x\scr F_x$. 

\begin{eg}
Let $U\ni 0$ be an open subset of $\Cbb^m$. Then $\fk m_{U,0}$ is the set of all series $\sum_{n_1,\dots,n_m\in \Nbb}a_{n_1,\dots,n_m}z_1^{n_1}\cdots z_m^{n_m}$ converging absolutely near $0$ such that $a_{0,\dots,0}=0$. Equivalently,
\begin{align*}
\fk m_{\Cbb^m,0}=z_1\scr O_{\Cbb^m,0}+\cdots+z_m\scr O_{\Cbb^m,0}.	
\end{align*}
\end{eg}

\begin{exe}
Let $W$ be a vector space, and let $\scr E=W\otimes\scr O_U$ where $U\subset\Cbb^m$. Let $x\in U$.  Show that the evaluation map
\begin{align}
(W\otimes\scr O_U)_x\rightarrow W, \qquad w\otimes f\mapsto f(x)w.	
\end{align}
descends to an isomorphism of vector spaces $(W\otimes\scr O_X)\big|x\simeq W$.
\end{exe}

\subsection{A criterion on local freeness}

This subsection is needed only in the Sec. \ref{lb153}.

Let $X$ be a complex manifold and $\scr E$ an $\scr O_X$-module. We say that $\scr E$ is of \textbf{finite type} (also called \textbf{finitely generated}) if  each $x\in X$ is contained in a neighborhood $U\subset X$ such that there exist finitely many sections $s_1,\dots,s_n\in\scr E(U)$ generating $\scr E_U$. Equivalently, each $x$ is contained in a neighborhood $U$ such that there is a surjective $\scr O_U$-module morphism $\Cbb^n\otimes\scr O_U\rightarrow\scr E_U$.

Warning: knowing that $\scr E(U)$ is a finitely generated $\scr O(U)$-module is not enough to show that $\scr E_U$ is generated by finitely many elements of $\scr E(U)$.

If $x\in U$ and $s_1,\dots,s_n\in\scr E(U)$ generate $\scr E_U$, then they clearly generate $\scr E_x$, and hence $s_1(x),\dots,s_n(x)$ span the fiber $\scr E|x$. In particular, $\scr E|x$ is finite-dimensional. Conversely, we have:

\begin{pp}[\textbf{Nakayama's lemma}]
Suppose $\scr E$ is of finite type. Choose $x\in X$ and a neighborhood $U\ni x$. Let $s_1,\dots,s_n\in\scr E(U)$ such that $s_1(x),\dots,s_N(x)$ span the fiber $\scr E|x$. Then there exists a neighborhood $V\subset U$ of $x$ such that $s_1|_V,\dots,s_n|_V$ generate $\scr E|_V$.
\end{pp}

\begin{proof}
By shrinking $U$, we may extend the list $s_1,\dots,s_n$ to $s_1,\dots,s_N\in\scr E(U)$ (where $N\geq n$) such that they generate $\scr E_U$. If $N=n$ then there is nothing to prove. 

Suppose $N>n$. Since $s_1(x),\dots,s_n(x)$ span $\scr E|x=\scr E_x/\fk m_x\scr E_x$, every element of $\scr E_x$, and in particular $s_N$, can be written as 
\begin{align*}
s_N=a_1s_1+\cdots+a_ns_n+\sigma\qquad \in\scr E_x
\end{align*}
where $a_1,\dots,a_n\in\Cbb$ and $\sigma\in\fk m_x\scr E_x$. Since $s_1,\dots,s_N$ generate the $\scr O_{X,x}$-module $\scr E_x$, we have $\sigma=f_1s_1+\cdots+f_Ns_N$ in $\scr E_x$ where $f_1,\dots,f_N\in\fk m_x$. So
\begin{align*}
s_N=g_1s_1+\cdots+g_Ns_N
\end{align*} 
in $\scr E_x$ where $g_1,\dots,g_N\in\scr O_{X,x}$ and $g_{n+1}(x)=\cdots=g_N(x)=0$. Since $g_N(x)=0$, $1-g_N$ is invertible in $\scr O_{X,x}$. So
\begin{align*}
s_N=(1-g_N)^{-1}\sum_{i=1}^{N-1}g_is_i
\end{align*}
in $\scr E_x$. So, after shrinking $U$ to a smaller neighborhood of $x$  on which $g_1,\dots,g_N,(1-g_N)^{-1}$ are holomorphic, the above equation holds in  $\scr E(U)$. This shows that $s_N$ is an $\scr O(U)$-linear combination of $s_1,\dots,s_{N-1}$. So $s_1,\dots,s_{N-1}$ generate $\scr E_U$. By repeating this argument, we see that $s_1,\dots,s_n$ generated $\scr E_U$ for a smaller $U$.
\end{proof}

\begin{thm}\label{lba1}
Assume that $\scr E$ is of finite type. Then the \textbf{rank function}
\begin{align}
r:X\rightarrow\Nbb,\qquad x\mapsto r(x)=\dim\scr E|x
\end{align}
is upper semicontinuous. (So $r(x)\geq r(y)$ for all $y$ in a neighborhood of $x$.)  Moreover, if the rank function is locally constant, then $\scr E$ is locally free.
\end{thm}

\begin{proof}
Let $n=r(x)$. Choose $s_1,\dots,s_n\in\scr E(U)$ (where $U\ni x$) such that $s_1(x),\dots,s_n(x)$ form a basis of $\scr E|x$. Then by Nakayama's Lemma, after shrinking $U$, $s_1,\dots,s_n$ generate $\scr E|_U$, and hence span $\scr E|y$ for all $y\in U$. This proves the upper semicontinuity.

Now suppose $r$ is constantly $n$ on $U$. Then, as $s_1(y),\dots,s_n(y)$ span $\scr E|y$, and since $\dim\scr E|y=n$, $s_1(y),\dots,s_n(y)$ are linearly independent. Let us show that $s_1,\dots,s_n$ generate freely $\scr E_U$ by showing that they are $\scr O_U$-linearly independent. Choose any open $V\subset U$ and $f_1,\dots,f_n\in\scr O(V)$ such that $f_1s_1+\cdots+f_ns_n=0$. Then for each $y\in V$, $\sum_{i=1}^n f(y)s_n(y)$ equals $0$ in $\scr E|y$. So $f_1(y)=\cdots=f_n(y)=0$ by the linear independence of  $s_1(y),\dots,s_n(y)$. So $f_1=\cdots=f_n=0$.
\end{proof}

\printindex

\noindent {\small \sc Yau Mathematical Sciences Center, Tsinghua University, Beijing, China.}

\noindent {\textit{E-mail}}: binguimath@gmail.com\qquad bingui@tsinghua.edu.cn
\end{document}